\pdfoutput=1 
\documentclass[10pt]{amsart}
\usepackage{amssymb}

\usepackage[theoremfont]{newtxtext}
\usepackage[scaled=0.95]{roboto} 
\usepackage[varbb, varvw, smallerops]{newtxmath}
\usepackage[scaled=0.87]{inconsolata} 
\usepackage[cal=cm, calscaled=0.95, frak=euler, frakscaled=0.98]{mathalfa}
\usepackage[T1]{fontenc}
\usepackage{comment}

\usepackage{tikz}
\usepackage{tikz-cd}
\usetikzlibrary{arrows}
\tikzcdset{arrow style=tikz,
diagrams={>={Stealth[round,length=4pt,width=6pt,inset=3pt]}}}
\usepackage{graphicx}

\newcommand{\vers}{\textup{Rigid Dualizing Complexes over Commutative Rings and their 
Functorial Properties}} 

\title[\vers]{Rigid Dualizing Complexes over Commutative Rings and their Functorial Properties}

\usepackage{hyperref}
\hypersetup{colorlinks=false}

\author{Mattia Ornaghi}
\address{Ornaghi: Dipartimento di Matematica,
Università degli Studi di Milano, Milano 20133, Italy}
\email{mattia12.ornaghi@gmail.com}

\author{Saurabh Singh}
\address{Singh: Department of  Mathematics,
Ben Gurion University, Be'er Sheva 84105, Israel}
\email{saurabhsingh61@gmail.com}

\author{Amnon Yekutieli}
\address{Yekutieli: Department of  Mathematics,
Ben Gurion University, Be'er Sheva 84105, Israel}
\email{amyekut@math.bgu.ac.il}

\newtheorem{thm}[equation]{Theorem}
\newtheorem{cor}[equation]{Corollary}
\newtheorem{prop}[equation]{Proposition}
\newtheorem{lem}[equation]{Lemma}
\theoremstyle{definition}
\newtheorem{dfn}[equation]{Definition}
\newtheorem{rem}[equation]{Remark}
\newtheorem{exa}[equation]{Example}

\newtheorem{que}[equation]{Question}

\newtheorem{conv}[equation]{Convention}

\newtheorem{setup}[equation]{Setup}

\numberwithin{equation}{section}

\newcommand{\iso}{\xrightarrow{%
\smash{\raisebox{-0.5ex}{\ensuremath{\scriptstyle \simeq  \mspace{2mu}}}}}} 
\newcommand{\xar}{\xrightarrow}

\newcommand{\sub}{\subseteq}
\newcommand{\opn}{\operatorname}
\newcommand{\cat}[1]{\operatorname{\mathsf{#1}}}
\newcommand{\catt}[1]{{\operatorname{\mathsf{#1}}}}

\newcommand{\cd}{\mspace{1.3mu}{\cdot}\mspace{1.3mu}}

\newcommand{\rmitem}[1]{\item[\textrm{(#1)}]}
\newcommand{\mfrak}[1]{\mathfrak{#1}}
\newcommand{\mcal}[1]{\mathcal{#1}}
\newcommand{\msf}[1]{\mathsf{#1}}

\newcommand{\mrm}[1]{\mathrm{#1}}
\newcommand{\mbb}[1]{\mathbb{#1}}
\newcommand{\OO}{\mcal{O}}
\newcommand{\MM}{\mcal{M}}

\newcommand{\KK}{\mcal{K}}

\newcommand{\Ga}{\Gamma}
\newcommand{\si}{\sigma}
\newcommand{\la}{\lambda}

\renewcommand{\th}{\theta}
\newcommand{\om}{\omega}
\newcommand{\al}{\alpha}
\newcommand{\be}{\beta}
\newcommand{\ga}{\gamma}
\newcommand{\ep}{\epsilon}
\newcommand{\ze}{\zeta}
\newcommand{\Om}{\Omega}
\newcommand{\De}{\Delta}
\renewcommand{\a}{\mfrak{a}}

\newcommand{\p}{\mfrak{p}}
\newcommand{\q}{\mfrak{q}}
\newcommand{\m}{\mfrak{m}}

\newcommand{\bk}{\bsym{k}}
\newcommand{\ba}{\bsym{a}}
\newcommand{\bb}{\bsym{b}}
\newcommand{\bc}{\bsym{c}}
\newcommand{\bm}{\bsym{m}}
\newcommand{\bg}{\bsym{g}}

\newcommand{\bt}{\bsym{t}}

\newcommand{\K}{\mathbb{K}}

\newcommand{\Z}{\mathbb{Z}}
\newcommand{\N}{\mathbb{N}}

\newcommand{\Hom}{\mcal{H}om}
\newcommand{\XX}{\mfrak{X}}

\newcommand{\gfrac}[2]{\genfrac{[}{]}{0pt}{}{#1}{#2}}
\newcommand{\smgfrac}[2]{\scalebox{1.1}{$\genfrac{[}{]}{0pt}{1}{#1}%
{\mspace{2mu} #2}$}}
\newcommand{\smfrac}[2]{\scalebox{1.1}%
{$\genfrac{}{}{0.45pt}{1}{#1}{#2 {}^{\vphantom{X}}}$}}
\newcommand{\tup}[1]{\textup{#1}}
\newcommand{\bsym}[1]{\boldsymbol{#1}}
\newcommand{\boplus}{\bigoplus\nolimits}
\newcommand{\ot}{\otimes}
\newcommand{\wtil}[1]{\widetilde{#1}}
\newcommand{\til}[1]{\tilde{#1}}

\renewcommand{\d}{\mathrm{d}}

\newcommand{\bwedge}{\bigwedge\nolimits}

\newcommand{\lb}{\linebreak}
\newcommand{\abs}[1]{\lvert #1 \rvert}

\newcommand{\bmat}[1]{\begin{bmatrix} #1 \end{bmatrix}}
\newcommand{\twoto}{\Rightarrow}
\newcommand{\twoiso}{\stackrel{\simeq\ }{\Longrightarrow}}
\newcommand{\eftover}{\mspace{0.0mu} /_{\mspace{-2mu} \mrm{eft}}\mspace{1.5mu}}
\renewcommand{\over}{\mspace{0mu} / \mspace{-1.5mu}}

\newcommand{\lsp}{\mspace{1.5mu}}
\newcommand{\cupprod}{{\ensuremath{\mspace{2mu} {\smile} \mspace{2mu}}}}

\begin{document}

\begin{abstract}
In this paper we treat Grothendieck Duality for noetherian rings via rigid dualizing complexes.\
This approach, relying on commutative differential graded rings, was introduced by the third author and J.J. Zhang around 2008.\
Despite promising results, the papers \cite{YZ1} and \cite{YZ2} 
suffered from several mistakes.\ Yet the results of these two papers are correct.\\
The aim of this article (together with \cite{Ye4}) 
is to replace and improve \cite{YZ1} and \cite{YZ2}.\
In particular, we prove that every ring $A$, essentially finite type over a regular base ring $\K$, has a unique rigid dualizing complex $R_A$.\ 
The rigid dualizing complexes have strong functorial properties, allowing us to construct the {\em twisted induction pseudofunctor},
which is our ring-theoretic version of the twisted inverse pseudofunctor $f^{!}$ of \cite{RD}.\\
In the upcoming papers \cite{OY}, \cite{Ye6} and \cite{Ye7}, the results here will be made geometric, 
providing a complete theory of rigid residue complexes on schemes, 
and -for the first time- on Deligne-Mumford stacks.
\end{abstract}

\date{\today}



\thanks{The first author is supported by FARE 2018 
HighCaSt, grant no. R18YA3ESPJ.\
The third author is supported by the Israel Science Foundation grant no. 824/12}

\maketitle

\pagestyle{headings}

\tableofcontents

\setcounter{section}{-1}
\section{Introduction}
\label{sec:intro}

Grothendieck Duality (abbreviated GD) has numerous applications across different fields of algebraic and arithmetic geometry, 
including moduli spaces, resolution of singularities, enumerative geometry, etc.\
This pivotal concept was introduced in the 1960s in the celebrated book \cite{RD} by R. Hartshorne, 
which expanded upon prenotes by A. Grothendieck. 
Much later, in 2000s, B. Conrad played a significant role, 
by completing and filling gaps in the proofs of \cite{RD}, and 
also proving the Base Change Theorem, see \cite{Co}.

The intricacy of Grothendieck Duality lies in its interplay 
between its local and global components, between the abstract and concrete formulations, 
and between the formal categorical statements and their geometrical applications.\ 
These interplays are both difficult and captivating.

The extensive efforts dedicated to the study of Grothendieck duality have primarily focused on 
elucidating the dynamics of the local-global relationship and addressing the disparities mentioned earlier.

Historically we can distinguish two main approaches to GD.\
The first one is due to Grothendieck and Hartshorne, see the aforementioned \cite{RD}.\
This can be considered as the most constructive, and it is based on residual dualizing complexes. 
The second one, extremely formal, is due to P. Deligne \cite{De1} and J.L. Verdier \cite{Ve}.\ 
Deligne's idea comes directly from Verdier duality and it passes through Nagata compactification \cite{De2}.\

Despite the importance of the topic and, after a first decade of interesting progress, 
the subsequent thirty years have been characterized by a gradual abandonment of the study of this field.\
There are several reasons for this but, in particular, 
the complexity of the foundations and the numerous highly technical proofs, 
make the whole theory (for most of the mathematicians) too massive and complicated.\

We can distinguish several mathematicians in the last decades 
who continued the study of GD, gaining constant and gradual developments.\ 
We list the three main groups.\

First consists of the third author, with J.J. Zhang and L. Shaul.\
The second group consists of J. Lipman (see \cite{Li}) and his collaborators 
L. Alonso, A. Jerem\'ias, L. Avramov, S. Iyengar, S. Nayak \cite{AILN} and P. Sastry \cite{Sa}.\
The third group is that of A. Neeman \cite{Ne}, collaborating with Lipman.\
The strategy of Lipman can be considered close to that of Deligne, while Neeman's strategy 
is very categorical, and it uses Brown Representability.\ 

We should mention that a revamped interested in GD was brought up very recently in 2019, by D. Clausen and P. Scholze.\ 
They approached to GD in a completely new way, using condensed mathematics, see \cite{Sc}.\

As stated before, a very promising approach was obtained by the third author with J.J. Zhang around 2008,
see \cite{YZ1} and \cite{YZ2}.\
They approached Grothendieck Duality via rigid dualizing complexes, using differential graded methods.\ 
Despite the encouraging results, several major flaws in these papers were found in 2010, see \cite{AILN}.\
Some of these mistakes were fixed by the third author in \cite{Ye4}.\
The aim of this present work is to correct the remaining mistakes in \cite{YZ1} and \cite{YZ2}, 
and to improve their results (especially regarding \'etale ring homomorphisms).\

Moreover, the paper \cite{OY}, which is in preparation, will introduce rigid residue complexes over rings.\
Whereas dualizing complexes live in derived categories, in which morphisms are defined only up to homotopy, 
and cannot be glued (cf. \cite[page 193]{RD}), 
rigid residue complexes are actual complexes and their functoriality is precise, not up to homotopy.\ 
In particular, they form quasi-coherent sheaves in the \'etale topology.\

These articles are part of a bigger project, whose 
final goal is establishing Grothendieck Duality, including global duality for proper maps, 
for Deligne-Mumford stacks, in the papers \cite{Ye6} and \cite{Ye7}.\
We strongly recommend reading the survey \cite{Ye9}, 
where the whole project is described.\

\medskip
Let us now give a preview of the main results of this paper.\

We fix a nonzero finite-dimensional regular noetherian base ring $\K$.\\ 
Our first major result is this:\ Given an essentially finite type (EFT) $\K$-ring $A$, it has a rigid dualizing complex
$(R_A, \rho_{A})$ relative to $\K$, and it is unique, up to a unique 
isomorphism in $\cat{D}(A)_{\mrm{rig} / \K}$, the category of rigid dualizing complexes (see Theorem \ref{thm:1550}).\
This first result allows us to define the {\em rigid autoduality functor} as follows:
\[ \opn{D}_{A}^{\mrm{rig}} : \cat{D}_{\mrm{f}}(A)^{\mrm{op}} \to 
\cat{D}_{\mrm{f}}(A) , \quad
\opn{D}_{A}^{\mrm{rig}} := \opn{RHom}_{A}(-, R_A) \]
which is a controvariant equivalence.\
Our second result is Theorem \ref{prop:1605}. 
Namely: there is a unique normalized pseudofunctor
\[ \opn{TwInd} : \cat{Rng} \eftover \K \to \cat{TrCat} \over \K , \]
called {\em twisted induction}.\ 
To an object $A \in \cat{Rng} \eftover \K$ it assigns the category
$\opn{TwInd}(A) := \cat{D}^{+}_{\mrm{f}}(A) \in \cat{TrCat} \over \K$.\
To a morphism $u:A\to B$ it associates the triangulated functor:
\[ \opn{TwInd}(u) = \opn{TwInd}_{u} := 
\opn{D}_{B}^{\mrm{rig}} \circ \opn{LInd}_u \circ \opn{D}_{A}^{\mrm{rig}}  \]
where $\opn{LInd}_u:=B\ot^{\mrm{L}}_A(\opn{-})$.\

This result is our ring-theoretic version of the twisted inverse pseudofunctor $f\mapsto f^{!}$ from \cite{RD}. 
To see this we provide two "concretizations" of this pseudofunctor:\ 
when $u:A\to B$ is finite, where there is a rigid trace morphism (Proposition \ref{prop:1610}), 
and when $u$ is essentially smooth, where top degree differential forms show up (Proposition \ref{prop:1611}). \
We observe that the pseudofunctor TwInd is constructed purely by algebraic methods.\ 
No geometry is involved, and without global duality, which is the 
key ingredient in all earlier aforementioned approaches to GD.\
To the best of our knowledge the only other purely algebraic construction is 
the one by Scholze and Clausen \cite{Sc}, and so far it works only for finite type $\Z$-rings.\

What follows is an outline of the paper section by section.\
In section \ref{sec:recall-dg} we furnish a concise foundational 
overview of commutative DG rings and their derived categories of DG modules.\
In particular we treat: the restriction and induction functors, semi-free DG rings, 
and nondegenarate backward and forward morphisms.\
We conclude this section with the notions of derived Morita property and derived tensor evaluation.\ 

The main tool we use to prove the uniqueness of the rigid dualizing complex is the \emph{squaring operation}, 
which is introduced in section \ref{sec:squaring}.\
Roughly speaking, the squaring operation is a quadratic endofunctor, resembling derived Hochschild cohomology, 
described in term of pairs of CDG rings.\ 
Given a pair of CDG rings $B/A$ (see Definition \ref{dfn:625}) 
and a DG $B$-module $M$, the squaring of $M$ is denoted by $\opn{Sq}_{A / \K}(M)$.\
A {\em rigidifying isomorphism} for $M$ is an isomorphism $\rho:M\to \opn{Sq}_{A / \K}(M)$ in $\cat{D}(B)$.\ 
A {\em rigid complex} $(M,\rho)$ consist of a complex $M$ and a rigidifying isomorphism $\rho$.
The important property of rigidity is this: if $(M,\rho)$ is a rigid complex and $M$ has the 
derived Morita property (e.g. if $M$ is either a dualizing complex or a tilting complex), 
then the identity is the only automorphism of $(M,\rho)$.\\
We recall that {\em noncommutative dualizing complexes} were introduced in 1992 by the third author \cite{Ye8}.\ 
{\em Noncommutative rigid dualizing complexes} were invented by M. Van den
Bergh in 1997, in his seminal paper \cite{VdB}.\
The idea of importing such concept of rigid dualizing complexes back to
commutative algebra and algebraic geometry is due to third author and J.J. Zhang. 
They added two key features: the functorialty of rigid complexes, and the passage from base field to base ring.
This was done in the papers \cite{YZ1} and \cite{YZ2} from around 2008.\\
The change from base field to base ring required the use of DG ring
resolutions.
Unfortunately we had several serious errors regarding the manipulation of DG
rings, and these affected some constructions and proofs.\
The errors in \cite{YZ1} and \cite{YZ2} were discovered by the authors of \cite{AILN} in
2010, and they also fixed one error. 
Some of the errors were fixed in 2016 \cite{Ye4}, 
and the rest will be fixed in this paper and in the subsequent \cite{OY}. \

Before continuing let us explain briefly the strategy of the proof of Theorem \ref{thm:1550}.\ 
First we take the structure homomorphism 
$\K\to A$.\ 
It factorizes as the composition of: 
a smooth homomorphism $\K\to A_1$, a finite homorphism $A_1\to A_2$ 
and a localization $A_2\to A$.\ 
We note that we have the tautological rigid complex $(\K,\rho^{\mrm{tau}})$ 
over $\K/\K$, cf. Example \ref{exa:675}.\
The challenging part is to understand how essentially smooth and finite homomorphisms induce a rigid complexes, 
by the procedures of {\em induction} and {\em coinduction} respectively.\ 
This will be done in the following sections.\

The main result of section \ref{sec:coinduced} (Theorem \ref{thm:2031}) deals with finite homomorphisms.\ 
Namely given a finite homomorphism $u:B\to C$ between EFT 
$\K$-rings, and a rigid complex $(M,\rho)\in \cat{D}(A)_{\mrm{rig} / \K}$, 
we can construct the {\em{coinduced rigidifying isomorphism}} (Definition \ref{dfn:760}).\

In section \ref{sec:cup-prod} we introduce the cup product for the squaring operation (Theorem \ref{thm:780}). 
In particular we prove that it is an isomorphism under suitable finiteness conditions (Theorem \ref{thm:810}). 
Moreover Definition \ref{dfn:1260} provides the notion of tensor product of rigid complexes.\
Such a tensor product will be used in section \ref{sec:twisted-induced} to construct the twisted induced rigid complexes.\
We conclude this section fixing a few small errors in the paper \cite{Ye4} (see Remark \ref{rem:650})

Section \ref{sec:regular} and \ref{sec:smooth} are of independent interest, beyond their application to the theory of rigid dualizing complexes. 
The main result of section \ref{sec:regular} is the Fundamental Local Isomorphism (Theorem \ref{thm:1080}). 
The section \ref{sec:smooth} contains results on essentially smooth ring homomorphisms.\ 
We make a deeper study of the essentially étale ring homomorphisms in section \ref{sec:etale}, in particular of the diagonal embedding.\

In section \ref{sec:twisted-induced} the following setup is in force:\ 
$A$ is a noetherian ring and $v:B\to C$ is an essentially smooth homomorphism, between essentially finite type $A$-rings, 
of constant differential relative dimension $n$.\ 
We construct the {\em twisted induced rigidifying isomorphism}; see Definition \ref{dfn:1235} and Definition \ref{dfn:1236}.\

In particular, in (most of) section \ref{sec:forward} and \ref{sec:induced-rigidity} 
we consider an essentially étale homomorphism $v:B\to C$ (cf. setup \ref{set:1281}).\
Under this setup, we construct a forward morphism 
$\opn{Sq}_{v / \K}:\opn{Sq}_{B / A}(M)\to\opn{Sq}_{C/A}(N)$ in $\cat{D}(B)$,
for any forward homomorphism $\lambda:M\to N$ in $\cat{D}(B)$ (see Definition \ref{dfn:1290}).\ 
Important properties of this operation are provided by Theorems \ref{thm:1292}, \ref{thm:895} and \ref{thm:1300}.\

In section \ref{sec:induced-rigidity} we endow the complex $N:=C\otimes_BM$ (recall that $C$ is essentially \'etale over $B$) 
with the induced rigidifying isomorphism and we prove that the standard nondegenerate 
forward morphism $q^{\opn{L}}_{v,M}:(M,\rho)\to (N,\sigma)$ is a rigid forward morphism over $v/A$ (Proposition \ref{prop:1700}).
We conclude section \ref{sec:induced-rigidity} with an open question 
(cf. Question \ref{que:1660}) on the construction of induced rigid complex in the previous setup. 
And we provide a positive answer of this question in a particular case (see Appendix \ref{sec:flat-case}).\

In section \ref{sec:dualizing} we apply all the material from the previous sections 
to prove the existence and uniqueness of rigid dualizing complexes (Theorem  \ref{thm:1550}). 
Moreover we discuss the functoriality of rigid dualizing complexes with respect to 
finite ring homomorphisms (Theorem \ref{thm:1551}) and essentially \'etale ring homomorphisms (Theorem \ref{thm:1552}).

In section \ref{sec:twisted} we use the constructions of rigid dualizing complexes 
from the previous section to construct the aforementioned twisted induction pseudofunctor. 
We should point out Remark \ref{rem:1555}, where we discuss the geometric counterpart 
of such a pseudofuntor which we call the \emph{twisted inverse image pseudofunctor} $f\mapsto f^{!}$.\

Section \ref{sec:relative} is about {\em relative rigid dualizing complexes}.\ 
We prove their existence, uniqueness, and some of their functorial properties.

In section \ref{sec:fin-etale} we work out an example of relative rigid dualizing complexes.\ 
We show that, when $A\to B$ is a finite étale homomorphism of noetherian rings, the operator trace 
$\opn{tr}^{\mrm{oper}}_{B / A} : B \to A$ (locally the matrix trace) is a nondegenerate rigid backward morphism.\

The paper ends with Appendix \ref{sec:flat-case}, in which a special case of Question \ref{que:1660} is answered. 
\medskip
\\
{\em Acknowledgments.} 
The authors would like to thank Brian Conrad, Ofer Gabber, Liran Shaul, James Zhang, 
Asaf Yekutieli and Michel Van den Bergh for advice on this work.

\section{Recalling Facts on Derived Categories of DG Modules}
\label{sec:recall-dg}

In this section we review some of the pertinent material on DG algebra and 
derived categories, following the book \cite{Ye5}.
A few new results are also proved here. 
Throughout ``DG'' stands for ``differential graded''.
We are going to ignore set-theoretic issues, as explained in 
\cite[Section 1.1]{Ye5}.
 
\begin{dfn} \label{dfn:615} \mbox{}
A {\em DG ring} is a graded 
ring $A = \bigoplus_{i \in \Z} A^i$, equipped with an additive operator 
$\d_A : A \to A$ of degree $1$ called the {\em differential}, satisfying 
$\d_A \circ \d_A = 0$, and the graded Leibniz rule
\begin{equation} \label{eqn:740}
\d_A(a \cd b) = \d_A(a) \cd b + (-1)^i \cd a \cd \d_A(b)
\end{equation}
for all $a \in A^i$ and $b \in A^j$.
The unit element of $A$ is $1_A$. 
\end{dfn}

We often denote the differential of $A$ by $\d$, for simplification. 
Of course the unit element $1_A \in \opn{Z}^0(A)$, i.e.\ it is a degree $0$ 
cocycle. Most earlier texts would call $A$ a {\em unital associative cochain DG 
algebra}. 
Rings are viewed as DG rings concentrated in degree $0$. 

\begin{dfn} \label{dfn:2040}
Given DG rings $A$ and $B$, a {\em DG ring homomorphism} 
$u : A \to B$ is a graded ring homomorphism (preserving units) that commutes 
with the differentials. 
\end{dfn}

\begin{dfn} \label{dfn:616} 
Let $A$ be a DG ring. 
\begin{enumerate}
\item $A$ is called {\em nonpositive} if $A^i = 0$ for all $i > 0$, or in other 
words if $A =  \bigoplus_{i \leq 0} A^i$.

\item  $A$ is called {\em strongly commutative} if 
$b \cd a = (-1)^{i \cd j} \cd a \cd b$
for all $a \in A^i$ and $b \in A^j$, and $a \cd a = 0$ if $i$ is odd. 

\item $A$ is called a {\em commutative DG ring} if it is both nonpositive and 
strongly commutative.
\end{enumerate}
\end{dfn}

We wish to emphasize one of the conditions in the definition above: 
for an odd degree element $a$ in a commutative DG ring $A$ the equality 
$a^2 = 0$ holds. This condition is absent from many publications 
treating DG theory (perhaps because it is automatic in charactersitic $0$), 
and it is of crucial importance in the arithmetic setting.

\begin{conv} \label{conv:615}
All DG rings in this paper are commutative (Definition \ref{dfn:616}(3)), 
unless explicitly stated otherwise. In particular, all rings are commutative. 
\end{conv}

Sometimes noncommutative rings are unavoidable. For instance, given a 
(commutative) ring $A$ and an $A$-module $M$, the ring 
$\opn{End}_A(M)$ of $A$-linear endomorphisms of $M$ is usually noncommutative. 
When encountering such a situation, we will mention explicitly that the ring 
$\opn{End}_A(M)$ is noncommutative. 

\begin{dfn} \label{dfn:1475}
The category of commutative DG rings, with DG ring homomorphisms, is denoted by 
$\cat{DGRng}$. The category $\cat{Rng}$ of commutative rings 
is a full subcategory of $\cat{DGRng}$. 
\end{dfn}

\begin{rem} \label{rem:670}
The notation in this paper for the category of commutative DG rings,
i.e.\ $\cat{DGRng}$, is not identical to that of \cite{Ye4}, where the notation 
$\catt{DGR}^{\leq 0}_{\mrm{sc}}$ was used, nor to the more expanded notation of 
\cite{Ye5}, which is $\catt{DGRng}^{\leq 0}_{\mrm{sc}}$.
In the current paper we drop the decoration ${}^{\leq 0}_{\mrm{sc}}$ to 
simplify the notation; this is possible thanks to Convention \ref{conv:615}.
Likewise, the category of commutative rings is denoted by 
$\cat{Rng}$ here, unlike the notation $\cat{Rng}_{\mrm{c}}$ used in \cite{Ye5}. 
 \end{rem}

When a DG ring $A$ is specified, there is the category of DG 
$A$-rings, denoted by $\cat{DGRng} \over A$, whose objects are 
pairs $(B, u_B)$, where $B$ is a DG ring, and $u_B : A \to B$ is a DG 
ring homomorphism, called the structural homomorphism of $B$. The morphisms in 
$\cat{DGRng} \over A$ are the ring homomorphisms that respect the 
structural homomorphisms from $A$. If $A$ is a ring, we also have 
the category $\cat{Rng} \over A$ of $A$-rings, and this is a full subcategory 
of $\cat{DGRng} \over A$. Very often we shall keep the structural homomorphisms 
implicit. 

A homomorphism of DG rings $u : A \to B$ is called a {\em quasi-isomorphism} if 
the graded ring homomorphism 
$\opn{H}(u) : \opn{H}(A) \to \opn{H}(B)$ is an isomorphism. 

Let $A$ be a DG ring. A {\em DG $A$-module} is a graded left $A$-module 
$M =  \bigoplus_{i \in \Z} M^i$, equipped with a differential
$\d_M : M \to M$ of degree $1$, satisfying 
$\d_M \circ \lsp \d_M = 0$ and the graded Leibniz rule (like (\ref{eqn:740}), 
but with $m \in M^j$ instead of $b \in A^j$).  
In case $A$ is a ring (i.e.\ $A = A^0$), a DG $A$-module is the same as a 
complex of $A$-modules. 
Since $A$ is commutative, we can make $M$ into a right DG $A$-module, by this 
formula: $m \cd a := (-1)^{i \cd j} \cd a \cd m$ 
for $a \in A^i$ and $m \in M^j$. 

Let $M$ and $N$ be DG $A$-modules. We can form these new DG $A$-modules: 
$M \ot_A N = \bigoplus_{i \in \Z} \, (M \ot_A N)^i$ 
and 
$\opn{Hom}_A(M, N) = \bigoplus_{i \in \Z} \, \opn{Hom}_A(M, N)^i$. 
See \cite[Section 3.3]{Ye5}. 

We shall say {\em DG category} instead of ``DG linear category'', cf.\ 
\cite[Definition 3.1.4]{Ye5}.
The DG category of DG $A$-modules is denoted by $\cat{C}(A)$. 
For $M, N \in \cat{C}(A)$, the DG module of morphisms between them is 
$\opn{Hom}_{\cat{C}(A)}(M, N) = \opn{Hom}_A(M, N)$.
The {\em strict subcategory} of $\cat{C}(A)$ is 
$\cat{C}_{\mrm{str}}(A)$; it has the same objects, but its morphisms are the 
degree $0$ cocycles:
$\opn{Hom}_{\cat{C}_{\mrm{str}}(A)}(M, N) = 
\opn{Z}^0 \bigl( \opn{Hom}_A(M, N) \bigr)$.
In words: the strict morphisms $\phi : M \to N$ are the degree $0$ 
homomorphisms that commute with the differentials. The category 
$\cat{C}_{\mrm{str}}(A)$ is an $A^0$-linear abelian category (recall that $A$ 
is nonpositive). 
If $A$ is a ring, then we also have the abelian category 
$\cat{M}(A) = \cat{Mod}(A)$ of $A$-modules, and there is a fully faithful 
functor $\cat{M}(A) \to \cat{C}_{\mrm{str}}(A)$.
See \cite[Section 3.4]{Ye5}. 

The {\em homotopy category} of DG $A$-modules is the category 
$\cat{K}(A)$, with the same objects as $\cat{C}(A)$, and whose morphisms are 
the homotopy classes of strict homomorphisms, i.e.\ 
$\opn{Hom}_{\cat{K}(A)}(M, N) = \opn{H}^0 \bigl( \opn{Hom}_A(M, N) \bigr)$.
There is an $A^0$-linear full functor 
$\opn{P} : \cat{C}_{\mrm{str}}(A) \to \cat{K}(A)$,
which is the identity of objects. The category $\cat{K}(A)$ is 
an $\opn{H}^0(A)$-linear triangulated category. 
See \cite[Section 5.4]{Ye5}.

The {\em derived category} of DG $A$-modules is the category 
$\cat{D}(A)$, obtained from $\cat{K}(A)$ by formally inverting the 
quasi-isomorphisms. This is also an $\opn{H}^0(A)$-linear triangulated 
category, and there is an $\opn{H}^0(A)$-linear triangulated functor 
$\opn{Q} : \cat{K}(A) \to \cat{D}(A)$, which is the identity of objects. 
If $A$ is a ring, then there is an equivalence of categories 
$\cat{M}(A) \to \cat{D}^0(A)$, where the latter is the full subcategory of 
$\cat{D}(A)$ on the complexes whose cohomology is concentrated in degree $0$. 
See \cite[Chapter 7]{Ye5}.

Let $A$ be a DG ring. 
A DG $A$-module $N$ is called {\em acyclic} if $\opn{H}^q(N) = 0$ for all $q$. 
A DG $A$-module $P$ is called {\em K-projective} if for every acyclic DG 
$A$-module $N$, the DG $A$-module 
$\opn{Hom}_A(P, N)$ is acyclic.   
A DG $A$-module $I$ is called {\em K-injective} if for every acyclic DG 
$A$-module $N$, the DG $A$-module 
$\opn{Hom}_A(N, I)$ is acyclic.   
A DG $A$-module $P$ is called {\em K-flat} if for every acyclic DG 
$A$-module $N$, the DG $A$-module $P \ot_A N$ is acyclic. 
If $P$ is K-projective then it is also K-flat. 
Every DG $A$-module $M$ has K-projective resolution
$\rho : P \to M$, i.e.\ a quasi-isomorphism in 
$\cat{C}_{\mrm{str}}(A)$ from a K-projective DG module $P$. 
This K-projective resolution is unique up to homotopy, in the sense that 
given another K-projective resolution $\rho' : P' \to M$, there is a homotopy 
equivalence $\chi : P \to P'$ such that $\rho' \circ \chi$ is homotopic to 
$\rho$, and moreover $\chi$ is unique up to homotopy. 
Likewise, $M$ has a K-injective resolution $M \to I$,
with the analogous uniqueness up to homotopy.
See \cite[Chapter 11]{Ye5}. 

The resolutions above allow us to derive functors, as explained in 
\cite[Chapters 8-10]{Ye5}. Let us mention the two most 
important derived functors. First there is the {\em right derived Hom bifunctor}
\[ \opn{RHom}_A : \cat{D}(A)^{\mrm{op}} \times \cat{D}(A) \to \cat{D}(A) . \]
For a pair of DG modules $M, N \in \cat{C}(A)$ there is a bifunctorial morphism 
\begin{equation} \label{eqn:618}
\eta^{\mrm{R}}_{M, N} : \opn{Hom}_A(M, N) \to  \opn{RHom}_A(M, N) \
\end{equation}
in $\cat{D}(A)$, and it is an isomorphism if $M$ is K-projective or $N$ is 
K-injective. There is a bifunctorial $\opn{H}^0(A)$-linear isomorphism 
$\opn{H}^0 \bigl( \opn{RHom}_A(M, N) \bigr)  \iso 
\opn{Hom}_{\cat{D}(A)}(M, N)$. 

The second is the {\em left derived tensor bifunctor}
\[ (- \ot^{\mrm{L}}_{A} -) : \cat{D}(A) \times \cat{D}(A) \to 
\cat{D}(A) . \]
For a pair of DG modules $M, N \in \cat{C}(A)$ there is a bifunctorial morphism 
\begin{equation} \label{eqn:621}
\eta^{\mrm{L}}_{M, N} : M \ot^{\mrm{L}}_{A} N \to M \ot_A N  
\end{equation}
in $\cat{D}(A)$, and it is an isomorphism if either $M$ or $N$ 
is K-flat. For details see \cite[Sections 12.2 and 12.3]{Ye5}. 

Given a homomorphism of DG rings $u : A \to B$, there is the {\em restriction 
functor} 
$\opn{Rest}_u = \opn{Rest}_{B / A} : \cat{C}(B) \to \cat{C}(A)$,
which is an exact DG functor, and hence it induces a triangulated functor
\begin{equation} \label{eqn:1870}
\opn{Rest}_u = \opn{Rest}_{B / A} : \cat{D}(B) \to \cat{D}(A) . 
\end{equation}
The {\em induction functor} 
$\opn{Ind}_u = \opn{Ind}_{B / A} : \cat{C}(A) \to \cat{C}(B)$,
$\opn{Ind}_{B / A} := B \ot_A (-)$,  
is DG functor, which is not exact (unless $B$ is K-flat as a DG $A$-module), 
and its left derived functor is the {\em left derived induction functor} 
\begin{equation} \label{eqn:1871}
\opn{LInd}_u = \opn{LInd}_{B / A} : \cat{D}(A) \to \cat{D}(B) , \quad
\opn{LInd}_{B / A} := B \ot^{\mrm{L}}_{A} (-) .
\end{equation}
The {\em coinduction functor} 
$\opn{CInd}_u = \opn{CInd}_{B / A} : \cat{C}(A) \to \cat{C}(B)$,
$\opn{CInd}_{B / A} := \opn{Hom}_A(B, -)$,
is DG functor, which is not exact (unless $B$ is K-projective as a DG 
$A$-module), and its right derived functor is the {\em right derived 
coinduction functor} 
\begin{equation} \label{eqn:1872}
\opn{RCInd}_u = \opn{RCInd}_{B / A} : \cat{D}(A) \to \cat{D}(B) , \quad
\opn{RCInd}_{B / A} := \opn{RHom}_A(B, -) . 
\end{equation}
The DG functors above are related as follows: $\opn{Ind}_{u}$ is left adjoint 
to $\opn{Rest}_{u}$, and $\opn{CInd}_{u}$ is right adjoint to 
$\opn{Rest}_{u}$. Likewise for the triangulated functors above. 
Moreover, for $M \in \cat{D}(A)$ and $N \in \cat{D}(B)$
there is an isomorphism 
\begin{equation} \label{eqn:1480}
\opn{RHom}_A(M, N) \cong \opn{RHom}_B(B \ot^{\mrm{L}}_{A} M, N) 
\end{equation}
in $\cat{D}(B)$, and it is functorial in $M$ and $N$. We call this the {\em 
adjunction isomorphism} for the DG ring homomorphism $A \to B$.
It is proved just like \cite[Proposition 12.10.12]{Ye5} -- we choose a 
K-projective resolution $P \to M$ over $A$; then, on the one hand 
$B \ot_A P$ is K-projective over $B$, and on the other hand there is the 
adjunction isomorphism 
$\opn{Hom}_A(P, N) \cong \opn{Hom}_B(B \ot_{A} P, N)$
in $\cat{C}_{\mrm{str}}(B)$. 

In case the DG ring homomorphism $u : A \to B$ is a quasi-isomorphism, the 
functors $\opn{Rest}_{B / A}$, $\opn{LInd}_{B / A}$ and 
$\opn{RCInd}_{B / A}$ are equivalences. Moreover, there are bifunctorial 
isomorphisms 
\begin{equation} \label{eqn:703}
\opn{Rest}_{B / A} \bigl( \opn{RHom}_B(M, N) \bigr) \iso 
\opn{RHom}_A \bigl( \opn{Rest}_{B / A}(M), \opn{Rest}_{B / A}(N) \bigr)
\end{equation}
and 
\begin{equation} \label{eqn:719}
\opn{Rest}_{B / A}(M) \ot^{\mrm{L}}_{A} \opn{Rest}_{B / A}(N) \iso 
\opn{Rest}_{B / A} (M \ot^{\mrm{L}}_{B} N) 
\end{equation}
in $\cat{D}(A)$, for every $M, N \in \cat{D}(B)$. See 
\cite[Proposition 12.10.4]{Ye5}. 

Let $A$ be a DG ring. A DG $A$-ring $\til{B}$ is called a {\em semi-free DG 
$A$-ring} (in the commutative sense) if there is an isomorphism of graded 
$A^{\natural}$-rings 
$\til{B}^{\natural} \cong A^{\natural}[X]$
for some nonpositive graded set $X$. Here $A^{\natural}$ and $B^{\natural}$ are 
the graded rings obtained by forgetting the differentials, and 
$A^{\natural}[X]$ is the commutative graded polynomial ring on the graded set 
$X$. See \cite[Example 3.1.23]{Ye5}. 

Now let $B$ be a DG $A$-ring. A {\em semi-free} (resp.\ {\em K-projective}, 
{\em K-flat}) {\em DG $A$-ring resolution of $B$} is a surjective 
quasi-isomorphism $\til{B} \to B$ 
in $\cat{DGRng} \over A$, s.t.\ $\til{B}$ is a semi-free DG $A$-ring
(resp.\ K-projective DG $A$-module, K-flat DG $A$-module). Semi-free DG ring 
resolutions always exist; they are also K-projective and K-flat. 
See \cite[Remark 12.8.22]{Ye5} and \cite[Theorem 3.21$\tup{(1)}$]{Ye4}.

Consider a DG ring homomorphism $u : A \to B$. Suppose 
$M \in \cat{C}(A)$, $N \in \cat{C}(B)$, and 
$\th : \opn{Rest}_{B / A}(N) \to M$
is a homomorphism in $\cat{C}_{\mrm{str}}(A)$. Then $\th$ is called a {\em 
backward homomorphism over $u$}, or a {\em backward homomorphism over $B / A$}. 
We shall often omit the restriction functor, and just talk about a 
backward homomorphism $\th : N \to M$ in $\cat{C}_{\mrm{str}}(A)$ over $u$ (or 
over $B / A$). Similarly there are backward morphisms $\th : N \to M$ in 
$\cat{D}(A)$ over $u$. 
For each $M \in \cat{C}(A)$ there is the {\em standard backward homomorphism},
also called {\em evaluation at $1$}, which is the backward homomorphism
\begin{equation} \label{eqn:2040}
\opn{tr}^{}_{u, M} : \opn{Hom}_A(B, M) \to M, \ 
\opn{tr}^{}_{u, M}(\th) := \th(1_B) 
\end{equation}
in $\cat{C}_{\mrm{str}}(A)$ over $u$. 
It has a derived variant
\begin{equation} \label{eqn:655}
\opn{tr}^{\mrm{R}}_{u, M} : \opn{RHom}_A(B, M) = \opn{RCInd}_u(M) \to M , 
\end{equation}
called the {\em standard derived backward morphism}, 
which is a backward morphism in $\cat{D}(A)$ over $u$. The morphism 
$\opn{tr}^{\mrm{R}}_{u, M}$ is
functorial in $M \in \cat{D}(A)$, and when $M = I$ is a K-injective DG 
$A$-module there is equality 
$\opn{tr}^{\mrm{R}}_{u, I} = \opn{tr}_{u, I}$. 

Due to the adjunction between $\opn{Rest}_{u}$ and $\opn{CInd}_{u}$,
to a backward homomorphism $\th : N \to M$ in $\cat{C}_{\mrm{str}}(A)$ over $u$
we associate the unique homomorphism 
$\opn{badj}_{u, M, N}(\th) : N \to \opn{Hom}_{A}(B, M)$ in 
$\cat{C}_{\mrm{str}}(B)$, such that 
$\opn{tr}_{u, M} \circ \opn{badj}_{u, M, N}(\th) = \th$. 
In the derived context, given any backward morphism $\th : N \to M$ in 
$\cat{D}(A)$ over $u$, there is a unique morphism
\begin{equation} \label{eqn:745}
\opn{badj}^{\mrm{R}}_{u, M, N}(\th) : N \to \opn{RHom}_A(B, M)
\end{equation}
in $\cat{D}(B)$, such that 
$\opn{tr}^{\mrm{R}}_{u, M} \circ 
\opn{badj}^{\mrm{R}}_{u, M, N}(\th) = \th$.
See the commutative diagram (\ref{eqn:1873}).

\begin{dfn} \label{dfn:2041}
Let $u : A \to B$ be a DG ring homomorphism, let 
$M \in \cat{D}(A)$, let $N \in \cat{D}(B)$, and let $\th : N \to M$
be a backward morphism in $\cat{D}(A)$ over $u$. 
The backward morphism $\th$ is called a {\em nondegenerate 
backward morphism} if the corresponding morphism 
$\opn{badj}^{\mrm{R}}_{u, M, N}(\th)$, see formula (\ref{eqn:745}), is an 
isomorphism in $\cat{D}(B)$. 
\end{dfn}

Note that the standard derived backward morphism 
$\opn{tr}^{\mrm{R}}_{u, M}$
from formula (\ref{eqn:655}) is nondegenerate; indeed, the corresponding 
morphism in $\cat{D}(B)$ is 
$\opn{id}_{N}$, where $N := \opn{RHom}_A(B, M)$. 
For this reason we refer to $\opn{tr}^{\mrm{R}}_{u, M}$ as the 
{\em standard nondegenerate derived backward morphism}. 

The category $\cat{D}(A)$ is $A^0$-linear, and $\cat{D}(B)$ is $B^0$-linear.
Given a backward morphism $\th : N \to M$ over $u$ and an element 
$b \in B^0$, we define the backward morphism 
$b \cd \th := \th \circ (b \cd \opn{id}_N) : N \to  M$. 
In this way $\opn{Hom}_{\cat{D}(A)}(N, M)$ becomes a $B^0$-module. 
Backward adjunction fits into this commutative diagram of $B^0$-modules:  
\begin{equation} \label{eqn:1873}
\begin{tikzcd} [column sep = 10ex, row sep = 6ex] 
\opn{Hom}_{\cat{D}(A)}(N, M)
\ar[dr, "{\opn{id}}"']
\ar[r, "{\opn{badj}^{\mrm{R}}_{u, M, N}}", "{\simeq}"']
&
\opn{Hom}_{\cat{D}(B)} \bigl(N, \opn{RCInd}_u(M) \bigr)
\ar[d, "{\opn{Hom}_{u}(\opn{id}, \opn{tr}^{\mrm{R}}_{u, M})}", "{\simeq}"']
\\
&
\opn{Hom}_{\cat{D}(A)}(N, M)
\end{tikzcd} 
\end{equation}

\begin{prop} \label{prop:2045}
Let $u : A \to B$ be a DG ring homomorphism, and for $i = 1, 2$ 
let $M_i \in \cat{D}(A)$, $N_i \in \cat{D}(B)$, and let  $\th_i : N_i \to M_i$
be nondegenerate backward morphisms in $\cat{D}(A)$ over $u$. 
\begin{enumerate}
\item Given a morphism $\phi : M_1 \to M_2$ in $\cat{D}(A)$, there exists a 
unique morphism $\psi : N_1 \to N_2$ in $\cat{D}(B)$, such that 
$\th_2 \circ \psi = \phi \circ \th_1$. 

\item If $\phi$ is an isomorphism in $\cat{D}(A)$, then 
$\psi$ is an isomorphism in $\cat{D}(B)$.
\end{enumerate} 
\end{prop}

The morphisms in the proposition are described in the first commutative diagram 
below. 
\begin{equation} \label{eqn:2045}
\begin{tikzcd} [column sep = 7ex, row sep = 6ex]
N_1
\ar[r, dashed, "{\psi}"]
\ar[d, "{\th_1}"']
&
N_2
\ar[d, "{\th_2}"]
\\
M_1
\ar[r, "{\phi}"]
&
M_2
\end{tikzcd} 
\qquad \qquad 
\begin{tikzcd} [column sep = 10ex, row sep = 5ex]
\opn{RCInd}_u(M_1)
\ar[r, dashed, "{\opn{RCInd}_u(\phi)}"]
\ar[d, "{\opn{tr}^{\mrm{R}}_{u, M_1}}"']
&
\opn{RCInd}_u(M_2)
\ar[d, "{\opn{tr}^{\mrm{R}}_{u, M_2}}"]
\\
M_1
\ar[r, "{\phi}"]
&
M_2
\end{tikzcd}
\end{equation}

\begin{proof}
Since $\th_i$ is a nondegenerate backward morphism, we can replace 
$N_i$ by the DG module $\opn{RCInd}_u(M_i) = \opn{RHom}_A(B, M_i)$, and then we 
can replace $\th_i$ by the standard nondegenerate backward morphism 
$\opn{tr}^{\mrm{R}}_{u, M_i}$.
Thus the first diagram in (\ref{eqn:2045}) can be replaced by the second one. 
The morphism $\opn{RCInd}_u(\phi)$ makes this second diagram commutative; and 
it is the only such morphism, because of adjunction. Since 
$\opn{RCInd}_u$ is a functor, if $\phi$ is an isomorphism then so is 
$\opn{RCInd}_u(\phi)$. 
\end{proof}

A morphism $\la : M \to \opn{Rest}_{B / A}(N)$ in $\cat{C}_{\mrm{str}}(A)$
is called a {\em forward homomorphism over $u$}, or a {\em forward homomorphism 
over $B / A$}. We shall usually omit the restriction 
functor, and just talk about a forward homomorphism $\la : M \to N$ in 
$\cat{C}_{\mrm{str}}(A)$ over $u$. There are also forward morphisms 
$\la : M \to N$ in $\cat{D}(A)$ over $u$. 
For each $M \in \cat{C}(A)$ there is the {\em standard forward homomorphism}
\begin{equation} \label{eqn:1476}
\opn{q}^{}_{u, M} : M \to B \ot_A M, \ m \mapsto 1_B \ot m ,  
\end{equation}
which is a forward homomorphism in $\cat{C}_{\mrm{str}}(A)$ over $u$. 
It has a derived variant 
\begin{equation} \label{eqn:656}
\opn{q}^{\mrm{L}}_{u, M} : M \to B \ot^{\mrm{L}}_{A} M , 
\end{equation}
called the {\em standard derived forward morphism}, 
which is a forward morphism in $\cat{D}(A)$ over $u$. 
The morphism $\opn{q}^{\mrm{L}}_{u, M}$ is functorial in $M$, and when 
$M = P$ is K-flat over $A$, then 
$\opn{q}^{\mrm{L}}_{u, P} = \opn{q}_{u, P}$. 

By the adjunction between $\opn{Rest}_{u}$ and $\opn{Ind}_{u}$, given any 
forward morphism $\la : M \to N$ in $\cat{C}_{\mrm{str}}(A)$ over $u$, there is 
a unique homomorphism 
$\opn{fadj}_{u, M, N}(\la) : B \ot_{A} M \to N$ in $\cat{C}_{\mrm{str}}(B)$
such that 
$\opn{fadj}_{u, M, N}(\la) \circ \opn{q}_{u, M} = \la$. 
 There is a derived variant: for a forward morphism $\la : M \to N$ in 
$\cat{D}(A)$ over $u$, there is a unique morphism 
\begin{equation} \label{eqn:741}
\opn{fadj}^{\mrm{L}}_{u, M, N}(\la) : B \ot^{\mrm{L}}_{A} M \to N
\end{equation}
in $\cat{D}(B)$, such that 
$\opn{fadj}^{\mrm{L}}_{u, M, N}(\la) \circ \opn{q}^{\mrm{L}}_{u, M} = \la$. 
See the commutative diagram (\ref{eqn:1874}).

\begin{dfn} \label{dfn:2045}
Let $u : A \to B$ be a DG ring homomorphism, let 
$M \in \cat{D}(A)$, let $N \in \cat{D}(B)$, and let $\la : M \to N$
be a forward morphism in $\cat{D}(A)$ over $u$. 
The forward morphism $\la$ is called a {\em nondegenerate 
forward morphism} if the corresponding morphism 
$\opn{fadj}^{\mrm{L}}_{u, M, N}(\la)$ from equation (\ref{eqn:741})
is an isomorphism in $\cat{D}(B)$. 
\end{dfn}

Observe that the standard derived forward morphism 
$\opn{q}^{\mrm{L}}_{u, M}$
from formula (\ref{eqn:656}) is nondegenerate; indeed, the corresponding 
morphism in $\cat{D}(B)$ is 
$\opn{id}_{N}$, where $N := B \ot^{\mrm{L}}_{A} M$. 
For this reason we refer to $\opn{q}^{\mrm{L}}_{u, M}$ as the 
{\em standard nondegenerate derived forward morphism}. 

Given a forward morphism $\la : M \to N$ over $u$ and an element 
$b \in B^0$, we define the forward morphism 
$b \cd \la :=  (b \cd \opn{id}_N) \circ \la : M \to N$. 
In this way $\opn{Hom}_{\cat{D}(A)}(M, N)$ becomes a $B^0$-module. 
Forward adjunction fits into this commutative diagram of $B^0$-modules:  
\begin{equation} \label{eqn:1874}
\begin{tikzcd} [column sep = 10ex, row sep = 6ex] 
\opn{Hom}_{\cat{D}(A)}(M, N)
\ar[dr, "{\opn{id}}"']
\ar[r, "{\opn{fadj}^{\mrm{L}}_{u, M, N}}", "{\simeq}"']
&
\opn{Hom}_{\cat{D}(B)} \bigl( \opn{LInd}_u(M) , N) \bigr)
\ar[d, "{\opn{Hom}_{u}(\opn{q}^{\mrm{L}}_{u, M}, \opn{id})}", "{\simeq}"']
\\
&
\opn{Hom}_{\cat{D}(A)}(M, N) 
\end{tikzcd} 
\end{equation}

\begin{prop} \label{prop:2046}
Let $u : A \to B$ be a DG ring homomorphism, and for $i = 1, 2$  
let $M_i \in \cat{D}(A)$, $N_i \in \cat{D}(B)$, and let  $\la_i : M_i \to N_i$
be nondegenerate forward morphisms in $\cat{D}(A)$ over $u$. 
\begin{enumerate}
\item Given a morphism $\phi : M_1 \to M_2$ in $\cat{D}(A)$, there exists a 
unique morphism $\psi : N_1 \to N_2$ in $\cat{D}(B)$, such that 
$\la_2 \circ \phi = \psi \circ \la_1$. 

\item If $\phi$ is an isomorphism in $\cat{D}(A)$, then 
$\psi$ is an isomorphism in $\cat{D}(B)$.
\end{enumerate} 
\end{prop}

The morphisms in the proposition are described in the first commutative diagram 
below. 
\begin{equation} \label{eqn:2046} 
\begin{tikzcd} [column sep = 7ex, row sep = 6ex]
M_1
\ar[r, "{\phi}"]
\ar[d, "{\la_1}"']
&
M_2
\ar[d, "{\la_2}"]
\\
N_1
\ar[r, dashed, "{\psi}"]
&
N_2
\end{tikzcd} 
\qquad \qquad 
\begin{tikzcd} [column sep = 10ex, row sep = 5ex]
M_1
\ar[r, "{\phi}"]
\ar[d, "{\opn{q}^{\mrm{L}}_{u, M_1}}"']
&
M_2
\ar[d, "{\opn{q}^{\mrm{L}}_{u, M_2}}"]
\\
\opn{LInd}_u(M_1)
\ar[r, dashed, "{\opn{LInd}_u(\phi)}"]
&
\opn{LInd}_u(M_2)
\end{tikzcd}
\end{equation}

\begin{proof}
The proof is analogous to that of Proposition \ref{prop:2045}.
The second diagram in (\ref{eqn:2046}) is the analogue of the second diagram in 
(\ref{eqn:2045}). 
\end{proof}

\begin{dfn} \label{dfn:676}
Let $A$ be a ring and $M \in \cat{D}(A)$. We say that $M$ has the {\em derived 
Morita property} over $A$ if the derived homothety morphism 
$\opn{hm}^{\mrm{R}}_M : A \to \opn{RHom}_A(M, M)$ 
in $\cat{D}(A)$ is an isomorphism. See 
\cite[Definitions 13.1.2 and 13.1.5]{Ye5}. 
\end{dfn}

Some texts use the name {\em semi-dualizing complex} for a complex $M$ that 
has the derived Morita property. See \cite[Remark 13.1.8]{Ye5} for a discussion 
of these names. 

\begin{prop} \label{prop:730}
Let $A$ be a nonzero ring.
The following conditions are equivalent for $M \in \cat{D}(A)$.
\begin{itemize}
\rmitem{i} $M$ has the  derived Morita property.

\rmitem{ii} The homothety ring homomorphism 
$A \to \opn{End}_{\cat{D}(A)}(M)$ is an isomorphism, and 
$\opn{Hom}_{\cat{D}(A)}(M, M[p]) = 0$ for all $p \neq 0$. 

\rmitem{iii} The $A$-module
$\opn{H}^0(\opn{RHom}_A(M, M))$ is free of rank $1$, with basis the element 
$\opn{id}_M$, and $\opn{H}^p(\opn{RHom}_A(M, M)) = 0$ for all $p \neq 0$. 
\end{itemize}
\end{prop}

Observe that a priori the ring $\opn{End}_{\cat{D}(A)}(M)$ in condition (ii) 
could be noncommutative. 

\begin{proof}
This is \cite[Proposition 13.1.6]{Ye5}.
\end{proof}

We now present two variations of Proposition \ref{prop:730}.

\begin{prop} \label{prop:1871}
Let $u : A \to B$ be a ring homomorphism, let $M \in \cat{D}(A)$, and
let $N \in \cat{D}(B)$.
Assume that $N$ has the derived Morita property over $B$, and that $B \neq 0$.
\begin{enumerate}
\item Let $\th : N \to M$ be a nondegenerate backward  morphism in $\cat{D}(A)$ 
over $u$. Then the $B$-module $\opn{Hom}_{\cat{D}(A)}(N, M)$ is free of rank 
$1$, with basis the backward morphism $\th$, and 
$\opn{Hom}_{\cat{D}(A)} \bigl( N, M[p] \bigr) = 0$ for every  $p \neq 0$.

\item Let $\la : M \to N$ be a nondegenerate forward 
morphism in $\cat{D}(A)$ over $u$. Then the $B$-module 
$\opn{Hom}_{\cat{D}(A)}(M, N)$ is free of rank $1$, with basis the forward 
morphism $\la$, and 
$\opn{Hom}_{\cat{D}(A)} \bigl( M, N[p] \bigr) = 0$
for every  $p \neq 0$.
\end{enumerate}
\end{prop}

\begin{proof} 
(1) We have an isomorphism of $B$-modules 
\[ \opn{badj}^{\mrm{R}}_{u, M, N} : \opn{Hom}_{\cat{D}(A)}(N, M) \iso 
\opn{Hom}_{\cat{D}(B)} \bigl( N, \opn{RCInd}_u(M) \bigr) . \]
Because $\th$ is a nondegenerate backward morphism, the morphism 
$\opn{badj}^{\mrm{R}}_{u, M, N}(\th) : N \to \opn{RCInd}_u(M)$
in $\cat{D}(B)$ is an isomorphism. Thus we get a $B$-module isomorphism \lb 
$\opn{Hom}_{\cat{D}(A)}(N, M) \iso \opn{Hom}_{\cat{D}(B)}(N, N)$,
and it sends $\th \mapsto \opn{id}_N$. Now we use Proposition \ref{prop:730}
(for $B$ and $N$). 

The assertion for $p \neq 0$ is proved similarly, and we leave it to 
the reader.

\medskip \noindent
(2) Forward adjunction is an isomorphism of $B$-modules 
\[ \opn{fadj}^{\mrm{L}}_{u, M, N} : \opn{Hom}_{\cat{D}(A)}(M, N) \iso 
\opn{Hom}_{\cat{D}(B)} \bigl( \opn{LInd}_u(M), N \bigr) , \]
see formula (\ref{eqn:741}). 
Because $\la$ is a nondegenerate forward morphism, the morphism 
$\opn{fadj}^{\mrm{L}}_{u, M, N}(\la) :  \opn{LInd}_u(M) \to N$
in $\cat{D}(B)$ is an isomorphism. Thus we get a $B$-module isomorphism 
$\opn{Hom}_{\cat{D}(A)}(M, N) \iso \opn{Hom}_{\cat{D}(B)}(N, N)$,
and it sends $\la \mapsto \opn{id}_N$. Now we use Proposition \ref{prop:730}
(for $B$ and $N$).

The assertion for $p \neq 0$ is proved similarly.
\end{proof}

If $A = B$ and $u = \opn{id}_A$, then backward and forward morphisms over $u$ 
are just morphisms in $\cat{C}_{\mrm{str}}(A)$ or $\cat{D}(A)$, respectively; 
they are nondegenerate iff they are 
isomorphisms. This material is discussed in detail in \cite[Section 12.6]{Ye5}.

If $B$ and $C$ are DG $A$-rings, then so is $B \ot_A C$. 
Here are several results involving tensor products of DG rings, which we shall 
need later in the paper. Their routine verification is left to the reader. 

Let $A \to B_i \xar{u_i} C_i$ be DG ring homomorphisms, for $i = 1, 2$, and 
let $M_i \in \cat{C}(B_i)$ and $N_i \in \cat{C}(C_i)$ be DG modules. 
There is a unique isomorphism 
\begin{equation} \label{eqn:1485}
(M_1 \ot_A M_2) \ot_{B_1 \ot_A B_1} (N_1 \ot_A N_2) \iso 
(M_1 \ot_{B_1} N_2) \ot_A (M_2 \ot_{B_2} N_2)
\end{equation}
in $\cat{C}_{\mrm{str}}(C_1 \ot_A C_2)$, which on pure tensors, 
for $m_i \in M^{j_i}$ and $n_i \in N^{k_i}$, has the formula 
\[ (m_1 \ot m_2) \ot (n_1 \ot n_2) \mapsto 
(-1)^{j_2 \cd k_1} \cd (m_1 \ot n_1) \ot (m_2 \ot n_2) . \]
In the special case when $B_1 = B_2$, i.e.\ when we are given 
DG ring homomorphisms $A \to B \xar{u_i} C_i$, and DG modules
$M \in \cat{C}(B)$ and $N_i \in \cat{C}(C_i)$, there is a unique isomorphism 
\begin{equation} \label{eqn:1486}
M \ot_{B \ot_A B} (N_1 \ot_A N_2) \iso N_1 \ot_B M \ot_B N_1
\end{equation}
in $\cat{C}_{\mrm{str}}(C_1 \ot_A C_2)$, 
which on pure tensors, for $m \in M^j$ and $n_i \in N^{k_i}$,
has the formula 
\[ m \ot (n_1 \ot n_2) \mapsto (-1)^{j \cd k_1} \cd n_1 \ot m \ot n_2 . \]
In this special case, there is a unique isomorphism
\begin{equation} \label{eqn:1487}
B \ot_{B \ot_A B} (C_1 \ot_A C_2) \iso C_1 \ot_B C_2
\end{equation}
in $\cat{DGRng} \over (C_1 \ot_A C_2)$, which on pure tensors, for
$b \in B^j$ and $c_i \in C^{k_i}$, has the formula 
\[ b \ot (c_1 \ot c_2) \mapsto (b \cd c_1) \ot c_2 = 
(-1)^{j \cd k_1} \cd c_1 \ot (b \cd c_2) . \]

In the paragraph above the DG $A$-rings $B_1$ and $B_2$ were distinct, and of 
course so were the DG modules $M_1$ and $M_2$. 
Several times later in the paper it would be advantageous to make artificial 
distinctions, as explained in the following notational convention. 

\begin{conv} \label{conv:1295}
We shall often make a notational distinction between the first and the second 
occurrences of the same DG ring or module in a tensor product. For instance, 
given a DG ring homomorphism $A \to B$, and a DG $B$-module $M$, we 
shall often write 
$B_1 \ot_{A} B_2 := B \ot_{A} B$
and 
$M_1 \ot_{A} M_2 := M \ot_{A} M \in \cat{C}(B_1 \ot_{A} B_2)$.
To emphasize: $M_i = M$ and $B_i = B$, and the subscripts only indicate the 
positions in the tensor product. A pure tensor in $M_1 \ot_{A} M_2$
we be typically denoted by $m_1 \ot m_2$, with $m_i \in M$. 

Copies of DG rings that stand in the middle position of a tensor product
will either remain undecorated by subscripts, or will get the 
subscript $0$, e.g.\ $B_1 \ot_{A_0} B_2 := B \ot_{A} B$.
\end{conv}

Here is some complementary material on categories. 
Recall that a linear category $\cat{C}$ is called {\em additive} if it has a 
zero object and finite coproducts (which are then usually called direct 
sums). 

\begin{exa} \label{exa:1570}
A nonzero DG ring $A$ can be viewed as a single object DG linear category 
$\cat{A}$; and it is not an additive category. On the other hand, the DG 
category $\cat{C}(A)$ of DG $A$-modules is an additive DG category.
\end{exa}

Given a collection $\{ \cat{C}_i \}_{i \in I}$ of categories, 
indexed by a set $I$, there is the product category
$\cat{C} := \prod_{i \in I} \cat{C}_i$. Its object set is 
$\opn{Ob}(\cat{C}) := \prod_{i \in I} \opn{Ob}(\cat{C}_i)$,
and for a pair of objects 
$C = \{ C_i \}_{i \in I}$ and $D = \{ D_i \}_{i \in I}$ in $\cat{C}$, the 
morphism set is
\begin{equation} \label{eqn:2100}
\opn{Hom}_{\cat{C}}(C, D) := 
\prod\nolimits_{i \in I} \opn{Hom}_{\cat{C}_i}(C_i, D_i) .
\end{equation}
Composition and identities in $\cat{C}$ are componentwise.
For every $i \in I$ there is the 
projection functor $P_i : \cat{C} \to \cat{C}_i$. The category $\cat{C}$
is universal, in the sense that for every category $\cat{E}$ the functor 
$\cat{Fun}(\cat{E}, \cat{C}) \to 
\prod\nolimits_{i \in I} \cat{Fun}(\cat{E}, \cat{C}_i)$, 
$F \mapsto \{ P_i \circ F \}_{i \in I}$,
is an isomorphism of categories. 

If the categories $\cat{C}_i$ are all linear, then 
$\cat{C} = \prod_{i \in I} \cat{C}_i$  is linear too, with componentwise 
addition of morphisms. In case the categories $\cat{C}_i$ are graded linear, 
then $\cat{C}$ is graded linear, and its graded abelian groups of morphisms 
are, using the notation of (\ref{eqn:2100}), 
\begin{equation} \label{eqn:2101}
\opn{Hom}_{\cat{C}}(C, D) := 
\boplus_{k \in \Z} \ 
\prod\nolimits_{i \in I} \opn{Hom}_{\cat{C}_i}(C_i, D_i)^k . 
\end{equation}

If the categories $\cat{C}_i$ are all additive, then 
$\cat{C} = \prod_{i \in I} \cat{C}_i$ is additive too, with componentwise 
direct sums. In it there is the full additive subcategory 
$\cat{C}' := \bigoplus_{i \in I} \cat{C}_i \sub \cat{C}$,
whose objects are the collections $C = \{ C_i \}_{i \in I}$
s.t.\ $C_i = 0$, the zero object of $\cat{C}_i$, for all but finitely 
many $i \in I$. In the additive case there are also embedding functors 
$E_i : \cat{C}_i \to \cat{C}$.
These functors satisfy $P_i \circ E_i = \opn{Id}_{\cat{C}_i}$.
If the indexing set $I$ is finite, so that $\cat{C}' = \cat{C}$, 
there is also an equality of functors
$\bigoplus_{i \in I} E_i \circ P_i = \opn{Id}_{\cat{C}}$.

The next two propositions are in the folklore, but we could not find 
references for them. 

\begin{prop} \label{prop:1540}
Let $\{ \cat{C}_i \}_{i \in I}$ be a finite collection of categories, with 
product $\cat{C} := \prod_{i \in I} \cat{C}_i$.
\begin{enumerate}
\item If the categories $\cat{C}_i$ are DG categories, then the category 
$\cat{C}$ has a unique structure of DG category, such that the 
functors $P_i : \cat{C} \to \cat{C}_i$ and 
$E_i : \cat{C}_i \to \cat{C}$ are DG functors. 

\item If the categories $\cat{C}_i$ are triangulated categories, then the 
category $\cat{C}$ has a unique structure of triangulated category, such that 
the functors $P_i : \cat{C} \to \cat{C}_i$ and 
$E_i : \cat{C}_i \to \cat{C}$ are triangulated functors.  
\end{enumerate}
\end{prop}

\begin{proof}
(1) The graded modules of morphism in $\cat{C}$ are as in formula 
(\ref{eqn:2101}). The differentials in $\cat{C}$ are componentwise; 
namely, for 
a degree $k$ morphism
$\phi = \{ \phi_i \}_{i \in I} : C \to D$ in $\cat{C}$, its differential is 
$\d(\phi) := \{ \d(\phi_i) \}_{i \in I}$. 
It is easy to see that this is the unique DG structure on $\cat{C}$ that's 
compatible with the functors $P_i$ and $E_i$.

\medskip \noindent 
(2) The morphism abelian groups in $\cat{C}$ are as in  formula 
(\ref{eqn:2100}). The translation functor of $\cat{C}$ is componentwise, i.e.\ 
for an object $C = \{ C_i \}_{i \in I} \in \cat{C}$ its translation is the 
object $C[1] := \{ C_i[1] \}_{i \in I}$, and the same for morphisms. The 
distinguished triangles of $\cat{C}$ are also componentwise. 
It is easy to verify that the axioms of a triangulated category hold,
that this is the only triangulated structure on $\cat{C}$ that is compatible 
with the functors $P_i$ and $E_i$. 
\end{proof}

\begin{prop} \label{prop:1475}
Let $\{ A_i \}_{i\in I}$ be a finite collection of DG rings, and define the DG 
ring $A := \prod_{i \in I} A_i$. 
\begin{enumerate}
\item The functor $F : \cat{C}(A) \to \prod_{i \in I} \cat{C}(A_i)$, 
$M \mapsto \{ A_i \ot_A M \}_{i \in I}$, 
is an equivalence of DG categories.

\item The functor $F : \cat{D}(A) \to \prod_{i \in I} \cat{D}(A_i)$, 
$M \mapsto \{ A_i \ot_A M \}_{i \in I}$, 
is an equivalence of triangulated categories.
\end{enumerate}
\end{prop}

The functor $F$ is item (2) makes sense, because $A_i$ is a K-flat right DG 
$A$-module.  

\begin{proof}
(1) The category $\cat{C}(A)$ is an additive DG category.
For every $i \in I$ consider the DG functor 
$E'_i : \cat{C}(A_i) \to \cat{C}(A)$, 
$E'_i := \opn{Rest}_{A_i / A}$, the restriction functor along 
the DG ring homomorphism $A \to A_i$. 
Then the DG functor 
$G : \prod_{i \in I} \cat{C}(A_i) \to \cat{C}(A)$, 
$G := \bigoplus_{i \in I} E'_i \circ P_i$, 
is a quasi-inverse of $F$. 

\medskip \noindent
(2) The same as (1), only here $F$ and $G$ are triangulated functors. 
\end{proof}

We end this section with a review of some finiteness properties for DG modules 
over DG rings, mostly following \cite[Section 12.4]{Ye5}.
There is also an improved version of the theorem on derived 
tensor evaluation, \cite[Theorem 12.9.10]{Ye5}.

A {\em pseudo-finite semi-free filtration} on a DG $A$-module $P$ is 
an ascending chain of DG submodules
$\{ F_j(P) \}_{j \geq -1}$ of $P$, such that 
$P = \bigcup F_j(P)$, $F_{-1}(P) = 0$, and there are $i_1 \in \Z$ and 
$r_j \in \N$ with 
$\opn{Gr}^F_j(P) \cong A[-i_1 + j]^{\oplus \lsp r_j}$ for all $j \geq 0$. 
See \cite[Definition 11.4.29]{Ye5}.
In other words, the DG $A$-module $P$ admits a semi-basis $(S, F)$,
where the graded set $S$ satisfies 
$S = \coprod_{i \leq i_1} S^i$, each $S^i$ is a finite set of cardinality 
$r_i$, the filtration $F$ on $S$ if 
$F_j(S) = \coprod_{i_1 - j \leq i \leq i_1} S^i$,
and $F_j(P)^{\natural}$ is a free graded $A^{\natural}$-module with basis 
$F_j(S)$. See \cite[Proposition 11.4.31]{Ye5}.
A DG $A$-module $P$ is called {\em pseudo-finite semi-free} if it admits a 
pseudo-finite semi-free filtration. A DG $A$-module $M$ is called 
{\em derived pseudo-finite} if there is an isomorphism $P \iso M$ in 
$\cat{D}(A)$ from a pseudo-finite semi-free DG module $P$. 

A {\em finite semi-free filtration} on a DG $A$-module $P$ is a pseudo-finite 
semi-free filtration $\{ F_j(P) \}_{j \geq -1}$ such that 
$P =  F_{j_1}(P)$ for some $j_1 \in \N$. 
A DG $A$-module $P$ is called {\em finite semi-free} if it admits some 
finite semi-free filtration. See \cite[Definition 14.1.2]{Ye5}. 
According to \cite[Theorem 14.1.22]{Ye5}, a DG $A$-module $L$ is {\em perfect} 
iff it is a direct summand in $\cat{D}(A)$ of a finite semi-free DG module $P$. 

The full subcategory of $\cat{D}(A)$ on the DG modules $M$ whose cohomology 
$\opn{H}(M) = \lb \boplus_{i \in \Z} \opn{H}^i(M)$ is bounded (resp.\ bounded 
above, resp.\ bounded below) is denoted by $\cat{D}^{\mrm{b}}(A)$ 
(resp.\ $\cat{D}^{-}(A)$, resp.\ $\cat{D}^{+}(A)$). These subcategories are 
triangulated. 

A DG ring $A$ is called {\em cohomologically pseudo-noetherian} if 
$\opn{H}^0(A)$ is a noetherian ring, and each $\opn{H}^i(A)$ is a finitely 
generated $\opn{H}^0(A)$-module. When $A$ is cohomologically pseudo-noetherian, 
the full subcategory $\cat{D}_{\mrm{f}}(A)$ of $\cat{D}(A)$, 
consisting of the DG modules $M$ such that every $\opn{H}^i(M)$ is a finitely 
generated $\opn{H}^0(A)$-module, is triangulated.
Combining conditions we obtain full triangulated subcategories 
$\cat{D}^{\star}_{\mrm{f}}(A)$ of $\cat{D}(A)$. The derived 
pseudo-finite DG $A$-modules are precisely the objects of 
$\cat{D}^{-}_{\mrm{f}}(A)$; see \cite[Theorem 11.4.40]{Ye5}.

A graded module $L$ is said to be {\em concentrated} in an integer interval 
$[q_0, q_1] \sub \Z$, for some $-\infty \leq q_0 \leq q_1 \leq \infty$,
if $\{ q \mid L^q \neq 0 \} \sub [q_0, q_1]$.
The {\em concentration} of $L$ is the smallest integer interval
$[q_0, q_1]$ is which $L$ is concentrated, and it is denoted by 
$\opn{con}(L)$. The {\em cohomological concentration} of a DG module $M$ is the
the concentration of the graded module $\opn{H}(M)$.
For instance, a DG $A$-module $M$ belongs to $\cat{D}^{\mrm{b}}(A)$ iff its 
cohomological concentration is bounded. 

A DG module $M \in \cat{D}(A)$ is said to have {\em flat concentration} inside 
an integer interval $[q_0, q_1]$ if 
\[  \opn{con}(\opn{H} (M \ot^{\mrm{L}}_A N)) \sub 
[q_0, q_1] + \opn{con}(\opn{H}(N)) \]
for every $N \in \cat{D}(A)$.
In this case we say that $M$ has {\em flat dimension} $\leq q_1 - q_0$. 
If $A$ is a ring, $M$ has flat concentration inside $[q_0, q_1]$, and 
$q_1 < \infty$, then according to \cite[Proposition 12.4.19]{Ye5}, together  
with a little extra argument contained in the proof of 
\cite[Proposition 12.4.13]{Ye5}, there is a quasi-isomorphism $P \to M$ in 
$\cat{C}_{\mrm{str}}(A)$, where $P$ is a complex of flat $A$-modules  
concentrated in the integer interval $[q_0, q_1]$.

Similarly, a DG module $M \in \cat{D}(A)$ is said to have {\em projective 
concentration} inside an integer interval $[q_0, q_1]$ if 
\[  \opn{con}(\opn{H}(\opn{RHom}_A(M, N))) \sub 
[-q_1, -q_0] + \opn{con}(\opn{H}(N)) \]
for all $N \in \cat{D}(A)$. 
In this case we say that $M$ has {\em projective dimension} $\leq q_1 - q_0$. 
If $A$ is a ring, $M$ has projective concentration inside $[q_0, q_1]$, and 
$q_1 < \infty$, then according to \cite[Proposition 12.4.16]{Ye5}, together 
the same little extra argument referred to above, there is a quasi-isomorphism 
$P \to M$ in $\cat{C}_{\mrm{str}}(A)$ from a complex $P$ of projective 
$A$-modules that is concentrated in degrees $[q_0, q_1]$.

A DG module $M \in \cat{D}(A)$ is said to have {\em injective 
concentration} inside an integer interval $[q_0, q_1]$ if
\[  \opn{con}(\opn{H}(\opn{RHom}_A(N, M))) \sub 
[q_0, q_1] - \opn{con}(\opn{H}(N)) \]
for all $N \in \cat{D}(A)$.
In this case we say that $M$ has {\em injective dimension} $\leq q_1 - q_0$. 
If $A$ is a ring, $M$ has injective concentration inside $[q_0, q_1]$, and 
$q_0 > -\infty$, then according to \cite[Proposition 12.4.13]{Ye5}, together 
the same little extra argument in the proof of that proposition, there 
is a quasi-isomorphism $M \to I$ in $\cat{C}_{\mrm{str}}(A)$ to a complex 
$I$ of injective $A$-modules that is concentrated in degrees $[q_0, q_1]$.

The derived tensor evaluation morphism was introduced in
\cite[Theorem 12.9.10]{Ye5}.

\begin{thm}[Derived Tensor Evaluation] \label{thm:2115}
Let $A$ and $B$ be DG rings, and let $L \in \cat{D}(A)$, 
$M \in \cat{D}(A \ot_{\Z} B)$ and $N \in \cat{D}(B)$ be DG modules. 
Assume the three conditions below hold: 
\begin{itemize}
\rmitem{i} The DG $A$-module $L$ is derived pseudo-finite. 

\rmitem{ii} The DG $(A \ot_{\Z} B)$-module $M$ is cohomologically bounded 
below.  

\rmitem{iii} The DG $B$-module $N$ has bounded below flat concentration.
\end{itemize}
Then the derived tensor evaluation morphism 
\[ \opn{ev}^{\mrm{R}, \mrm{L}}_{L, M, N} : 
\opn{RHom}_{A}(L, M) \ot^{\mrm{L}}_{B} N \to 
\opn{RHom}_{A}(L, M \ot^{\mrm{L}}_{B} N) \]
in $\cat{D}(\Z)$ is an isomorphism.
\end{thm}

\begin{proof}
The proof is in five steps. 

\medskip \noindent
Step 1. Since the DG ring $A \ot_{\Z} B$ is nonpositive, we can replace $M$ by 
a suitable smart truncation $\opn{smt}^{\geq i}(M)$, 
see \cite[Proposition 7.3.8]{Ye5}; so we 
might as well assume that $M$ is a bounded below DG $(A \ot_{\Z} B)$-module. 
Let $P \to  L$ be a quasi-isomorphism in 
$\cat{C}_{\mrm{str}}(A)$, where $P$ is a pseudo-finite semi-free DG $A$-module,
with pseudo-finite semi-free filtration $\{ F_j(P) \}_{j \geq -1}$.
By definition there are $i_1 \in \Z$ and $r_j \in \N$ be such that 
\begin{equation} \label{eqn:2123}
\opn{Gr}^F_j(P) \cong A[-i_1 + j]^{\oplus \lsp r_j}
\end{equation}
for all $j \geq 0$.
Let $Q \to  N$ be a K-flat DG module resolution over $B$.  
We keep these resolutions throughout the proof. 

\medskip \noindent
Step 2. Assume that $P$ is a finite semi-free DG $A$-module.
Then, as in step 1 of the proof of 
\cite[Theorem 12.9.7]{Ye5}, the homomorphism 
\[ \opn{ev}_{P, M, Q} : 
\opn{Hom}_{A}(P, M) \ot_{B} Q \to 
\opn{Hom}_{A}(P, M \ot_{B} Q) \]
in $\cat{C}_{\mrm{str}}(\Z)$ is an isomorphism.

\medskip \noindent
Step 3. Let $j \in \N$ be arbitrary. Consider the DG $A$-module 
$\bar{F}_j(P) := P / F_j(P)$.
This is a semi-free DG $A$-module concentrated in the degree interval 
$[- \infty, i_1 - j - 1]$, where $i_1$ is the integer from equation 
(\ref{eqn:2123}). There is a short exact sequence 
\begin{equation} \label{eqn:2118}
0 \to F_j(P) \to P \to \bar{F}_j(P) \to  0 
\end{equation}
in $\cat{C}_{\mrm{str}}(A)$, which is split in 
$\cat{G}_{\mrm{str}}(A^{\natural})$. By applying
$\opn{Hom}_A(-, M)$ to it we obtain a short exact sequence 
\[ 0 \to \opn{Hom}_{A}(\bar{F}_j(P), M)\to 
\opn{Hom}_{A}(P, M) \to \opn{Hom}_{A}(F_j(P), M) \to  0 \]
in $\cat{C}_{\mrm{str}}(A)$.
This gives rise to a distinguished triangle 
\[ \opn{Hom}_{A}(\bar{F}_j(P), M) \to 
\opn{Hom}_{A}(P, M) \to \opn{Hom}_{A}(F_j(P), M) \xar{\ \triangle\ } \]
in $\cat{D}(A)$. Since $Q$ is K-flat, we get a distinguished triangle 
\begin{equation} \label{eqn:2119}
\opn{Hom}_{A}(\bar{F}_j(P), M) \ot_A Q \to 
\opn{Hom}_{A}(P, M) \ot_A Q \to \opn{Hom}_{A}(F_j(P), M) \ot_A Q 
\xar{\ \triangle\ }
\end{equation}
in $\cat{D}(A)$.

Similarly, from (\ref{eqn:2118}) we deduce the existence of a short 
exact sequence 
\[ 0  \to \opn{Hom}_{A}(\bar{F}_j(P), M \ot_A Q) \to 
\opn{Hom}_{A}(P, M \ot_A Q) \to \opn{Hom}_{A}(F_j(P), M \ot_A Q) \to 0 \]
in $\cat{C}_{\mrm{str}}(A)$, and thus distinguished triangle 
\begin{equation} \label{eqn:2120}
\opn{Hom}_{A}(\bar{F}_j(P), M \ot_A Q) \to 
\opn{Hom}_{A}(P, M \ot_A Q) \to \opn{Hom}_{A}(F_j(P), M \ot_A Q) 
\xar{\ \triangle\ }
\end{equation}
in $\cat{D}(A)$.

\medskip \noindent
Step 4. Let $[k_0, \infty]$ be a bounded below integer interval containing the 
flat concentration of $N \in \cat{D}(B)$.
Let $[l_0, \infty]$ be a bounded below integer interval containing the  
concentration of the DG module $M$. Note that $k_0, l_0 \in \Z$. 

We know that the DG module $\bar{F}_j(P)$ is concentrated inside the integer 
interval $[-\infty, i_1 - j - 1]$.  
Therefore the DG module $\opn{Hom}_{A}(\bar{F}_j(P), M)$
is concentrated inside the integer 
interval $[-i_1 + j + 1 + l_0, \infty]$.  
We conclude that the graded module 
\begin{equation} \label{eqn:2121}
\opn{H} \bigr( \opn{Hom}_{A}(\bar{F}_j(P), M) \ot_B Q \bigr) \cong
\opn{H} \bigr( \opn{Hom}_{A}(\bar{F}_j(P), M) \ot^{\mrm{L}}_B N \bigr)
\end{equation}
is concentrated inside the integer 
interval $[-i_1 + j + 1 + l_0 + k_0, \infty]$. 

Since 
$M \ot_B Q \cong M \ot^{\mrm{L}}_B N$,
the concentration of $\opn{H}(M \ot_B Q)$ is contained inside the integer 
interval $[l_0 + k_0, \infty]$.
The projective concentration of the DG $A$-module $\bar{F}_j(P)$ is inside 
$[-\infty, i_1 - j - 1]$. We see that the graded module 
\begin{equation} \label{eqn:2122}
\opn{H} \bigr( \opn{Hom}_{A}(\bar{F}_j(P), M \ot_B Q) \bigr) \cong
\opn{H} \bigr( \opn{RHom}_{A}(\bar{F}_j(P), M \ot_B Q) \bigr)
\end{equation}
is concentrated inside the integer 
interval $[-i_1 + j + 1 + l_0 + k_0, \infty]$.

\medskip \noindent
Step 5. We now prove that for every $i \in \Z$ the homomorphism 
\[ \opn{H}^i(\opn{ev}_{P, M, Q}) :
\opn{H}^i \bigl( \opn{Hom}_{A}(P, M) \ot_{B} Q \bigr) \to 
\opn{H}^i \bigl( \opn{Hom}_{A}(P, M \ot_{B} Q) \bigr) \]
is an isomorphism. 
This will imply that $\opn{ev}_{P, M, Q}$ is a quasi-isomorphism
in $\cat{C}_{\mrm{str}}(\Z)$, and 
hence that $\opn{ev}^{\mrm{R}, \mrm{L}}_{L, M, N}$ is an isomorphism in 
$\cat{D}(\Z)$.  

Given $i$, let $j$ be an integer satisfying 
$j \geq i_1 - l_0 - k_0 + i + 1$.
The intervals of concentration of the graded modules in 
equations (\ref{eqn:2121}) and (\ref{eqn:2122}) tell us that 
\[ \opn{H}^{i'} \bigr( \opn{Hom}_{A}(\bar{F}_j(P), M) \ot_B Q \bigr) = 0 \]
and 
\[ \opn{H}^{i'} \bigr( \opn{Hom}_{A}(\bar{F}_j(P), M \ot_B Q) \bigr) = 0 \]
for all $i' \leq i + 1$. 
The long exact cohomology sequences for the distinguished triangles 
(\ref{eqn:2119}) and (\ref{eqn:2120}) imply that 
\[ \opn{H}^{i} \bigr( \opn{Hom}_{A}(P, M) \ot_B Q \bigr) \to 
\opn{H}^{i} \bigr( \opn{Hom}_{A}(F_j(P), M) \ot_B Q \bigr) \]
and 
\[ \opn{H}^{i} \bigr( \opn{Hom}_{A}(P, M \ot_B Q) \bigr) \to 
\opn{H}^{i} \bigr( \opn{Hom}_{A}(F_j(P), M \ot_B Q) \bigr) \]
are both bijective. 

Finally, consider the commutative diagram of $\Z$-modules 
\begin{equation} \label{eqn:2115}
\begin{tikzcd} [column sep = 6ex, row sep = 5ex] 
\opn{H}^i \bigl( \opn{Hom}_{A}(P, M) \ot_{B} Q \bigr)
\arrow[r, "{}"]
\ar[d, "{\opn{H}^i(\opn{ev}_{P, M, Q})}"']
&
\opn{H}^i \bigl( \opn{Hom}_{A}(F_j(P), M) \ot_{B} Q \bigr)
\ar[d, "{\opn{H}^i(\opn{ev}_{F_j(P), M, Q})}"]
\\
\opn{H}^i \bigl( \opn{Hom}_{A}(P, M \ot_{B} Q) \bigr)
\arrow[r, "{}"]
&
\opn{H}^i \bigl( \opn{Hom}_{A}(F_j(P), M \ot_{B} Q) \bigr)
\end{tikzcd}
\end{equation}
We know that the horizontal arrows are bijective. 
Since $F_j(P)$ is a finite semi-free DG $A$-module, step 2 tells us that 
$\opn{H}^i(\opn{ev}_{F_j(P), M, Q})$. We conclude that 
$\opn{H}^i(\opn{ev}_{P, M, Q})$ is bijective. 
\end{proof}

\begin{rem} \label{rem:2115}
Theorem \ref{thm:2115} holds also for {\em noncommutative} DG rings $A$ and 
$B$, provided they are {\em nonpositive}. The obvious modification is that $M$ 
should a DG module over $A \ot_{\Z} B^{\mrm{op}}$. The proof is essentially the 
same. 

We see that Theorem \ref{thm:2115} is an improvement upon 
\cite[Theorem 12.9.10]{Ye5} -- condition (iii) here is weaker than condition 
(c)(iii) there.

There is a minor error in \cite[Theorem 12.9.10]{Ye5}: in condition (c) there it 
is necessary to require the DG rings $A$ and $B$ to be {\em nonpositive}. 
Otherwise smart truncation of DG $(A \ot_{\Z} B^{\mrm{op}})$-modules might 
not be possible. 
\end{rem}

\begin{rem} \label{rem:2116}
Perfect DG modules have several characterizations, see 
\cite[Theorem 14.1.22]{Ye5}.
One of them is this: a DG $A$-module $L \in \cat{D}(A)$ is perfect iff for 
every $N \in \cat{D}(A)$ the derived tensor-evaluation morphism 
\[ \opn{ev}^{\mrm{R}, \mrm{L}}_{L, A, N} : 
\opn{RHom}_{A}(L, A) \ot^{\mrm{L}}_{A} N \to 
\opn{RHom}_{A}(L, N) \]
in $\cat{D}(\Z)$ is an isomorphism. (This is a slight modification of condition 
(iii) of \cite[Theorem 14.1.22]{Ye5}, but it is equivalent to it, 
as can be seen in the proof there.) Another characterization is this: 
if $A$ and $B$ are DG rings , and 
$F : \cat{D}(A) \to \cat{D}(B)$
is an equivalence of triangulated categories, then a DG module 
$L \in \cat{D}(A)$ is perfect iff the DG module 
$F(L) \in \cat{D}(B)$ is perfect; see \cite[Corollary 12.1.4]{Ye5}.

\end{rem}

\section{The Squaring Operation and Rigid Complexes}
\label{sec:squaring}

In this section we recall the main results of the paper \cite{Ye4} on the 
squaring operation. We then use these results to define rigid complexes and 
rigid morphisms between them. 

Convention \ref{conv:615} is in place: DG rings are commutative by default
(Definition \ref{dfn:616}(3)). The category of commutative DG rings is denoted 
by $\cat{DGRng}$, and its full subcategory of commutative rings is 
$\cat{Rng}$.

The squaring operation is best described in terms of pairs of DG rings; so we 
start with the next definition. 

\begin{dfn} \label{dfn:625} \mbox{}
\begin{enumerate}
\item By a {\em pair of DG rings} $B / A$ we mean a homomorphism 
$A \xar{u} B$ in $\cat{DGRng}$. 

\item Suppose $B' / A' = (A' \xar{u'} B')$ is another pair of DG rings. A 
{\em morphism of pairs} $w / v : B / A \to B' / A'$
consists of DG ring homomorphisms $v : A \to A'$ and 
$w : B \to B'$, such that $u' \circ v = w \circ u$. Namely the diagram 
\[ \begin{tikzcd} [column sep = 6ex, row sep = 5ex] 
A
\arrow[r, "{u}"]
\ar[d, "{v}"']
&
B
\ar[d, "{w}"]
\\
A'
\arrow[r, "{u'}"]
&
B'
\end{tikzcd} \]
in  $\cat{DGRng}$ is commutative. 

\item The pairs of DG rings form a category $\cat{PDGRng}$, with obvious 
compositions. 
\end{enumerate}
\end{dfn}

The pronunciation of the expression ``$B / A$'' is ``B relative to A''.

\begin{rem} \label{rem:2100}
Let $\cat{Arr}$ be the category with objects $0$ and $1$, and with three 
morphisms: a morphism $\ep : 0 \to  1$, and the two identity automorphisms.
Then the category $\cat{PDGRng}$ and the functor category 
$\cat{Fun}(\cat{Arr}, \cat{DGRng})$ are isomorphic. 
\end{rem}

\begin{dfn} \label{dfn:626}\mbox{}
\begin{enumerate}
\item A pair of DG rings $\til{B} / \til{A}$ is called a {\em K-flat pair} if 
$\til{B}$ is K-flat as a DG $\til{A}$-module.

\item Given a pair of DG rings $B / A$, a {\em K-flat resolution} of $B / A$ 
is a morphism of pairs 
$s / r : \til{B} / \til{A} \to B / A$, 
such that $r : \til{A} \to A$ and $s : \til{B} \to B$ are surjective 
quasi-isomorphisms, and $\til{B} / \til{A}$ is a K-flat pair. 

\item Let $w / v : B / A \to B' / A'$ be a morphism of pairs of DG rings.  
A {\em K-flat resolution} of $w / v$ 
consists of K-flat resolutions 
$s / r : \til{B} / \til{A} \to B / A$
and $s' / r' : \til{B}' / \til{A}' \to B' / A'$, 
together with a morphism of pairs 
$\til{w} / \til{v} : \til{B} / \til{A} \to \til{B}' / \til{A}'$,
such that 
\[ (s' / r') \circ (\til{w} / \til{v}) = (w / v) \circ (s / r) \]
in $\cat{PDGRng}$. 
\end{enumerate}
\end{dfn}

Items (2) and (3) of the definition are shown in the next two diagrams. These 
are commutative diagrams in $\cat{DGRng}$ and $\cat{PDGRng}$, respectively. 
The expressions "QI" and "KF" stand for "quasi-isomorphism" and "K-flat", 
respectively. 
\begin{equation} \label{eqn:2102}
 \begin{tikzcd} [column sep = 6ex, row sep = 5ex] 
\til{A}
\arrow[r, "{\til{u}, \ \mrm{KF}}"]
\ar[d, two heads, "{r, \ \mrm{QI}}"']
&
\til{B}
\ar[d, two heads, "{s, \ \mrm{QI}}"]
\\
A
\arrow[r, "{u}"]
&
B
\end{tikzcd}
\qquad \qquad 
\begin{tikzcd} [column sep = 6ex, row sep = 5ex] 
\til{B} / \til{A}
\arrow[r, "{\til{w} / \til{v}}"]
\ar[d, "{s / r}"']
&
\til{B}' / \til{A}'
\ar[d, "{s' / r'}"]
\\
B / A
\arrow[r, "{w / v}"]
&
B' / A'
\end{tikzcd}
\end{equation}

It is known (see \cite[Corollary 3.25]{Ye4}) that K-flat resolutions of pairs, 
and of morphisms of pairs, exist.  

\begin{dfn} \label{dfn:1000}
Fix a pair of DG rings $B / A$. 
\begin{enumerate}
\item The set of K-flat resolutions $\til{B} / \til{A}$ of $B / A$, in the 
sense of Definition \ref{dfn:626}(2), is denoted by 
$\cat{KFRes}(B / A)$. 

\item We make $\cat{KFRes}(B / A)$ into a category, in which the morphisms 
$\til{w} / \til{v} : \til{B} / \til{A} \to \til{B}' / \til{A}'$ 
are the K-flat resolutions of the morphism
$\opn{id}_B / \opn{id}_A : B / A \to B / A$, 
in the sense of Definition \ref{dfn:626}(3).  
\end{enumerate}
\end{dfn}

Let $B / A$ be a pair of DG rings. Given a DG $B$-module $M$, and a 
resolution $\til{B} / \til{A} \in \cat{KFRes}(B / A)$, the 
{\em resolved square of $M$} is the object
\begin{equation} \label{eqn:625}
\opn{Sq}_{B / A}^{\til{B} / \til{A}} (M) := 
\opn{RHom}_{\til{B} \ot_{\til{A}} \til{B}}(B, M \ot_{\til{A}}^{\mrm{L}} M)
\in \cat{D}(B) . 
\end{equation}
The DG $B$-module structure on 
$\opn{Sq}_{B / A}^{\til{B} / \til{A}} (M)$
comes from the action of $B$ on the first argument of $\opn{RHom}$. 
We use the name ``resolved square'' because it depends on the K-flat 
resolution $\til{B} / \til{A}$. In Theorem \ref{thm:631} below this dependence 
is going to be removed. 

The resolved square 
$\opn{Sq}_{B / A}^{\til{B} / \til{A}} (M)$ of $M$ can be presented explicitly 
as 
follows. We first choose a K-projective resolution 
$\til{P} \to M$ in $\cat{C}_{\mrm{str}}(\til{B})$,
so that
$\til{P} \ot_{\til{A}} \til{P} \cong M \ot_{\til{A}}^{\mrm{L}} M$
in $\cat{D}(\til{B} \ot_{\til{A}} \til{B})$. 
Next we choose a K-injective resolution 
$\til{P} \ot_{\til{A}} \til{P} \to \til{I}$ 
in $\cat{C}_{\mrm{str}}(\til{B} \ot_{\til{A}} \til{B})$. These choices are 
unique up to homotopy. They give rise to a canonical isomorphism 
$\opn{Sq}_{B / A}^{\til{B} / \til{A}} (M) \cong 
\opn{Hom}_{\til{B} \ot_{\til{A}} \til{B}}(B, \til{I})$ 
in $\cat{D}(B)$, as explained in formulas (\ref{eqn:618}) and 
(\ref{eqn:621}).

Consider a morphism of pairs of DG rings 
$w / v : B / A \to B' / A'$, such that $v : A \to A'$ is a quasi-isomorphism.
Choose some K-flat resolution
$\til{w} / \til{v} : \til{B} / \til{A} \to \til{B}' / \til{A}'$
of $w / v$. Let $M \in \cat{D}(B)$ and $M' \in \cat{D}(B')$ be DG modules, and 
let $\th : M' \to M$ be a backward morphism in 
$\cat{D}(B)$ over $w$. 
There is a backward morphism
\begin{equation} \label{eqn:628}
\opn{Sq}_{w / v}^{\til{w} / \til{v}} (\th) : 
\opn{Sq}_{B' / A'}^{\til{B}' / \til{A}'} (M') \to
\opn{Sq}_{B / A}^{\til{B} / \til{A}} (M)
\end{equation}
in $\cat{D}(B)$ over $w$, which we call the {\em resolved square of $\th$}. 

Here is an explicit formula for the morphism
$\opn{Sq}_{w / v}^{\til{w} / \til{v}} (\th)$, following 
\cite[Section 5 and Remark 7.5]{Ye4}.
Let $\til{P} \to M$ be a K-projective resolution in 
$\cat{C}_{\mrm{str}}(\til{B})$, and let 
$\til{P} \ot_{\til{A}} \til{P} \to \til{I}$ 
be a K-injective resolution in 
$\cat{C}_{\mrm{str}}(\til{B} \ot_{\til{A}} \til{B})$, as above.
Likewise we choose a K-projective resolution 
$\til{P}' \to M'$ in $\cat{C}_{\mrm{str}}(\til{B}')$,
and a K-injective resolution
$\til{P}' \ot_{\til{A}'} \til{P}' \to \til{I}'$
in $\cat{C}_{\mrm{str}}(\til{B}' \ot_{\til{A}'} \til{B}')$.
Next we choose a K-projective resolution 
$\til{\ga} : \til{Q}' \to \til{P}'$ of $\til{P}'$
in $\cat{C}_{\mrm{str}}(\til{B})$. 
There is a homomorphism 
$\til{\th} : \til{Q}' \to \til{P}$ 
in $\cat{C}_{\mrm{str}}(\til{B})$, unique up to homotopy, such that the 
diagram 
\[ \begin{tikzcd} [column sep = 10ex, row sep = 5ex] 
\til{P}
\ar[d, "{\simeq}"]
&
\til{Q}'
\arrow[l, dashed, "{\opn{Q}(\til{\th})}"']
\arrow[r, "{\opn{Q}(\til{\ga})}", "{\simeq}"']
&
\til{P}'
\ar[d, "{\simeq}"]
\\
M
&
&
M'
\arrow[ll, "{\th}"']
\end{tikzcd} \]
in $\cat{D}(\til{B})$ is commutative. 
(Notice that in this diagram,  $\opn{Q}$ is the localization functor, whereas 
$\til{Q}'$ is a DG module; the fonts are distinct. Also the restriction functor 
$\opn{Rest}_{\til{w}}$ is suppressed.) 
Since $\til{v} : \til{A} \to \til{A}'$ is a quasi-isomorphism of DG rings, 
$\til{Q}'$ is K-flat over $\til{A}$, and $\til{P}'$ is K-flat over $\til{A}'$,
it follows that 
$\til{\ga} \ot_{\til{v}} \til{\ga} : 
\til{Q}' \ot_{\til{A}} \til{Q}' \to \til{P}' \ot_{\til{A}'} \til{P}'$ 
is a quasi-isomorphism
(see \cite[Proposition 2.6(1)]{Ye4}). Hence there is a homomorphism 
$\til{\ze}: \til{I}' \to \til{I}$ in 
$\cat{C}_{\mrm{str}}(\til{B} \ot_{\til{A}} \til{B})$, unique up to homotopy,
such that the diagram 
\begin{equation} \label{eqn:670}
\begin{tikzcd} [column sep = 10ex, row sep = 5ex] 
\til{P} \ot_{\til{A}} \til{P}
\ar[d, "\tup{qu-iso}"]
&
\til{Q}' \ot_{\til{A}} \til{Q}'
\arrow[l, "{\til{\th} \ot \til{\th}}"']
\arrow[r, "{\til{\ga} \ot_{\til{v}} \til{\ga}}", "\tup{qu-iso}"']
&
\til{P}' \ot_{\til{A}'} \til{P}'
\ar[d, "\tup{qu-iso}"]
\\
\til{I}
&
&
\til{I}'
\arrow[dashed, ll, "{\til{\ze}}"']
\end{tikzcd} 
\end{equation}
in $\cat{C}_{\mrm{str}}(\til{B} \ot_{\til{A}} \til{B})$
is commutative up to homotopy. 
Then the homomorphism 
\begin{equation} \label{eqn:890}
\opn{Hom}_{\til{w} \ot_{\til{v}} \til{w}}(w, \til{\ze}) :
\opn{Hom}_{\til{B}' \ot_{\til{A}'} \til{B}'}(B', \til{I}') \to 
\opn{Hom}_{\til{B} \ot_{\til{A}} \til{B}}(B, \til{I}) 
\end{equation}
in $\cat{C}_{\mrm{str}}(B)$ represents
$\opn{Sq}_{w / v}^{\til{w} / \til{v}} (\th)$.

The morphism $\opn{Sq}_{w / v}^{\til{w} / \til{v}} (\th)$ in (\ref{eqn:628}) 
is functorial in all arguments. By this we mean that if
$w' / v' : B' / A' \to B'' / A''$
is another morphism of pairs of DG rings, such that 
$v' : A' \to A''$ is a quasi-isomorphism, with a K-flat resolution 
$\til{w}' / \til{v}' : \til{B}' / \til{A}' \to \til{B}'' / \til{A}''$, 
and if $\th' : M'' \to M'$
is a backward morphism in $\cat{D}(B')$ over $w'$, then 
\begin{equation} \label{eqn:880}
\opn{Sq}_{(w' / v') \circ (w / v)}^{(\til{w}' / \til{v}') \circ 
(\til{w} / \til{v})} (\th \circ \th') = 
\opn{Sq}_{w / v}^{\til{w} / \til{v}} (\th) \circ
\opn{Sq}_{w' / v'}^{\til{w}' / \til{v}'} (\th') ,
\end{equation}
as backward morphisms 
$\opn{Sq}_{B'' / A''}^{\til{B}'' / \til{A}''} (M'') \to
\opn{Sq}_{B / A}^{\til{B} / \til{A}} (M)$
in $\cat{D}(B)$ over $w' \circ w$. 
This is easy to see from the uniqueness up to homotopy of the resolutions of 
the DG modules. 

In case $B' / A' = B / A$ and $w / v = \opn{id}_B / \opn{id}_A$, 
so $\th : M' \to M$ is a morphism in $\cat{D}(B)$, we shall 
often use the following variant of formula (\ref{eqn:628}):
\begin{equation} \label{eqn:1490} 
\opn{Sq}_{B / A}^{\til{w} / \til{v}} (\th) := 
\opn{Sq}_{\opn{id}_B / \opn{id}_A}^{\til{w} / \til{v}} (\th) :
\opn{Sq}_{B / A}^{\til{B}' / \til{A}'} (M') \to
\opn{Sq}_{B / A}^{\til{B} / \til{A}} (M) . 
\end{equation}

Here is the first main result of \cite{Ye4}. 

\begin{thm}[Existence of Squares, {\cite[Theorem 0.3.4]{Ye4}}] \label{thm:631}
Let $A \to B$ be a homomorphism in $\cat{DGRng}$, and let $M \in \cat{D}(B)$.
There is a DG $B$-module $\opn{Sq}_{B / A}(M)$, 
together with a collection 
\[ \bsym{\opn{sq}}_{B / A, M} =
\bigl\{ \opn{sq}_{B / A, M}^{\til{B} / \til{A}} \bigr\}
_{\til{B} / \til{A} \, \in \, \cat{KFRes}(B / A)} \]
of isomorphisms 
\[ \opn{sq}^{\til{B} / \til{A}}_{B / A, M} : 
\opn{Sq}_{B / A}(M) \iso \opn{Sq}_{B / A}^{\til{B} / \til{A}}(M) \]
in $\cat{D}(B)$,
satisfying the following condition\tup{:}
\begin{enumerate}
\item[($*$)] For every morphism 
$\til{w} / \til{v} : \til{B} / \til{A} \to \til{B}' / \til{A}'$
in $\cat{KFRes}(B / A)$,
the diagram 
\[ \begin{tikzcd} [column sep = 14ex, row sep = 6ex] 
\opn{Sq}_{B / A}(M)
\ar[dr, "{\opn{sq}^{\til{B}' / \til{A}'}_{B / A, M}}"]
\ar[d, "{\opn{sq}^{\til{B} / \til{A}}_{B / A, M}}"']
\\
\opn{Sq}_{B / A}^{\til{B}' / \til{A}'}(M)
\ar[r, "{\opn{Sq}_{B / A}^{\til{w} / \til{v}}(\mrm{id}_M)}"']
&
\opn{Sq}_{B / A}^{\til{B} / \til{A}}(M)
\end{tikzcd} \]
of isomorphisms in $\cat{D}(B)$ is commutative. 
\end{enumerate}

The object 
$\opn{Sq}_{B / A}(M) \in \cat{D}(B)$
is unique, up to a unique isomorphism that respects the collection 
of isomorphisms $\bsym{\opn{sq}}_{B / A, M}$. 
\end{thm}

Let us make explicit the uniqueness property of the object
$\opn{Sq}_{B / A}(M)$ with its collection of isomorphisms
$\bsym{\opn{sq}}_{B / A, M}$. 
Suppose we are given another DG module 
$\opn{Sq}'_{B / A}(M) \in \cat{D}(B)$,
and a collection $\bsym{\opn{sq}}'_{B / A, M}$
of isomorphisms
$\opn{sq}'^{\, \til{B} / \til{A}}_{B / A, M} : 
\opn{Sq}'_{B / A}(M) \iso \opn{Sq}_{B / A}^{\til{B} / \til{A}}(M)$
in $\cat{D}(B)$, satisfying condition ($*$). 
Then there is a unique isomorphism 
$\si : \opn{Sq}_{B / A}(M) \iso \opn{Sq}'_{B / A}(M)$
in $\cat{D}(B)$ satisfying this condition: for every resolution 
$\til{B} / \til{A} \in \cat{KFRes}(B / A)$
the diagram 
\[ \begin{tikzcd} [column sep = 8ex, row sep = 6ex] 
\opn{Sq}_{B / A}(M)
\ar[d, "{\opn{sq}^{\til{B} / \til{A}}_{B / A, M}}"']
\ar[r, "{\si}"]
&
\opn{Sq}'_{B / A}(M)
\ar[dl, "{\opn{sq}'^{\, \til{B} / \til{A}}_{B / A, M}}"]
\\
\opn{Sq}_{B / A}^{\til{B} / \til{A}}(M)
\end{tikzcd} \]
of isomorphisms in $\cat{D}(B)$ is commutative. 

Note that in the paper \cite{Ye4}, the isomorphism 
$\opn{sq}^{\til{B} / \til{A}}_{B / A, M}$ was denoted 
by $\opn{sq}^{\til{B} / \til{A}}$.

\begin{dfn} \label{dfn:670}
Let $A \to B$ be a homomorphism in $\cat{DGRng}$, and let $M \in \cat{D}(B)$.
The DG $B$-module $\opn{Sq}_{B / A}(M)$ from Theorem \ref{thm:631}
is called the {\em square of $M$ over $B / A$}.
\end{dfn}
We shall usually keep the collection of isomorphisms
$\bsym{\opn{sq}}_{B / A, M}$, which comes with the object $\opn{Sq}_{B / A}(M)$ 
and determines it, implicit.\
The pronunciation of the expression ``the square of $M$ over $B / A$'' is this:
``the square of $M$ over $B$ relative to $A$''. 

\begin{dfn} \label{dfn:630} \mbox{}
\begin{enumerate}
\item By a {\em triple of DG rings} $C / B / A$ we mean homomorphisms
$A \xar{u} B \xar{v} C$ in $\cat{DGRng}$. 

\item A {\em morphism of triples} of DG rings 
\[ t / s / r : \til{C} / \til{B} / \til{A} \to C / B / A \]
is a commutative diagram
\[ \begin{tikzcd} [column sep = 6ex, row sep = 5ex] 
\til{A}
\arrow[r, "{\til{u}}"]
\ar[d, "{r}"]
&
\til{B}
\ar[d, "{s}"]
\arrow[r, "{\til{v}}"]
&
\til{C}
\ar[d, "{t}"]
\\
A
\arrow[r, "{u}"]
&
B
\arrow[r, "{v}"]
&
C
\end{tikzcd} \]
in $\cat{DGRng}$.

\item A {\em K-flat resolution} of a triple of DG rings $C / B / A$ 
is a morphism of triples
$t / s / r$ as in item (2), in which $r$, $s$ and $t$ are surjective 
quasi-isomorphisms, and $\til{u}$ and $\til{v}$ are K-flat homomorphisms. 
\end{enumerate}
\end{dfn}

Given a triple $C / B / A$ of DG rings, the pronunciation of the expression 
``$C / B / A$'' is ``C relative to B relative to A''.

K-flat resolutions of triples exist; see \cite[Proposition 7.9]{Ye4}.  
Note that a triple of DG rings 
$C / B / A = (A \xar{u} B \xar{v} C)$ can be viewed as a morphism of 
pairs of DG rings $v / \opn{id}_A : B / A \to C / A$. 
Since $\til{v} \circ \til{u}$ is K-flat too, a K-flat resolution 
$\til{C} / \til{B} / \til{A}$ of 
$C / B / A$ can be viewed as a K-flat resolution 
$\til{B} / \til{A} \to \til{C} / \til{A}$
of $B / A \to C / A$. This will be used in the next theorem, which is the 
second main result of \cite{Ye4}. 

\begin{thm}[Square of a Backward Morphism, {\cite[Theorem 0.3.5]{Ye4}}] 
\label{thm:632}
Let $A \to B \xar{v} C$ be homomorphisms in $\cat{DGRing}$, 
let $M \in \cat{D}(B)$, let $N \in \cat{D}(C)$, and let 
$\th : N \to M$ be a backward morphism in $\cat{D}(B)$ over $v$.
There is a unique backward morphism 
\[ \opn{Sq}_{v / \mrm{id}_A}(\th) \ : \
\opn{Sq}_{C / A}(N) \ \to \ \opn{Sq}_{B / A}(M) \]
in $\cat{D}(B)$ over $v$, satisfying the condition\tup{:}
\begin{enumerate}
\item[($**$)] For every K-flat resolution
$\til{A} \to \til{B} \xar{\til{v}} \til{C}$
of the triple of DG rings $A \to B \xar{v} C$, the diagram 
\[ \begin{tikzcd} [column sep = 12ex, row sep = 6ex] 
\opn{Sq}_{C / A}(N)
\ar[r, "{\opn{Sq}_{v / \mrm{id}_A}(\th)}"]
\ar[d, "{\opn{sq}^{\til{C} / \til{A}}_{C / A, N}}"', "{\simeq}"]
&
\opn{Sq}_{B / A}(M)
\ar[d, "{\opn{sq}^{\til{B} / \til{A}}_{B / A, M}}", "{\simeq}"']
\\
\opn{Sq}_{C / A}^{\til{C} / \til{A}}(N)
\ar[r, "{\opn{Sq}_{v / \mrm{id}_A}^{\til{v} / \mrm{id}_{\til{A}}}(\th)}"]
&
\opn{Sq}_{B / A}^{\til{B} / \til{A}}(M)
\end{tikzcd} \]
in $\cat{D}(B)$ is commutative. 
\end{enumerate}
\end{thm}

We sometimes write 
\begin{equation} \label{eqn:1500}
\opn{Sq}_{v / A}(\th) = \opn{Sq}_{C / B / A}(\th) 
:= \opn{Sq}_{v / \mrm{id}_A}(\th) .
\end{equation}
The expression ``over $u / A$'' is pronounced ``over $u$ relative 
to $A$''. 

The next proposition is implicit in \cite{Ye4}. 

\begin{prop} \label{prop:890}
Let $A \to B \xar{v} B' \xar{v'} B''$ be homomorphisms in $\cat{DGRing}$, 
let $M \in \cat{D}(B)$, $M' \in \cat{D}(B')$ and $M'' \in \cat{D}(B'')$
be DG modules, let $\th : M' \to M$ be a backward morphism in $\cat{D}(B)$ 
over $v$, and let  $\th' : M'' \to M'$  be a backward morphism in $\cat{D}(B')$ 
over $v'$. Then there is equality 
\[ \opn{Sq}_{v' \circ v / \mrm{id}_A}(\th \circ \th') =  
\opn{Sq}_{v / \mrm{id}_A}(\th) \circ 
\opn{Sq}_{v' / \mrm{id}_A}(\th') \]
of backward morphisms 
$\opn{Sq}_{B'' / A}(M'') \to \opn{Sq}_{B / A}(M)$
in $\cat{D}(B)$ over $v' \circ v$. 
\end{prop}

\begin{proof}
Choose a K-flat DG ring resolution 
$A \to \til{B} \xar{\til{v}} \til{B}' \xar{\til{v}'} \til{B}''$
of the homomorphisms $A \to B \xar{v} B' \xar{v'} B''$;
by this we mean the obvious generalization of Definition \ref{dfn:630}(3).
According to \cite[Proposition 5.17 and Definitions 7.3 and 7.4]{Ye4}
there is equality 
\[ \opn{Sq}^{\til{v}' \circ \til{v} / \mrm{id}_A}_{v' \circ v / \mrm{id}_A}
(\th \circ \th') =  
\opn{Sq}^{\til{v} / \mrm{id}_A}_{v / \mrm{id}_A}(\th) \circ 
\opn{Sq}^{\til{v}' / \mrm{id}_A}_{v' / \mrm{id}_A}(\th') \]
of backward morphisms 
$\opn{Sq}^{\til{B}'' / A}_{B'' / A}(M'') \to 
\opn{Sq}^{\til{B} / A}_{B / A}(M)$
in $\cat{D}(B)$ over $v' \circ v$. (This is formula (\ref{eqn:880}) above.)
Now we use condition ($**$) of Theorem \ref{thm:632}.
\end{proof}

\begin{dfn} \label{dfn:880}
Let $A \to B$ be a homomorphism in $\cat{DGRing}$, and 
$\th : M \to N$ be a morphism in $\cat{D}(B)$. We define 
the morphism 
\[ \opn{Sq}_{B / A}(\th) : \opn{Sq}_{B / A}(M) \to \opn{Sq}_{B / A}(N) \]
in $\cat{D}(B)$ to be 
$\opn{Sq}_{B / A}(\th) := \opn{Sq}_{\mrm{id}_B / \mrm{id}_A}(\th)$.
\end{dfn}

Like in the proof of Proposition \ref{prop:890}, it is easy to see that 
for a pair of DG rings $B / A$ and $M \in \cat{D}(B)$, there is equality 
$\opn{Sq}_{B / A}(\mrm{id}_M) = \mrm{id}_{\opn{Sq}_{B / A}(M)}$
of automorphisms of $\opn{Sq}_{B / A}(M)$ in 
$\cat{D}(B)$.
We see that:

\begin{cor} \label{cor:880}
Given a DG ring homomorphism $A \to B$, there is a functor 
$\opn{Sq}_{B / A} : \cat{D}(B) \to \cat{D}(B)$,
whose values on objects is specified in Definition \ref{dfn:670}, 
and whose values on morphisms is specified in Definition \tup{\ref{dfn:880}}. 
\end{cor}

The functor $\opn{Sq}_{B / A}$ is not linear -- it is a {\em quadratic 
functor}, as Theorem \ref{thm:672} below shows. 
The category $\cat{D}(C)$ is $C^0$-linear. (In fact it is 
$\opn{H}^0(C)$-linear.)
Therefore, in the situation of Theorem \ref{thm:632}, the backward morphism 
$\th : N \to M$ can be multiplied by an element $c \in C^0$. More precisely, we 
have the backward morphism 
$c \cd \th := \th \circ (c \cd \opn{id}_N) : N \to M$. 

\begin{thm}[Squaring is a Quadratic Functor, {\cite[Theorem 7.16]{Ye4}}] 
\label{thm:672}
In the situation of Theorem \tup{\ref{thm:632}}, for every element $c \in C^0$ 
there is equality 
\[ \opn{Sq}_{v / \mrm{id}_A}(c \cd \th) = 
c^2 \cd \opn{Sq}_{v / \mrm{id}_A}(\th) \]
of backward morphisms 
$\opn{Sq}_{C / A}(N) \to \opn{Sq}_{B / A}(M)$ in $\cat{D}(B)$ over $v$. 
\end{thm}

In our paper we will not need the full generality of Theorems \ref{thm:631}, 
\ref{thm:632} and \ref{thm:672} -- for us the DG rings $A, B$ and $C$ will 
always be rings. 

\begin{prop} \label{prop:1491}
Let $A \to B$ be a ring homomorphism, and let 
$B = \prod_{i \in I} B_i$ be a finite decomposition of the ring $B$, with 
projection homomorphisms $p_i : B \to B_i$. Let $M \in \cat{D}(B)$, and define
$M_i := B_i \ot_B M \in \cat{D}(B_i)$.
There is an isomorphism $M \cong \boplus_i M_i$ in $\cat{D}(B)$, and there are 
backward morphisms $e_i : M_i \to M$ in $\cat{D}(B)$ over $p_i$. 
By Theorem \ref{thm:632}, for every $i$ there is a backward morphism 
$\opn{Sq}_{p_i / A}(e_i) : \opn{Sq}_{B_i / A}(M_i) \to 
\opn{Sq}_{B / A}(M)$
in $\cat{D}(B)$ over $p_i$. Then 
\[ \sum\nolimits_{i} \opn{Sq}_{p_i / A}(e_i) :
\boplus_i \opn{Sq}_{B_i / A}(M_i) \to \opn{Sq}_{B / A}(M) \]
is an isomorphism in $\cat{D}(B)$.
\end{prop}

\begin{proof}
For each index $i$ we choose a K-flat DG ring resolution $\til{B}_i \to B_i$ 
relative to $A$. Define $\til{B} := \prod_i \til{B}_i$. Then 
$\til{B} \to B$ is a K-flat DG ring resolution relative to $A$. 
There is a morphism 
\[ \sum\nolimits_i e_i \ot^{\mrm{L}}_{{A}} e_i :
\bigoplus\nolimits_i M_i \ot^{\mrm{L}}_{{A}} M_i \to 
M \ot^{\mrm{L}}_{{A}} M \]
in $\cat{D}(\til{B} \ot_{{A}} \til{B})$. According to Theorem \ref{thm:632} 
we can present $\sum\nolimits_{i} \opn{Sq}_{p_i / A}(e_i)$ as follows:
\[ \begin{aligned}
& \boplus_i \opn{Sq}_{B_i / A}^{\til{B}_i / {A}}(M_i) = 
\boplus_i \opn{RHom}_{\til{B}_i \ot_{{A}} \til{B}_i}(B_i, 
M_i \ot^{\mrm{L}}_{{A}} M_i)
\\
& \quad 
\cong^{\tup{(1)}} 
\boplus_i \opn{RHom}_{\til{B} \ot_{{A}} \til{B}}(B, 
M_i \ot^{\mrm{L}}_{{A}} M_i)
\\
& \quad 
\iso^{\tup{(2)}}
\opn{RHom}_{\til{B} \ot_{{A}} \til{B}}(B, 
M \ot^{\mrm{L}}_{{A}} M)
= \opn{Sq}^{\til{B} / {A}}_{B / A}(M \ot^{\mrm{L}}_{{A}} M) .
\end{aligned} \]
Here $\cong^{\tup{(1)}}$ is the direct sum of the adjunction isomorphisms for 
the DG ring homomorphisms
$p_i \ot p_i : \til{B} \ot_{{A}} \til{B} \to 
\til{B}_i \ot_{{A}} \til{B}_i$,
coupled with the DG ring isomorphisms 
$(\til{B}_i \ot_{{A}} \til{B}_i) \ot_{\til{B} \ot_{{A}} \til{B}} B 
\cong \til{B}_i$. The isomorphism $\iso^{\tup{(2)}}$ comes from 
$\sum\nolimits_i e_i \ot^{\mrm{L}}_{{A}} e_i$, with the vanishing 
$\opn{RHom}_{\til{B} \ot_{{A}} \til{B}}(B, 
M_i \ot^{\mrm{L}}_{{A}} M_j)=0$, for $i\not=j$.
\end{proof}

Recall that a ring homomorphism $u : A \to B$ is called a {\em finite type ring
homomorphism}, and $B$ is called a {\em finite type $A$-ring}, if
there exists a surjection of 
$A$-rings $A[t_1, \ldots, t_n] \to B$ from the polynomial ring  in $n$ 
variables, for some natural number $n$. 

A ring  homomorphism $u : A \to B$ is called a {\em localization ring
homomorphism}, and $B$ is called a {\em localization of $A$}, if
there exists an isomorphism of $A$-rings $B \cong A_S$, 
where $S \sub A$ is a multiplicatively closed subset, and 
$A_S = A[S^{-1}]$ is the localization of $A$ with respect to $S$. 

\begin{dfn} \label{dfn:1065}
A ring homomorphism $u : A \to B$ is called an {\em essentially finite type ring
homomorphism}, and $B$ is called an {\em essentially finite type $A$-ring}, if
$u = v \circ \til{u}$, where $\til{u} : A \to \til{B}$ is a finite type ring 
homomorphism, and $v : \til{B} \to B$ is a localization ring homomorphism. 
We sometimes use the abbreviation ``EFT'' for ``essentially finite type''. 
\end{dfn}

The definition is depicted in this commutative diagram of rings:
\[ \begin{tikzcd} [column sep = 12ex, row sep = 8ex] 
A
\ar[rr, bend left = 25, start anchor = north east, end anchor = north west, 
"{u}"]
\ar[r, "{\til{u}}"]
&
\til{B}
\ar[r, "{v}"]
&
B
\end{tikzcd} \]

It is easy to see that if $u : A \to B$ and $v : B \to  C$ are EFT ring 
homomorphisms, then so is $v \circ u : A \to C$. 
If $A$ is noetherian and $B$ is an EFT $A$-ring, then $B$ is noetherian too. 
The EFT property is stable under base change: if $u : A \to B$ is an EFT ring 
homomorphism, and if $A \to A'$ is an arbitrary ring homomorphism, then the 
induced ring homomorphism $u' : A' \to A' \ot_A B$ is EFT. 

Recall from Section \ref{sec:recall-dg} that fixing a ring $A$, there is
the category $\cat{Rng} \over A$ of $A$-rings. An object of $\cat{Rng} \over A$ 
is a ring $B$ equipped with a ring homomorphisms $u : A \to B$, 
called the structural homomorphism. The morphisms in $\cat{Rng} \over A$
are the ring homomorphisms  that respect the structural homomorphisms from $A$.

\begin{dfn} \label{dfn:1066}
Let $A$ be a noetherian ring $A$. We denote by $\cat{Rng} \eftover A$ 
the full subcategory of $\cat{Rng} \over A$ on the EFT $A$-rings. Namely, the 
condition is that the structural homomorphism $u : A \to B$ is EFT.  
\end{dfn}

Clearly all the rings in $\cat{Rng} \eftover A$ are noetherian, and all the 
homomorphisms in $\cat{Rng} \eftover A$ are EFT homomorphisms.

For the rest of the current section, and in many subsequent parts of the paper, 
we will invoke the next convention on finiteness. 

\begin{conv} \label{conv:1070}
All rings are assumed to be noetherian, and all ring homomorphisms 
are assumed to be essentially finite type. 
\end{conv}

Despite this convention, in some important definitions and results (such as 
Definition \ref{dfn:675} below) we will mention these finiteness conditions 
explicitly. 

Let $A$ be a ring. At the end of Section \ref{sec:recall-dg} we reviewed the 
full triangulated subcategories 
$\cat{D}^{\star}, \cat{D}^{\star}_{\mrm{f}} \sub \cat{D}(A)$. 
We also recalled the definition of flat dimension of a complex 
$M \in \cat{D}(A)$.

Here is the main definition of our paper. 

\begin{dfn} \label{dfn:675}
Let $A \to B$ be an EFT homomorphism between noetherian rings. 
\begin{enumerate}
\item A {\em rigid complex over $B / A$} is a pair 
$(M, \rho)$, consisting of a complex 
$M \in \cat{D}^{\mrm{b}}_{\mrm{f}}(B)$,
which has finite flat dimension over $A$, and an isomorphism 
$\rho : M \iso \opn{Sq}_{B / A}(M)$
in $\cat{D}(B)$, called a {\em rigidifying isomorphism}. 

\item Let $(M, \rho)$ and $(N, \si)$ be two rigid complexes over $B / A$. 
A {\em rigid morphism} 
$\th : (M, \rho) \to (N, \si)$
over $B / A$ is a morphism $\th : M \to N$ in 
$\cat{D}(B)$, such that the diagram 
\[ \begin{tikzcd} [column sep = 8ex, row sep = 6ex] 
M
\ar[r, "{\rho}", "{\simeq}"']
\ar[d, "{\th}"']
&
\opn{Sq}_{B / A}(M)
\ar[d, "{\opn{Sq}_{B / A}(\th)}"]
\\
N
\ar[r, "{\si}", "{\simeq}"']
&
\opn{Sq}_{B / A}(N)
\end{tikzcd} \]
in $\cat{D}(B)$ is commutative. Here $\opn{Sq}_{B / A}(\th)$ is the 
morphism from Definition \ref{dfn:880}.

\item The category of rigid complexes and rigid morphisms over $B / A$ is 
denoted by $\cat{D}(B)_{\mrm{rig} / A}$. 
\end{enumerate}
\end{dfn}

\begin{exa} \label{exa:675}
Here are some examples of rigid complexes. First a truly silly example: 
the zero complex $M = 0$ is rigid. 

Next, take a nonzero ring $A$, and let $B := A$ and 
$M := A \in \cat{D}^{\mrm{b}}_{\mrm{f}}(A)$.
It is easy to see that $\opn{Sq}_{A / A}(A) = A$, so 
$\rho^{\mrm{tau}}_{A / A} := \opn{id}_A$ is a rigidifying isomorphism, and 
$(A, \rho_{A / A}^{\mrm{tau}})$ is a nonzero object in 
$\cat{D}(A)_{\mrm{rig} / A}$. We call $(A, \rho_{A / A}^{\mrm{tau}})$ the {\em 
tautological rigid complex} over $A / A$. 

More interesting rigid complexes will show up is Sections 
\ref{sec:coinduced}, \ref{sec:twisted-induced} and \ref{sec:induced-rigidity} 
of the paper. 
\end{exa}

\begin{thm} \label{thm:675}
Let $A \to B$ be an EFT homomorphism of noetherian rings, let
$(M, \rho) \in \cat{D}(B)_{\mrm{rig} / A}$, and assume $M$ has the derived 
Morita property over $B$. Then the only automorphism of
$(M, \rho)$ in $\cat{D}(B)_{\mrm{rig} / A}$ is the identity. 
\end{thm}

\begin{proof}
Let $\theta$ be an automorphism of $(M,\rho)$ in $\cat{D}(B)_{\mrm{rig} / A}$.
By Proposition \ref{prop:730} there is a unique element 
$b \in B^{\times}$ such that $\th = b \cd \opn{id}_{M}$ as automorphisms of $M$ 
in $\cat{D}(B)$. 
According to Theorem \ref{thm:672} we have
\begin{equation} \label{eqn:730}
\opn{Sq}_{B / A}(\th) = 
\opn{Sq}_{B / A}(b \cd \opn{id}_{M})=
b^2 \cd \opn{Sq}_{B / A}(\opn{id}_{M}) 
\end{equation}
as automorphism of $M$ in $\cat{D}(B)$. 

Next, by Definition \ref{dfn:675}, we have the following commutative diagrams
\[ \begin{tikzcd} [column sep = 8ex, row sep = 6ex] 
M
\ar[r, "{\rho}", "{\simeq}"']
\ar[d, "{\opn{id}_{M}}"', "{\simeq}"]
&
\opn{Sq}_{B / A}(M)
\ar[d, "{\opn{Sq}_{B / A}(\opn{id}_{M})}", "{\simeq}"']
\\
M
\ar[r, "{\rho}", "{\simeq}"']
&
\opn{Sq}_{B / A}(M)
\end{tikzcd} 
\mspace{50mu}
\begin{tikzcd} [column sep = 8ex, row sep = 6ex] 
M
\ar[r, "{\rho}", "{\simeq}"']
\ar[d, "{\th}"', "{\simeq}"]
&
\opn{Sq}_{B / A}(M)
\ar[d, "{\opn{Sq}_{B / A}(\th)}", "{\simeq}"']
\\
M
\ar[r, "{\rho}", "{\simeq}"']
&
\opn{Sq}_{B / A}(M)
\end{tikzcd} \]
of isomorphisms in $\cat{D}(B)$.
With formula (\ref{eqn:730}) they tell us that 
$b \cd \opn{id}_{M} = b^2 \cd \opn{id}_{M}$.
Once again using Proposition \ref{prop:730}, we get $b = b^2$. 
But $b$ is invertible in $B$, so we have $b = 1$ and $\th = \opn{id}_{M}$. 
\end{proof}

\begin{prop} \label{prop:1490}
In the situation of Proposition \ref{prop:1491}, assume that for every $i$ we 
are given a rigidifying isomorphism 
$\rho_i : M_i \iso \opn{Sq}_{B_i / A}(M_i)$
in $\cat{D}(B_i)$. Then there is a unique  rigidifying isomorphism 
$\rho : M \iso \opn{Sq}_{B / A}(M)$
in $\cat{D}(B)$ s.t.\ the diagram 
\[ \begin{tikzcd} [column sep = 8ex, row sep = 5ex] 
\boplus_i M_i
\ar[r, "{\sum\nolimits_{i} \rho_i}", "{\simeq}"']
\ar[d, "{\sum\nolimits_{i} e_i}"', "{\simeq}"]
&
\boplus_i \opn{Sq}_{B_i / A}(M_i)
\ar[d, "{\sum\nolimits_{i} \opn{Sq}_{p_i / A}(e_i)}", "{\simeq}"']
\\
M
\ar[r, "{\rho}", "{\simeq}"']
&
\opn{Sq}_{B / A}(M)
\end{tikzcd} \]
in $\cat{D}(B)$ is commutative. Moreover, every rigidifying isomorphism 
$\rho : M \iso \opn{Sq}_{B / A}(M)$
arises this way. 
\end{prop}

\begin{proof}
According to Propositions \ref{prop:1475}(2) and \ref{prop:1491}
we know that 
\[ \opn{Hom}_{\cat{D}(B)}(M, \opn{Sq}_{B / A}(M)) \cong 
\prod\nolimits_i \opn{Hom}_{\cat{D}(B_i)}(M_i, \opn{Sq}_{B_i / A}(M_i)) . \]
Under this bijection a rigidifying isomorphism $\rho$ goes to 
a collection $\{ \rho_i \}$ of rigidifying isomorphisms.
\end{proof}

\begin{rem} \label{rem:650}
We take this opportunity to list a few small errors in the paper \cite{Ye4} by 
the third author, and to fix them. 
First, we need to adjust notations. For a DG ring $A$, the strict category
of complexes $\underline{\cat{C}}_{\mrm{str}}({A})$ here is denoted by $\underline{\cat{M}}({A})$ in \cite{Ye4}.
The derived category is $\underline{\cat{D}}({A})$ in both papers.
In Theorem 0.3.3 of that paper, and in the text immediately preceding it, one 
needs to assume that the DG ring homomorphism $v : A' \to A$ is a {\em 
quasi-isomorphism}. This very same correction (or its variant, with the DG ring 
homomorphism $v_k : A_k \to A_{k  -1}$) has to be made in several locations in 
that paper: Setup 5.7, Definition 5.8, 
Setup 6.1, Definition 7.4 and Theorem 7.6. The condition that $v$ is a 
quasi-isomorphism implies that the homomorphism 
$\til{v} : \til{A}' \to \til{A}$ between the corresponding resolutions is also 
a quasi-isomorphism; and this is crucial: it implies that the DG module 
homomorphism $\til{\ga} \ot_{\til{v}} \til{\ga}$ in diagram (\ref{eqn:670}) 
above is a quasi-isomorphism. Without this property we would not be able to 
produce the homomorphism $\til{\ze}$ in that diagram. This repeated error did 
not effect Theorems 0.3.4 and 0.3.5 of \cite{Ye4}, since the quasi-isomorphism 
was built into their assumptions. 

A second error is in the proof of Theorem 3.22 of \cite{Ye4}. Just before 
equation (3.24), where we choose an element $b' \in \til{B}^{-i - 1}$
such that $\d(b') = c$, the element $b'$ should satisfy $v(b') = 0$. 
We forgot to take care of this. Here is the fix: Let 
$K := \opn{Ker}(v)$, so there is an exact sequence 
$0 \to K \to \til{B} \xar{v} B \to 0$ 
in $\cat{C}_{\mrm{str}}(\til{B})$. 
Since $v$ is a quasi-isomorphism, the DG module $K$ is acyclic. Now 
$v(c) = 0$ and $\d(c) = 0$, so $c \in \opn{Z}^{-i}(K)$. Becuase 
$\opn{H}^{-i}(K) = 0$, there exists an element $b' \in K^{-i - 1}$
such that $\d(b') = c$. By construction the element $b'$ satisfies 
$v(b') = 0$. 
 
There are two typographic mistakes in \cite{Ye4} that we wish to point out. 
In Definition 5.6 the expression 
$\opn{Hom}_{B^{\mrm{en}}}(B, \til{I})$
should be changed to 
$\opn{Hom}_{\til{B}^{\mrm{en}}}(B, \til{I})$; 
and in the proof of Theorem 6.11, 5 lines before the end, the expression 
$\opn{R}(y^{\dag} / r)$ should be changed to 
$\opn{R}(s^{\dag} / r)$.

The text between Definition 5.3 and Remark 5.4 does not make sense 
as written. However, if we take $\til{B}/\til{A}$ to be $\til{B}_1/\til{A}_1=\til{B}_2/\til{A}_2$,
namely in the setting of Definition 5.3, then the homomorphism 
$\til{\ot}^{\mrm{en}}:=\til{B}\ot_{\til{A}}\til{B}:\til{M}^{\mrm{en}}_0\to \til{M}^{\mrm{en}}_1$
is what we need there.
\end{rem}

\section{Rigid Backward Morphisms and Coinduced Rigidity}
\label{sec:coinduced}

The main result of this section is Theorem \ref{thm:2031}. It says that 
if $u : B \to C$ is a finite $A$-ring homomorphism, and if $\th : N \to M$ is a 
nondegenerate backward morphism over $u$, then $\opn{Sq}_{u / A}(\th)$ is also 
nondegenerate. This allows us, in Theorems \ref{thm:2030} and \ref{thm:680},
to construct the {\em coinduced rigidifying isomorphism}.

Recall that according to Convention \ref{conv:615} all rings and DG rings are 
commutative. In this section we also adopt Convention \ref{conv:1070}, 
so all rings are noetherian and all ring homomorphisms are EFT. 

We start by extending Definition \ref{dfn:675}(2) to a relative setup.

\begin{dfn} \label{dfn:677}
Let $A$ be a noetherian ring, let $u : B \to C$ be a homomorphism of EFT 
$A$-rings, let $(M, \rho) \in \cat{D}(B)_{\mrm{rig} / A}$, and let
$(N, \si) \in \cat{D}(C)_{\mrm{rig} / A}$. A {\em rigid backward morphism}
\[ \th : (N, \si) \to (M, \rho) \]
in $\cat{D}(B)$ over $u / A$, or over $C / B / A$, is a backward morphism 
$\th : N \to M$ in $\cat{D}(B)$ over $u$, such that the diagram 
\[ \begin{tikzcd} [column sep = 8ex, row sep = 5ex] 
N
\ar[r, "{\si}", "{\simeq}"']
\ar[d, "{\th}"']
&
\opn{Sq}_{C / A}(N)
\ar[d, "{\opn{Sq}_{u / A}(\th)}"]
\\
M
\ar[r, "{\rho}", "{\simeq}"']
&
\opn{Sq}_{B / A}(M)
\end{tikzcd} \]
in $\cat{D}(B)$ is commutative. 
\end{dfn}

The backward morphism $\opn{Sq}_{u / A}(\th)$ in the definition 
is the one from Theorem \ref{thm:632} and formula (\ref{eqn:1500}). 

The next proposition says that the composition of rigid backward morphisms is 
rigid. 

\begin{prop} \label{prop:1495}
Let $A \to B \xar{u} C \xar{v} D$ be homomorphisms of rings. Let 
$(L, \rho) \in \cat{D}(B)_{\mrm{rig} / A}$,
$(M, \si) \in \cat{D}(C)_{\mrm{rig} / A}$ and 
$(N, \tau) \in \cat{D}(D)_{\mrm{rig} / A}$ be rigid complexes. Let 
$\th : (M, \si) \to (L, \rho)$ and $\ze : (N, \tau) \to (M, \si)$ be rigid 
backward morphisms over $u / A$ and $v / A$, respectively. Then 
$\th \circ \ze : (N, \tau) \to (L, \rho)$
is a rigid backward morphism over $(v \circ u) / A$. 
\end{prop}

\begin{proof}
This is an immediate consequence of Proposition \ref{prop:890}.
\end{proof}

Existence and uniqueness of rigid backward morphisms are much harder to prove. 
The key to understanding them is Theorem \ref{thm:2031} below. 

Nondegenerate backward morphisms were introduced in Section 
\ref{sec:recall-dg}. 
Recall that given a ring homomorphism $u : B \to  C$, and complexes
$M \in \cat{D}(B)$ and $N \in \cat{D}(C)$, a backward 
morphism $\th : N \to  M$ in $\cat{D}(B)$ over $u$ is called nondegenerate if 
the corresponding morphism 
\begin{equation} \label{eqn:2105}
\opn{badj}^{\mrm{R}}_{u, M, N}(\th) : N \to \opn{RCInd}_{u}(M) 
= \opn{RHom}_{B}(C, M)
\end{equation}
in $\cat{D}(C)$ is an isomorphism. 
 
\begin{thm}[Nondegeneracy] \label{thm:2031}
Let $A$ be a noetherian ring, let $u : B \to C$ be a finite homomorphism of EFT 
$A$-rings, let $M \in \cat{D}(B)$, let $N \in \cat{D}(C)$, and let 
$\th : N \to  M$ be a nondegenerate backward morphism in $\cat{D}(B)$ over 
$u$. Assume that $M$ and $N$ have finite flat dimensions over $A$. 
Then 
\[ \opn{Sq}_{u / A}(\th) : \opn{Sq}_{C / A}(N) \to \opn{Sq}_{B / A}(M) \]
is a nondegenerate backward morphism in $\cat{D}(B)$ over $u$. 
\end{thm}

The proof of this theorem requires much preparation and several lemmas. 

A DG ring $\til{B}$ is called {\em cohomologically pseudo-noetherian} if 
the graded ring $\opn{H}(\til{B})$ is pseudo-noetherian; namely if the ring
$\opn{H}^0(\til{B})$ is noetherian, and the $\opn{H}^0(\til{B})$-modules 
$\opn{H}^q(\til{B})$ are all finitely generated.  See 
\cite[Definitions 11.4.36 and 11.4.37]{Ye5}.
(If the graded ring $\opn{H}(\til{B})$ is both bounded and pseudo-noetherian, 
then it is noetherian; but when $\opn{H}(\til{B})$ is not bounded,  
it might fail to be noetherian -- cf.\ \cite[Example 11.4.38]{Ye5}.)

For a DG ring $A$ we denote by $A^{\natural}$ the graded ring obtained from $A$ 
by forgetting its differential. Similarly, for a DG $A$-module $M$ we write 
$M^{\natural}$ for the underlying graded $A^{\natural}$-module. 

\begin{lem} \label{lem:605}
Let $\til{B} \to C$ be a homomorphism of DG rings, 
where $\til{B}$ is cohomologically pseudo-noetherian, $C$ is a ring, and 
the ring homomorphism $\opn{H}^0(\til{B}) \to C$ is finite. 
Then there exists a K-flat DG $\til{B}$-ring resolution
$v : \til{C} \to C$ of $C$, such that 
\[ \tag{\dag} \til{C}^{\natural} \cong \bigoplus_{q \leq 0} \, 
\til{B}^{\natural} [-q]^{\oplus \mu_q} \]
as graded $\til{B}^{\natural}$-modules, and the multiplicities $\mu_q$ are 
finite. 
\end{lem}

In other words, $\til{C}$ is pseudo-finite semi-free as a DG 
$\til{B}$-module, see \cite[Definition 11.4.29]{Ye5}.

\begin{proof}
We shall construct $\til{C}$ as the union of an increasing 
sequence of DG $\til{B}$-rings
$F_0(\til{C}) \sub F_1(\til{C}) \sub \cdots$, 
which will be defined recursively. 
At the same time we shall construct a compatible increasing sequence 
of DG $\til{B}$-ring homomorphisms
$v_i : F_i(\til{C}) \to C$,
and an increasing sequence of graded sets 
$F_i(X) \sub F_i(\til{C})$. The homomorphism $v$ will be the 
union of the $v_i$, and the graded set $X$ will be the union of the graded sets 
$F_i(X)$. 

$F_{0}(X) := \emptyset$, and for $i \geq 1$, the following conditions will hold:
\begin{enumerate}
\rmitem{i} $\mrm{H}(v_i) : \mrm{H}(F_i(\til{C})) \to C$
is surjective in degrees $\geq -i$.

\rmitem{ii} $\mrm{H}(v_i) : \mrm{H}(F_i(\til{C})) \to C$
is bijective in degrees $\geq -i + 1$.

\rmitem{iii} $F_i(X) = F_{i - 1}(X) \sqcup Y_i$, where $Y_i$ is a finite graded 
set concentrated in degree $-i$. 

\rmitem{iv} For $i = 0$ the graded ring $F_0(\til{C})^{\natural}$ 
is finite free as a graded $\til{B}^{\natural}$-module, with basis concentrated 
in degree $0$. For $i \geq 1$ we have 
$\mrm{d}(F_i(X)) \sub F_{i-1}(\til{C})$, and
\[ F_i(\til{C})^{\natural} \cong 
F_0(\til{C})^{\natural} \ot_{\Z} \Z[F_i(X)]  \]
as graded $F_0(\til{C})^{\natural}$-rings. 
\end{enumerate}
These conditions imply that the DG ring homomorphism 
$v : \til{C} \to C$ has the required properties. 

The recursive construction is divided into two cases: $i = 0$ and $i \geq 1$. 

\medskip \noindent 
$\vartriangleright$ \hspace{0.3ex} The case $i = 0$. 
Since $C$ is a finite $\opn{H}^0(\til{B})$-ring, it is also a finite 
$\til{B}^0$-ring. Choose elements $c_1, \ldots, c_m \in C$ that generate it as 
a $\til{B}^0$-ring. Since each $c_k$ is integral over $\til{B}^0$, 
there is some monic polynomial 
$p_k(t) \in \til{B}^0[t]$ such that $p_k(c_k) = 0$ in $C$. Let 
$y_1, \ldots, y_m$ be distinct variables of degree $0$. Define the set
$Y_0 := \{ y_1, \ldots, y_m \}$ and the $\til{B}^0$-ring
\[ F_0(\til{C})^0 := \til{B}^0[Y_0] / \big( p_1(y_1), \ldots, p_m(y_m) \big) . 
\]
As a $\til{B}^0$-module it is free  of finite rank. 
Then define the DG $\til{B}$-ring 
$F_0(\til{C}) := \til{B} \ot_{\til{B}^0} F_0(\til{C})^0$.
Let $v_0 : F_0(\til{C}) \to C$ be the 
surjective DG $\til{B}$-ring homomorphism that sends the class of $y_k$ to 
$c_k$. Define $F_0(X) := \emptyset$. 
Then conditions (i)-(iv) hold for $i = 0$.

\medskip \noindent 
$\vartriangleright$ \hspace{0.3ex} The case $i \geq 1$. Now the DG ring $F_{i - 
1}(\til{C})$, the DG ring homomorphism 
$v_{i - 1} : F_{i - 1}(\til{C}) \to C$ and the graded set 
$F_{i - 1}(X)$ have already been constructed, satisfying conditions (i)-(iv)
with index $i - 1$.  
We are going to construct the DG ring $F_{i}(\til{C})$, etc.  

Consider the $\opn{H}^0(\til{B})$-module 
\[  N_{i - 1} := \opn{Ker} \Bigl( \mrm{H}^{-i + 1}(v_{i - 1}) : 
\mrm{H}^{-i + 1} (F_{i - 1} (\til{C})) \to C^{-i + 1} \Bigr) . \]
(Of course $C^{-i + 1} = 0$ for $i \geq 2$.)
Recall that $\til{B}$ is a cohomologically pseudo-noetherian DG ring. 
By construction the $\opn{H}^0(\til{B})$-module 
$\mrm{H}^{-i + 1} (F_{i - 1} (\til{C}))$ is finite, and therefore so is 
$N_{i - 1}$. 
Choose a finite set $Y_{i}$ of degree $-i$ variables, and a degree $1$ 
function 
\begin{equation} \label{eqn:605}
\d : Y_{i} \to \opn{Z}^{-i + 1} (F_{i - 1} (\til{C}))
\end{equation}
such that cohomology classes of the cocycles $\d(y)$, for $y \in Y_{i}$,
generate $N_{i - 1}$ as an $\opn{H}^0(\til{B})$-module. 
 
Define the graded set 
$F_{i}(X) := F_{i - 1}(X) \sqcup Y_{i}$ and the DG ring 
\[ F_{i}(\til{C}) := 
F_0(\til{C}) \otimes_{\Z} \Z[F_i(X)] . \]
The differential $\mrm{d}$ of $F_{i}(\til{C})$ is the unique differential 
that extends the given differential of $F_{i - 1}(\til{C})$ and the function
$\mrm{d} : Y_{i} \to F_{i - 1}(\til{C})$ from (\ref{eqn:605}). 
It is easy to see that conditions (i)-(iv) hold for index $i$.
\end{proof}

\begin{lem} \label{lem2.4} 
Let $A$ be a ring, and let $P$ and $N$ be bounded below complexes of 
$A$-modules. Assume 
that each $P^i$ is a flat $A$-module, and that $N$ has finite 
flat dimension over $A$. Then the canonical morphism 
$P \otimes^{\mrm{L}}_{A} N \to P \otimes_{A} N$
in $\msf{D}(A)$ is an isomorphism. 
\end{lem}

The catch here is that the complex $P$ of flat $A$-modules is bounded {\em 
below}, not above. 

\begin{proof}
Choose a bounded flat resolution $\varphi:Q \to N$ over $A$.
We have to show that 
$\opn{id}_P\otimes\varphi:P \otimes_{A} Q \to P \otimes_{A} N$
is a quasi-isomorphism. 
Let $L$ be the standard cone of the homomorphism $\varphi$
in $\cat{C}_{\mrm{str}}(A)$, see \cite[Definition 4.2.1]{Ye5}.
According to \cite[Definition 4.2.5 and 5.4.2]{Ye5} we get a distinguished triangle 
$Q \xar{\varphi} N  \to L \xar{\bigtriangleup}$
in $\cat{K}(A)$. Applying the DG functor $P \ot_A (-)$ to it gives 
a distinguished triangle 
\[ P \ot_A Q \to P \ot_A N \to P \ot_A L \xar{\bigtriangleup} \]
in $\cat{K}(A)$.
Due to the long exact cohomology sequence, it is enough to show that the 
complex $P \otimes_{A} L$ is acyclic. 

It suffices to prove that for any given an integer $i$ the module 
$\mrm{H}^i (P \otimes_{A} L)$ is zero. 
We know that $L$ and $P$ are bounded below complexes.
Therefore, for any sufficiently large integer $i_1$ the complex 
$P' := \opn{stt}^{\leq i_1}(P)$, the stupid truncation of $P$ below $i_1$
(see \cite[Definition 11.2.11]{Ye5}) satisfies 
$\mrm{H}^i (P \otimes_{A} L) \cong \mrm{H}^i (P' \otimes_{A} L)$. 
Now $P'$ is bounded complex of flat modules, so it is K-flat. 
On the other had, the complex $L$ is acyclic. It follows that 
$P' \otimes_{A} L$ is acyclic, and in particular
$\mrm{H}^i (P' \otimes_{A} L) = 0$. 
\end{proof}

Let $A \to B \xar{u} C$ be as in Theorem \ref{thm:2031}.
In the next lemma we use the shorthand notation
\begin{equation} \label{eqn:2031}
u^{\flat} := \opn{RCInd}_u = \opn{RHom}_{B}(C, -) : \cat{D}(B) \to  
\cat{D}(C)
\end{equation}
for the derived coinduction functor along $u$. 
(The algebro-geometric source for this notation is explained in Remark 
\ref{rem:1555}.)
For every $M \in \cat{D}(B)$ there is a nondegenerate backward morphism
$\opn{tr}^{\mrm{R}}_{u, M} : u^{\flat}(M) \to M$ in $\cat{D}(B)$ over $u$,
and it is functorial in $M$. See Section \ref{sec:recall-dg}.

\begin{lem} \label{lem:2030}
Let $A \to B \xar{u} C$ be as in Theorem \ref{thm:2031}, and
let $M \in \cat{D}(B)$. 
Assume that the complexes 
$M$ and $u^{\flat}(M) \in \cat{D}(C)$ have finite flat 
dimensions over $A$. Then there is an isomorphism 
\[ \ga : 
u^{\flat} \bigl( \opn{Sq}_{B / A}(M) \bigr) \iso 
\opn{Sq}_{C / A} \bigl( u^{\flat}(M) \bigr) \]
in $\cat{D}(C)$, such that the diagram
\[ \tag{$\heartsuit$}
\begin{tikzcd} [column sep = 6ex, row sep = 5ex, every label/.append style={font=\footnotesize}] 
u^{\flat} \bigl( \opn{Sq}_{B / A}(M) \bigr)
\ar[r, "{\ga}", "{\simeq}"']
\ar[d, "{\opn{tr}^{\mrm{R}}_{u, \opn{Sq}_{B / A}(M)}}"']
&
\opn{Sq}_{C / A} \bigl( u^{\flat}(M) \bigr)
\ar[d, "{\opn{Sq}_{u / A}(\opn{tr}^{\mrm{R}}_{u, M})}"]
\\
\opn{Sq}_{B / A}(M)
\ar[r, "{\opn{id}}", "{\simeq}"']
&
\opn{Sq}_{B / A}(M)
\end{tikzcd} \]
in $\cat{D}(B)$ is commutative. 
\end{lem}

\begin{proof}
The proof is divided into a few steps. 

\medskip \noindent
Step 1. We start by choosing DG ring resolutions. 
Let $\til{B} \to B$ be a K-flat DG $A$-ring resolution of $B$, such that 
$\til{B}$ is semi-free as a DG $A$-module; for instance, $\til{B}$ could be a 
commutative semi-free DG ring resolution of $B$, as in 
\cite[Theorem 3.21$\mrm{(1)}$]{Ye4}.
Since $\opn{H}(\til{B}) = B$, we see that the DG ring $\til{B}$ is 
cohomologically pseudo-noetherian, and the ring homomorphism 
$\opn{H}^0(\til{B}) \to C$ is finite. These are the assumptions of Lemma 
\ref{lem:605}. Therefore we can find a K-flat DG $\til{B}$-ring 
resolution $\til{C} \to C$ of $C$ for which 
formula ($\dag$) of that lemma holds. The DG ring homomorphism 
$\til{B} \to \til{C}$ is denoted by $\til{u}$. 
As a DG $\til{B}$-module, $\til{C}$ is  
pseudo-finite semi-free  (see \cite[Definition 11.4.29]{Ye5}).
Of course $A \to \til{B} \xar{\til{u}} \til{C}$ is a K-flat resolution of 
$A \to B \xar{u} C$, in the sense of Definition \ref{dfn:630}.

With these choices, and in view of condition ($**$) in Theorem \ref{thm:632},
it suffices to find an isomorphism
\begin{equation} \label{eqn:2032}
\til{\ga} : u^{\flat} \bigl( \opn{Sq}^{\til{B} / A}_{B / A}(M) \bigr)
\iso 
\opn{Sq}^{\til{C} / A}_{C / A} \bigl( u^{\flat}(M) \bigr)
\end{equation}
in $\cat{D}(C)$, which makes the diagram 
\begin{equation} \label{eqn:764}
\begin{tikzcd} [column sep = 6ex, row sep = 5ex, every label/.append style={font=\footnotesize}] 
u^{\flat} \bigl( \opn{Sq}^{\til{B} / A}_{B / A}(M) \bigr)
\ar[r, "{\til{\ga}}", "{\simeq}"']
\ar[d, "{\opn{tr}^{\mrm{R}}_{u, \opn{Sq}^{\til{B} / A}_{B / A}(M)}}"']
&
\opn{Sq}^{\til{C} / A}_{C / A} \bigl( u^{\flat}(M) \bigr)
\ar[d, "{\opn{Sq}^{\til{u} / A}_{u / A}(\opn{tr}^{\mrm{R}}_{u, M})}"]
\\
\opn{Sq}^{\til{B} / A}_{B / A}(M)
\ar[r, "{\opn{id}}", "{\simeq}"']
&
\opn{Sq}^{\til{B} / A}_{B / A}(M)
\end{tikzcd}
\end{equation}
in $\cat{D}(B)$ commutative. 

\medskip \noindent
Step 2. Now we choose DG module resolutions. 
Since $H(M)$ is bounded, there is $\til{P}' \to M$ bounded above semi-free DG $\til{B}$-module 
resolution of $M$. As a complex of $A$-modules, $\til{P}'$ is a bounded above 
complex of free modules. For each integer $i_0$ there is the smart truncation 
$\opn{smt}^{\geq i_0}(\til{P}')$, which is a bounded DG $\til{B}$-module, 
and it comes with a canonical homomorphism 
$\til{P}' \to \opn{smt}^{\geq i_0}(\til{P}')$
in $\cat{C}_{\mrm{str}}(\til{B})$; see 
\cite[Definition 7.3.6 and Proposition 7.3.8]{Ye5}.
Since $M$ has finite flat dimension over $A$, it follows 
that for a sufficiently small integer $i_0$ the smart truncation 
$\til{P} := \opn{smt}^{\geq i_0}(\til{P}')$ has these properties: the 
canonical homomorphism $\til{P}' \to \til{P}$ is a quasi-isomorphism, and the 
$A$-module $\til{P}^{i_0}$ is flat; cf.\ \cite[Proposition 12.4.19]{Ye5}.
Thus $\til{P}$ is a bounded DG $\til{B}$-module, and all the $\til{P}^{i}$
are flat $A$-modules.
The quasi-isomorphisms $\til{P}' \to M$ and $\til{P}' \to \til{P}$ give rise to 
an isomorphism $M \cong \til{P}$ in $\cat{D}(\til{B})$. 

\medskip \noindent
Step 3. In this step we are going construct the isomorphism $\til{\ga}$.
We have these isomorphisms in $\cat{D}(\til{C})$~:
\begin{equation} \label{eqn:702}
\begin{aligned}
&
u^{\flat}(M) = \opn{RHom}_B(C, M) 
\cong^{\mrm{(i)}}
\opn{RHom}_{\til{B}}(\til{C}, M) 
\\
& \qquad 
\cong^{\mrm{(ii)}}
\opn{RHom}_{\til{B}}(\til{C}, \til{P}) 
\cong^{\mrm{(iii)}}
\opn{Hom}_{\til{B}}(\til{C}, \til{P}) .
\end{aligned}
\end{equation}
The isomorphism $\cong^{\mrm{(i)}}$ is an instance of (\ref{eqn:703}); the 
isomorphism $\cong^{\mrm{(ii)}}$ comes from $M \cong \til{P}$; and the 
isomorphism $\cong^{\mrm{(iii)}}$ is because $\til{C}$ is a K-projective DG 
$\til{B}$-module. 
Note that $\opn{Hom}_{\til{B}}(\til{C}, \til{P})$ is a bounded below complex of 
flat $A$-module and it has finite flat dimension over $A$.

In what follows we are going to employ Convention \ref{conv:1295} regarding 
subscript indexing of repeated tensor factors. 
There is an obvious homomorphism 
\begin{equation} \label{eqn:704}
\opn{Hom}_{\til{B}_1}(\til{C}_1, \til{P}_1) \ot_{\til{A}}
\opn{Hom}_{\til{B}_2}(\til{C}_2, \til{P}_2) \to 
\opn{Hom}_{\til{B}_1 \ot_{\til{A}} \til{B}_2}
(\til{C}_1 \ot_{\til{A}} \til{C}_2 , \til{P}_1 \ot_{\til{A}} \til{P}_2)
\end{equation}
in $\cat{C}_{\mrm{str}}(\til{C}_1 \ot_{\til{A}} \til{C}_2)$. 
Since $\til{C}$ is a pseudo-finite semi-free DG $\til{B}$-module, and 
$\til{P}$ is a bounded DG $\til{B}$-module, 
it follows that 
$\til{C}_1 \ot_{\til{A}} \til{C}_2$ is a pseudo-finite semi-free DG 
$(\til{B}_1 \ot_{\til{A}} \til{B}_2)$-module, and  
$\til{P}_1 \ot_{\til{A}} \til{P}_2$ is a bounded DG 
$(\til{B}_1 \ot_{\til{A}} \til{B}_2)$-module. 
Therefore the homomorphism (\ref{eqn:704}) is actually an isomorphism
in $\cat{C}_{\mrm{str}}(\til{C}_1 \ot_{\til{A}} \til{C}_2)$.

Consider the following sequence of isomorphisms in  $\cat{D}(C)$.
\begin{equation} \label{eqn:600}
\begin{aligned}
& \opn{Sq}^{\til{C} / A}_{C / A} \bigl( u^{\flat}(M) \bigr) = 
\opn{RHom}_{\til{C}_1 \ot_{A} \til{C}_2} 
\bigl( C , u^{\flat}(M)_1 \ot^{\mrm{L}}_{A} u^{\flat}(M)_2 \bigr)
\\
& \qquad 
\cong^{\mrm{(1)}} \opn{RHom}_{\til{C}_1 \ot_{A} \til{C}_2} 
\bigl( C , \opn{Hom}_{\til{B}_1}(\til{C}_1, \til{P}_1) \ot^{\mrm{L}}_{A}
\opn{Hom}_{\til{B}_2}(\til{C}_2, \til{P}_2) \bigr)
\\
& \qquad 
\cong^{\mrm{(2)}} \opn{RHom}_{\til{C}_1 \ot_{A} \til{C}_2} \bigl( C , 
\opn{Hom}_{\til{B}_1}(\til{C}_1, \til{P}_1) \ot_{A}
\opn{Hom}_{\til{B}_2}(\til{C}_2, \til{P}_2) \bigr)
\\
& \qquad 
\cong^{\mrm{(3)}} 
\opn{RHom}_{\til{C}_1 \ot_{A} \til{C}_2} 
\bigl( C, \opn{Hom}_{\til{B}_1 \ot_{A} \til{B}_2}(\til{C}_1 \ot_{A} \til{C}_2, 
\til{P}_1 \ot_{A} \til{P}_2) \bigr)
\\
& \qquad 
\cong^{\mrm{(4)}} 
\opn{RHom}_{\til{C}_1 \ot_{A} \til{C}_2} 
\bigl( C, \opn{RHom}_{\til{B}_1 \ot_{A} \til{B}_2}(\til{C}_1 \ot_{A} \til{C}_2, 
M_1 \ot^{\mrm{L}}_{A} M_2) \bigr) 
\\
& \qquad 
\cong^{\mrm{(5)}}  
\opn{RHom}_{\til{B}_1 \ot_{A} \til{B}_2} (C , M_1 \ot^{\mrm{L}}_{A} M_2) 
\\
& \qquad 
\cong^{\mrm{(6)}}   
\opn{RHom}_{B} \bigl( C, \opn{RHom}_{\til{B}_1 \ot_{A} \til{B}_2}
(B , M_1 \ot^{\mrm{L}}_{A} M_2) \bigr) 
= u^{\flat} \bigl( \opn{Sq}^{\til{B} / A}_{B / A}(M) \bigr) . 
\end{aligned}
\end{equation}
Here are the explanations of these isomorphisms. 
The isomorphism $\cong^{\mrm{(1)}}$ is by the isomorphism 
$u^{\flat}(M) \cong \opn{Hom}_{\til{B}}(\til{C}, \til{P})$
in $\cat{D}(\til{C})$ from equation (\ref{eqn:702}). 
The isomorphism $\cong^{\mrm{(2)}}$ comes from Lemma \ref{lem2.4}; it
applies because the DG $\til{C}$-module 
$\opn{Hom}_{\til{B}}(\til{C}, \til{P})$ 
is a bounded below complex of flat $A$-modules, and it also has finite flat 
dimension over $A$. Isomorphism $\cong^{\mrm{(3)}}$ is due to the isomorphism 
(\ref{eqn:704}). The isomorphism $\cong^{\mrm{(4)}}$ relies on these facts: 
$\til{C}_1 \ot_{A} \til{C}_2$ is K-projective as a DG 
$(\til{B}_1 \ot_{\til{A}} \til{B}_2)$-module, 
and $\til{P}_1 \ot_{A} \til{P}_2 \cong M_1 \ot^{\mrm{L}}_{A} M_2$
in $\cat{D}(\til{B}_1 \ot_{\til{A}} \til{B}_2)$.
The isomorphism $\cong^{\mrm{(5)}}$ is derived Hom-tensor adjunction for 
the DG ring homomorphism 
$\til{B}_1 \ot_{\til{A}} \til{B}_2 \to \til{C}_1 \ot_{\til{A}} \til{C}_2$. 
Finally, the isomorphism $\cong^{\mrm{(6)}}$ is derived Hom-tensor adjunction 
for the DG ring homomorphism 
$\til{B}_1 \ot_{\til{A}} \til{B}_2 \to B$. 

We define the isomorphism $\til{\ga}$ in formula (\ref{eqn:2032}) to be the 
composition of the isomorphisms in (\ref{eqn:600}),
going from bottom to top. 

\medskip \noindent
Step 4. In this step we are going to prove that the isomorphism $\til{\ga}$ 
makes diagram (\ref{eqn:764}) commutative. 
Recall that 
\[ \opn{Sq}^{\til{B} / A}_{B / A}(M) = 
\opn{RHom}_{\til{B}_1 \ot_{A} \til{B}_2} (B , M_1 \ot^{\mrm{L}}_{A} M_2) 
\cong 
\opn{RHom}_{\til{B}_1 \ot_{A} \til{B}_2} (B , \til{P}_1 \ot_{A} \til{P}_2) . \]
Each of the objects in (\ref{eqn:600}) admits an obvious morphism to 
$\opn{Sq}^{\til{B} / A}_{B / A}(M)$
in the category $\cat{D}(B)$. For the first and the last objects these are the 
vertical morphisms appearing in diagram (\ref{eqn:764}); for the first, 
compare to formula (\ref{eqn:890}).
A rather easy calculation shows that the isomorphisms (1)-(6) in 
(\ref{eqn:600}) commute with these obvious morphisms to 
$\opn{Sq}^{\til{B} / A}_{B / A}(M)$. 
It follows that diagram (\ref{eqn:764}) is commutative. 
\end{proof}

\begin{proof}[Proof of Theorem \tup{\ref{thm:2031}}]
Let's use the shorthand $u^{\flat}$ from formula (\ref{eqn:2031}). 
We are given a nondegenerate backward morphism 
$\th : N \to M$ over $u$. This means that there is an isomorphism 
$N \cong u^{\flat}(M)$, and we can replace $\th : N \to  M$ with the standard 
nondegenerate backward morphism 
$\opn{tr}^{\mrm{R}}_{u, M} : u^{\flat}(M) \to M$. 
We need to prove that the backward morphism 
\[  \opn{Sq}_{u / A}(\opn{tr}^{\mrm{R}}_{u, M}) : 
\opn{Sq}_{C / A}(u^{\flat}(M)) \to \opn{Sq}_{B / A}(M) \]
is nondegenerate. 

We have at our disposal the standard nondegenerate backward morphism 
\[ \opn{tr}^{\mrm{R}}_{u, \opn{Sq}_{B / A}(M)} : 
u^{\flat}(\opn{Sq}_{B / A}(M)) \to \opn{Sq}_{B / A}(M) \]
associated to the object 
$\opn{Sq}_{B / A}(M) \in \cat{D}(B)$. This is the left vertical arrow in the 
commutative diagram ($\heartsuit$) in Lemma \ref{lem:2030}. Because the 
horizontal arrows in that diagram are isomorphisms, it follows that the 
right vertical arrow there is also nondegenerate; but that arrow is 
$\opn{Sq}_{u / A}(\opn{tr}^{\mrm{R}}_{u, M})$.
\end{proof}

The next theorem is a generalization of Theorem \ref{thm:675}. 

\begin{thm} \label{thm:2030}
Let $A$ be a noetherian ring, let $u : B \to C$ be a finite homomorphism of EFT 
$A$-rings, let $(M, \rho) \in \cat{D}(B)_{\mrm{rig} / A}$, let
$(N, \si) \in \cat{D}(C)_{\mrm{rig} / A}$, and let 
$\th' : N \to  M$ be a nondegenerate backward morphism in $\cat{D}(A)$ over 
$u$. Assume that $N$ has the derived Morita property over $C$.
Then there exists a unique nondegenerate rigid backward morphism 
\[ \th : (N, \si) \to (M, \rho) \]
over $u$ relative to $A$. 
\end{thm}

\begin{proof}
Consider the diagram 
\begin{equation} \label{eqn:2030}
\begin{tikzcd} [column sep = 8ex, row sep = 5ex] 
N
\ar[r, "{\si}", "{\cong}"']
\ar[d, "{\th'}"']
&
\opn{Sq}_{C / A}(N)
\ar[d, "{\opn{Sq}_{u / A}(\th')}"]
\\
M
\ar[r, "{\rho}", "{\cong}"']
&
\opn{Sq}_{B / A}(M)
\end{tikzcd} 
\end{equation}
in $\cat{D}(B)$, which is not commutative in general. 
According to Theorem \ref{thm:2031} the backward morphism 
$\opn{Sq}_{u / A}(\th')$ is nondegenerate. Proposition 
\ref{prop:1871}(1) tells us that the $C$-module 
$\opn{Hom}_{\cat{D}(B)}(N, \opn{Sq}_{B / A}(M))$
is free of rank $1$, and the nondegenerate backward morphisms 
$\rho \circ \th'$ and 
$\opn{Sq}_{u / A}(\th') \circ \si$ 
are both bases of it. This implies that there is a unique element 
$c \in C^{\times}$ such that 
\begin{equation} \label{eqn:2035}
c \cd (\rho \circ \th') = \opn{Sq}_{u / A}(\th') \circ \si . 
\end{equation}

Define the backward morphism 
\[ \th := c^{-1} \cd \th' : N \to  M \]
in $\cat{D}(B)$ over $u$. It is nondegenerate too. 
Using Theorem \ref{thm:672} and equation 
(\ref{eqn:2035}) we obtain these equalities
\[ \begin{aligned}
&
\opn{Sq}_{u / A}(\th) \circ \si = 
\opn{Sq}_{u / A}(c^{-1} \cd  \th') \circ \si = 
c^{-2} \cd \opn{Sq}_{u / A}(\th') \circ \si 
\\ 
& \qquad 
= c^{-2} \cd c \cd (\rho \circ \th') = \rho \circ (c^{-1} \cd \th') =
\rho \circ \th . 
\end{aligned} \]
We see that $\th$ is a nondegenerate rigid backward morphism. 

The uniqueness of the element $c \in C^{\times}$ implies that $\th$ is the only 
nondegenerate rigid backward morphism $N \to M$. 
\end{proof}

\begin{thm} \label{thm:680}
Let $A$ be a noetherian ring, let $u : B \to C$ be a finite homomorphism 
between EFT $A$-rings, and let 
$(M, \rho) \in \cat{D}(B)_{\mrm{rig} / A}$.
Define the complex 
\[ N := \opn{RCInd}_{u}(M) = \opn{RHom}_B(C, M) \in \cat{D}(C) . \]
Assume that $N$ has finite flat dimension over $A$. 
Then $N$ has a unique rigidifying isomorphism 
$\si : N \iso \opn{Sq}_{C / A}(N)$
in $\cat{D}(C)$ such that 
\[ \opn{tr}^{\mrm{R}}_{u, M} : (N, \si) \to (M, \rho) \]
is a nondegenerate rigid backward morphism over $u$ relative to $A$. 
\end{thm}

\begin{proof} 
In the proof we are going to use the shorthand $u^{\flat}$ from 
formula (\ref{eqn:2031}). 
Consider the solid diagram
\begin{equation} \label{eqn:706}
\begin{tikzcd} [column sep = 10ex, row sep = 5ex] 
u^{\flat}(M)
\ar[r, dashed, "{\si}", "{\simeq}"']
\ar[d, "{\opn{tr}^{\mrm{R}}_{u, M}}"']
&
\opn{Sq}_{C / A}(u^{\flat}(M))
\ar[d, "{\opn{Sq}_{u / A}(\opn{tr}^{\mrm{R}}_{u, M})}"]
\\
M
\ar[r, "{\rho}", "{\simeq}"']
&
\opn{Sq}_{B / A}(M)
\end{tikzcd}
\end{equation}
in $\cat{D}(B)$. The standard backward morphism 
$\opn{tr}^{\mrm{R}}_{u, M}$ is nondegenerate. 
The complex $M$ has finite flat dimension over $A$ by Definition 
\ref{dfn:675}(1), and for the complex $N = u^{\flat}(M)$ this condition is 
assumed. According to Theorem \ref{thm:2031} the backward morphism 
$\opn{Sq}_{u / A}(\opn{tr}^{\mrm{R}}_{u, M})$
is also nondegenerate. By Proposition \ref{prop:2045} there is a unique 
morphism $\si$ on the dashed arrow that makes diagram (\ref{eqn:706}) 
commutative, and this morphism $\si$ is an isomorphism. 
We obtain a rigid complex $(N, \si)$ over $C$ relative to $A$. 
The commutativity of diagram (\ref{eqn:706}) means that 
$\opn{tr}^{\mrm{R}}_{u, M}$ is a rigid backward morphism.
\end{proof}

\begin{dfn}[Coinduced Rigidity] \label{dfn:760}
Under the assumptions and notation of Theorem \ref{thm:680}, the rigidifying 
isomorphism 
$\si : N \iso \opn{Sq}_{C / A}(N)$ 
from Theorem \ref{thm:680}, in the category $\cat{D}(C)$, is called the {\em 
rigidifying isomorphism coinduced from $\rho$}, 
and it is denoted by 
$\opn{RCInd}^{\mrm{rig}}_{u / A}(\rho)$. 

The rigid complex 
\[ \opn{RCInd}^{\mrm{rig}}_{u / A} (M, \rho) := (N, \si) \in 
\cat{D}(C)_{\mrm{rig} / A} \]
is called the {\em rigid complex over $C / A$  coinduced from $(M, \rho)$ along 
$u$}. 
\end{dfn}

We already saw the tautological rigid complex 
$(A, \rho^{\mrm{tau}}_{A / A}) \in \cat{D}(A)_{\mrm{rig} / A}$.
Here is another example of a nonzero rigid complex. 

\begin{exa} \label{exa:1505}
Let $u : A \to B$ be a finite ring homomorphism, and asssume that $B$ is a 
nonzero ring, and it has finite flat dimension as an $A$-module. 
As a complex of $A$-modules, $B$ is perfect.
(Note that if $A$ is a regular ring, then $B$ 
is automatically of finite flat dimension over $A$.)
Then the complex $M :=  \opn{RCInd}_{B / A}(A) \in \cat{D}(B)$ is perfect 
over $A$, and so it has finite flat 
dimension over $A$. By Theorem \ref{thm:680} there is a 
unique rigidifying isomorphism 
\[ \rho := \opn{RCInd}^{\mrm{rig}}_{u / A}(\rho^{\mrm{tau}}_{A / A}) 
: M \iso \opn{Sq}_{B / A}(M) \]
in $\cat{D}(B)$, such that 
$\opn{tr}^{\mrm{R}}_{B / A} : (M, \rho) \to (A, \rho^{\mrm{tau}}_{A / A})$
is a nondegenerate rigid backward morphism over $B / A$. 
We get a nonzero rigid complex 
$(M, \rho) \in \cat{D}(B)_{\mrm{rig} / A}$. 
\end{exa}

A special case of the last example deserves a definition:

\begin{dfn} \label{dfn:1505}
Let $A$ be a noetherian ring, and let $u : A \to B$ be a finite flat ring 
homomorphism. Define the 
{\em dual module of $B$ relative to $A$} to be 
\[ \De^{\mrm{fifl}}_{B / A} := \opn{Hom}_A(B, A) = 
\opn{RCInd}_{u}(A) \in \cat{M}(A) . \]
Its coinduced rigidifying isomorphism is denoted by 
\[ \rho_{B / A}^{\mrm{fifl}} :=
\opn{RCInd}^{\mrm{rig}}_{u / A}(\rho^{\mrm{tau}}_{A / A})
: \De^{\mrm{fifl}}_{B / A} \iso \opn{Sq}_{B / A}(\De^{\mrm{fifl}}_{B / A}) . \]
The rigid complex 
$\bigl( \De^{\mrm{fifl}}_{B / A}, \rho_{B / A}^{\mrm{fifl}} \bigr) 
\in \cat{D}(B)_{\mrm{rig} / A}$
is called the {\em standard rigid complex of the finite flat homomorphism 
$u : A \to B$}. 
As we shall see later, in Section \ref{sec:relative}, this is a relative rigid 
dualizing compex. 
\end{dfn}

\section{The Cup Product Morphism and the Derived Tensor Product of Rigid 
Complexes}
\label{sec:cup-prod}

In this section we introduce the {\em cup product morphism for the squaring 
operation} (Theorem \ref{thm:780}), and prove it is an isomorphism 
under some finiteness conditions (Theorem \ref{thm:810}). In 
Definition \ref{dfn:1260}
we introduce the {\em tensor product of rigid complexes}.
The tensor product of rigid complexes will be used in Section 
\ref{sec:twisted-induced} to construct 
twisted induced rigid complexes. 

Recall that the category $\cat{Rng}$ of commutative rings is a full subcategory 
of the category $\cat{DGRng}$ of commutative DG rings
(Definition \ref{dfn:616}(3)). According to Convention \ref{conv:615}, all 
rings and DG rings are commutative by default.
At the start of this section (until Theorem \ref{thm:810} ) we do not impose 
any finiteness conditions on DG rings and DG module. 
 
We are going to adhere to Convention \ref{conv:1295}, regarding the subscript 
enumeration of repeated tensor factors. For instance, given a DG ring 
homomorphism $\til{B} \xar{} \til{C}$, there is the multiplication DG ring 
homomorphism
\[ \opn{mult}_{\til{C} / \til{B}} : 
\til{C}_1 \ot_{\til{B}_0} \til{C}_2 \to \til{C}_0 ,
\quad \til{c}_1 \ot \til{c}_2 \mapsto \til{c}_1 \cd \til{c}_2 . \] 

From here until Definition \ref{1510} (inclusive), we work in the 
following setup: 

\begin{setup} \label{set:1395} 
There is a triple of rings $C / B / A = (A \xar{u} B \xar{v} C)$, as 
in Definition \ref{dfn:630}(2); a K-flat triple of DG rings 
$\til{C} / \til{B} / \til{A} = (\til{A} \xar{\til{u}} \til{B} \xar{\til{v}} 
\til{C})$;
and a resolution 
$t / s /r : \til{C} / \til{B} / \til{A} \to C / B / A$
of triples, as in Definition \ref{dfn:630}(3). See diagram (\ref{eqn:2055}).
\end{setup}

What Setup \ref{set:1395} says is that there is a commutative diagram 
\begin{equation} \label{eqn:2055}
\begin{tikzcd} [column sep = 6ex, row sep = 5ex] 
\til{A}
\arrow[r, "{\til{u}}"]
\ar[d, "{r}"]
&
\til{B}
\ar[d, "{s}"]
\arrow[r, "{\til{v}}"]
&
\til{C}
\ar[d, "{t}"]
\\
A
\arrow[r, "{u}"]
&
B
\arrow[r, "{v}"]
&
C
\end{tikzcd} 
\end{equation}
in $\cat{DGRng}$, such that $\til{B}$ is K-flat as a DG $\til{A}$-module, 
$\til{C}$ is K-flat as a DG $\til{B}$-module, and hence 
$\til{C}$ is K-flat as a DG $\til{A}$-module. The homomorphisms 
$r, s, t$ are surjective quasi-isomorphisms in $\cat{DGRng}$.  

Let us define the auxiliary DG ring  
\begin{equation} \label{eqn:2056}
\til{D}^{} := (\til{C}_1 \ot_{\til{A}} \til{C}_2) 
\ot_{\til{B}_1 \ot_{\til{A}} \til{B}_2} \til{B}^{}_0 .
\end{equation}
The DG ring homomorphism
\begin{equation} \label{eqn:1510}
g : \til{D}^{} 
\xar{} \til{C}_1 \ot_{\til{B}_0} \til{C}_2 , \quad 
(\til{c}_1 \ot \til{c}_2) \ot \til{b}_0 \mapsto 
\til{c}_1 \ot (\til{c}_2 \cd \til{b}_0) =  
(-1)^{k_2 \cd j} \cd (\til{c}_1 \cd \til{b}_0) \ot \til{c}_2 ,
\end{equation}
for $\til{c}_i \in \til{C}^{k_i}$ and $\til{b}_0 \in \til{C}^{j}$,
is an isomorphism. Hence there is an equivalence of triangulated categories 
\begin{equation} \label{eqn:1446}
\opn{Rest}_{g} : 
\cat{D}(\til{C}_1 \ot_{\til{B}_0} \til{C}_2) 
\xar{} \cat{D}(\til{D}^{}) , 
\end{equation}
which respects the operations 
$\opn{RHom}(-, -)$ and $(- \ot^{\mrm{L}} -)$.
See formulas (\ref{eqn:703}) and (\ref{eqn:719}).
The reason we need the auxiliary DG ring $\til{D}$ is explained in Remark 
\ref{rem:1445}. 

Given a DG module $M \in \cat{C}(\til{B}_1 \ot_{\til{A}} \til{B}_2)$,
consider the DG $\til{A}$-module 
\begin{equation} \label{eqn:2047}
F := \opn{Hom}_{\til{B}_1 \ot_{\til{A}} \til{B}_2}(B_0, M) \in \cat{C}(A) . 
\end{equation}
Since $B_0$ is generated as a DG $(\til{B}_1 \ot_{\til{A}} \til{B}_2)$-module 
by the element $1_{B_0}$, an element $f \in F$ is completely determined by 
its value $f(1_{B_0}) \in M$.
The DG $\til{A}$-module $F$ has three actions by the DG ring 
$\til{B} = \til{B}_0 = \til{B}_1 = \til{B}_2$. 

\begin{lem} \label{lem:2110}
The three actions of the DG ring $\til{B}$ on the DG $\til{A}$-module 
$F$ from formula (\ref{eqn:2047}) are equal. 
\end{lem}

\begin{proof}
Take homogeneous elements $\til{b} \in \til{B}^k$ and $f \in F^j$. 
For $i = 0, 1, 2$ we shall write $\til{b}_i := b$ for the element $b$ sitting 
in the copy $\til{B}_i$ of $\til{B}$. It suffices to calculate the element 
$(\til{b}_i \cd f)(1_{B_0}) \in M$.
For $i = 0$ we have 
$(\til{b}_0 \cd f) (1_{B_0}) = (-1)^{k \cd j} \cd f(\til{b})$. 
Next, because $f$ is $(\til{B}_1 \ot_{\til{A}} \til{B}_2)$-linear, we get 
\[ (\til{b}_1 \cd f) (1_{B_0}) = ((\til{b} \ot 1_{B_0}) \cd f)(1_{B_0}) = 
(-1)^{k \cd j} \cd f( (\til{b} \ot 1_{B_0}) \cd 1_{B_0}) = 
(-1)^{k \cd j} \cd f(\til{b}) \]
and 
\[ (\til{b}_2 \cd f) (1_{B_0}) = ((1_{B_0} \ot \til{b}) \cd f)(1_{B_0}) = 
(-1)^{k \cd j} \cd f( (1_{B_0} \ot \til{b}) \cd 1_{B_0}) = 
(-1)^{k \cd j} \cd f(\til{b}) . \qedhere \]
\end{proof}

Given complexes $M \in \cat{D}(B)$ and $N \in \cat{D}(C)$, we have the objects
$M_1 \ot^{\mrm{L}}_{\til{A}} M_2 \in \cat{D}(\til{B}_1 \ot_{\til{A}} 
\til{B}_2)$ 
and 
\begin{equation} \label{eqn:2057}
(N_1 \ot^{\mrm{L}}_{\til{B}_0} N_2) \ot^{\mrm{L}}_{\til{B}_0}
\opn{RHom}_{\til{B}_1 \ot_{\til{A}} \til{B}_2}
(B_0, M_1 \ot^{\mrm{L}}_{\til{A}} M_2)  \in 
\cat{D}(\til{D}) ,
\end{equation}
where $\til{D}$ is the DG ring from equation (\ref{eqn:2056}).
The action of the DG ring $\til{B}_0$ on the object in 
(\ref{eqn:2057}) is through $B_0$, and the DG ring 
$\til{C}_1 \ot_{\til{A}} \til{C}_2$ acts through 
$N_1 \ot^{\mrm{L}}_{\til{B}_0} N_2$. 
We can give an explicit presentation of the object in (\ref{eqn:2057}), with 
its various actions, by choosing DG module resolutions. Let $\til{P} \to M$ be 
a K-projective DG module resolution over 
$\til{B}$, let 
$\til{P}_1 \ot_{\til{A}} \til{P}_2 \to \til{I}$
be a K-injective DG module resolution over 
$\til{B}_1 \ot_{\til{A}} \til{B}_2$, 
and let $\til{Q} \to N$ be a K-projective DG module resolution over 
$\til{C}$. Note that $\til{Q}$ is a K-flat DG $\til{B}$-module.
Then there is a canonical isomorphism
\begin{equation} \label{eqn:2058}
\begin{aligned}
&
(N_1 \ot^{\mrm{L}}_{\til{B}_0} N_2) \ot^{\mrm{L}}_{\til{B}_0}
\opn{RHom}_{\til{B}_1 \ot_{\til{A}} \til{B}_2}
(B_0, M_1 \ot^{\mrm{L}}_{\til{A}} M_2) 
\\ & \qquad 
\cong (\til{Q}_1 \ot_{\til{B}_0} \til{Q}_2)  \ot_{\til{B}_0}
\opn{Hom}_{\til{B}_1 \ot_{\til{A}} \til{B}_2}(B_0, \til{I}) .
\end{aligned}
\end{equation}
in $\cat{D}(\til{D})$.

\begin{lem} \label{lem:2055}
With the DG module resolutions chosen above, there is a unique isomorphism 
\[ \tag{$\divideontimes$}
\begin{aligned}
&
\phi_{\til{Q}, \til{I}} : 
(\til{Q}_1 \ot_{\til{B}_0} \til{Q}_2)  \ot_{\til{B}_0}
\opn{Hom}_{\til{B}_1 \ot_{\til{A}} \til{B}_2}(B_0, \til{I}) 
\\ & \qquad 
\iso 
\til{Q}_1 \ot_{\til{B}_1} \til{Q}_2 \ot_{\til{B}_2} 
\opn{Hom}_{\til{B}_1 \ot_{\til{A}} \til{B}_2}(B_0, \til{I}) 
\end{aligned} \] 
in $\cat{C}_{\mrm{str}}(\til{D})$, with formula 
\[ \phi_{\til{Q}, \til{I}} 
\bigl( \til{q}_1 \ot \til{q}_2) \ot \chi \bigr) := 
\til{q}_1 \ot \til{q}_2 \ot \chi  \]
for $\til{q}_i \in \til{Q}$ and 
$\chi \in \opn{Hom}_{\til{B}_1 \ot_{\til{A}} \til{B}_2}(B_0, \til{I})$.
\end{lem}
 
\begin{proof}
If we replace the operations $(- \ot_{\til{B}_i} -)$ in formula 
($\divideontimes$) with $(- \ot_{\til{A}} -)$,
then both objects there 
become the same, and the isomorphism, let's denote it by 
$\phi'_{\til{Q}, \til{I}}$, is the identity 
automorphism of this object. The passage from $(- \ot_{\til{A}} -)$
to $(- \ot_{\til{B}_i} -)$ amounts to passing to quotients modulo the relations 
arising from the various actions of the DG rings $\til{B}_i$, $i = 0, 1, 2$. 
According to Lemma \ref{lem:2110}  these three actions of the DG ring
$\til{B} = \til{B}_i$ are 
equal. Therefore the relations are also equal, and the isomorphism 
$\phi'_{\til{Q}, \til{I}}$ induces an isomorphism 
$\phi_{\til{Q}, \til{I}}$.
\end{proof}

\begin{lem} \label{lem:1431} 
Let $M \in \cat{D}(B)$ and $N \in \cat{D}(C)$.
There is a unique isomorphism 
\[ \tag{\dag} 
\begin{aligned}
&
\phi^{\lsp \mrm{L}}_{M, N} : 
(N_1 \ot^{\mrm{L}}_{\til{B}_0} N_2) \ot^{\mrm{L}}_{\til{B}_0}
\opn{RHom}_{\til{B}_1 \ot_{\til{A}} \til{B}_2}
(B_0, M_1 \ot^{\mrm{L}}_{\til{A}} M_2) 
\\ & \qquad 
\iso 
N_1 \ot^{\mrm{L}}_{\til{B}_1} \bigr( N_2 \ot^{\mrm{L}}_{\til{B}_2} 
\opn{RHom}_{\til{B}_1 \ot_{\til{A}} \til{B}_2}
(B_0, M_1 \ot^{\mrm{L}}_{\til{A}} M_2) \bigr)
\end{aligned} \]
in $\cat{D}(\til{D}^{})$, such that for every 
K-projective DG module resolution $\til{P} \to M$ over $\til{B}$,
K-injective DG module resolution
$\til{P}_1 \ot_{\til{A}} \til{P}_2 \to \til{I}$
over $\til{B}_1 \ot_{\til{A}} \til{B}_2$, 
and K-projective DG module resolution $\til{Q} \to N$ over $\til{C}$, 
the isomorphism $\phi^{\lsp \mrm{L}}_{M, N}$ 
in $\cat{D}(\til{D})$ is presented by the isomorphism 
$\phi_{\til{Q}, \til{I}}$ in $\cat{C}_{\mrm{str}}(\til{D})$
from Lemma \ref{lem:2055}.
\end{lem}

\begin{proof}
Since the resolutions $\til{P} \to M$, 
$\til{P}_1 \ot_{\til{A}} \til{P}_2 \to \til{I}$
and $\til{Q} \to N$ are unique up to homotopy equivalences, which are 
themselves unique up to homotopy, the isomorphisms 
$\phi_{\til{Q}, \til{I}}$ are unique up to homotopy. Therefore the 
resulting isomorphisms 
$\phi^{\lsp \mrm{L}}_{M, N} = 
\opn{Q}(\phi_{\til{Q}, \til{I}})$ 
in $\cat{D}(\til{D})$ are all equal.
\end{proof}

\begin{rem} \label{rem:1445}
The reader might wonder at this point why we need the DG ring 
 $\til{D}^{}$ at all. The reason is that the second object in formula 
($\dag$), as written, does not belong to the derived category 
$\cat{D}(\til{C}_1 \ot_{\til{B}_0} \til{C}_2)$; 
it is just an object of the derived category 
$\cat{D}(\til{D})$. Because of the equivalence (\ref{eqn:1446}) 
we can make certain operations in $\cat{D}(\til{D})$, and later transfer them 
to $\cat{D}(\til{C}_1 \ot_{\til{B}_0} \til{C}_2)$. This is what will be done in 
the following lemma.
\end{rem}

The {\em derived tensor evaluation morphism} for commutative DG rings
was constructed in \cite[Theorem 12.10.14]{Ye5}.

\begin{lem} \label{lem:1430}
Let $M \in \cat{D}(B)$ and $N \in \cat{D}(C)$.
There is a unique morphism 
\[ \tag{$\heartsuit$} 
\begin{aligned}
&
(N_1 \ot^{\mrm{L}}_{\til{B}_0} N_2) \ot^{\mrm{L}}_{\til{B}_0}
\opn{RHom}_{\til{B}_1 \ot_{\til{A}} \til{B}_2}
(B_0, M_1 \ot^{\mrm{L}}_{\til{A}} M_2) 
\\ & \quad 
\to 
\opn{RHom}_{\til{C}_1 \ot_{\til{A}} \til{C}_2}
\bigl( \til{C}_1 \ot_{\til{B}_0} \til{C}_2, 
(M_1 \ot^{\mrm{L}}_{B_1} N_1) 
\ot^{\mrm{L}}_{\til{A}} (M_2 \ot^{\mrm{L}}_{B_2} N_2) \bigr)
\end{aligned} \]
in $\cat{D}(\til{C}_1 \ot_{\til{B}_0} \til{C}_2)$, 
whose image under the equivalence $\opn{Rest}_{g}$ from formula 
(\ref{eqn:1446}) is the composed morphism 
\[ \tag{$\til{\heartsuit}^{}$} 
\begin{aligned}
& (N_1 \ot^{\mrm{L}}_{\til{B}_0} N_2) \ot^{\mrm{L}}_{\til{B}_0}
\opn{RHom}_{\til{B}_1 \ot_{\til{A}} \til{B}_2}
(B_0, M_1 \ot^{\mrm{L}}_{\til{A}} M_2)
\\
& \quad \iso^{\mrm{(a)}}
N_1 \ot^{\mrm{L}}_{\til{B}_1} \bigl( N_2 \ot^{\mrm{L}}_{\til{B}_2}
\opn{RHom}_{\til{B}_1 \ot_{\til{A}} \til{B}_2}
(\til{B}_0, M_1 \ot^{\mrm{L}}_{\til{A}} M_2) \bigr)
\\
& \quad \to^{\mrm{(b)}}
N_1 \ot^{\mrm{L}}_{\til{B}_1} 
\opn{RHom}_{\til{B}_1 \ot_{\til{A}} \til{B}_2}
\bigl( \til{B}_0, M_1 \ot^{\mrm{L}}_{\til{A}} (M_2 \ot^{}_{\til{B}_2} N_2) 
\bigr)
\\
& \quad \to^{\mrm{(c)}}
\opn{RHom}_{\til{B}_1 \ot_{\til{A}} \til{B}_2}
\bigl( \til{B}_0, (M_1 \ot^{\mrm{L}}_{\til{B}_1} N_1)
\ot^{\mrm{L}}_{\til{A}} (M_2 \ot^{}_{\til{B}_2} N_2) \bigr)
\\
& \quad \iso^{\mrm{(d)}}
\opn{RHom}_{\til{C}_1 \ot_{\til{A}} \til{C}_2}
\bigl( \til{C}_1 \ot_{\til{B}_0} \til{C}_2, 
(M_1 \ot^{\mrm{L}}_{\til{B}_1} N_1) 
\ot^{\mrm{L}}_{\til{A}} (M_2 \ot^{\mrm{L}}_{\til{B}_2} N_2) \bigr) 
\\
& \quad \iso^{\mrm{(e)}}
\opn{RHom}_{\til{C}_1 \ot_{\til{A}} \til{C}_2}
\bigl( \til{C}_1 \ot_{\til{B}_0} \til{C}_2, 
(M_1 \ot^{\mrm{L}}_{B_1} N_1) 
\ot^{\mrm{L}}_{\til{A}} (M_2 \ot^{\mrm{L}}_{B_2} N_2) \bigr)
\end{aligned} \]
in $\cat{D}(\til{D}^{})$.
Here the isomorphism $\iso^{\mrm{(a)}}$ is the isomorphism 
$\phi^{\lsp \mrm{L}}_{M, N}$from Lemma \ref{lem:1431}.
The morphisms $\to^{\mrm{(b)}}$ and $\to^{\mrm{(c)}}$ are derived tensor 
evaluation morphisms. The isomorphism 
$\iso^{\mrm{(d)}}$ is by adjunction for the DG ring homomorphism 
$\til{B}_1 \ot_{\til{A}} \til{B}_2 \to \til{C}_1 \ot_{\til{A}} \til{C}_2$,
like in formula(\ref{eqn:1480}), together with the isomorphism $g$ from 
formula (\ref{eqn:1510}). 
Finally, the  isomorphism $\iso^{\mrm{(e)}}$ is due to the isomorphism 
$M \ot^{\mrm{L}}_{\til{B}} N \cong M \ot^{\mrm{L}}_{B} N$
in $\cat{D}(\til{C})$, that we have because of the quasi-isomorphism 
$\til{B} \to B$, combined with formula (\ref{eqn:719}). 
\end{lem}

To clarify, the action of the DG ring 
$\til{C}_1 \ot_{\til{B}_0} \til{C}_2$
on the first object in ($\heartsuit$) is through 
$N_1 \ot^{\mrm{L}}_{\til{B}_0} N_2$,
and its action on the second object in ($\heartsuit$) is through itself.

\begin{proof}
The only fact we need to stress is that the first (resp.\ second) object 
appearing in ($\heartsuit$) is sent by the equivalence $\opn{Rest}_{g}$
to the first (resp.\ last) object appearing in 
$\til{\heartsuit}^{}$.
\end{proof}

\begin{dfn} \label{1510}
Let $M \in \cat{D}(B)$ and $N \in \cat{D}(C)$. 
Define the morphism
\[ \opn{cup}^{\til{C} / \til{B} / \til{A}}_{C / B / A, M, N}  \, : \, 
\opn{Sq}^{\til{B} / \til{A}}_{B / A}(M) \ot^{\mrm{L}}_{B}
\opn{Sq}^{\til{C} / \til{B}}_{C / B}(N) \, \to \,
\opn{Sq}^{\til{C} / \til{A}}_{C / A}(M \ot^{\mrm{L}}_{B} N) \]
in $\cat{D}(C)$, called the {\em resolved cup product morphism}, to be the 
composition of all the morphisms in $\cat{D}(C)$ in formula ($\til{\lozenge}$) 
below.
\[ \tag{$\til{\lozenge}$}
\begin{aligned}
& \opn{Sq}^{\til{B} / \til{A}}_{B / A}(M) \ot^{\mrm{L}}_{B}
\opn{Sq}^{\til{C} / \til{B}}_{C / B}(N) 
\iso^{(\mrm{1})} 
\opn{Sq}^{\til{B} / \til{A}}_{B / A}(M) \ot^{\mrm{L}}_{\til{B}}
\opn{Sq}^{\til{C} / \til{B}}_{C / B}(N)
\\
& \quad =
\opn{Sq}^{\til{B} / \til{A}}_{B / A}(M) \ot^{\mrm{L}}_{\til{B}_0}
\opn{RHom}_{\til{C}_1 \ot_{\til{B}_0} \til{C}_2}
(C_0, N_1 \ot^{\mrm{L}}_{\til{B}_0} N_2)
\\
& \quad \xar{\mspace{12mu}}^{(\mrm{2})} 
\opn{RHom}_{\til{C}_1 \ot_{\til{B}_0} \til{C}_2} 
\bigl( C_0, (N_1 \ot^{\mrm{L}}_{\til{B}_0} N_2) \ot^{\mrm{L}}_{\til{B}_0}
\opn{Sq}^{\til{B} / \til{A}}_{B / A}(M) \bigr)
\\
& \quad \xar{\mspace{12mu}}^{(\mrm{3})} 
\opn{RHom}_{\til{C}_1 \ot_{\til{B}_0} \til{C}_2} 
\bigl( C_0, \opn{RHom}_{\til{C}_1 \ot_{\til{A}} \til{C}_2} 
(\til{C}_1 \ot_{\til{B}_0} \til{C}_2, 
L_1 \ot^{\mrm{L}}_{\til{A}} L_2) \bigr) 
\\
& \quad \iso^{(\mrm{4})} 
\opn{RHom}_{\til{C}_1 \ot_{\til{A}} \til{C}_2} 
(C_0, L_1 \ot^{\mrm{L}}_{\til{A}} L_2) 
= \opn{Sq}^{\til{C} / \til{A}}_{C / A}(L) 
= \opn{Sq}^{\til{C} / \til{A}}_{C / A}(M \ot^{\mrm{L}}_{B} N) . 
\end{aligned} \]
Here $L := M \ot_B^{\mrm{L}} N \in \cat{D}(C)$. 
The morphisms in formula ($\til{\lozenge}$) are as follows. 
The isomorphism $\iso^{(\mrm{1})}$ comes from the DG ring quasi-isomorphism 
$\til{B} \to B$, together with formula (\ref{eqn:719}). 
The morphism $\xar{\mspace{12mu}}^{(\mrm{2})}$ 
is an instance of the derived 
tensor-evaluation morphism, see \cite[Theorem 12.10.14]{Ye5}. 
The morphism $\to^{(\mrm{3})}$ is is gotten by applying 
$\opn{RHom}_{\til{C}_1 \ot_{\til{B}_0} \til{C}_2}(C_0, -)$
to the morphism ($\heartsuit$) from Lemma \ref{lem:1430}.
Lastly, the isomorphism $\iso^{(\mrm{4})}$ is by adjunction for the DG ring 
homomorphism 
$\til{C}_1 \ot_{\til{A}} \til{C}_2 \to \til{C}_1 \ot_{\til{B}_0} \til{C}_2$,
see formula (\ref{eqn:1480}).
\end{dfn}

The name ``resolved cup product morphism'' indicates that it is dependent on the
K-flat triple of DG rings resolution $\til{C} / \til{B} / \til{A}$
of $C / B / A$. The next theorem removes this dependence. 

\begin{thm}[Cup Product Morphism] \label{thm:780}
Let $C / B / A$ be a triple of rings, let $M \in \cat{D}(B)$, and   
let $N \in \cat{D}(C)$. 
There is a unique morphism 
\[ \opn{cup}_{C / B / A, M, N} \, : \, 
\opn{Sq}_{B / A}(M) \ot^{\mrm{L}}_{B}
\opn{Sq}_{C / B}(N) \, \to \,
\opn{Sq}_{C / A}(M \ot^{\mrm{L}}_{B} N)  \]
in $\cat{D}(C)$, called the {\em cup product morphism}, with this 
property\tup{:}
\begin{itemize}
\item[($\diamondsuit$)]
Given a K-flat DG ring resolution $\til{C} / \til{B} / \til{A}$ of the triple 
of 
rings $C / B / A$, the diagram 
\[ \begin{tikzcd} [column sep = 18ex, row sep = 6ex] 
\opn{Sq}_{B / A}(M) \ot^{\mrm{L}}_{B} \opn{Sq}_{C / B}(N) 
\arrow[r, "{\opn{cup}_{C / B / A, M, N}}"]
\arrow[d, "{ \opn{sq}_{B / A, M}^{\til{B} / \til{A}} \ot^{\mrm{L}}_{B} 
\opn{sq}_{C / B, N}^{\til{C} / \til{B}} }", "{\simeq}"']
&
\opn{Sq}_{C / A} (M \ot^{\mrm{L}}_{B} N)
\arrow[d, "{ \opn{sq}_{C / A, M \ot^{\mrm{L}}_{B} N}^{\til{C} / \til{A}} }", 
"{\simeq}"']
\\
\opn{Sq}^{\til{B} / \til{A}}_{B / A}(M) \ot^{\mrm{L}}_{B} 
\opn{Sq}^{\til{C} / \til{B}}_{C / B}(N) 
\arrow[r, "{\opn{cup}^{\til{C} / \til{B} / \til{A}}_{C / B / A, M, N}}"]
&
\opn{Sq}^{\til{C} / \til{A}}_{C / A} (M \ot^{\mrm{L}}_{B} N) 
\end{tikzcd} \]
in $\cat{D}(C)$, where 
$\opn{cup}^{\til{C} / \til{B} / \til{A}}_{C / B / A, M, N}$
is the resolved cup product morphism from Definition \ref{1510}, is commutative.
\end{itemize}
\end{thm}

\begin{proof}
Let us choose a universal K-flat resolution 
$t^{\mrm{u}} / s^{\mrm{u}} / r^{\mrm{u}} : 
\til{C}^{\mrm{u}} / \til{B}^{\mrm{u}} / \til{A}^{\mrm{u}}
\to C / B / A$ 
of the triple of DG rings $C / B / A$. By this we mean that 
$\til{A}^{\mrm{u}} \to A$ is a commutative semi-free DG ring resolution 
relative to $\Z$; $\til{B}^{\mrm{u}} \to B$ is a commutative semi-free DG ring 
resolution relative $\til{A}^{\mrm{u}}$; and $\til{C}^{\mrm{u}} \to C$ is a 
commutative semi-free DG ring resolution relative to $\til{B}^{\mrm{u}}$. This 
is possible by \cite[Theorem 3.21$(1)$]{Ye4}. We define the morphism
$\opn{cup}_{C / B / A, M, N}$
to be the unique morphism that makes the diagram in ($\diamondsuit$) 
commutative for the universal resolution 
$\til{C}^{\mrm{u}} / \til{B}^{\mrm{u}} / \til{A}^{\mrm{u}}$.

Given an arbitrary K-flat resolution 
$t / s / r : \til{C} / \til{B} / \til{A} \to C / B / A$
of the triple of DG rings $C / B / A$,
we must prove that the diagram in ($\diamondsuit$) is commutative for 
$\til{C} / \til{B} / \til{A}$. According to 
\cite[Theorem $\tup{3.22(1)}$]{Ye4} there exists a DG ring homomorphism 
$\til{r} : \til{A}^{\mrm{u}} \to \til{A}$ 
such that $r^{\mrm{u}} = r \circ \til{r}$ (the base DG ring here is $\Z$). 
Similarly there is a DG $\til{A}^{\mrm{u}}$-ring homomorphism 
$\til{s} : \til{B}^{\mrm{u}} \to \til{B}$ 
such that $s^{\mrm{u}} = s \circ \til{s}$, and the same for $\til{C}$. 
We thus have a commutative diagram 
\begin{equation} \label{eqn:840}
\begin{tikzcd} [column sep = 12ex, row sep = 5ex] 
&
\til{C} / \til{B} / \til{A} 
\ar[d, "{t / s / r}"]
\\
\til{C}^{\mrm{u}} / \til{B}^{\mrm{u}} / \til{A}^{\mrm{u}}
\ar[r, "{t^{\mrm{u}} / s^{\mrm{u}} / r^{\mrm{u}}}"']
\ar[ur, "{\til{t} / \til{s} / \til{r}}"]
&
C / B / A
\end{tikzcd} 
\end{equation}
of triples of DG rings. 

Consider the next diagram in $\cat{D}(C)$:
\begin{equation} \label{eqn:841}
\begin{tikzcd} [column sep = 16ex, row sep = 5ex]
\opn{Sq}^{\til{B} / \til{A}}_{B / A}(M) \ot^{\mrm{L}}_{B} 
\opn{Sq}^{\til{C} / \til{B}}_{C / B}(N) 
\arrow[r, "{\opn{cup}^{\til{C} / \til{B} / \til{A}}_{C / B / A, M, N}}"]
\arrow[d, "{ \opn{Sq}_{B / A}^{\til{s} / \til{r}}(\opn{id}) \ot^{\mrm{L}}_{B} 
\opn{Sq}_{C / B}^{\til{t} / \til{s}}(\opn{id}) }", "{\cong}"']
&
\opn{Sq}^{\til{C} / \til{A}}_{C / A} (M \ot^{\mrm{L}}_{B} N) 
\arrow[d, "{ \opn{Sq}_{C / A}^{\til{t} / \til{r}}(\opn{id}) }", 
"{\cong}"']
\\
\opn{Sq}^{\til{B}^{\mrm{u}} / \til{A}^{\mrm{u}}}_{B / A}(M) \ot^{\mrm{L}}_{B} 
\opn{Sq}^{\til{C}^{\mrm{u}} / \til{B}^{\mrm{u}}}_{C / B}(N) 
\arrow[r, "{\opn{cup}^{\til{C}^{\mrm{u}} / \til{B}^{\mrm{u}} / 
\til{A}^{\mrm{u}}}_{C / B / A, M, N}}"']
&
\opn{Sq}^{\til{C}^{\mrm{u}} / \til{A}^{\mrm{u}}}_{C / A} 
(M \ot^{\mrm{L}}_{B} N) 
\end{tikzcd}
\end{equation}
The morphisms $\opn{Sq}_{B / A}^{\til{s} / \til{r}}(\opn{id})$ etc.\ are 
special cases of (\ref{eqn:628}).
All the morphisms that are involved in the construction of the resolved cup 
product $\opn{cup}^{- / - / -}_{C / B / A, M, N}$ in Definition \ref{1510},
namely the morphisms in ($\til{\lozenge}$), 
respect the morphisms induced by 
$\til{t} / \til{s} / \til{r}$. 
Therefore diagram (\ref{eqn:841}) is commutative. Comparing it to the 
commutative diagram in condition ($*$) of Theorem \ref{thm:631},
we conclude that the diagram in ($\diamondsuit$) is commutative for 
$\til{C} / \til{B} / \til{A}$.
\end{proof}

Recall that a complex $M \in \cat{D}(A)$ is called a {\em perfect complex} if 
it is isomorphic, in $\cat{D}(A)$, to a bounded complex of finitely generated 
projective $A$-modules. See \cite[Section 14.1]{Ye5} for several equivalent 
characterizations of perfect complexes, and of perfect DG modules over 
DG rings.

\begin{thm}[Nondegeneracy] \label{thm:810} 
Let $C / B / A$ be a triple of rings, and let $M \in \cat{D}(B)$ and
$N \in \cat{D}(C)$ be DG modules.
Assume that all conditions \tup{(i)--(iv)} below hold\tup{:}
\begin{itemize}
\rmitem{i} $A$ is a noetherian ring, and $B$ and $C$ are essentially finite 
type $A$-rings.

\rmitem{ii} $M$ has finite flat dimension over $A$. 

\rmitem{iii} $N$ has finite flat dimension over $B$.

\rmitem{iv} Either \tup{(a)} or \tup{(b)} below is satisfied\tup{:}
\begin{itemize}
\rmitem{a} The ring $C$ is flat over $B$, and $C$ is a perfect complex over 
$C \ot_B C$. 

\rmitem{b} The complex $\opn{Sq}_{B / A}(M)$ has finite flat dimension over 
$B$. 
\end{itemize}
\end{itemize}
Then the cup product morphism 
\[  \opn{cup}_{C / B / A, M, N} \, : \, 
\opn{Sq}_{B / A}(M) \ot^{\mrm{L}}_{B}
\opn{Sq}_{C / B}(N) \, \to \,
\opn{Sq}_{C / A}(M \ot^{\mrm{L}}_{B} N)  \]
in $\cat{D}(C)$ is an isomorphism. 
\end{thm}

The proof of Theorem \ref{thm:810} requires the next lemma. 
In this lemma we use the ingredients of the theorem, and also Setup
\ref{set:1395}. 

\begin{lem} \label{lem:1436}
Let $M \in \cat{D}(B)$ and $N \in \cat{D}(C)$, and 
assume that conditions \tup{(i)--(iii)} of the theorem hold.
Then the morphism ($\heartsuit$) in Lemma \ref{lem:1430},
in the category $\cat{D}(\til{C}_1 \ot_{\til{B}_0} \til{C}_2)$,
is an isomorphism.
\end{lem}

\begin{proof}
Since the restriction functor 
$\opn{Rest} : \cat{D}(\til{C}_1 \ot_{\til{B}_0} \til{C}_2) \to \cat{D}(\Z)$
is conservative, is suffices to prove that ($\heartsuit$) is an isomorphism in 
$\cat{D}(\Z)$. Looking at the sequence of morphisms 
($\til{\heartsuit}^{}$) we see that we only need to prove that the 
morphisms  $\to^{\mrm{(b)}}$ and $\to^{\mrm{(c)}}$ there are isomorphisms. 

To show that the morphism $\to^{\mrm{(b)}}$ is an isomorphism, it suffices to 
prove that the derived tensor evaluation morphism 
\[ \begin{aligned}
& 
N_2 \ot^{\mrm{L}}_{\til{B}_2}
\opn{RHom}_{\til{B}_1 \ot_{\til{A}} \til{B}_2}
(\til{B}_0, M_1 \ot^{\mrm{L}}_{\til{A}} M_2) 
\\
& \quad \to
\opn{RHom}_{\til{B}_1 \ot_{\til{A}} \til{B}_2}
\bigl( \til{B}_0, (M_1 \ot^{\mrm{L}}_{\til{A}} M_2) \ot^{}_{\til{B}_2} N_2 
\bigr)
\end{aligned} \]
in $\cat{D}(\til{B}_1)$ is an isomorphism. 
We are going to use Theorem \ref{thm:2115}.
The DG ring $\til{B}_1 \ot_{\til{A}} \til{B}_2$ is cohomologically 
pseudo-noetherian, and the object $\til{B}_0$ belongs to 
$\cat{D}^-_{\mrm{f}}(\til{B}_1 \ot_{\til{A}} \til{B}_2)$,
so condition (i) of this theorem holds. Since $M$ has finite flat dimension 
over $A$, it also has finite flat dimension over $\til{A}$, so 
$M_1 \ot^{\mrm{L}}_{\til{A}} M_2$ is cohomologically bounded, and condition 
(ii) of the theorem holds. Lastly, since $N$ has finite flat dimension over 
$B$, it also has finite flat dimension over $\til{B}$, so condition (iii) of 
Theorem \ref{thm:2115} holds. 

Proving that the morphism $\to^{\mrm{(c)}}$ is an isomorphism is very much the 
same; the only change is that now we need to show that 
$M_1 \ot^{\mrm{L}}_{\til{A}} M_2 \ot^{}_{\til{B}_2} N_2$
has bounded cohomology, but this is shown as above. 
\end{proof}

\begin{proof}[Proof of Theorem \tup{\ref{thm:810}}]
We need to study the morphisms
$\xar{\mspace{12mu}}^{(\mrm{2})}$  
and $\xar{\mspace{12mu}}^{(\mrm{3})}$ in Definition \ref{1510}.
According to conditions (i)--(iii) and Lemma \ref{lem:1436}, 
the morphism $\xar{\mspace{12mu}}^{(\mrm{3})}$ is an isomorphism.
It remains to prove that the morphism 
$\xar{\mspace{14mu}}^{(\mrm{2})}$ in Definition \ref{1510}
is an isomorphism.

We now place ourselves in the situation of Definition \ref{1510}. 
In order to prove that the morphism $\xar{\mspace{14mu}}^{(\mrm{2})}$  
is an isomorphism, we may consider it as a morphism in
$\cat{D}(\Z)$; this is due to the fact that the restriction functor
$\cat{D}(C) \to \cat{D}(\Z)$ is conservative. 
Here is the morphism in question: 
\begin{equation} \label{eqn:2110}
\begin{aligned}
&
\opn{Sq}^{\til{B} / \til{A}}_{B / A}(M) \ot^{\mrm{L}}_{\til{B}_0}
\opn{RHom}_{\til{C}_1 \ot_{\til{B}_0} \til{C}_2}
(C_0, N_1 \ot^{\mrm{L}}_{\til{B}_0} N_2)
\\
& \quad \xar{\mspace{12mu}}^{(\mrm{2})} 
\opn{RHom}_{\til{C}_1 \ot_{\til{B}_0} \til{C}_2} 
\bigl( C_0, (N_1 \ot^{\mrm{L}}_{\til{B}_0} N_2) \ot^{\mrm{L}}_{\til{B}_0}
\opn{Sq}^{\til{B} / \til{A}}_{B / A}(M) \bigr) . 
\end{aligned}
\end{equation}
It is a derived tensor evaluation morphism in $\cat{D}(\Z)$.

Under condition (iv)(a), $C_0$ is a perfect DG module over $C_1 \ot_{B_0} C_2$. 
The equivalence 
\[ \opn{Rest}_{t \ot_{s} t} : \cat{D}(C_1 \ot_{B_0} C_2) \to 
\cat{D}(\til{C}_1 \ot_{\til{B}_0} \til{C}_2) \]
and \cite[Corollary 14.1.24]{Ye5} tell us that 
$C_0$ is a perfect DG module over
$\til{C}_1 \ot_{\til{B}_0} \til{C}_2$.
Hence, by \cite[Theorem 14.1.22]{Ye5}, the morphism (\ref{eqn:2110}) is an 
isomorphism. 

Now let's assume condition (iv)(b). We are given that 
$\opn{Sq}^{\til{B} / \til{A}}_{B / A}(M)$
has finite flat dimension over $B$; hence it is also  has finite flat dimension
over $\til{B}$. 
By condition (i) the DG $(\til{C}_1 \ot_{\til{B}_0} \til{C}_2)$-module $C_0$ 
is derived pseudo-finite (cf.\ proof of Lemma \ref{lem:1436}). 
By condition (iii) the DG module
$N_1 \ot^{\mrm{L}}_{\til{B}_0} N_2$ has bounded cohomology. 
We see that the conditions of Theorem \ref{thm:2115} are satisfied, so 
(\ref{eqn:2110}) is an isomorphism. 
\end{proof}

\begin{lem} \label{lem:1530}
Let $A$ be a noetherian ring, and let $B \to C$ be a homomorphism in
$\cat{Rng} \eftover A$. Suppose $M \in \cat{D}^{\mrm{b}}_{\mrm{f}}(B)$ has 
finite flat dimension over $A$, and $N \in \cat{D}^{\mrm{b}}_{\mrm{f}}(C)$ has 
finite flat dimension over $B$. Then $M \ot^{\mrm{L}}_{B} N$ belongs to 
$\cat{D}^{\mrm{b}}_{\mrm{f}}(C)$, and it has finite flat dimension over $A$. 
\end{lem}

The easy proof is left to the reader.
This lemma  is needed to show that the complex
$M \ot^{\mrm{L}}_{B} N$ in the definition below satisfies the conditions in
Definition \ref{dfn:675}(1).

\begin{dfn}[Derived Tensor Product of Rigid Complexes] \label{dfn:1260}
Let $A$ be a noetherian ring, and let $B \to C$ be a homomorphism in
$\cat{Rng} \eftover A$. 
Let $(M, \rho) \in \cat{D}(B)_{\mrm{rig} / A}$ and
$(N, \si) \in \cat{D}(C)_{\mrm{rig} / B}$
be rigid complexes.
Assume that the cup product morphism 
\[  \opn{cup}_{C / B / A, M, N} \, : \, 
\opn{Sq}_{B / A}(M) \ot^{\mrm{L}}_{B}
\opn{Sq}_{C / B}(N) \, \to \,
\opn{Sq}_{C / A}(M \ot^{\mrm{L}}_{B} N)  \]
in $\cat{D}(C)$ is an isomorphism. Then:
\begin{enumerate}
\item We define the rigidifying isomorphism 
\[ \rho \cupprod \si \,  : \, M \ot^{\mrm{L}}_{B} N \, \iso \, 
\opn{Sq}_{C / A}(M \ot^{\mrm{L}}_{B} N) \]
to be the unique isomorphism in $\cat{D}(C)$ for which the diagram 
\[ \begin{tikzcd} [column sep = 8ex, row sep = 6ex]
M \ot^{\mrm{L}}_{B} N 
\ar[r, "{\rho \cupprod \si}", "{\simeq}"']
\ar[d, "{\rho \, \ot^{\mrm{L}}_{B} \si}"', "{\simeq}"]
&
\opn{Sq}_{C / A}(M \ot_B N) 
\\
\opn{Sq}_{B / A}(M) \ot^{\mrm{L}}_{B}
\opn{Sq}_{C / B}(N)
\ar[ur, "{\opn{cup}_{C / B / A, M, N}}"', "{\simeq}"]
\end{tikzcd} \]
is commutative. 
The rigidifying isomorphism $\rho \cupprod \si$ is called the {\em cup product} 
of $\rho$ and $\si$. 

\item The rigid complex 
\[ (M, \rho) \ot^{\mrm{L, rig}}_{B} (N, \si) := 
\bigl( M \ot^{\mrm{L}}_{B} N, \rho \cupprod \si \bigr)
\in \cat{D}(C)_{\mrm{rig} / A} \]
is called the {\em derived tensor product of $(M, \rho)$ and $(N, \si)$ over 
$B$ relative to $A$}. 
\end{enumerate}
\end{dfn}

\begin{rem} \label{rem:825}
We shall use Theorem \ref{thm:810} in the following situations.
The first situation is when $B \to C$ is essentially smooth, and then 
condition (iii)(a) holds. (See also Remark \ref{rem:1050}.)
The cup product isomorphism, with Definition \ref{dfn:1260}, will be used to obtain induced 
rigid structures, in Section \ref{sec:twisted-induced}.

The other situation that will be important for us -- it will be in the proof of 
the Residue Theorem in the paper \cite{Ye6} -- is when $B$ is a local 
Gorenstein ring and $M$ is the 
rigid dualizing complex of $B / A$. Then 
$\opn{Sq}_{B / A}(M) \cong M \cong B[p]$ for some integer $p$, and 
condition (iii)(b) holds. 
\end{rem}

The following property will be useful to prove Theorem \ref{thm:1905}.

\begin{thm}[Backward Functoriality] \label{thm:845}
Let $A \to B \to C \xar{w} C'$ be ring homomorphisms, let 
$M \in \cat{D}(B)$, $N \in \cat{D}(C)$ and $N' \in \cat{D}(C')$
be DG modules, and let $\psi : N' \to N$ in $\cat{D}(C)$
be a backward morphism over $w$. 
Then the diagram 
\[ \tag{$\triangledown$}
\begin{tikzcd} [column sep = 16ex, row sep = 5ex] 
\opn{Sq}_{B / A}(M) \ot^{\mrm{L}}_{B} \opn{Sq}_{C' / B}(N') 
\arrow[r, "{\opn{cup}_{C' / B / A, M, N'}}"]
\arrow[d, "{ \opn{id} \ot^{\mrm{L}}_{B} \opn{Sq}_{w / B}(\psi) }", "{\cong}"']
&
\opn{Sq}_{C' / A} (M \ot^{\mrm{L}}_{B} N')
\arrow[d, "{ \opn{Sq}_{w / A}(\opn{id} \ot^{\mrm{L}}_{B} \, \psi) }", 
"{\cong}"']
\\
\opn{Sq}_{B / A}(M) \ot^{\mrm{L}}_{B} \opn{Sq}_{C / B}(N) 
\arrow[r, "{\opn{cup}_{C / B / A, M, N}}"']
&
\opn{Sq}_{C / A} (M \ot^{\mrm{L}}_{B} N)
\end{tikzcd} \]
in $\cat{D}(C)$ is commutative.
\end{thm}

\begin{proof} The proof is done is several steps. 

\smallskip \noindent
Step 1. Let us choose DG ring resolutions 
$\til{A} \xar{\til{u}} \til{B} \xar{\til{v}} \til{C} \xar{\til{w}} \til{C}'$ 
of 
$A \xar{u} B \xar{v} C \xar{w} C'$, 
where $\til{u}$ and $\til{v}$ are K-flat, and $\til{w}$ is K-projective. 
We have a commutative diagram 
\[ \begin{tikzcd} [column sep = 8ex, row sep = 4ex] 
\til{A}
\arrow[r, "{\til{u}}"]
\ar[d, "{r}"]
&
\til{B}
\ar[d, "{s}"]
\arrow[r, "{\til{v}}"]
&
\til{C}
\ar[d, "{t}"]
\arrow[r, "{\til{w}}"]
&
\til{C}'
\ar[d, "{t'}"]
\\
A
\arrow[r, "{u}"]
&
B
\arrow[r, "{v}"]
&
C
\arrow[r, "{w}"]
&
C'
\end{tikzcd} \]
in $\cat{DGRng}$, with vertical surjective quasi-isomorphisms. 
Note that $\til{C} / \til{B} / \til{A}$ is a K-flat resolution of the triple of 
DG rings $C / B / A$, and $\til{C}' / \til{B} / \til{A}$ is a K-flat resolution 
of the triple of DG rings $C' / B / A$.

Next we are going to choose DG module resolutions 
$\til{P}$, $\til{Q}$ and $\til{I}$
as in Lemma \ref{lem:1431}. We also choose
a K-projective DG module resolution $\til{Q}' \to N'$ over 
$\til{C}'$. Since $\til{C}'$ is K-projective over $\til{C}$, it follows that 
$\til{Q}'$ is K-projective as a DG $\til{C}$-module. 

\medskip \noindent
Step 2. Lemma \ref{lem:1431} is stated with the triples of DG rings 
$C / B / A$ and $\til{C} / \til{B} / \til{A}$; the DG modules $M, N, \til{P}, 
\til{Q}, \til{I}$; and the auxiliary DG ring $\til{D}$. This lemma yields 
the isomorphism ($\dag$), which is presented by its resolved form 
($\til{\dag}$). 

There is also a version of Lemma \ref{lem:1431} with 
primed objects: the triples of DG rings $C' / B / A$ and 
$\til{C}' / \til{B} / \til{A}$; the DG modules
$M, N', \til{P}, \til{Q}', \til{I}$; and the auxiliary DG ring 
\[ \til{D}' := (\til{C}'_1 \ot_{\til{A}} \til{C}'_2) 
\ot_{\til{B}_1 \ot_{\til{A}} \til{B}_2} \til{B}^{}_0 . \]
The primed version of Lemma \ref{lem:1431} gives us a primed isomorphism 
($\dag'$), which is presented by its resolved form ($\til{\dag}'$).

The backward morphism $\psi : N' \to N$ in $\cat{D}(\til{C})$
is presented by a backward homomorphism
$\til{\psi} : \til{Q}' \to \til{Q}$
in $\cat{C}_{\mrm{str}}(\til{C})$, which is unique up to homotopy. 
The homomorphism $\til{\psi}$ gives rise to a backward morphism in 
$\cat{C}_{\mrm{str}}(\til{D})$ from 
($\til{\dag}$) to ($\til{\dag}'$), and all this is a commutative diagram in 
$\cat{C}_{\mrm{str}}(\til{D})$. Passing to derived categories, 
we get a backward morphism in $\cat{D}(\til{D})$ from ($\dag$) to ($\dag'$),
and together this is a commutative diagram  in $\cat{D}(\til{D})$.

\medskip \noindent
Step 3. There is a primed version of Lemma (\ref{lem:1430}), with a 
morphism ($\heartsuit'$) in the category
$\cat{D}(\til{C}'_1 \ot_{\til{B}_0} \til{C}'_2)$, 
and a sequence of morphisms ($\til{\heartsuit}'$) in 
$\cat{D}(\til{D}')$. 
The backward morphism $\psi : N' \to N$ gives rise to a backward morphism in 
the category $\cat{D}(\til{D})$ from ($\til{\heartsuit}'$) to 
($\til{\heartsuit}$). 
The whole resulting diagram in $\cat{D}(\til{D})$ is commutative; for the 
commutativity at the isomorphism $\iso^{(1)}$ and $\iso^{(1')}$ 
we rely on step 2. 

Passing to $\cat{D}(\til{C}_1 \ot_{\til{B}_0} \til{C}_2)$ using the equivalence 
$\opn{Rest}_g$, we obtain a backward morphism from ($\heartsuit'$)
to ($\heartsuit$), giving us a commutative diagram in 
$\cat{D}(\til{C}_1 \ot_{\til{B}_0} \til{C}_2)$.

\medskip \noindent
Step 4. The primed version of Definition \ref{1510} produces a cup product 
morphism \lb 
$\opn{cup}^{\til{C}' / \til{B} / \til{A}}_{C' / B / A, M, N'}$. 
The diagram 
\begin{equation} \label{eqn:1025}
\begin{tikzcd} [column sep = 18ex, row sep = 4ex] 
\opn{Sq}^{\til{B} / \til{A}}_{B / A}(M) \ot^{\mrm{L}}_{B} 
\opn{Sq}^{\til{C}' / \til{B}}_{C' / B}(N') 
\arrow[r, "{\opn{cup}^{\til{C}' / \til{B} / \til{A}}_{C' / B / A, M, N'}}"]
\arrow[d, "{ \opn{id} \ot^{\mrm{L}}_{B} 
\opn{Sq}^{\til{w}' / \til{B}}_{w / B}(\psi) }"', "{\simeq}"]
&
\opn{Sq}^{\til{C}' / \til{A}}_{C' / A} (M \ot^{\mrm{L}}_{B} N')
\arrow[d, "{ \opn{Sq}^{\til{w} / \til{A}}_{w / A}
(\opn{id} \, \ot^{\mrm{L}}_{B} \, \psi) }", "{\simeq}"']
\\
\opn{Sq}^{\til{B} / \til{A}}_{B / A}(M) \ot^{\mrm{L}}_{B} 
\opn{Sq}^{\til{C} / \til{B}}_{C / B}(N) 
\arrow[r, "{\opn{cup}^{\til{C} / \til{B} / \til{A}}_{C / B / A, M, N}}"]
&
\opn{Sq}^{\til{C} / \til{A}}_{C / A} (M \ot^{\mrm{L}}_{B} N)
\end{tikzcd}
\end{equation}
in $\cat{D}(C)$ is commutative. Here we used step 3 for the commutativity at 
the morphisms $\to^{(\mrm{3})}$ and $\to^{(\mrm{3}')}$. 

\medskip \noindent
Step 5. To finish the proof let us look at the cubical diagram 
(\ref{eqn:1880}) in the category $\cat{D}(C)$.
The top face of the cube is diagram ($\triangledown$), and the
the objects are labelled with the abbreviations 
\[ \opn{Sq}(M) \ot \opn{Sq}(N') := \opn{Sq}_{B / A}(M) \ot^{\mrm{L}}_{B} 
\opn{Sq}_{C' / B}(N') \]
etc., due to lack of space. The bottom face of the cube is diagram 
(\ref{eqn:1025}), and the the objects are labelled with the abbreviations
\[ \wtil{\opn{Sq}}(M) \ot \wtil{\opn{Sq}}(N') := 
\opn{Sq}^{\til{B} / \til{A}}_{B / A}(M) \ot^{\mrm{L}}_{B} 
\opn{Sq}^{\til{C}' / \til{B}}_{C' / B}(N') \]
etc. The horizontal arrows are not labelled, but they are shown in diagrams 
($\triangledown$) and (\ref{eqn:1025}). 
The vertical arrows in diagram (\ref{eqn:1880}) are labelled, and they 
are all isomorphisms.

\begin{equation} \label{eqn:1880}
\begin{tikzcd}[row sep = 6ex, column sep = 2ex]
& 
\opn{Sq}(M) \ot \opn{Sq}(N')
\arrow[dd, "{\mspace{35mu} \opn{sq}^{\til{B} / \til{A}}_{B / A, M} 
\ot^{\mrm{L}}_{B} \opn{sq}^{\til{C}' / \til{B}}_{C' / B, N'}}" description,
near start, "{\simeq}" near end]
\arrow[dl]
\ar[rr] 
& 
& 
\opn{Sq}(M \ot N')
\arrow[dl] 
\arrow[dd, "{\mspace{10mu} 
\opn{sq}^{\til{C}' / \til{A}}_{C' / A, M \ot^{\mrm{L}}_{B} N'}}"'
description, "{\simeq}" near end] 
\\
\opn{Sq}(M) \ot \opn{Sq}(N)
\arrow[dd, "{{\opn{sq}^{\til{B} / \til{A}}_{B / A, M} 
\ot^{\mrm{L}}_{B} \opn{sq}^{\til{C} / \til{B}}_{C / B, N} }}" description,
"{\simeq}" near end] 
\ar[rr, crossing over]
& 
& 
\opn{Sq}(M \ot N)
\\
& 
\wtil{\opn{Sq}}(M) \ot \wtil{\opn{Sq}}(N')
\arrow[dl] 
\arrow[rr]
& 
& 
\wtil{\opn{Sq}}(M \ot N')
\arrow[dl] 
\\
\wtil{\opn{Sq}}(M) \ot \wtil{\opn{Sq}}(N)
\ar[rr]
& 
& 
\wtil{\opn{Sq}}(M\ot N)
\arrow[from=uu, crossing over, 
"{\mspace{0mu} \opn{sq}^{\til{C} / \til{A}}_{C / A, M \ot^{\mrm{L}}_{B} N}}"'
description, near start, "{\simeq}" near end] 
\end{tikzcd} 
\end{equation}
The left and right vertical faces of the cube (\ref{eqn:1880})
are commutative, by condition ($**$) of Theorem \ref{thm:632}.
The front and rear vertical faces of the cube (\ref{eqn:1880}) 
are commutative, because they are instances of the commutative diagram 
($\diamondsuit$) in Theorem \ref{thm:780}. From step 4 we know that the bottom 
face, i.e.\ diagram (\ref{eqn:1025}), is commutative. The conclusion is that the 
top face, namely diagram ($\triangledown$), is commutative.
\end{proof}

\section{Regular Surjections of Rings and the Fundamental Local Isomorphism} 
\label{sec:regular}

In this section of the paper we recall a few facts about regular sequences 
and related concepts, and sharpen some known results. The main result here is 
Theorem \ref{thm:1080}, called the {\em Fundamental Local Isomorphism}, which 
is a slight modification of \cite[Proposition III.7.2]{RD}, and we give a full 
proof of it. The content of this section could be of independent interest, 
beyond its application to the theory rigid dualizing complexes.  

Throughout this section Conventions \ref{conv:615} and \ref{conv:1070} are in 
force. Thus all rings are commutative and noetherian by default.

Let us introduce some useful notation. 
Given a ring $A$ and a multiplicatively closed set $S \sub A$, the localization 
of $A$ with respect to $S$ is the ring $A_S = A[S^{-1}]$ gotten by inverting 
the elements of $S$. It is equipped with the universal ring homomorphism 
$\opn{q}_{A, S} : A \to A_S$. 
As usual, if $S = \{ s^k \}_{k \in \N}$ is a principal multiplicatively closed 
set, generated by some element $s \in A$, then we write 
$A_s = A[s^{-1}] := A_S$. If $S = A - \p$, the complement of some prime ideal 
$\p$, then we write $A_{\p} := A_S$. The corresponding ring homomorphisms are 
$\opn{q}_{A, s} : A \to A_s$ and 
$\opn{q}_{A, \p} : A \to A_{\p}$.
Given an $A$-module (or a complex of 
$A$-modules) $M$, we write $M_S = M[S^{-1}] := A_S \ot_A M$. The corresponding 
$A$-module homomorphism is 
$\opn{q}_{M, S} : M \to M_S$, $\opn{q}_{M, S}(m) := 1 \ot m$.

\begin{rem} \label{rem:1535}
In the terminology of Section \ref{sec:recall-dg}, the $A$-module homomorphism 
$\opn{q}_{M, S} : M \to M_S$ is the standard nondegenerate forward
homomorphism
in $\cat{M}(A)$ over the ring homomorphism $\opn{q}_{A, S} : A \to A_S$; see 
formula (\ref{eqn:1476}). 
\end{rem}

For an element $s \in A$ we identify the affine scheme 
$\opn{Spec}(A_s)$ with the principal open set 
\begin{equation} \label{eqn:1532}
\opn{NZer}_{\opn{Spec}(A)}(s) := 
\{ \p \mid s \notin \p \} \sub \opn{Spec}(A) . 
\end{equation}
Given a surjective ring homomorphism $A \to B$ with kernel $\a$, we identify 
the underlying topological space of the affine scheme $\opn{Spec}(B)$ with the 
closed subset 
\begin{equation} \label{eqn:1533}
\opn{Zer}_{\opn{Spec}(A)}(\a) = \{ \p \mid \a \sub \p \} \sub \opn{Spec}(A) . 
\end{equation}

Suppose $\ba = (a_1, \ldots, a_r)$ is a sequence of elements of $A$, and 
$f : A \to B$ is a ring homomorphism. Then we let 
$f(\ba) :=  (f(a_1), \ldots, f(a_r))$
be the corresponding sequence of elements of $B$. 
Likewise for a sequence $\bm$ of elements of a module $M$ and a module 
homomorphism $\phi : M \to N$. 

In the literature there are several notions of regularity for a sequence 
of elements of a ring. We are interested in three of them. 

An element $a \in A$ is called a {\em regular element} if it is not a 
zero-divisor, i.e.\ if multiplication by $a$ is an injective endomorphism of 
$A$. A sequence $\ba = (a_1, \ldots, a_n)$  in $A$ is called a {\em
regular sequence} if for every $i = 1, \ldots, n$ the image of the element 
$a_i$ in the ring $A / (a_1, \ldots, a_{i - 1})$ is regular, and the ring 
$A / (a_1, \ldots, a_{n})$ is nonzero. 
This definition is consistent with \cite[Section 16]{Ma}.

Let $\a \sub A$ be the ideal generated by the sequence 
$\ba = (a_1, \ldots, a_n)$ .
Consider the associated graded ring 
$\opn{Gr}^{\a}(A) = \boplus_{k \geq 0} \opn{Gr}_k^{\a}(A)$,
where $\opn{Gr}_k^{\a}(A) = \a^k / \a^{k + 1}$, and in particular 
$\opn{Gr}_0^{\a}(A) = A / \a$. 
(In the book \cite[Section 15.1]{Ye5} such a graded ring is called an {\em 
algebraically graded ring}, as opposed to the {\em cohomologically graded 
rings} that underly DG rings, see \cite[Section 3.1]{Ye5}.) 
Let $\bt = (t_1, \ldots, t_n)$ be a sequence of variables, all of degree $1$, 
and let $\opn{Gr}^{\a}_0(A)[\bt]$ be the graded polynomial ring. 
There is a canonical graded $\opn{Gr}^{\a}_0(A)$-ring homomorphism 
\begin{equation} \label{eqn:1086}
\opn{Gr}^{\a}_0(A)[\bt] \to \opn{Gr}^{\a}(A) ,
\end{equation}
that sends the variable $t_i$ to the class $\bar{a}_i$ of the element $a_i$ in
$\opn{Gr}_1^{\a}(A) = \a / \a^2$.
The sequence $\ba$ is called a {\em quasi-regular sequence} if the homomorphism 
(\ref{eqn:1086}) is bijective, and $\opn{Gr}^{\a}_0(A) \neq 0$.
See \cite[Section 16]{Ma} and \cite[Section 15.1]{EGA-0IV}. 

The third notion of regularity is less familiar -- it is not mentioned 
explicitly in either of the references above, but it is mentioned in 
\cite[Remark tag =
\texttt{\href{https://stacks.math.columbia.edu/tag/065M}{065M}}]{SP}
and studied in
\cite[Section tag =
\texttt{\href{https://stacks.math.columbia.edu/tag/062D}{062D}}]{SP}.
To a finite sequence $\ba = (a_1, \ldots, a_n)$ of elements in the ring $A$
we associate the {\em Koszul complex} 
$\opn{K}(A; \ba) = \bigoplus_{-n \leq i \leq 0} \opn{K}^{i}(A; \ba)$.
See \cite[Section 16]{Ma},
\cite[Section
tag = \texttt{\href{https://stacks.math.columbia.edu/tag/062D}{062D}}]{SP}
and \cite[Examples 3.3.9 and 3.3.11]{Ye5}.
Observe that in \cite{Ma} and \cite{SP} the Koszul complex has lower 
(homological) indices.

\begin{dfn} \label{dfn:1540}
We call the sequence 
$\ba = (a_1, \ldots, a_n)$  in $A$ a {\em Koszul regular sequence}  
if the Koszul complex $\opn{K}(A; \bsym{a})$ satisfies 
$\opn{H}^p \bigl( \opn{K}(A; \bsym{a}) \bigr) = 0$ for all $p \neq 0$, and 
$\opn{H}^0 \bigl( \opn{K}(A; \bsym{a}) \bigr) \neq 0$.
\end{dfn}

Note that $\opn{H}^0 \bigl( \opn{K}(A; \bsym{a}) \bigr)$ is canonically 
isomorphic to $A / \a$, where $\a \sub A$ is the ideal generated by the 
sequence $\ba$. Thus, $\ba$ is a Koszul regular sequence iff the augmentation
homomorphism $\opn{K}(A; \bsym{a}) \to A / \a$  
is a free resolution of $A / \a$ as an $A$-module, and $A / \a \neq 0$.

Each of the three regularity conditions has two clauses in it: a vanishing 
clause (e.g.\ the vanishing of the kernel of multiplication by $a_i$ in the 
case of regularity), and a nonvanishing clause (that $A / \a \neq 0$). 

\begin{prop} \label{prop:1090}
Let $\ba = (a_1, \ldots, a_n)$ be a sequence of elements of $A$, and let 
$\a \sub A$ be the ideal generated by $\ba$. Let $S \sub A$ be a 
multiplicatively closed subset such that 
$A_S \ot_A (A / \a) \neq 0$. 
If $\ba$ is a regular sequence (resp.\  quasi-regular sequence, resp.\  
Koszul regular sequence), then $\opn{q}_{A, S}(\ba)$ is a regular 
sequence (resp.\  quasi-regular sequence, resp.\  Koszul regular sequence) in 
$A_S$. 
\end{prop}

\begin{proof}
The vanishing clause is satisfied because of the flatness of the ring 
homomorphism $\opn{q}_{A, S} : A \to A_S$.
The nonvanishing clause is satisfied by assumption. 
\end{proof}

Here is a partial converse to the previous proposition. 

\begin{prop} \label{prop:1091}
Let $\ba = (a_1, \ldots, a_n)$ be a sequence of elements of $A$, let 
$\a \sub A$ be the ideal generated by $\ba$, and let $B := A / \a$.
\begin{enumerate}
\item Assume that $B \neq 0$, and for every $\p \in \opn{Spec}(B)$ the sequence 
$\opn{q}_{A, \p}(\ba)$ is a quasi-regular sequence (resp.\  Koszul 
regular sequence) in the ring $A_{\p}$. Then $\ba$ is a quasi-regular sequence 
(resp.\  Koszul regular sequence) in $A$. 

\item Let $\p \in \opn{Spec}(B)$, and assume that the sequence 
$\opn{q}_{A, \p}(\ba)$ is a regular sequence in the ring $A_{\p}$. 
Then there exists an element $s \in A - \p$ such that sequence 
$\opn{q}_{A, s}(\ba)$ is a regular sequence in the ring $A_{s}$. 
\end{enumerate}
\end{prop}

\begin{proof} 
(1) This is because the $A$-modules $\opn{Gr}_k^{\a}(A)$ and 
$\opn{H}^p \bigl( \opn{K}(A; \bsym{a}) \bigr)$
are supported in $\opn{Spec}(B)$.

\medskip \noindent
(2) For every $i = 1, \ldots, n$ let 
$A_i := A / (a_1, \ldots, a_{i - 1})$, and let $N_i \sub A_i$ be the the kernel 
of multiplication by $a_i$. Then $N_i$ is a finitely generated $A$-module.
The regularity of the sequence $\opn{q}_{A, \p}(\ba)$ in the local 
ring $A_{\p}$ means that $A_{\p} \ot_A N_i = 0$ for all $i$. Hence there is 
some element $s \in A - \p$ such that $A_{s} \ot_A N_i = 0$ for all $i$. We see 
that the sequence $\opn{q}_{A, s}(\ba)$ in $A_s$ is regular. 
\end{proof}

The next theorem collects several classical results in the local case.

\begin{thm} \label{thm:1090}
Assume $A$ is a noetherian local ring with maximal ideal $\m$. Let 
$\ba = (a_1, \ldots, a_n)$ be a sequence of elements of $\m$. The following 
three conditions are equivalent:
\begin{itemize}
\rmitem{i} The sequence $\ba$ is regular.

\rmitem{ii} The sequence $\ba$ is quasi-regular.

\rmitem{iii}  The sequence $\ba$ is Koszul regular.
\end{itemize}
\end{thm}

\begin{proof} 

\smallskip \noindent
(i) $\Rightarrow$ (ii): See \cite[Theorem 16.2$\tup{(i)}$]{Ma}.

\medskip \noindent
(ii) $\Rightarrow$ (i): See \cite[Theorem 16.3]{Ma}, noting that $\m$ is 
the radical of $A$.  

\medskip \noindent
(i) $\Rightarrow$ (iii): This is \cite[Theorem 16.5$\tup{(i)}$]{Ma}.

\medskip \noindent
(iii) $\Rightarrow$ (ii): This is \cite[Theorem 16.5$\tup{(ii)}$]{Ma}.
\end{proof}

\begin{cor} \label{cor:1090}
Let $\ba$ be a finite sequence of elements of $A$.
The following two conditions are equivalent:
\begin{itemize}
\rmitem{i} The sequence $\ba$ is quasi-regular.

\rmitem{ii} The sequence $\ba$ is Koszul regular.
\end{itemize}
\end{cor}

\begin{proof}
Let $\a \sub A$ be the ideal generated by the sequence $\ba$, and let 
$B := A / \a$. 
According to Proposition \ref{prop:1090} and Proposition \ref{prop:1091}(1),
$\ba$ is quasi-regular (resp.\ Koszul regular) iff for every 
$\p \in \opn{Spec}(B)$ the sequence $\opn{q}_{A, \p}(\ba)$ is a 
quasi-regular sequence (resp.\  Koszul regular sequence) in the local ring 
$A_{\p}$. Now use Theorem \ref{thm:1090}. 
\end{proof}

\begin{cor} \label{cor:1093}
Let $\ba$ be a regular sequence of elements of $A$. Then $\ba$ is  Koszul 
regular. 
\end{cor}

\begin{proof} 
According to Proposition \ref{prop:1090}, for every 
$\p \in \opn{Spec}(B)$ the sequence $\opn{q}_{A, \p}(\ba)$ is a 
regular sequence in the local ring $A_{\p}$. By Theorem \ref{thm:1090}, 
$\opn{q}_{A, \p}(\ba)$ is a Koszul regular sequence in $A_{\p}$. Now 
use Proposition \ref{prop:1091}(1).
\end{proof}

\begin{cor} \label{cor:1091}
Let $\ba$ be a Koszul regular sequence of elements of $A$, let 
$\a \sub A$ be the ideal generated by $\ba$, and let $B := A / \a$.
Then for every $\p \in \opn{Spec}(B)$ there exists an element $s \in A - \p$ 
such that the sequence $\opn{q}_{A, s}(\ba)$ in $A_s$ is regular. 
\end{cor}

\begin{proof}
According to Proposition \ref{prop:1090} the sequence 
$\opn{q}_{A, \p}(\ba)$ is 
a Koszul regular sequence in the local ring $A_{\p}$. By Theorem \ref{thm:1090},
$\opn{q}_{A, \p}(\ba)$ is a regular sequence in $A_{\p}$. Now use 
Proposition \ref{prop:1091}(2).
\end{proof}

The next definition does not seem to appear in standard commutative algebra 
books, but it can be found in \cite[Section 16.9]{EGA-IV} and 
\cite[Section
tag = \texttt{\href{https://stacks.math.columbia.edu/tag/07CU}{07CU}}]{SP}.

\begin{dfn} \label{dfn:1080}
Let $\a \sub A$ be an ideal, and let $B := A / \a$. The ideal $\a$ is called a 
{\em regular ideal} (resp.\ {\em quasi-regular ideal}) if $B \neq 0$,
and for every $\p \in \opn{Spec}(B)$ there exists some element $s \in A - \p$ 
such that the ideal $\a_s \sub A_s$ is generated by a regular sequence
(resp.\ quasi-regular sequence).
\end{dfn}

\begin{dfn} \label{dfn:1081}
A surjective ring homomorphism $f : A \to B$ is called a {\em regular 
surjection} (resp.\ {\em quasi-regular surjection})
if the ideal $\a := \opn{Ker}(f)$ is a regular ideal (resp.\ quasi-regular 
ideal) of $A$, in the sense of Definition \ref{dfn:1080}. 
In particular the ring $B$ is nonzero. 
\end{dfn}

\begin{prop} \label{prop:1130}
An ideal $\a \sub A$ is regular iff it is quasi-regular.
\end{prop}

Remember that our rings are noetherian, unlike in 
\cite[Section 16]{EGA-IV}. 

\begin{proof}
Combine Corollaries \ref{cor:1090}, \ref{cor:1093} and \ref{cor:1091}.
\end{proof}

\begin{cor} \label{cor:1130}
A ring homomorphism $A \to B$ is a regular surjection iff it is a 
quasi-regular surjection. 
\end{cor}

\begin{proof}
This is immediate from Proposition \ref{prop:1130}.
\end{proof}

In light of Proposition \ref{prop:1130} and Corollary \ref{cor:1130},
from now on we shall only talk about regular ideals and regular surjections of 
rings (and not mention the quasi-regular ones). 

\begin{rem} \label{rem:1130}
Here we are going to compare the algebraic Definition \ref{dfn:1081} to the 
corresponding geometric notion from \cite{EGA-IV}. Let $Y \sub X$ be a closed 
embedding of noetherian schemes. 
According to \cite[Definition 16.9.2]{EGA-IV}, $Y \sub X$ is called a 
{\em regular closed embedding} of schemes if every point $y \in Y$ 
has an affine open neighborhood 
$V = \opn{Spec}(B) = U \cap Y \sub Y$, 
where $U = \opn{Spec}(A) \sub X$ is an affine open subscheme, 
and the ring homomorphism $A \to B$ is a regular 
surjection. This is depicted in Figure \ref{fig:1}. 
The geometric result matching Corollary 
\ref{cor:1130} is  \cite[Proposition 16.9.10]{EGA-IV}. 
Note that the coherent $\OO_V$-module corresponding to the $B$-module 
$\a / \a^2$ is called the {\em conormal sheaf} of $V / U$. 
\end{rem}

\begin{figure}[!tb]
\centering
\includegraphics[scale=0.18]{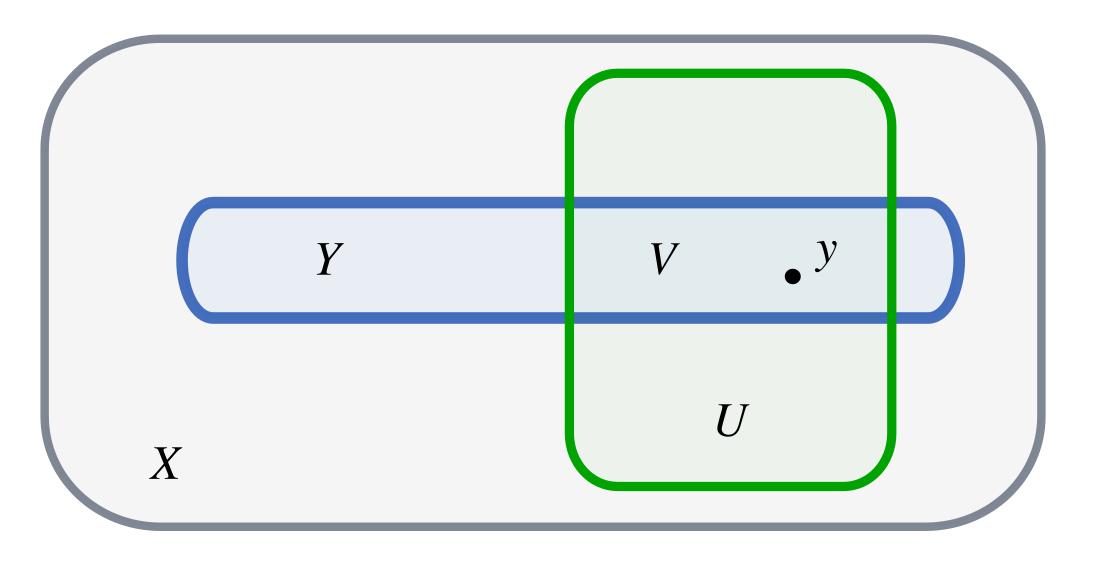}
\caption{An illustration for Remark \ref{rem:1130}.}
\label{fig:1}
\end{figure}

\newpage

Before continuing, we need to say a more about Koszul complexes. 
Given a sequence $\bsym{a} = (a_1, \ldots, a_n)$ in $A$, its associated
Koszul complex $\opn{K}(A; \bsym{a})$ is actually a semi-free commutative DG 
$A$-ring. This means that as a graded $A$-ring
(i.e.\ after forgetting the differential), there is an isomorphism 
\begin{equation} \label{eqn:1100}
\opn{K}(A; \bsym{a})^{\natural} \cong A[\bt] ,
\end{equation}
where the latter is the strongly commutative polynomial ring on a sequence 
of degree $-1$ variables $\bt = (t_1, \ldots, t_n)$. 
Note that here we use {\em cohomological grading}, as in  
\cite[Section 3.1]{Ye5}, and the Koszul sign rule holds. Strong commutativity 
for cohomologically graded rings
implies that $t_j \cd t_i = - t_i \cd t_j$ and $t_i \cd t_i = 0$.
With respect to the graded ring isomorphism (\ref{eqn:1100}), the differential 
of $\opn{K}(A; \bsym{a})$ is $\d(t_i) = a_i$. 
See \cite[Definition 3.1.22, Example 3.1.23 and Example 3.3.11]{Ye5}.
Note that $\opn{K}^{-n}(A, \bsym{a})$ is a free $A$-module of rank $1$, with 
basis the element $t_1 \cdots t_n$.
Letting $\a \sub A$ be the ideal generated by $\ba$, there is a unique
DG $A$-ring homomorphism $\opn{K}(A; \bsym{a}) \to A / \a$,
and it induces an $A$-ring isomorphism
$\opn{H}^0 \bigl( \opn{K}(A; \bsym{a}) \bigr) \cong A / \a$.

Here are several definitions that introduce useful notation.  

\begin{dfn} \label{dfn:1920}
Let $C = \bigoplus_{i \in \Z} C^i$ be a strongly commutative 
cohomologically graded ring, for instance the Koszul complex 
$\opn{K}(A, \bsym{a}) = \bigoplus_{i  \leq 0}\opn{K}^i(A, \bsym{a})$ 
mentioned above, or the de Rham 
complex $\Om_{B / A} = \bigoplus_{i \geq 0} \Om^i_{B / A}$
associated to a ring homomorphism $A \to B$. 
The multiplication in $C$ is denoted either by $c_1 \cdot c_2$ or by 
$c_1 \wedge c_2$.
Given a sequence $\bc = (c_1, \ldots, c_p)$ of odd degree elements of $C$, 
we write 
\[ \opn{wdg}(\bc) := c_1 \wedge \cdots \wedge c_p \in C , \]
and call this element the {\em wedge of $\bc$}. 
\end{dfn}

In the situation of the definition above, let $A := C^0$; the action of an 
element $a \in A$ on an element $c \in C$ is $a \cd c = c \cd a$. Suppose 
$\bc = (c_1, \ldots, c_n)$ and $\bc' = (c'_1, \ldots, c'_n)$ are
sequences of odd degree elements of $C$, viewed as rows, and 
$\ba = [a_{i, j}] \in \opn{Mat}_n(A)$ is a matrix
satisfying $\bc' = \bc \cdot \ba$, namely 
$c'_j = \sum_{i} c_i \cdot a_{i, j}$.  
Due to the strong commutativity of $C$ (i.e.\ the Koszul sign rule) we have 
\begin{equation} \label{eqn:1920}
\opn{wdg}(\bc') =
\opn{wdg}(\bc \cd \ba) = \opn{det}(\ba) \cd \opn{wdg}(\bc) .
\end{equation}

\begin{dfn} \label{dfn:1921}
Let $L$ be a free $A$-module of rank $1$, with basis element $\la$. 
Given another $A$-module $M$, and an element $\mu \in M$, we denote by 
$\smfrac{\, \mu \,}{\la} \in \opn{Hom}_{A}(L, M)$ 
the $A$-linear homomorphism 
$\smfrac{\, \mu \,}{\la} : L \to  M$,   
$\smfrac{\, \mu \,}{\la}(a \cd \la) := a \cd \mu$. 
\end{dfn}

\begin{dfn}[Generalized Fraction] \label{dfn:1922}
Let $\bsym{a} = (a_1, \ldots, a_n)$ be a sequence of elements in the ring $A$,
let $M$ be an $A$-module, and let $\mu \in M$ be an element. 
The {\em generalized fraction}
\[ \smgfrac{\mu}{\bsym{a}} \in 
\opn{H}^n \bigl( \opn{Hom}_{A} \bigl( \opn{K}(A; \bsym{a}), M \bigr) \bigr) \]
is the cohomology class of the homomorphism 
\[ \smfrac{\mu}{\opn{wdg}(\bt)} = \smfrac{\mu}{t_1 \cdots t_n} \in
\opn{Hom}_{A} \bigl( \opn{K}^{-n}(A; \bsym{a}), M \bigr) =
\opn{Hom}_{A} \bigl( \opn{K}(A; \bsym{a}), M \bigr)^{n} . \]
Here the fraction $\smfrac{\mu}{\opn{wdg}(\bt)}$ is the homomorphism from
Definition \ref{dfn:1921}, for the free $A$-module 
$L := \opn{K}^{-n}(A; \bsym{a})$ and its basis
$\la := \opn{wdg}(\bt) = t_1 \cdots t_n$ gotten from the isomorphism
(\ref{eqn:1100}). 
\end{dfn}

The definition makes sense because $\smfrac{\mu}{\opn{wdg}(\bt)}$ is an element 
of top degree in the complex 
$\opn{Hom}_{A} \bigl( \opn{K}(A; \bsym{a}), M \bigr)$,
so it is a cocycle. 

\begin{dfn} \label{dfn:1923}
Suppose $\bsym{a} = (a_1, \ldots, a_n)$ is a Koszul regular sequence in $A$,
and $\a \sub A$ is the ideal generated by $\ba$. Let $B := A / \a$. 
The augmentation homomorphism $\opn{K}(A; \bsym{a}) \to B$ is a
quasi-isomorphism, so $\opn{K}(A; \bsym{a})$ is a free resolution of $B$ as an 
$A$-module. 
\begin{enumerate}
\item By the universal property of the right derived functor, for every 
complex $M \in \cat{D}(A)$ there is a canonical isomorphism 
\[ \opn{kpres}_{\ba} : 
\opn{Hom}_{A} \bigl( \opn{K}(A; \bsym{a}), M \bigr) \iso 
\opn{RHom}_{A}(B, M) \]
in $\cat{D}(A)$, called the {\em Koszul presentation}. 

\item For an $A$-module $M$ and for $i \geq 0$ we get a canonical isomorphism 
of $A$-modules 
\[ \opn{kpres}^i_{\ba} := \opn{H}^i(\opn{kpres}_{\ba}) :
\opn{H}^i \bigl(  \opn{Hom}_{A} \bigl( \opn{K}(A; \bsym{a}), M \bigr) \bigr) 
\iso \opn{Ext}^i_{A}(B, M) \]
\end{enumerate}
\end{dfn}

In the situation of Definition \ref{dfn:1923}(2), for $i = n$ and
$\mu \in M$, we sometimes use the shorthand 
\begin{equation} \label{eqn:1921}
\smgfrac{\mu}{\bsym{a}} \in \opn{Ext}^n_{A}(B, M) ,
\end{equation}
leaving the canonical isomorphism $\opn{kpres}^n_{\ba}$ implicit.
Of course every cohomology class in $\opn{Ext}^n_{A} (B, M )$ can be expressed
this way as a generalized fraction -- but not uniquely. 

\begin{lem} \label{lem:1111}
Let $\ba = (a_1, \ldots, a_n)$ be a Koszul regular sequence in $A$, let $\a 
\sub A$ be the ideal generated by $\ba$, and let $B := A / \a$.
\begin{enumerate}
\item The $B$-module $\a / \a^2$ is free of rank $n$, with basis the 
sequence $\bar{\bsym{a}} = (\bar{a}_1, \ldots, \bar{a}_n)$, 
where $\bar{a}_i$ is the image of $a_i$ in $\a / \a^2$. 

\item The $B$-module $\bwedge^n_B (\a / \a^2)$ is free of rank $1$, with basis 
the element 
\[ \opn{wdg}(\bar{\ba}) = 
\bar{a}_1  \wedge \cdots \wedge \bar{a}_n
\in \bwedge^n_B (\a / \a^2) . \]
\end{enumerate}
\end{lem}

\begin{proof}
(1) By Corollary \ref{cor:1090}, $\ba$ is a quasi-regular sequence in $A$.
Looking at formula (\ref{eqn:1086}), 
we see that $\a / \a^2 = \opn{Gr}_1^{\a}(A)$ is a free module over the ring
$B = \opn{Gr}_0^{\a}(A)$, with the basis the sequence $\bar{\bsym{a}}$.

\medskip \noindent 
(2) This is clear from (1). 
\end{proof}

\begin{lem} \label{lem:1110}
Let $\a \sub A$ be an ideal, let $B := A / \a$, let $\bsym{a}$ and $\ba'$ be 
Koszul regular sequences in $A$ of length $n$ that generate $\a$, and let 
$\bg \in \opn{Mat}_{n}(A)$ 
be a matrix such that $\ba' = \ba \cd \bg$. 
\begin{enumerate}
\item Let $\bar{\bg} \in \opn{Mat}_{n}(B)$ be the image of $\bg$ under the 
surjection $A \to B$. Then the matrix $\bar{\bg}$ is invertible.  

\item Let $M$ be an $A$-module and $\mu \in M$. Then 
$\smgfrac{\mu}{\ba'} = \opn{det}(\bar{\bg})^{-1} \cd \smgfrac{\mu}{\ba}$
in $\opn{Ext}^n_{A} (B, M )$.
\end{enumerate}
\end{lem}

Notice that by definition the ring $B$ is nonzero. 
A matrix $\bg$ as above always exists, yet it need not be invertible
in $\opn{Mat}_{n}(A)$. 

\begin{proof} 
(1) By Lemma \ref{lem:1111}(1), the $B$-module 
$\a / \a^2 = \opn{Gr}_1^{\a}(A)$ is a free module, with basis the sequence 
$\bar{\bsym{a}} = (\bar{a}_1, \ldots, \bar{a}_n)$.
The same is true for the sequence $\bar{\ba}'$ in $\a / \a^2$. We know 
that $\bar{\ba}' = \bar{\ba} \cd \bar{\bg}$. This proves that 
$\bar{\bg} \in \opn{GL}_n(B)$.  

\medskip \noindent
(2) Let $\bt = (t_1, \ldots, t_n)$ and $\bt' =  (t'_1, \ldots, t'_n)$ be 
sequences of degree $-1$ variables, and fix graded $A$-ring isomorphisms
$\opn{K}(A; \ba) \cong A[\bt]$ and $\opn{K}(A; \ba') \cong A[\bt']$ 
as in formula (\ref{eqn:1100}). 
Write the matrix $\bg$ as $\bg = [g_{i, j}]$ with $g_{i, j} \in A$. 
There is a unique graded $A$-ring homomorphism
$u_{\bg} : A[\bt'] \to A[\bt]$
such that $u_{\bg}(t'_j) = \sum_i t_i \cd g_{i, j}$.
Due to the Koszul sign rule, in degree $-n$ we get 
\begin{equation} \label{eqn:1110}
u_{\bg}(\opn{wdg}(\bt')) = 
u_{\bg}(t'_1 \cdots t'_n) =
u_{\bg}(t'_1) \cdots u_{\bg}(t'_n) =
\opn{wdg}(\bt \cd \bg) =
\opn{det}(\bg) \cd \opn{wdg}(\bt)
\end{equation}
in $A$.

Now the differentials of $\opn{K}(A; \ba)$ and $\opn{K}(A; \ba')$ satisfy 
$\d(t_i) = a_i$ and $\d(t'_i) = a'_i$, respectively, and they are $A$-linear. 
Hence 
\[ \d(u_{\bg}(t'_j)) = \d \bigl( \sum\nolimits_i t_i \cd g_{i, j} \bigr) = 
\sum\nolimits_i \d(t_i) \cd g_{i, j} = 
\sum\nolimits_i a_i \cd g_{i, j} = a'_j =
u_{\bg}(a'_j) = u_{\bg}(\d(t'_j)) . \]
This means that $u_{\bg} \circ \d = \d \circ u_{\bg}$, so it is a DG
$A$-ring homomorphism
\begin{equation} \label{eqn:1120}
u_{\bg} : \opn{K}(A; \ba') \to \opn{K}(A; \ba) .
\end{equation}
We get a commutative diagram 
\[ \begin{tikzcd} [column sep = 4ex, row sep = 4ex] 
\opn{K}(A; \ba') 
\arrow[r]
\arrow[rr, bend left = 15, start anchor = north, end anchor = north, 
"{u_{\bg}}"]
&
B
&
\opn{K}(A; \ba)
\ar[l]
\end{tikzcd} \]
of DG $A$-rings, where the arrows to $B$ are 
the augmentation quasi-iso\-morphisms. We conclude that in cohomology the 
diagram of $B$-module isomorphisms
\[ \begin{tikzcd} [column sep = 6ex, row sep = 4ex]
\opn{H}^n \bigl(  \opn{Hom}_{A} \bigl( \opn{K}(A; \bsym{a}'), M \bigr) \bigr) 
\arrow[r, "{\opn{kpres}^n_{\ba'}}", "{\simeq}"']
&
\opn{Ext}^n_{A} (B, M)
&
\opn{H}^n \bigl( \opn{Hom}_{A} \bigl( \opn{K}(A; \bsym{a}), M \bigr) \bigr)
\arrow[l, "{\opn{kpres}^n_{\ba}}"', "{\simeq}"]
\arrow[ll, bend right = 15, start anchor = north, end anchor = north, 
"{\opn{H}^n(\opn{Hom}_{A}(u_{\bg}, \opn{id}_M))}"', "{\simeq}"]
\end{tikzcd} \]
is commutative.

Consider the elements
$\smfrac{\mu}{\opn{wdg}(\bt)} \in 
\opn{Hom}_{A} \bigl( \opn{K}^{-n}(A; \bsym{a}), M \bigr)$ 
and 
$\opn{wdg}(\bt') \in \opn{K}^{-n}(A; \bsym{a}')$. 
Let us calculate:
\begin{equation} \label{eqn:1121}
\begin{aligned}
&
\opn{Hom}_A(u_{\bg}, \opn{id}_M) \bigl( \smfrac{\mu}{\opn{wdg}(\bt)} \bigr)
(\opn{wdg}(\bt')) 
= \smfrac{\mu}{\opn{wdg}(\bt)} \bigl( u_{\bg}(\opn{wdg}(\bt')) \bigr) 
\\ & \quad   
= \smfrac{\mu}{\opn{wdg}(\bt)} \big(\opn{det}(\bg) \cd \opn{wdg}(\bt) \bigr) 
= \opn{det}(\bg) \cd \mu 
= \opn{det}(\bg) \cd 
\smfrac{\mu}{\opn{wdg}(\bt')} (\opn{wdg}(\bt')) .  
\end{aligned}
\end{equation}
We are making use of formula (\ref{eqn:1110}). 
The conclusion is that 
\[ \opn{Hom}_A(u_{\bg}, \opn{id}_M) \bigl( \smfrac{\mu}{\opn{wdg}(\bt)} \bigr)
= \opn{det}(\bg) \cd \smfrac{\mu}{\opn{wdg}(\bt')} . \]
Thus in $\opn{Ext}^n_{A} (B, M)$ there is equality
$\smgfrac{\mu}{\ba} =  \opn{det}(\bar{\bg}) \cd \smgfrac{\mu}{\ba'}$.
Dividing both sides by $\opn{det}(\bar{\bg})$ we get the desired equality. 
\end{proof}

\begin{prop} \label{prop:1096}
Let $\a \sub A$ be a regular ideal, with $B := A / \a$. Then:
\begin{enumerate}
\item The $B$-module $\a / \a^2$ is projective. 

\item Assume that the projective $B$-module $\a / \a^2$ has constant rank $n$. 
Let $s \in A$ be an element such that the localized ring $B_s$ is nonzero, and 
such that the ideal $\a_s \sub A_s$ is generated by some Koszul regular
sequence $\ba$. Then the length of $\ba$ is $n$.
\end{enumerate}
\end{prop}

\begin{proof} 
(1) The condition is local on $\opn{Spec}(B)$. So it is enough to prove that
for every $s \in A$ such that $B_s \neq 0$, and such that the ideal 
$\a_s \sub A_s$ is generated by a Koszul regular sequence, the $B_s$-module
$\a_s / \a_s^2$ is free. This is done in Lemma \ref{lem:1111}(1).

\medskip \noindent
(2) By Lemma \ref{lem:1111}(1) the rank of the free $B$-module $\a_s / \a_s^2$
equals the length of $\ba$. 
\end{proof}

\begin{dfn} \label{dfn:1095}
Let $f : A \to B$ be a regular surjection, with kernel $\a$. 
We say that $f$ has {\em constant codimension $n$} if the projective 
$B$-module $\a / \a^2$ has constant rank $n$. 
\end{dfn}

In case $\a / \a^2$ is a projective $B$-module of constant rank $n$, the top 
exterior power $\bwedge^n_B (\a / \a^2)$ is a projective $B$-module of constant 
rank $1$.

\begin{dfn} \label{dfn:1085} 
Let $f : A \to B$ be a regular surjection of constant codimension $n$, with 
kernel $\a$. The {\em relative dualizing module of $B / A$} is the $B$-module
\[ \De_{B / A}^{\mrm{rs}} := 
\opn{Hom}_B \bigl( \bwedge^n_B (\a / \a^2) , B \bigr) . \]
\end{dfn}

The superscript ``rs'' refers to ``regular surjection''. 
The $B$-module $\De_{B / A}^{\mrm{rs}}$ is also projective of constant rank $1$.

Given a multiplicatively closed subset $S \sub A$, there are canonical 
$A_S$-module isomorphisms
$\a_S \cong \opn{Ker}(A_S \to B_S)$, 
$A_S \ot_A (\a / \a^2) \cong (\a_S / \a_S^2)$ 
and 
\begin{equation} \label{eqn:1096}
A_S \ot_A \De^{\mrm{rs}}_{B / A} \cong \De^{\mrm{rs}}_{B_S / A_S} .  
\end{equation}

\begin{prop} \label{prop:1095}
Let $\a \sub A$ be a regular ideal, and let $B := A / \a$.
Let $\ba = (a_1, \ldots, a_n)$ be a sequence in $\a$, with image
the sequence $\bar{\ba} = (\bar{a}_1, \ldots, \bar{a}_n)$
in $\a / \a^2$. Assume that for some $\p \in \opn{Spec}(B)$ the sequence 
$\opn{q}_{A, \p}(\bar{\ba})$ is a basis of the $B_{\p}$-module 
$(\a / \a^2)_{\p}$. Then there exists an element $s \in A - \p$ 
such that the sequence $\opn{q}_{A, s}(\ba)$ in the ring $A_s$ is 
regular, and it generates the ideal $\a_{s}$.
\end{prop}

\begin{proof}
Since $\a \sub A$ is a regular ideal, the ideal $\a_{\p} \sub A_{\p}$ is 
generated by some regular sequence $\ba'$. 
According to Proposition \ref{prop:1096}(2) the length of $\ba'$ is $n$. 
Consider the length $n$ sequence $\opn{q}_{A, \p}(\ba)$ in $\a_{\p}$. 
There is a matrix $\bsym{g} \in \opn{Mat}_{n}(A_{\p})$ such that 
$\opn{q}_{A, \p}(\ba) = \ba' \cd \bsym{g}$.  

As explained above, the sequence $\bar{\ba}'$, the image of $\ba'$ in the ring 
$B_{\p} = A_{\p} / \a_{\p}$, is a basis of the $B_{\p}$-module 
$\a_{\p} / \a^2_{\p}$. We are given that $\bar{\ba}$ is a basis
of $\a_{\p} / \a^2_{\p}$. It follows that the matrix 
$\bar{\bsym{g}} \in \opn{Mat}_{n}(B_{\p})$, the image of $\bsym{g}$, 
is invertible. Because $A_{\p} \to B_{\p}$ is a surjection of local rings, we 
see that $\bsym{g} \in \opn{GL}_n(A_{\p})$.
This implies that the sequence $\opn{q}_{A, \p}(\ba)$ generates the 
ideal $\a_{\p}$. 

By Theorem \ref{thm:1090} the sequence $\ba'$ in $A_{\p}$ is Koszul regular. 
Let 
$u_{\bg} : \opn{K}(A_{\p}, \opn{q}_{A, \p}(\ba)) \to 
\opn{K}(A_{\p}, \ba')$
be the DG ring homomorphism induced by $\bg$; like in (\ref{eqn:1120}), 
but in reverse direction.
Since $\bsym{g} \in \opn{GL}_n(A_{\p})$, it follows that 
$u_{\bg}$ is a DG ring isomorphism. We conclude that 
$\opn{q}_{A, \p}(\ba)$ is a Koszul regular sequence in $A_{\p}$ too. 
Again applying Theorem \ref{thm:1090}, this says that 
$\opn{q}_{A, \p}(\ba)$ is a 
regular sequence in $A_{\p}$. 

By Proposition \ref{prop:1091} there is an element $s \in A - \p$ 
such that the sequence $\opn{q}_{A, s}(\ba)$ in the ring $A_s$ is 
regular. After possibly changing $s$ we can also make sure that 
sequence $\opn{q}_{A, s}(\ba)$ generates the ideal $\a_{s}$.
\end{proof}

\begin{lem} \label{lem:1105}
Let $\a \sub A$ be an ideal, let $B := A / \a$, and let $\bsym{a}$ be a Koszul
regular sequence in $A$ that generates $\a$. 
\begin{enumerate}
\item The $B$-module $\De^{\mrm{rs}}_{B / A}$ is free, with basis the element 
$\smfrac{1}{\opn{wdg}(\bar{\ba})}$. 

\item Suppose $\ba'$ is another Koszul regular sequence in $A$ that generates 
$\a$, and $\bg \in \opn{Mat}_{n}(A)$ is a matrix such that 
$\ba' = \ba \cd \bg$. Then 
$\smfrac{1}{\opn{wdg}(\bar{\ba}')} = \opn{det}(\bar{\bg})^{-1} \cd 
\smfrac{1}{\opn{wdg}(\bar{\ba})}$
in $\De^{\mrm{rs}}_{B / A}$.
\end{enumerate}
\end{lem}

In the lemma, $\opn{wdg}(\bar{\ba}) \in \bwedge^n_B (\a / \a^2)$ 
is the element from Lemma \ref{lem:1111},  
$\smfrac{1}{\opn{wdg}(\bar{\ba})} \in \De^{\mrm{rs}}_{B / A}$
is the element from Definition \ref{dfn:1921}, 
with $L := \bwedge^n_B (\a / \a^2)$, $M := A$ and 
$\la := \opn{wdg}(\bar{\ba})$. The matrix 
$\bar{\bg} \in \opn{GL}_n(B)$ is the one from Lemma \ref{lem:1110},  

\begin{proof} 
(1) By Lemma \ref{lem:1111}(2) the $B$-module $\a / \a^2$ is free of rank $1$, 
with basis $\opn{wdg}(\bar{\ba})$. 

\medskip \noindent 
(2) Passing from $\a$ to $\a / \a^2$ we have 
$\bar{\ba}'  = \bar{\ba} \cd \bar{\bg}$.
Because of the multilinearity and antisymmetry of the expression
$\opn{wdg}(\bar{\ba}')$, it follows that
$\opn{wdg}(\bar{\ba}') = \opn{wdg}(\bar{\ba}) \cd \opn{det}(\bar{\bg})$.
A calculation that's almost the same as (\ref{eqn:1121})
shows that 
$\smfrac{1}{\opn{wdg}(\bar{\ba})} = 
\opn{det}(\bar{\bg}) \cd \smfrac{1}{\opn{wdg}(\bar{\ba}')}$
in $\De_{B / A}^{\mrm{rs}}$. 
Finally we divide both sides by $\opn{det}(\bar{\bg})$. 
\end{proof}

\begin{lem} \label{lem:1100}
Let $\a \sub A$ be an ideal, and let $B := A / \a$. 
Assume that $\a$ is generated by some Koszul regular sequence, and 
the $B$-module $\a / \a^2$ is free of rank $n$. 
Let $M$ be an $A$-module. Then:
\begin{enumerate}
\item There is a unique $B$-module isomorphism
\[ \opn{fund}_{B / A, M} : 
\De^{\mrm{rs}}_{B / A} \otimes_A M \iso \opn{Ext}^n_{A}(B, M)  \]
with this property: 
\begin{itemize}
\item[($*$)] For every Koszul regular sequence $\bsym{a}$ that generates the 
ideal $\a$, and for every element $\mu \in M$, we have 
\[ \opn{fund}_{B / A, M} 
\bigl( \smfrac{1}{\opn{wdg}(\bar{\bsym{a}})} \ot \mu \bigr) = 
\opn{kpres}^n_{\ba} \bigl( \smgfrac{\mu}{\bsym{a}} \bigr) . \]
\end{itemize}

\item For every $p > n$ we have $\opn{Ext}^p_{A}(B, M) = 0$. 

\item If $M$ is a flat $A$-module, then 
$\opn{Ext}^p_{A}(B, M) = 0$ also for all $0 \leq p < n$.  
\end{enumerate}
\end{lem}

\begin{proof}  
(1) Choose some Koszul regular sequence $\ba$ that generates $\a$. It has 
length $n$ by Proposition \ref{prop:1096}(2). Using the isomorphism 
(\ref{eqn:1100}), the element 
$\opn{wdg}(\bt) = t_1 \cdots t_n$ is a basis of the free 
$A$-module $\opn{K}^{-n}(A, \ba)$. 
We get an isomorphism of $A$-modules 
\[ M \iso \opn{Hom}_{A} \bigl( \opn{K}^{-n}(A; \bsym{a}), M \bigr) ,
\quad \mu \mapsto \smfrac{\mu}{\opn{wdg}(\bt)}   . \]
This isomorphism sends the submodule $\a \cd M \sub M$ bijectively to the 
submodule
of coboundaries in the $A$-module 
\[ \opn{Hom}_{A} \bigl( \opn{K}^{-n}(A; \bsym{a}), M \bigr) = 
\opn{Hom}_{A} \bigl( \opn{K}(A; \bsym{a}), M \bigr)^{n} . \]
Therefore in cohomology there is an induced
isomorphism of $B$-modules 
\begin{equation} \label{eqn:1124}
B \ot_A M \iso
\opn{H}^n \bigl( \opn{Hom}_{A} \bigl( \opn{K}(A; \bsym{a}), M \bigr) \bigr) . 
\end{equation}
Using the canonical isomorphism $\opn{kpres}^n_{\ba}$ of Definition
\ref{dfn:1923} implicitly, see formula (\ref{eqn:1921}), 
formula (\ref{eqn:1124}) yields an isomorphism
\[ \psi_{\ba} : B \ot_A M \iso \opn{Ext}^n_{A}(B, M)  , \quad 
1 \otimes \mu \mapsto \smgfrac{\mu}{\bsym{a}} . \]

Next, the element 
$\smfrac{1}{\opn{wdg}(\bar{\bsym{a}})}$ is a basis of the free $B$-module
$\De^{\mrm{rs}}_{B / A}$; so there is an isomorphism of $B$-modules 
\[ \th_{\ba} : \De_{B / A}^{\mrm{rs}} \ot_A M \iso B \ot_A M  ,
\quad \smfrac{1}{\opn{wdg}(\bar{\bsym{a}})} \ot \mu 
\mapsto 1 \otimes \mu . \]
Composing we obtain an isomorphism of $B$-modules
\[ \opn{fund}_{B / A, M} := \psi_{\ba} \circ \th_{\ba} : 
\De^{\mrm{rs}}_{B / A} \ot_A M \iso \opn{Ext}^n_{A}(B, M) , \quad 
\smfrac{1}{\opn{wdg}(\bar{\bsym{a}})} \ot \mu \mapsto
\smgfrac{\mu}{\bsym{a}}  . \]
This isomorphism satisfies condition ($*$) for the sequence 
$\ba$, and it is clearly unique with this property. 

It remains to prove that condition ($*$) holds for every other Koszul regular 
sequence $\ba'$  that generates the ideal $\a$. Here is the calculation, using 
the abbreviation $\phi := \opn{fund}_{B / A, M}$. 
\[ \begin{aligned}
&
\phi \bigl( \smfrac{1}{\opn{wdg}(\bar{\bsym{a}}')} \ot \mu \bigr)
=^{(1)} \phi \bigl( \opn{det}(\bar{\bg})^{-1} \cd 
\smfrac{1}{\opn{wdg}(\bar{\bsym{a}})} \ot \mu \bigr) 
\\
& \quad 
= \opn{det}(\bar{\bg})^{-1} \cd 
\phi \bigl( \smfrac{1}{\opn{wdg}(\bar{\bsym{a}})} \ot \mu \bigr)
= \opn{det}(\bar{\bg})^{-1} \cd \, \smgfrac{\mu}{\bsym{a}} =^{(2)} 
\smgfrac{\mu}{\bsym{a}'}  .
\end{aligned} \]
Here equality $=^{(1)}$ is by Lemma \ref{lem:1105}(2), and 
equality $=^{(2)}$ is by Lemma \ref{lem:1110}(2). 

\medskip \noindent
(2) This is because the complex $\opn{K}(A; \ba)$ is concentrated in degrees 
$-n, \ldots, 0$. 

\medskip \noindent
(3) There is an isomorphism of complexes 
\[ \opn{Hom}_{A} \bigl( \opn{K}(A; \bsym{a}), M \bigr) \cong 
\opn{Hom}_{A} \bigl( \opn{K}(A; \bsym{a}), A \bigr) \ot_A M . \]
Since $M$ is flat, for every $p$ there is an isomorphism of $B$-modules 
\[ \opn{H}^p \bigl( \opn{Hom}_{A} \bigl( \opn{K}(A; \bsym{a}), A \bigr) \ot_A M
\bigr) \cong
\opn{H}^p \bigl( \opn{Hom}_{A} \bigl( \opn{K}(A; \bsym{a}), A \bigr) \bigr)
\ot_A M . \]
Now there is an isomorphism of complexes (not canonical)
$\opn{Hom}_{A} \bigl( \opn{K}(A; \bsym{a}), A \bigl) \cong \lb 
\opn{K}(A; \bsym{a})[-n]$.
Therefore 
$\opn{H}^p \bigl( \opn{Hom}_{A} \bigl( \opn{K}(A; \bsym{a}), A \bigr) = 0$
for all $p \neq n$. 
\end{proof}

Given a multiplicatively closed subset $S \sub A$, there is a canonical 
$A_S$-module isomorphism 
\begin{equation} \label{eqn:1122}
A_S \ot_A \opn{Ext}^n_{A}(B, M) \cong \opn{Ext}^n_{A_S}(B_S, M_S) .  
\end{equation}

Next is a slight modification of \cite[Proposition III.7.2]{RD}.
Recall the notion of regular surjection of rings of constant codimension, from 
Definitions \ref{dfn:1081} and \ref{dfn:1095}. The relative dualizing module 
$\De_{B / A}^{\mrm{rs}}$ was introduced in Definition \ref{dfn:1085}. 
We remind that all rings are noetherian by default.

\begin{thm}[Fundamental Local Isomorphism] \label{thm:1080}
Let $A \to B$ be a regular surjection of noetherian rings, of constant 
codimension $n$, with kernel $\a$. Let $M$ be an $A$-module. 
\begin{enumerate}
\item There is a unique $B$-module isomorphism
\[ \opn{fund}_{B / A, M} : \De^{\mrm{rs}}_{B / A} \otimes_A M 
\iso \opn{Ext}^n_{A}(B, M) \]
with this property: 
\begin{itemize}
\item[($\dag$)] For every element $s \in A$ such that
$B_s \neq 0$, and such that the ideal $\a_s \sub A_s$ is generated by a Koszul 
regular sequence, the diagram
\[ \begin{tikzcd} [column sep = 14ex, row sep = 5ex] 
A_s \ot_A  (\De^{\mrm{rs}}_{B / A} \otimes_A M)
\arrow[d, "{\simeq}"]
\arrow[r, "{\opn{id} \ot_{A} \, \opn{fund}_{B / A, M}}", "{\simeq}"']
&
A_s \ot_A  \opn{Ext}^n_{A}(B, M)
\arrow[d, "{\simeq}"']
\\
\De^{\mrm{rs}}_{B_s / A_s} \otimes_{A_s} M_s
\arrow[r, "{\opn{fund}_{B_s / A, M_s}}", "{\simeq}"']
&
\opn{Ext}^n_{A_s}(B_s, M_s) 
\end{tikzcd} \] 
of $B_s$-modules is commutative. Here the vertical isomorphisms are from 
(\ref{eqn:1096}) and (\ref{eqn:1122}), and the isomorphism 
$\opn{fund}_{B_s / A, M_s}$ is the one from Lemma \ref{lem:1100}(1). 
\end{itemize}

\item For every $p > n$ we have $\opn{Ext}^p_{A}(B, M) = 0$. 

\item If $M$ is a flat $A$-module, then 
$\opn{Ext}^p_{A}(B, M) = 0$ also for all $0 \leq p < n$.  
\end{enumerate}
\end{thm}

\begin{proof} 
(1) Take an element $s \in A$ satisfying the conditions in property 
($\dag$). Thus 
$\phi_s := \opn{fund}_{B_s / A, M_s}$ is the unique $B_s$-module isomorphism 
such that 
$\phi_s \bigl( \smfrac{1}{\opn{wdg}(\bar{\bsym{a}})} \ot \mu \bigr) =
\smgfrac{\mu}{\bsym{a}}$
for every regular sequence $\ba$ in $A_s$ that generates the ideal 
$\a_s \sub A_s$ and for every element $\mu \in M$; this is condition ($*$) in 
Lemma \ref{lem:1100}(1). 

Consider two such elements $s_1, s_2 \in A$. The isomorphisms 
$\phi_{s_i}$, for $i = 1, 2$, induce isomorphisms of 
$B_{s_1 \cd s_2}$-modules 
\[ \opn{id} \ot \, \phi_{s_i} : 
A_{s_1 \cd s_2} \ot_{A_{s_i}}  (\De^{\mrm{rs}}_{B_{s_i} / A_{s_i}} 
\otimes_{A_{s_i}} M_{s_i}) \iso
A_{s_1 \cd s_2} \ot_{A_{s_i}} \opn{Ext}^n_{A_{s_i}}(A_{s_i}, M_{s_i})  . \]
Both these isomorphisms satisfy condition ($*$) in Lemma \ref{lem:1100}(1),
and therefore they are equal. (Indeed, they are both equal to 
$\phi_{s_1 \cd s_2}$.) 

Let's write $Y := \opn{Spec}(B)$. 
Given an element $s \in A$, define the principal affine open set
$Y_s := \opn{Spec}(B_s) = \opn{NZer}_{Y}(s) \sub Y$.
The principal affine open sets $Y_s \sub Y$, for elements $s \in A$ satisfying 
the conditions in property ($\dag$), cover $Y$. We have shown that 
the isomorphisms $\opn{id} \ot \, \phi_{s}$ satisfy the cocycle condition
on double intersections 
$Y_{s_i} \cap Y_{s_j} = Y_{s_i \cd s_j}$. 
Therefore they glue to a unique $B$-module isomorphism $\phi$ for which 
condition ($\dag$) holds. 

\medskip \noindent
(2, 3) These are clear from Lemma \ref{lem:1100}(2, 3) respectively.
\end{proof}

\begin{dfn} \label{dfn:1543}
Given a nonzero ring $B$, the noetherian affine scheme $\opn{Spec}(B)$ is a 
finite disjoint union of nonempty connected closed-open subschemes
$\opn{Spec}(B) = \coprod\nolimits_i \, \opn{Spec}(B_i)$.
This is the {\em connected component decomposition of $\opn{Spec}(B)$}. 

The corresponding ring isomorphism $B \cong \prod_{i} B_i$ is called the 
{\em connected component decomposition of $B$}. 
\end{dfn}

\begin{cor} \label{cor:1131}
Let $A \to B$ be a regular surjection rings. Then $B$ is perfect as a complex 
of $A$-modules.
\end{cor}

\begin{proof}
Let $B \cong \prod_{i} B_i$ the connected component decomposition of $B$.
Each $A \to B_i$ is a regular surjection of rings, of some constant codimension 
$n_i$. According to Theorem \ref{thm:1080}, the projective dimension of 
$B_i$ as an $A$-module is $n_i$. Then the projective dimension of $B$ as an 
$A$-module is $n := \opn{max} \, \{ n_i \} < \infty$.
By \cite[Theorem 14.1.33]{Ye5} the $A$-module $B$ is a perfect complex over $A$.
\end{proof}

\begin{prop} \label{prop:1260}
In the situation of Theorem \ref{thm:1080}, let $A \to A'$ be a localization 
homomorphism. Define $B' := A' \ot_A B$ and $M' := A' \ot_A M$. Assume that 
$B' \neq 0$. Then:
\begin{enumerate}
\item The ring homomorphism $A' \to B'$ is a regular surjection. 

\item The diagram of $B$-modules 
\[ \begin{tikzcd} [column sep = 12ex, row sep = 5ex] 
\De_{B / A}^{\mrm{rs}} \otimes_A M
\arrow[d, "{\opn{q}_{B' / B}}"]
\arrow[r, "{\opn{fund}_{B / A, M}}", "{\simeq}"']
&
\opn{Ext}^n_{A}(B, M)
\arrow[d, "{\opn{q}_{B' / B}}"']
&
\\
\De_{B' / A'}^{\mrm{rs}} \otimes_{A'} M'
\arrow[r, "{\opn{fund}_{B' / A', M'}}", "{\simeq}"']
&
\opn{Ext}^n_{A'}(B', M') 
\end{tikzcd} \] 
is commutative. Here the homomorphisms $\opn{q}_{B' / B}$ are the 
localization homomorphisms.
\end{enumerate} 
\end{prop}

\begin{proof} 
(1) By Proposition \ref{prop:1090}, regularity of a sequence is 
preserved under localization.

\medskip \noindent
(2) This is immediate from condition (\dag) of Theorem \ref{thm:1080}, i.e.\ 
both $\opn{fund}_{B / A, M}$ and $\opn{fund}_{B' / A', M'}$ are locally 
determined by the formula in property 
($*$) of Lemma \ref{lem:1100}.
\end{proof}

\section{Essentially Smooth Ring Homomorphisms}
\label{sec:smooth}

This section contains results on essentially smooth ring homomorphisms, 
generalizing some of the results on smoothness from \cite{EGA-I}, 
\cite{EGA-IV} and \cite{EGA-0IV}. The content of this section could be of 
independent interest, beyond its application to the theory of rigidity. 

According to Conventions \ref{conv:615} and \ref{conv:1070}, both in force in 
this section, all rings are commutative and noetherian,
and all ring homomorphisms are EFT. 

Let us recall the definitions of formally smooth and formally 
\'etale ring homomorphisms, following \cite[Definition 17.1.1]{EGA-IV}, 
\cite[Section 19.0]{EGA-0IV} and \cite[Section 28]{Ma} (with the discrete 
topologies on all rings). 
A ring homomorphism $f : A \to B$ is called {\em formally smooth}, and $B$ is 
called a {\em formally  smooth $A$-ring}, if $f$ has the following lifting 
property: given an $A$-ring $C$, a nilpotent ideal $\mathfrak{c} \sub C$, and an $A$-ring 
homomorphism $\bar{g} : B \to C / \mathfrak{c}$, {\em there exists} an $A$-ring 
homomorphism $g : B \to C$ lifting $\bar{g}$. If moreover the lifting 
$g : B \to C$ is always {\em unique}, then the homomorphism $f : A \to B$ is 
called {\em formally \'etale}, and $B$ is called a {\em formally \'etale 
$A$-ring}. See next diagram, in which $\pi: C \to C / \mathfrak{c}$ is the canonical 
surjection.  
\begin{equation} \label{eqn:1015}
\begin{tikzcd} [column sep = 8.6ex, row sep = 5ex] 
A
\arrow[r]
\ar[d, "{f}"']
&
C
\ar[d, "{\pi}"]
\\
B
\ar[r, "{\bar{g}}"']
\ar[ur, dashed, "{g}"]
&
C / \mathfrak{c}
\end{tikzcd}
\end{equation}

In case the homomorphism $f$ is finite 
type and formally smooth (resp.\ formally \'etale), then $f$ is called a {\em 
smooth} (resp.\ {\em \'etale}) homomorphism, and $B$ is called a {\em smooth} 
(resp.\ {\em \'etale}) $A$-ring.
See \cite[Definition 17.3.1]{EGA-IV}, which is slightly more general (because 
the rings are not assumed to be noetherian there).

Essentially finite type ring homomorphisms were introduced in Definition 
\ref{dfn:1065}. 

\begin{dfn} \label{dfn:1015}
Let $f : A \to B$ be a homomorphism between noetherian rings. 
The homomorphism $f$ is called an {\em essentially smooth} (resp.\ 
{\em essentially \'etale}) {\em homomorphism}, 
and $B$ is called an {\em essentially smooth}  (resp.\ {\em essentially 
\'etale}) {\em $A$-ring}, if $f$ is
essentially finite type and formally smooth (resp.\ formally \'etale). 
\end{dfn}

Here is the most basic example of an essentially \'etale homomorphism:

\begin{exa} \label{exa:1050}
If $f : A \to B$ is a localization homomorphism between noetherian rings,
then $f$ is essentially \'etale. All that's needed is to prove that $f$ is 
formally \'etale, and this is implicit in the proof of  
\cite[Proposition 19.3.5$\tup{(iv)}$]{EGA-0IV}. Here is the full argument. Let
$S \sub A$ be a multiplicatively closed set such that $B \cong A[S^{-1}]$.
Given an $A$-ring $C$, a nilpotent ideal $\mathfrak{c} \sub C$, and an $A$-ring 
homomorphism $\bar{g} : B \to C / \mathfrak{c}$, let $T \sub C$ be the image of $S$. 
So $\bar{g}(f(S)) = \pi(T) \sub C / \mathfrak{c}$,
where $\pi : C \to C / \mathfrak{c}$ is the surjection. We see that $\pi(t)$ is invertible 
in $C / \mathfrak{c}$ for every $t \in T$. Since $\mathfrak{c}$ is a nilpotent ideal, it follows that 
$t$ is invertible in $C$. Therefore there's a unique $A$-ring homomorphism 
$g : B \to C$, and it makes diagram (\ref{eqn:1015}) commute. 
\end{exa}

\begin{prop} \label{prop:1015}
Let $f : A \to B$ be an essentially smooth (resp.\ essentially \'etale) 
ring homomorphism.
\begin{enumerate}
\item If $g : B \to C$ is another essentially smooth (resp.\ 
essentially \'etale) homomorphism, then $g \circ f$ is essentially smooth 
(resp.\ essentially \'etale).

\item Let $A \to C$ be any homomorphism to a noetherian ring $C$. Then the 
induced homomorphism $C \to B \ot_A C$ is essentially smooth (resp.\ 
essentially \'etale). 
\end{enumerate}
\end{prop}

\begin{proof} 
(1) We know that the composition of two EFT homomorphisms is EFT.
And by \cite[Proposition 17.1.3]{EGA-IV}(ii), the composition of two 
formally smooth (resp.\ formally \'etale) homomorphisms is formally smooth 
(resp.\ formally \'etale).

\medskip \noindent(2) We know that the homomorphism $C \to B \ot_A C$ 
is EFT, and that the ring $B \ot_A C$ is noetherian. 
According to \cite[Proposition 17.1.3]{EGA-IV}(iii), taking $S = X$ there, the 
homomorphism $C \to B \ot_A C$ is formally smooth (resp.\ formally \'etale). 
\end{proof}

\begin{exa} \label{exa:1015}
Suppose $f : A \to B$ is a smooth (resp.\ \'etale) ring homomorphism, 
$\q \sub B$ is a prime ideal, and $\p := f^{-1}(\q) \sub A$. 
Then the induced homomorphism $A_{\p} \to B_{\q}$ is essentially smooth
(resp.\ essentially \'etale). 
This is a consequence of Proposition \ref{prop:1015} and Example \ref{exa:1050}.
\end{exa}

\begin{prop} \label{prop:1055}
Let $f : A \to B$ be a ring homomorphism. The following conditions are 
equivalent:
\begin{enumerate}
\rmitem{i} The homomorphism $f$ is essentially \'etale. 

\rmitem{ii} The homomorphism $f$ is essentially smooth, and  
$\Om^1_{B / A} = 0$.
\end{enumerate}
\end{prop}

\begin{proof}
The homomorphism $f$ is called {\em formally unramified} if there is always 
{\em at most one} lifting $g$ in diagram (\ref{eqn:1015}); see 
\cite[Definition 17.1.1]{EGA-IV}. Thus $f$ is formally \'etale 
iff it is both formally smooth and formally unramified. 
But by \cite[Proposition 17.2.1]{EGA-IV} or \cite[Proposition 20.7.4]{EGA-0IV}
the homomorphism $f$ is formally unramified iff $\Om^1_{B / A} = 0$.
\end{proof}

A {\em principal localization} of a ring $B$ is a $B$-ring that's isomorphic
to $B_s = B[s^{-1}]$ for some element $s \in B$. The corresponding affine open 
set $\opn{Spec}(B_s) \cong \opn{NZer}_{\opn{Spec}(B)}(s) \sub \opn{Spec}(B)$ 
is called a {\em principal affine open set}. See formula (\ref{eqn:1532}) 
regarding notation. 

\begin{thm} \label{thm:1015}
Let $A \to B$ be an essentially smooth (resp.\ essentially \'etale) 
homomorphism between noetherian rings. There is an affine open covering 
$\opn{Spec}(B) = \bigcup_i \opn{Spec}(B_i)$,
such that each $B_i$ is a principal localization of $B$, and 
the ring homomorphism $A \to B_i$ is the composition of a smooth (resp.\ 
\'etale) homomorphism $A \to \til{B}_i$ with a localization homomorphism 
$\til{B}_i \to B_i$. 
\end{thm}

\begin{proof}
Define $X := \opn{Spec}(A)$, $Y := \opn{Spec}(B)$  
and $\phi := \opn{Spec}(f) : Y \to X$. 
Choose a finite type $A$-subring $B^{\mrm{ft}} \sub B$ such that $B$ is a 
localization of $B^{\mrm{ft}}$. We can view the underlying topological space 
of $Y$ as a topological subspace of 
$Y^{\mrm{ft}} := \opn{Spec}(B^{\mrm{ft}})$. 

Take a point $y \in Y$, and let $x := \phi(y) \in X$. 
By Example \ref{exa:1015} the local ring 
$B_y \cong \mcal{O}_{Y, y} \cong \mcal{O}_{Y^{\mrm{ft}}, y}$ 
is a formally smooth (resp.\ formally \'etale) $A_x$-ring. According to 
\cite[Th\'eor\`eme 17.5.1]{EGA-IV}
and \cite[Definition 17.3.7]{EGA-IV}, 
there is an open neighborhood $V^{\mrm{ft}}$ of $y$ in $Y^{\mrm{ft}}$ 
which is smooth (resp.\ \'etale) over $X$. 
Choose an element $s \in B^{\mrm{ft}}$ such 
that the localization $B^{\mrm{ft}}[s^{-1}]$ 
satisfies
$y \in \opn{Spec}(B^{\mrm{ft}}[s^{-1}]) \sub V^{\mrm{ft}}$.
By \cite[Proposition 17.1.6]{EGA-IV} and \cite[Remarques 17.1.2]{EGA-IV} 
we know that $B^{\mrm{ft}}[s^{-1}]$ is a smooth (resp.\ \'etale) $A$-ring.
Let $i$ be an index corresponding to the point $y$, 
and define the $A$-rings $\til{B}_i := B^{\mrm{ft}}[s^{-1}]$ and
$B_i := B[s^{-1}]$. Then 
$\til{B}_i$ is a smooth (resp.\ \'etale) $A$-ring, 
$B_i$ is a localization of $\til{B}_i$ and a principal localization of $B$,
$Y_i := \opn{Spec}(B_i)$ is a principal affine open set in $Y$, and 
$y \in Y_i$. 
See Figure \ref{fig:15}, in which $\til{Y}_i := \opn{Spec}(\til{B}_i)$. 

Performing the construction above for all points $y \in Y$ we obtain a covering 
$Y = \bigcup_i Y_i$ with the required properties. 
\end{proof}

Of course the covering in the theorem can be made finite, because 
$\opn{Spec}(B)$ is quasi-compact. 

\begin{rem} \label{rem:1560}
In the situation of Theorem \ref{thm:1015}, and its proof, the map of schemes
$\til{\phi}_i : \til{Y}_i \to X$ is smooth. 
The map of schemes $Y_i \to \til{Y}_i$ is an embedding of topological spaces, 
but it usually {\em is not an embedding of schemes}, and moreover it 
is usually not a finite type map of schemes. 
The reason is this: let $S_i \sub \til{B}_i$ be a multiplicatively closed set 
such that $B_i \cong (\til{B}_i)_{S_i}$. If $S_i$ is finitely generated as a 
multiplicative monoid, then the ring homomorphism $\til{B}_i \to B_i$ is finite 
type; but otherwise $\til{B}_i \to B_i$ is just essentially finite type. 
\end{rem}

\begin{figure}[!t]
\centering
\includegraphics[scale=0.17]{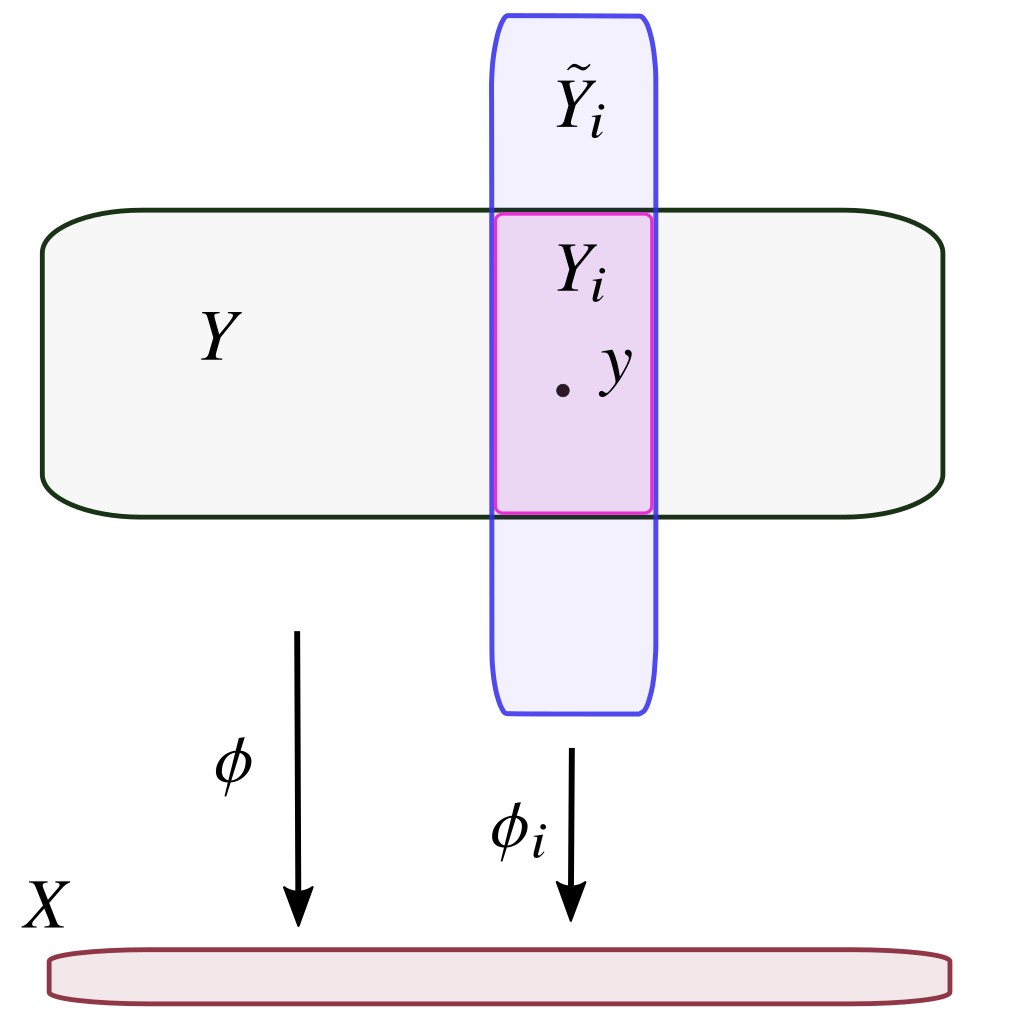}
\caption{Illustration for the proof of Theorem \ref{thm:1015} and for Remark 
\ref{rem:1560}.} 
\label{fig:15}
\end{figure}

Here is a quick review of some facts on {\em differential forms}, following 
\cite[Section 20]{EGA-0IV}, \cite[Section 16]{EGA-IV} and 
\cite[Section 25]{Ma}.
Given an EFT ring homomorphism $A \to B$, the module of differential $1$-forms 
of $B / A$ is $\Om^1_{B / A}$. There is a universal $A$-linear derivation 
$\d_{B / A} : B \to \Om^1_{B / A}$; namely for every $B$-module $N$ the 
$A$-linear homomorphism 
$\opn{Hom}_B(\Om^1_{B / A}, N) \to \opn{Der}_A(B, N)$, 
$\phi \mapsto \phi \circ \d_{B / A}$, 
is bijective. 

There is a standard presentation of the universal derivation
$\bigl( \Om^1_{B / A}, \d_{B / A} \bigr)$, as we explain below. 

\begin{dfn} \label{dfn:1544}
Given a ring homomorphism $A \to B$, we have the multiplication homomorphism
$\opn{mult}_{B / A} : B \ot_A B \to B$, 
$\opn{mult}_{B / A}(b_1 \ot b_2) := b_1 \cd b_2$. 
Its kernel is the diagonal ideal
$\opn{I}_{B / A} := \opn{Ker} (\opn{mult}_{B / A}) \sub B \ot_A B$. 
\end{dfn}

Geometrically, letting $X := \opn{Spec}(A)$ and $Y := \opn{Spec}(B)$, 
the diagonal embedding
$\opn{diag}_{Y / X} : Y \to Y \times_X Y$,
is 
$\opn{diag}_{Y / X} := \opn{Spec}(\opn{mult}_{B / A})$.
In other words, $\opn{I}_{B / A}$ is the defining ideal of the diagonal:
$\opn{diag}_{Y / X}(Y) = \opn{Zer}_{Y \times_X Y}(\opn{I}_{B / A})
\sub Y \times_X Y$. 
This is why we call it the diagonal ideal.

\begin{dfn} \label{dfn:1935}
Let $A \to B$ be a ring homomorphism. For an element $b \in B$ we define 
$\til{\d}_{B / A}(b) := b \ot 1 - 1 \ot b \in \opn{I}_{B / A}$.
\end{dfn}

It is not hard to show that the elements $\til{\d}_{B / A}(b)$ generate 
the ideal $\opn{I}_{B / A} \sub B \ot_A B$. 
There's a unique $B$-module 
isomorphism
\begin{equation} \label{eqn:1154}
\opn{I}_{B / A} \, / \, \opn{I}_{B / A}^2 \iso \Om^1_{B / A}
\end{equation}
with formula  
\begin{equation} \label{eqn:1136}
\bigl( \til{\d}_{B / A}(b) + \opn{I}_{B / A}^2 \bigr) \mapsto \d_{B / A}(b)  
\end{equation}
for $b \in B$. 
Since $B \ot_A B$ is a noetherian ring, it follows that $\opn{I}_{B / A}$ is a 
finitely generated ideal, and hence $\Om^1_{B / A}$ is a 
finitely generated $B$-module. Here is another consequence of this 
presentation: because the elements
$\til{\d}_{B / A}(b)$ generate $\opn{I}_{B / A}$ as an ideal, it follows that 
the $1$-forms $\d_{B / A}(b)$ generate $\Om^1_{B / A}$ as a $B$-module.

Given a sequence $\bb = (b_1, \ldots, b_n)$ of elements in the ring $B$, we let 
\begin{equation} \label{eqn:newa}
\d_{B / A}(\bb) := \bigl( \d_{B / A}(b_1), \ldots, \d_{B / A}(b_n) \bigr)
\end{equation}
and 
\begin{equation} \label{eqn:newb}
\til{\d}_{B / A}(\bb) := \bigl( \til{\d}_{B / A}(b_1), \ldots, 
\til{\d}_{B / A}(b_n) \bigr),
\end{equation}
which are seq\-uences in $\Om^1_{B / A}$ and $\opn{I}_{B / A}$ resp. 

The module of differentials $\Om^1_{- / -}$ has the following  
functorialities. Given a commutative diagram 
\begin{equation} \label{eqn:1665}
\begin{tikzcd} [column sep = 5ex, row sep = 4ex] 
A
\ar[r, "{}"]
\ar[d, "{f}"']
&
B
\ar[d, "{g}"]
\\
A'
\ar[r]
&
B'
\end{tikzcd} 
\end{equation}
in $\cat{Rng}$, there is an induced forward homomorphism of $B$-modules 
$\Om^1_{g / f} : \Om^1_{B / A} \to \Om^1_{B' / A'}$
whose formula is 
$\Om^1_{g / f}(\d_{B / A}(b)) = \d_{B' / A'}(g(b))$.

Now suppose that in diagram (\ref{eqn:1665}) we have $A = A'$. Renaming objects 
and morphisms, we get these ring homomorphisms
$A \xar{f} B \xar{g} C$. Let's write 
$\Om^1_{g / A} := \Om^1_{g / \opn{id}_A} : \Om^1_{B / A} \to \Om^1_{C / A}$ and 
$\Om^1_{C / f} := \Om^1_{\opn{id}_C / f} : \Om^1_{C / A} \to \Om^1_{C / B}$.  
These homomorphisms sit inside an exact sequence of $C$-modules 
\begin{equation} \label{eqn:1070}
C \ot_B \Om^1_{B / A} \xar{\opn{id} \ot \, \Om^1_{g / A}} \Om^1_{C / A} 
\xar{\Om^1_{C / f}} \Om^1_{C / B} \to 0 ,
\end{equation}
called the {\em first fundamental exact sequence}.  
If $g : B \to C$ is formally smooth, then the sequence 
\begin{equation} \label{eqn:1071}
0 \to C \ot_B \Om^1_{B / A} \xar{\opn{id} \ot \, \Om^1_{g / A}} \Om^1_{C / A} 
\xar{\Om^1_{C / f}} \Om^1_{C / B} \to 0
\end{equation}
is split exact. See \cite[Theorem 20.5.7]{EGA-0IV} or 
\cite[Theorem 25.1]{Ma}. A necessary and 
sufficient condition for this split-exactness is that $g$ is  
formally smooth relative to $A$ (a property to be recalled later). 

The last property of $\Om^1_{- / -}$ we want to mention here is this.
Assume that diagram (\ref{eqn:1665}) is cocartesian, i.e.\ 
$B' \cong A' \ot_A B$. Then the homomorphism of $B'$-modules 
\begin{equation} \label{eqn:11665}
\opn{id} \ot \, \Om^1_{g / f} : 
B' \ot_B \Om^1_{B / A} \to \Om^1_{B' / A'}
\end{equation}
is bijective. See \cite[Proposition 20.5.5]{EGA-IV}.

\begin{prop} \label{prop:1141}
Let $A \to B \xar{g} C$ be EFT ring homomorphisms, and assume that $g$ is 
essentially \'etale. Then the homomorphism 
$\opn{id} \ot \, \Om^1_{g / A} : C \ot_B \Om^1_{B / A} \to \Om^1_{C / A}$
is bijective. 
\end{prop}

\begin{proof}
By Proposition \ref{prop:1055} we know that 
$\Om^1_{C / B} = 0$. Now use the exact sequence (\ref{eqn:1071}).
\end{proof}

\begin{thm} \label{thm:1140}
Let $f : A \to B$ be an essentially smooth homomorphism between noetherian 
rings. Then:
\begin{enumerate}
\item The ring $B$ is flat over $A$.

\item The module of differentials $\Om^1_{B / A}$ is a finitely generated 
projective $B$-module.
\end{enumerate}
\end{thm}

\begin{proof} 
(1) Take an affine open covering $\opn{Spec}(B) = \bigcup_i \opn{Spec}(B_i)$
as in Theorem \ref{thm:1015}. For each $i$ the $A$-ring $\til{B}_i$ is smooth, 
and hence, according to \cite[Corollaire 17.5.2]{EGA-IV}, $\til{B}_i$ is flat 
over $A$. Because $\til{B}_i \to B_i$ is a localization, we see that $B_i$ is a 
flat $A$-ring. Due to the local nature of flatness, it follows that $B$ is a 
flat $A$-ring. 

\medskip \noindent
(2) We already know that $\Om^1_{B / A}$ is a finitely generated $B$-modules
(because $A \to B$ is EFT). Showing that $\Om^1_{B / A}$ is projective 
(which in this case is the same as being flat) is a local question. 
Take the open covering of $\opn{Spec}(B)$ from above. 
It suffices to prove that for every $i$ the $B_i$-module
$\Om^1_{B_i / A}$ is a projective $B_i$-module.
Since the $A$-ring $\til{B}_i$ from Theorem \ref{thm:1015} is smooth, according 
to \cite[Proposition 17.2.3$\tup{(i)}$]{EGA-IV} the 
$\til{B}_i$-module $\Om^1_{\til{B}_i / A}$ is a finitely generated projective 
module. By Proposition \ref{prop:1141} we know that 
$B_i \ot_{\til{B}_i} \Om^1_{\til{B}_i / A} \cong \Om^1_{B_i / A}$,
and this is is a finitely generated projective $B_i$-module. 
\end{proof}

Suppose $N$ is a finitely generated projective $B$-module. For every prime 
ideal $\q \in \opn{Spec}(B)$ the $B_{\q}$-module $N_{\q}$ is free, and its  
rank is denoted by $\opn{rank}_{B_{\q}}(N_{\q})$. The function 
$\q \mapsto \opn{rank}_{B_{\q}}(N_{\q})$ is constant on connected components of 
$\opn{Spec}(B)$. Sometimes the rank is constant on $\opn{Spec}(B)$, and then we 
denote it by $\opn{rank}_{B}(N)$. 
In the next definition this is done for the $B$-module 
$N := \Om^1_{B / A}$, which by Theorem \ref{thm:1015} is finitely 
generated projective.

\begin{dfn} \label{dfn:1067}
Let $f : A \to B$ be an essentially smooth homomorphism between noetherian 
rings. If the projective $B$-module $\Omega^1_{B / A}$ has constant rank $n$, 
then we call $f$ an {\em essentially smooth homomorphism of constant 
differential relative dimension $n$}, and $B$ is called an {\em essentially 
smooth $A$-ring of constant differential relative dimension $n$}.
\end{dfn}

By Proposition \ref{prop:1055}, an essentially \'etale homomorphism is the same 
as an essentially smooth homomorphism of differential relative dimension $0$.

\begin{rem} \label{rem:1065}
If $f : A \to B$ is a smooth homomorphism, then the differential relative  
dimension of $f$ equals the geometric relative dimension of $f$. Namely, for a 
prime ideal $\q \sub B$ with preimage $\p := f^{-1}(\q) \sub A$,
there is equality 
$\opn{rank}_{B_{\q}}(\Omega^1_{B_{\q} / A_{\p}}) = 
\opn{dim} \bigl( \bk(\p) \ot_{A_{\p}} B_{\q} \bigr)$.
Here $\opn{dim}$ is Krull dimension, and $\bk(\p)$ is the residue field
of $A_{\p}$. See \cite[Proposition 17.10.2]{EGA-IV}.

This can be false when $f : A \to B$ is essentially smooth (but not finite 
type). Take for example a field $A$, and let $B := A(t)$, the field of rational 
functions in a single variable $t$. Then the differential relative dimension of 
$f$ is $1$, whereas the geometric relative dimension of $f$ is $0$. 
\end{rem}

\begin{prop} \label{prop:1140}
Let $A \to B$ be an essentially smooth ring homomorphism of constant 
differential relative dimension $n$. Given $\p \in \opn{Spec}(B)$, there exists 
a sequence $\bb = (b_1, \ldots, b_n)$ in $B$ and an element $s \in B - \p$, 
s.t.\ the sequence $\d_{B / A}(\bb)$ is a basis of the the 
$B_s$-module 
$B_s \ot_B \Om^1_{B / A} \cong \Om^1_{B_s / A}$. 
\end{prop}

\begin{proof}
We know, from Theorem \ref{thm:1140}, that $\Om^1_{B / A}$ is a projective 
$B$-module of rank $n$. Therefore $\bk(\p) \ot_B \Om^1_{B / A}$ 
is a free $\bk(\p)$-module of rank $n$. We also know that 
$\Om^1_{B / A}$ is generated as a $B$-module by the differential forms
$\d_{B / A}(b)$, $b \in B$. This implies that there exists a sequence 
$\bb = (b_1, \ldots, b_n)$ in $B$ 
s.t.\ the image of the sequence of forms $\d_{B / A}(\bb)$ in 
$\bk(\p) \ot_B \Om^1_{B / A}$ is a basis of this $\bk(\p)$-module.
By the Nakayama Lemma, the  sequence of forms $\d_{B / A}(\bb)$
is a basis of the free $B_{\p}$-module 
$B_{\p} \ot_B \Om^1_{B / A}$. The standard argument shows that this remains 
true on a suitable principal open neighborhood 
$\opn{Spec}(B_s)$ of $\p$ in $\opn{Spec}(B)$.
\end{proof}

\begin{rem} \label{rem:1220}
We shall encounter several situations of localization of tensor products of 
rings. First, let $s \in B$. Then there is a unique isomorphism of
$(B \ot_A B)$-rings 
$B_s \ot_A B_s \cong (B \ot_A B)_{s \ot s}$. The image of 
$\til{s} := s \ot s$ in $B$, under the homomorphism 
$\opn{mult}_{B / A}$, is $s^2$; but $B_{s^2} = B_s$, so 
$(B \ot_A B)_{\til{s}} \ot_{B \ot_A B} B \cong B_s$. 

Second, suppose we are given an element $\til{s} \in B \ot_A B$, with image
$s := \opn{mult}_{B / A}(\til{s}) \in B$. Then 
$(B \ot_A B)_{\til{s}} \ot_{B \ot_A B} B \cong B_s$. 
However, in general the $(B \ot_A B)$-rings
$(B \ot_A B)_{\til{s}}$ and $B_s \ot_A B_s$ are not isomorphic -- see 
Remark \ref{rem:1180}.
\end{rem}

\begin{thm} \label{thm:1141}
Let $A \to B$ be an essentially smooth homomorphism between noetherian rings. 
Then the multiplication homomorphism 
$\opn{mult}_{B / A} : B \ot_A B \to B$
is a regular surjection. 
\end{thm}

See Definition \ref{dfn:1081} regarding the concept of regular surjection, and 
Definition \ref{dfn:1544} regarding the multiplication homomorphism. 

We shall need a lemma for the proof. An $A$-ring homomorphism 
$g : B \to B'$ induces an $A$-ring homomorphism 
$g \ot g : B \ot_A B \to B' \ot_A B'$,
and this induces a $(B \ot_A B)$-module homomorphism 
$\opn{I}_{g / A} : \opn{I}_{B / A} \to \opn{I}_{B' / A}$. 
There is a commutative diagram 
\begin{equation} \label{eqn:1666}
\begin{tikzcd} [column sep = 5ex, row sep = 4ex] 
\opn{I}_{B / A}
\ar[r, "{\opn{I}_{g / A}}"]
\ar[d]
&
\opn{I}_{B' / A}
\ar[d]
\\
\Om^1_{B / A}
\ar[r, "{\Om^1_{g / A}}"]
&
\Om^1_{B' / A}
\end{tikzcd} 
\end{equation}
of $(B \ot_A B)$-modules, in which the vertical arrows come from 
(\ref{eqn:1154}). 

\begin{lem} \label{lem:1195}
Let $A$ be a ring, and let $g : B \to B'$ be a localization 
homomorphism of $A$-rings. Then the $(B' \ot_A B')$-module  homomorphism  
\[ \opn{id} \ot \, \opn{I}_{g / A} :
(B' \ot_A B') \ot_{B \ot_A B}  \opn{I}_{B / A} \to \opn{I}_{B' / A} \]
is bijective. 
\end{lem}

\begin{proof}
We know that the ring homomorphism
$(B' \ot_A B') \ot_{B \ot_A B} B \to B'$
is bijective. Consider the commutative diagram with exact rows 
\[ \begin{tikzcd} [column sep = 5ex, row sep = 4ex] 
0
\ar[r]
&
\opn{I}_{B / A}
\ar[r]
\ar[d, "{\opn{I}_{g / A}}"]
&
B \ot_A B
\ar[r, "{\opn{mult}}"]
\ar[d, "{g \ot g}"]
&
B
\ar[r]
\ar[d, "{g}"]
&
0
\\
0
\ar[r]
&
\opn{I}_{B' / A}
\ar[r]
&
B' \ot_A B'
\ar[r, "{\opn{mult}}"]
&
B'
\ar[r]
&
0
\end{tikzcd} \]
We see that the second row is gotten from the first row by applying the 
operation \lb $(B' \ot_A B') \ot_{B \ot_A B} (-)$.
\end{proof}

\begin{proof}[Proof of Theorem \tup{\ref{thm:1141}}]
Let's write $X := \opn{Spec}(A)$ and $Y := \opn{Spec}(B)$. 
Geometrically we have to prove that the diagonal embedding 
$\opn{diag}_{Y / X} : Y \to Y \times_X Y$
is a regular closed embedding of schemes. 
For convenience we sometimes identify $Y$ with the closed subscheme 
$\opn{diag}_{Y / X}(Y) \sub Y \times_X Y$. 
Thus, according to Definitions \ref{dfn:1080} and \ref{dfn:1081}, given a point 
$y \in Y \sub Y \times_X Y$, we must prove that there is some element 
$\til{s} \in B \ot_A B$ such that 
$y \in \opn{Spec} \bigl( (B \ot_A B)_{\til{s}} \bigr) = 
\opn{NZer}_{Y \times_X Y}(\til{s}) \sub Y \times_X Y$,
and such that the ideal 
\[(B \ot_A B)_{\til{s}} \ot_{B \ot_A B} \opn{I}_{B / A} \sub 
(B \ot_A B)_{\til{s}} \]
is generated by a regular sequence. 
According to Theorem \ref{thm:1015} we can find an affine open neighborhood
$Y' = \opn{Spec}(B')$ of $y$ in $Y$, 
s.t.\ the ring $B'$ is a principal localization of $B$, and the ring 
homomorphisms $A \to B'$ factors through a smooth homomorphism
$A \to \til{B}'$ and a localization homomorphism 
$g : \til{B}' \to B'$. 
Note that $Y'$ can be seen as a topological subspace of the scheme
$\til{Y}' := \opn{Spec}(\til{B}')$. See Figure \ref{fig:5}.
Likewise the product 
$Y' \times_X Y' = \opn{Spec}(B' \ot_A B')$ 
is a topological subspace of the scheme
$\til{Y}' \times_X \til{Y}' =  \opn{Spec}(\til{B}' \ot_A \til{B}')$.
See Figure \ref{fig:6}.

\begin{figure}[h]
\centering
\includegraphics[scale=0.13]{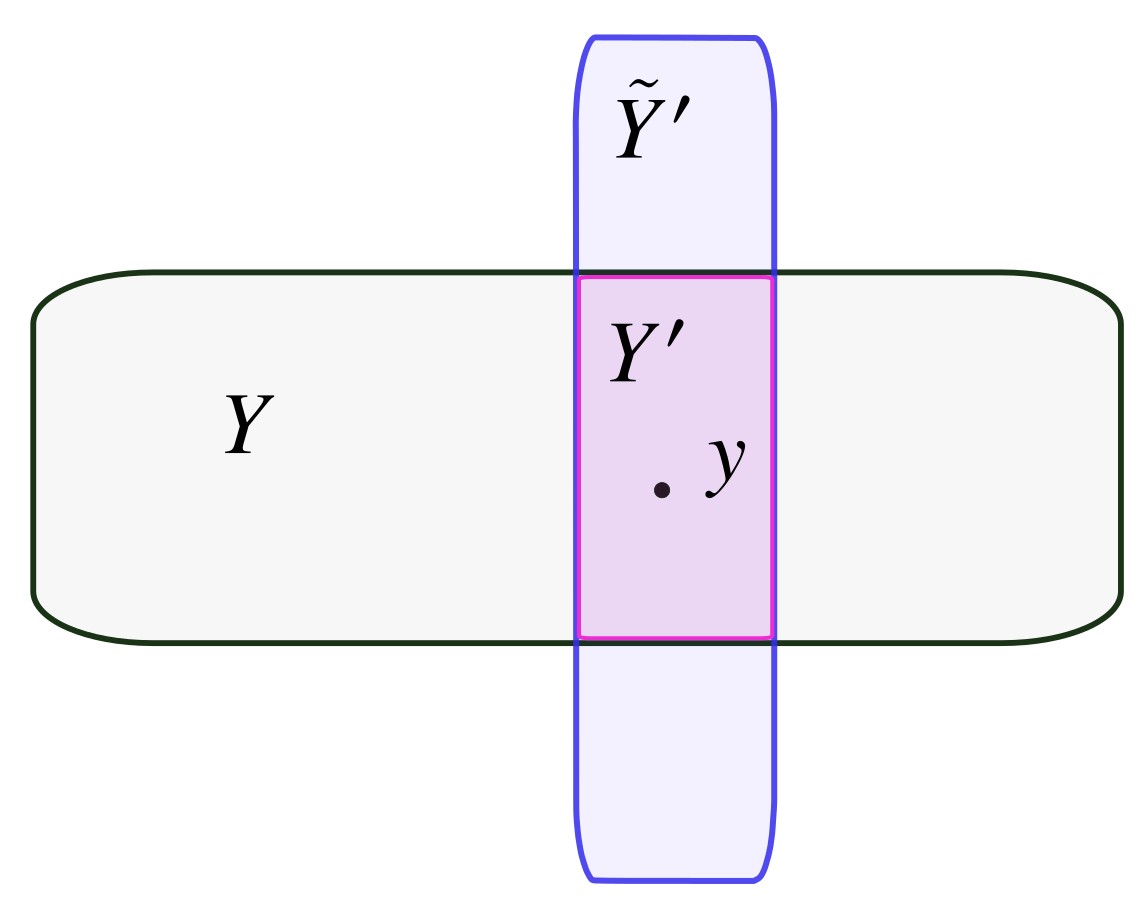}
\caption{
$y \in Y' \sub Y$, and $Y' \sub \til{Y}'$.} 
\label{fig:5}
\end{figure}

\begin{figure}[h]
\centering
\includegraphics[scale=0.23]{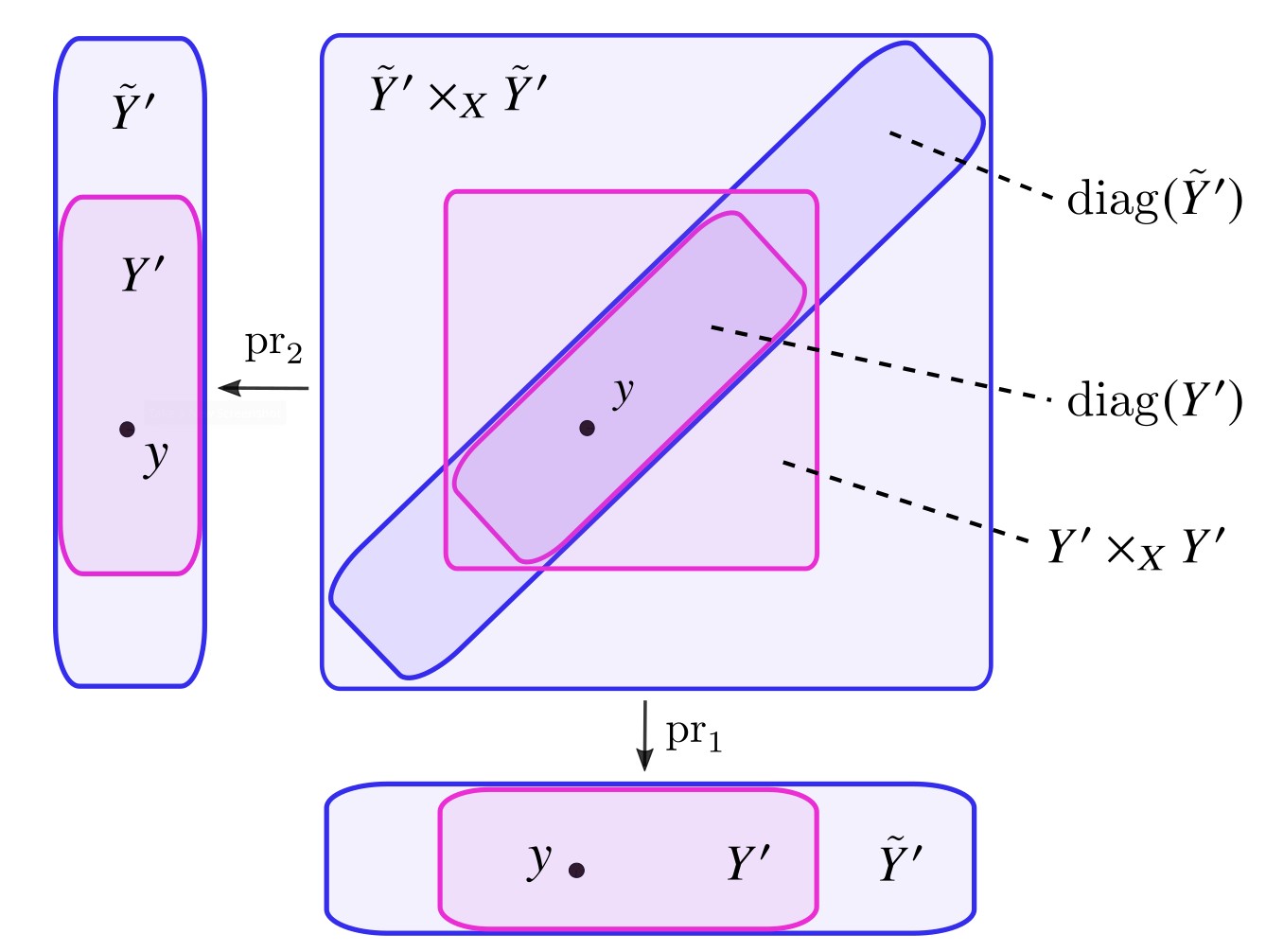}
\caption{
$Y' \times_X Y'$ is a 
topological subspace of $\til{Y}' \times_X \til{Y}'$, and \\
$\opn{diag}(Y') = \opn{diag}(\til{Y}') \cap (Y' \times_X Y')$.} 
\label{fig:6}
\end{figure}

Since $\til{Y}' \to X$ is a smooth map of schemes, by
\cite[Proposition 17.12.4]{EGA-IV} the diagonal map 
$\opn{diag}_{\til{Y}' / X} : \til{Y}' \to \til{Y}' \times_X \til{Y}'$
is a quasi-regular closed embedding of schemes. This means that the 
multiplication homomorphism 
$\opn{mult}_{\til{B}' / A} : \til{B}' \ot_A \til{B}' \to \til{B}'$
is a quasi-regular surjection of rings. But then, according to Corollary 
\ref{cor:1130}, the homomorphism $\opn{mult}_{\til{B}' / A}$
is a regular surjection of rings. 

The definition of a regular surjection of rings says that we can find an 
element $\til{r} \in \til{B}' \ot_A \til{B}'$ such that
$y \in \opn{Spec} \bigl( (\til{B}' \ot_A \til{B}')_{\til{r}} \bigr) = 
\opn{NZer}_{\til{Y}' \times_X \til{Y}'}(\til{r}) \sub 
\til{Y}' \times_X \til{Y}'$,
and such that the ideal 
\[ (\til{B}' \ot_A \til{B}')_{\til{r}} 
\ot_{\til{B}' \ot_A \til{B}'} \opn{I}_{\til{B}' / A} \sub
(\til{B}' \ot_A \til{B}')_{\til{r}} \]
is generated by a regular sequence, say $\til{\bb}$. 
Let's write 
$\til{V} := \opn{NZer}_{\til{Y}' \times_X \til{Y}'}(\til{r})$; 
it is shown in Figure \ref{fig:7}.

\begin{figure}[!t]
\centering
\includegraphics[scale=0.22]{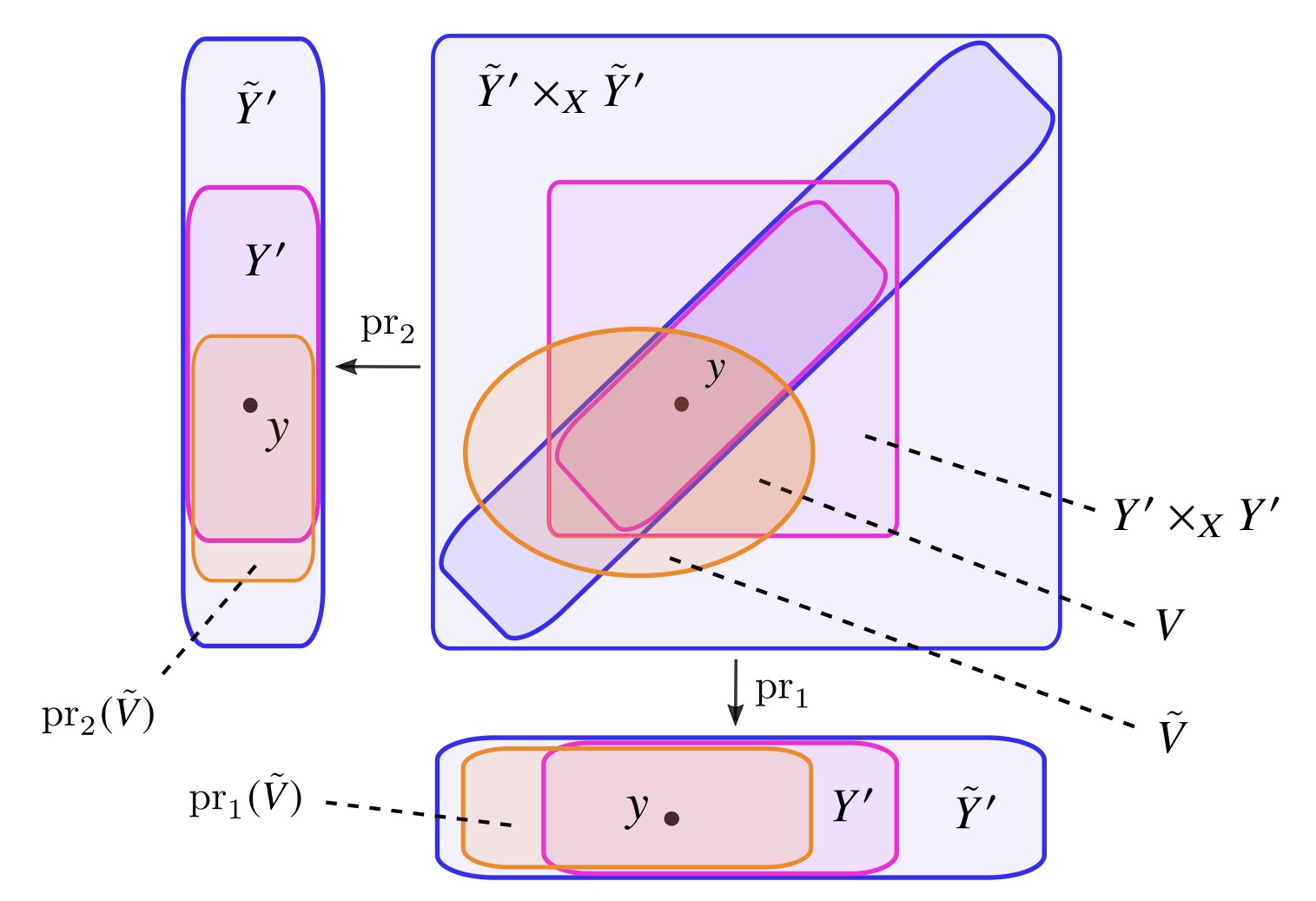}
\caption{For the proof of Theorem \ref{thm:1141}: 
$V = \til{V} \cap (Y' \times_X Y')$ and $y \in V$.}
\label{fig:7}
\end{figure}

Consider the localization homomorphism 
$g \ot g : \til{B}' \ot_A \til{B}' \to B' \ot_A B'$,
and define the element
$r := (g \ot g)(\til{r}) \in B' \ot_A B'$.
Writing $V :=  \opn{NZer}_{Y' \times_X Y'}(r)$, we have 
$V = \til{V} \cap (Y' \times_X Y')$.
See Figure \ref{fig:7}.

Because $g \ot g$ is a localization homomorphism, according to Lemma 
\ref{lem:1195} there is equality
$\opn{I}_{B' / A} = (B' \ot_A B') \ot_{\til{B}' \ot_A \til{B}'}
\opn{I}_{\til{B}' / A}$
of ideals of the ring $B' \ot_A B'$.
Therefore 
\[ (B' \ot_A B')_r \ot_{B' \ot_A B'} \opn{I}_{B' / A} 
= (B' \ot_A B')_r \ot_{(\til{B}' \ot_A \til{B}')_{\til{r}}}
\bigl( (\til{B}' \ot_A \til{B}')_{\til{r}} 
\ot_{\til{B}' \ot_A \til{B}'} \opn{I}_{\til{B}' / A} \bigr) \]
as ideals of the ring $(B' \ot_A B')_r$. 
This means that this ideal is generated by the sequence of elements 
$\bb := (g \ot g)_{\til{r}}(\til{\bb})$, where  
$(g \ot g)_{\til{r}} : (\til{B}' \ot_A \til{B}')_{\til{r}} \to 
(B' \ot_A B')_r$
is the localization ring homomorphism obtained from $g \ot g$ by inverting 
$\til{r}$. 
By Proposition \ref{prop:1090} the sequence $\bb$ in the ring 
$(B' \ot_A B')_r$ is regular. 

Finally, we know that $B'$ is a principal localization of $B$, i.e.\
$B' \cong B_t$ for some element $t \in B$. Then 
$B' \ot_A B' \cong (B \ot_A B)_{t \ot t}$.
We can express $r \in B' \ot_A B'$ as 
$r = r' \cd (t \ot t)^{-m}$ for some $r' \in B \ot_A B$ and $m \geq 0$.
Let us define $s := r' \cd (t \ot t) \in B \ot_A B$.
Then 
$(B \ot_A B)_s \cong (B' \ot_A B')_r$
as $(B \ot_A B)$-rings. This concludes the proof.
\end{proof}

\begin{rem} \label{rem:1180}
Take a point $y \in Y \sub Y \times_X Y$, in the notation of the proof of the 
theorem, and let $V$ be an open neighborhood of $y$ in $Y \times_X Y$.
The reader might wonder if there exists an open neighborhood $U$ of $y$ in $Y$ 
s.t.\ $U \times_X U \sub V$. The answer is negative, as the counterexample 
below shows. 

Take $A := \mbb{R}$ and $B := \mbb{C}$. Then 
$Y = \opn{Spec}(B)$ has one point, which we call $y$. The product 
$Y \times_X Y$ has two points: the diagonal copy of $y$, and another point, say 
$y'$. Take $V := Y = \{ y \}$, which is closed and open in 
$Y \times_X Y$. Clearly $V$ is not $U \times_X U$ for any open set  $U$ in  
$Y$. What is true is that 
$V = \opn{NZer}_{Y \times_X Y}(s)$, 
for the element 
$s := \mrm{i} \otimes 1 + 1 \otimes \mrm{i} \in B \ot_A B$.
\end{rem}

The relative dualizing module $\De_{-/-}^{\mrm{rs}}$ of a regular surjection was
introduced Definition \ref{dfn:1085}.

\begin{cor} \label{cor:1925}
Let $A \to B$ be an essentially smooth ring homomorphism, of constant 
relative differential dimension $n$. 
Write $B^{\mrm{en}} := B \ot_A B$, so
$\opn{mult}_{B / A} : B^{\mrm{en}} \to B$ is a regular surjection. Then the
relative dualizing module
$\De_{B / B^{\mrm{en}}}^{\mrm{rs}}$
admits a canonical isomorphism of $B$-modules 
\[ \opn{Hom}_{B}(\Om^n_{B / A}, B) \iso
\De_{B / B^{\mrm{en}}}^{\mrm{rs}} =
\opn{Hom}_{B} \bigl( \bwedge^n_B(\opn{I}_{B / A} \, / \, \opn{I}_{B / A}^2),
B \bigr) . \]
\end{cor}

\begin{proof}
Formula (\ref{eqn:1154}) gives a canonical isomorphism 
$\opn{I}_{B / A} \, / \, \opn{I}_{B / A}^2 \iso \Om^1_{B / A}$,
from which we get a canonical isomorphism
$\bwedge^n_B(\opn{I}_{B / A} \, / \, \opn{I}_{B / A}^2) \iso
\Om^n_{B / A}$.
Now just take duals.
\end{proof}

\begin{cor} \label{cor:1060}
Let $A \to B$ be an essentially smooth ring homomorphism. Then $B$ is a 
perfect complex over the ring $B \ot_A B$. 
\end{cor}

\begin{proof}
Combine Theorem \ref{thm:1141} and Corollary \ref{cor:1131}.
\end{proof}

\begin{dfn} \label{dfn:1175}
Let $A \to B$ be an essentially smooth homomorphism between noetherian rings, 
of constant differential relative dimension $n$. An {\em \'etale coordinate 
system of $B$ relative to $A$} is a sequence 
$\bb = (b_1, \ldots, b_n)$
of elements of $B$, such that $A$-ring homomorphism 
$A[\bt] \to B$ from the polynomial ring 
$A[\bt] = A[t_1, \ldots, t_n]$,  sending $t_i \mapsto b_i$, is essentially 
\'etale. 
\end{dfn}

To study \'etale coordinate systems we need to recall the notion of relative 
formal smoothness from \cite[Definition 19.9.1]{EGA-0IV}.
Suppose $A \to B \to C$ are ring homomorphisms (with no finiteness 
conditions). We say that $B \to C$ is {\em formally smooth relative to A}
if it has the following lifting property: given a $B$-ring $D$, a nilpotent 
ideal $\mathfrak{d} \sub D$, and a $B$-ring homomorphism $\bar{g} : C \to D / \mathfrak{d}$,
such there exists an an $A$-ring homomorphism $g_A : C \to D$ lifting 
$\bar{g}$, then  there exists a $B$-ring 
homomorphism $g_B : C \to D$ lifting $\bar{g}$. 
See next commutative diagrams, where $\pi : D \to D / \mathfrak{d}$ is the canonical 
surjection.
\begin{equation} \label{eqn:1177}
\begin{tikzcd} [column sep = 8.6ex, row sep = 5ex] 
B
\arrow[r]
\ar[d]
&
D
\ar[d, "{\pi}"]
\\
C
\ar[r, "{\bar{g}}"']
&
D / \mathfrak{d}
\end{tikzcd}
\qquad
\begin{tikzcd} [column sep = 8.6ex, row sep = 6ex] 
A
\arrow[r]
\ar[d]
&
D
\ar[d, "{\pi}"]
\\
C
\ar[r, "{\bar{g}}"']
\ar[ur, dashed, "{g_A}"]
&
D / \mathfrak{d}
\end{tikzcd}
\qquad  
\begin{tikzcd} [column sep = 8.6ex, row sep = 6ex] 
B
\arrow[r]
\ar[d]
&
D
\ar[d, "{\pi}"]
\\
C
\ar[r, "{\bar{g}}"']
\ar[ur, dashed, "{g_B}"]
&
D / \mathfrak{d}
\end{tikzcd}
\end{equation}

\begin{lem} \label{lem:1175}
Let $A \to B \to C$ be ring homomorphisms, and assume $A \to C$ is formally 
smooth. Then the following conditions are equivalent:
\begin{itemize}
\rmitem{i} $B \to C$ is formally smooth.

\rmitem{ii} $B \to C$ is formally smooth relative to $A$.
\end{itemize}
\end{lem}

\begin{proof}
Immediate from the definitions, because a lifting $g_A$ always exists.
\end{proof}

\begin{thm} \label{thm:1175}
Let $A \to B$ be an essentially smooth homomorphism between noetherian rings,
of constant differential relative dimension $n$, and write 
$B^{\mrm{en}} := B \ot_A B$. Let 
$\bb = (b_1, \ldots, b_n)$ be a sequence of elements in $B$.
The following conditions are equivalent:
\begin{itemize}
\rmitem{i} The sequence $\bb$ is an \'etale coordinate system of $B$ relative
to $A$.

\rmitem{ii} The sequence $\d_{B / A}(\bb)$ is a basis of the $B$-module 
$\Om^1_{B / A}$. 

\rmitem{iii} For every prime ideal $\p \sub B$ there exists an element 
$\til{s} \in B^{\mrm{en}}$ such that 
$\opn{mult}_{B / A}(\til{s}) \notin \p$, and such that the sequence 
$\opn{q}_{B^{\mrm{en}} ,\til{s}}(\til{\d}_{B / A}(\bb))$ in the ring 
$(B^{\mrm{en}})_{\til{s}}$ is 
regular, and it generates the ideal 
$(B^{\mrm{en}})_{\til{s}} \ot_{B^{\mrm{en}}} \opn{I}_{B / A} 
\sub (B^{\mrm{en}})_{\til{s}}$.
\end{itemize}
See formulas (\ref{eqn:newa}) and (\ref{eqn:newb}) regarding the sequences 
$\d_{B / A}(\bb)$ and $\til{\d}_{B / A}(\bb)$.
\end{thm}

In condition (iii) we have the localization ring homomorphism 
$\opn{q}_{B^{\mrm{en}}, \til{s}} : B^{\mrm{en}} \to (B^{\mrm{en}})_{\til{s}}$. 
Here is the geometric interpretation
of the requirement that $\opn{mult}_{B / A}(\til{s}) \notin \p$. Let's write
$X := \opn{Spec}(A)$, $Y := \opn{Spec}(B)$ and $y := \p \in Y$. Then 
$\opn{mult}_{B / A}(\til{s}) \notin \p$ iff 
$\opn{diag}_{Y / X}(y) \in \opn{NZer}_{Y \times_X Y}(\til{s}) 
\sub Y \times_X Y$.
 
\begin{proof} 
(i) $\Rightarrow$ (ii):
The ring homomorphism $f : A \to A[\bt]$ is smooth of constant differential 
relative dimension $n$, and the ring homomorphism $g : A[\bt] \to B$,
$g(t_i) = b_i$, is essentially \'etale. Because $g$ is essentially \'etale,
there is a split short exact sequence 
\begin{equation} \label{eqn:1175}
0 \to B \ot_{A[\bt]} \Om^1_{A[\bt] / A} 
\xar{\opn{id} \ot \, \Om^1_{g / A}} \Om^1_{B / A} 
\xar{\Om^1_{B / f}} \Om^1_{B / A[\bt]} \to 0 \ ; 
\end{equation}
this is an instance of the sequence (\ref{eqn:1071}).
By Proposition \ref{prop:1055} we know that the module
$\Om^1_{B / A[\bt]}$ is zero. It follows that the homomorphism 
\begin{equation} \label{eqn:1176}
\opn{id} \ot \, \Om^1_{g / A} : 
B \ot_{A[\bt]} \Om^1_{A[\bt] / A} \to \Om^1_{B / A} 
\end{equation}
is an isomorphism. 
Since the sequence 
$\d_{A[\bt] / A}(\bt)$ is a basis of the $A[\bt]$-module 
$\Om^1_{A[\bt] / A}$, we see that the sequence $\d_{B / A}(\bb)$ is a basis of 
the $B$-module $\Om^1_{B / A}$. 

\medskip \noindent
(ii) $\Rightarrow$ (i):
Here the homomorphism (\ref{eqn:1176}) is an isomorphism. 
According to \cite[Theorem 20.5.7]{EGA-0IV}, the homomorphism 
$A[\bt] \to B$ is formally smooth relative to $A$.
But $A \to C$ is formally smooth, so by Lemma \ref{lem:1175} we know that 
$A[\bt] \to B$ is formally smooth.
Therefore $A[\bt] \to B$ is essentially smooth. 
The exact sequence 
\[ B \ot_{A[\bt]} \Om^1_{A[\bt] / A} 
\xar{\opn{id} \ot \, \Om^1_{g / A}} 
\Om^1_{B / A} \xar{\Om^1_{B / f}} \Om^1_{B / A[\bt]} \to 0,  
\]
which is part of (\ref{eqn:1175}), with the fact that 
(\ref{eqn:1176}) is an isomorphism, say that 
$\Om^1_{B / A[\bt]} = 0$. From Proposition \ref{prop:1055} 
we see that $A[\bt] \to B$ is essentially \'etale.
Thus $\bb$ is an \'etale coordinate system.

\medskip \noindent
(ii) $\Rightarrow$ (iii): By Theorem \ref{thm:1141} the multiplication 
homomorphism $\opn{mult}_{B / A} :B^{\mrm{en}} \to B$ is a regular surjection, 
i.e.\ the ideal $\opn{I}_{B / A} \sub B^{\mrm{en}}$ is regular. 
Now use Proposition \ref{prop:1095}, applied to the 
$B$-module 
$\Om^1_{B / A} \cong \opn{I}_{B / A} / (\opn{I}_{B / A})^2$,
noting formulas (\ref{eqn:1154}) and (\ref{eqn:1136}). 

\medskip \noindent
(iii) $\Rightarrow$ (ii): By Lemma \ref{lem:1111}(1) the sequence 
$\d_{B / A}(\bb)$ is a basis of the $B_{\p}$-module 
$\Om^1_{B_{\p} / A} \cong B_{\p} \ot_B \Om^1_{B / A}$
for every $\p \in \opn{Spec}(B)$. This implies that $\d_{B / A}(\bb)$ is a 
basis 
of the $B$-module $\Om^1_{B / A}$.
\end{proof}

\begin{cor} \label{cor:1175}
Let $A \to B$ be an essentially smooth homomorphism of constant 
differential relative dimension $n$. Given a prime ideal $\p \sub B$, there 
exists an element $s \in B - \p$ such the ring $B_s = B[s^{-1}]$ admits some 
\'etale coordinate system relative to $A$. 
\end{cor}

\begin{proof}
This is by Proposition \ref{prop:1140} and Theorem \ref{thm:1175}. In fact, 
Proposition \ref{prop:1140} says that we can find a sequence $\bb$ in $B$
whose image in $B_s$ is an \'etale coordinate system relative to $A$. 
\end{proof}

\begin{cor} \label{cor:1930} 
Let $A \to B$ be an essentially smooth homomorphism of constant relative 
differential dimension $n$, and assume there exists an \'etale coordinate system 
$\bb$ for $B$ relative to $A$. Then:
\begin{enumerate}
\item The $B$-module $\Om^n_{B / A}$ is free with
basis the element $\opn{wdg}(\d(\bb))$.

\item The $B$-module
$\bwedge^n_B(\opn{I}_{B / A} \, / \, \opn{I}_{B / A}^2)$
is free  with basis the element
$\opn{wdg}(\til{\d}(\bb))$.

\item The isomorphism
\[ \opn{Hom}_{B}(\Om^n_{B / A}, B) \iso
\De_{B / B^{\mrm{en}}}^{\mrm{rs}} \]
from Corollary \ref{cor:1925} sends
$\smfrac{1}{\opn{wdg}(\d(\bb))} \mapsto
\smfrac{1}{\opn{wdg}(\til{\d}(\bb))}$.
\end{enumerate}
See Definition \ref{dfn:1920} regarding $\opn{wdg}$ and Definition \ref{dfn:1921} regarding fractions.
\end{cor}

\begin{proof}
Item (1) is clear from the implication (i) $\Rightarrow$ (ii) in Theorem 
\ref{thm:1175}. This, with formula (\ref{eqn:1154})  give us item (2).
Item (3) is clear from Corollary \ref{cor:1925} and its proof.
\end{proof}

The Koszul presentation $\opn{kpres}^i_{\ba, M}$ of $\opn{Ext}^i$ was 
introduced in Definition \ref{dfn:1923}. 

\begin{cor} \label{cor:1931} 
Let $A \to B$ be an essentially smooth homomorphism of constant relative 
differential dimension $n$, and write $B^{\mrm{en}} := B \ot_A B$. 
Assume there exists an \'etale coordinate system $\bb= (b_1, \ldots, b_n)$ for $B$ relative to $A$. 
Given a $B^{\mrm{en}}$-module $M$, there is a unique $B$-module homomorphism 
\[ \opn{kprox}^n_{\til{\d}(\bb), M} :
\opn{H}^n \bigl( \opn{Hom}_{B^{\mrm{en}}} 
\bigl( \opn{K}(B^{\mrm{en}}; \til{\d}(\bb)), M \bigr) \bigr) \to 
\opn{Ext}^n_{B^{\mrm{en}}}(B, M)
 \]
called {\em Koszul approximation} with this property: 
\begin{itemize}
\item[($*$)] For every element $\til{s} \in B^{\mrm{en}}$ 
such that the ring $(B^{\mrm{en}})_{\til{s}}$ is nonzero,  the sequence 
\begin{equation}\label{neww}
\opn{q}_{B^{\mrm{en}} ,\til{s}}(\til{\d}_{B / A}(\bb))
\end{equation}
in  
$(B^{\mrm{en}})_{\til{s}}$ is regular, and it generates the ideal 
$(B^{\mrm{en}})_{\til{s}} \ot_{B^{\mrm{en}}} \opn{I}_{B / A} 
\sub (B^{\mrm{en}})_{\til{s}}$,  
the diagram 
\[ \begin{tikzcd} [column sep = 10ex, row sep = 5ex] 
\opn{H}^n \bigl( \opn{Hom}_{B^{\mrm{en}}} 
\bigl( \opn{K}(B^{\mrm{en}}; \til{\d}(\bb)), M \bigr) \bigr)
\arrow[r, "{\opn{kprox}^n_{\til{\d}(\bb), M}}"]
\ar[d, "{\opn{q}_{B^{\mrm{en}} ,\til{s}}}"']
&
\opn{Ext}^n_{B^{\mrm{en}}}(B, M)
\ar[d, "{\opn{q}_{B^{\mrm{en}} ,\til{s}}}"]
\\
\opn{H}^n \bigl( \opn{Hom}_{(B^{\mrm{en}})_{\til{s}}} 
\bigl( \opn{K}((B^{\mrm{en}})_{\til{s}}; \til{\d}(\bb)), M_{\til{s}} 
\bigr) \bigr)
\arrow[r, "{\opn{kpres}^n_{\til{\d}(\bb), M_{\til{s}}}}", "{\simeq}"']
&
\opn{Ext}^n_{(B^{\mrm{en}})_{\til{s}}}(B_{\til{s}}, M_{\til{s}})
\end{tikzcd} \]
is commutative. 
\end{itemize}
\end{cor}

\begin{proof}
Let's write $X := \opn{Spec}(A)$ and $Y := \opn{Spec}(B)$.
Note that the module
$\opn{Ext}^n_{B^{\mrm{en}}}(B, M)$
is supported on $\opn{diag}_{Y / X}(X) \cong X$.
By Theorem \ref{thm:1175}, the principal open sets
$\opn{NZer}_{Y \times_X Y}(\til{s}) \sub Y \times_X Y =  
\opn{Spec}(B^{\mrm{en}})$, 
for elements $\til{s} \in B^{\mrm{en}}$
satisfying the conditions in ($*$), cover 
$\opn{diag}_{Y / X}(X)$. For every such element
$\til{s}$ the homomorphism
\begin{equation} \label{eqn:1935}
\opn{kpres}^n_{\til{\d}(\bb), M_{\til{s}}} \circ \opn{q}_{B^{\mrm{en}} ,\til{s}}
:
\opn{H}^n \bigl( \opn{Hom}_{B^{\mrm{en}}} 
\bigl( \opn{K}(B^{\mrm{en}}; \til{\d}(\bb)), M \bigr) \bigr) \to 
\opn{Ext}^n_{(B^{\mrm{en}})_{\til{s}}}(B_{\til{s}}, M_{\til{s}})
\end{equation}
exists. Now there is a canonical isomorphism 
\[ (B^{\mrm{en}})_{\til{s}} \ot_{B^{\mrm{en}}}  
\opn{Ext}^n_{B^{\mrm{en}}}(B, M)
\iso \opn{Ext}^n_{(B^{\mrm{en}})_{\til{s}}}(B_{\til{s}}, M_{\til{s}}) . \]
Because of the explicit formula for 
$\opn{kpres}^n_{\til{\d}(\bb), M_{\til{s}}}$, the homomorphisms (\ref{eqn:1935})
agree on double intersections
\[ \opn{NZer}_{Y \times_X Y}(\til{s}_1) \cap 
\opn{NZer}_{Y \times_X Y}(\til{s}_2) 
= \opn{NZer}_{Y \times_X Y}(\til{s}_1 \cd \til{s}_2) . \]
Hence they glue uniquely to a homomorphism 
$\opn{kprox}^n_{\til{\d}(\bb), M}$.
\end{proof}

\begin{rem} \label{rem:2125}
The homomorphism $\opn{kprox}^n_{\til{\d}(\bb), M}$ need not be an isomorphism.
That is why we call it $\opn{kprox}$; it is defined on the whole $\opn{Spec}(B^{\mrm{en}})$ 
but it approximates the local isomorphism $\opn{kpres}$.
One obstruction is that the sequence $\til{\d}(\bb)$ might not generate the
ideal $\opn{I}_{B / A}$ in $B^{\mrm{en}}$, in which case the homomorphism 
$\opn{H}^0 \bigl( \opn{K}(B^{\mrm{en}}; \til{\d}(\bb)) \bigr) \to B$
will not be bijective.

Indeed, this is what happens in the very easy case of Remark \ref{rem:1180},
where $A := \mbb{R}$ and $B := \mbb{C}$. The relative dimension here is
$n = 0$, the empty sequence $\bb$ is an \'etale coordinate sequence,
but the ideal $\opn{I}_{B / A}$ is not zero.
\end{rem}

\begin{rem} \label{rem:1050}
It is possible to characterize essentially smooth homomorphisms in several 
ways. Indeed, the following conditions 
are equivalent for an essentially finite type ring homomorphism $f : A \to B$ 
between noetherian rings:
\begin{itemize}
\rmitem{i} The ring $B$ is essentially smooth over $A$.

\rmitem{ii} For every prime ideal $\q \sub B$, with 
$\p := f^{-1}(\q) \sub A$, the ring $B_{\q}$ is essentially smooth over 
$A_{\p}$.

\rmitem{iii} The ring $B$ is flat over $A$, and for every prime ideal
$\q \sub B$, with $\p := f^{-1}(\q) \sub A$, the $\bk(\p)$-ring 
$B \ot_A \bk(\p)$ is geometrically regular.

\rmitem{iv} The ring $B$ is flat over $A$, and $B$ is a perfect complex over 
the ring $B \ot_A B$. 
\end{itemize}

For condition (iii), recall that given a field $K$, a noetherian $K$-ring 
$C$ is said to be geometrically regular if for every finite field extension 
$K \to L$ the ring $C \ot_K L$ is regular.

We shall not require these characterizations in our paper. Some of the 
implications were proved above. As for the rest, the interested reader can 
find proofs for them in the locations listed below. 
The proof of the implication (ii) $\Rightarrow$ (i) is hinted at in 
the proof of Theorem \ref{thm:1015}. The equivalence between (i) and (iii)
is \cite[Corollaire 17.5.2]{EGA-IV}. (Warning: the phrasing of 
\cite[Th\'eor\`eme 17.5.1 and Th\'eor\`eme 6.8.6]{EGA-IV} is confusing -- 
regularity there means geometric regularity.)
For the proof of the implication (iv) $\Rightarrow$ (i) see \cite[Lemma 23.10.1]{SP}.
\end{rem}

\section{Twisted Induced Rigidity}
\label{sec:twisted-induced}

The purpose of this section is the construction of the {\em twisted induced 
rigidifying isomorphism}; see Definitions \ref{dfn:1235} and \ref {dfn:1236}.
Recall that according to Conventions \ref{conv:615} and \ref{conv:1070}, both 
of which are in force in this section, all rings are commutative and noetherian 
by default, and all ring homomorphisms are EFT by default. 

For most of this section, until and including Definition 
\ref{dfn:1225}, we work in the following setup:

\begin{setup} \label{set:1165}
There is an essentially smooth homomorphism $A \to B$ between 
noetherian rings, of constant differential relative dimension $n$.
We write $B^{\mrm{en}} := B \ot_A B$.
\end{setup}

The module of relative differential forms $\Om^1_{B / A}$ is a projective 
$B$-module of constant rank $n$. There is a canonical $B$-module isomorphism
$\opn{I}_{B / A} / \opn{I}^2_{B / A} \iso \Om^1_{B / A}$, 
see Definition \ref{dfn:1544} and formulas (\ref{eqn:1154}) and 
(\ref{eqn:1136}). 

Recall that the de Rham complex 
$\Om_{B/A} = \bigoplus_{i = 0}^{n} \, \Om^i_{B/A}$ 
is a strongly commutative central DG $A$-ring.
As a graded ring it is the strongly commutative graded tensor $B$-ring on the 
graded $B$-module $\Om^1_{B / A}$, where $\Om^1_{B / A}$ is concentrated 
in degree $1$. In the ungraded sense, this just means that
$\Om^i_{B / A} = \bwedge^i_B(\Om^1_{B / A})$. 
In particular, $\Om^0_{B / A} = B$, and $\Om^n_{B / A}$ is a projective 
$B$-module of rank $1$. The differential $\d_{B / A}$ of $\Om_{B / A}$
is the unique degree $1$ derivation on this graded ring that extends the 
universal derivation
$\d_{B / A} : B \to \Om^{1}_{B / A}$.
An $A$-ring homomorphism $g : B \to C$ extends uniquely to a DG $A$-ring 
homomorphism $\Om_{g / A} : \Om_{B / A} \to \Om_{C / A}$. 

The multiplication homomorphism 
$\opn{mult}_{B / A} : B^{\mrm{en}} \to B$
is a regular surjection of rings, of constant codimension $n$ (Theorem 
\ref{thm:1141}). Recall the relative dualizing module 
\begin{equation} \label{eqn:1207}
\De^{\mrm{rs}}_{B / B^{\mrm{en}}} = 
\opn{Hom}_B \bigl( \bwedge^n_B (\opn{I}_{B / A} / \opn{I}_{B / A}^2) , B \bigr)
\end{equation}
associated to the regular surjection $\opn{mult}_{B / A}$, see 
Definition \ref{dfn:1085}. 
The $B$-module isomorphism (\ref{eqn:1154}) induces, upon taking top exterior 
powers, a $B$-module isomorphism 
$\bwedge^n_B (\opn{I}_{B / A} / \opn{I}_{B / A}^2) \cong \Om^n_{B / A}$.
These are rank $1$ projective $B$-modules. Therefore we have a canonical  
isomorphism of $B$-modules 
\begin{equation} \label{eqn:1215}
\De^{\mrm{rs}}_{B / B^{\mrm{en}}} \ot_B \Om^n_{B / A} \iso B , \quad 
\phi \ot \om  \mapsto \phi(\om) ,  
\end{equation}
see Corollary \ref{cor:1925}.

The isomorphism (\ref{eqn:1215}) respects localization over $g$, in the following sense.
Suppose $g : B \to C$ is a localization $A$-ring homomorphism. Let's write 
$C^{\mrm{en}} := C \ot_A C$
and $g^{\mrm{en}} := g \ot g : B^{\mrm{en}} \to C^{\mrm{en}}$. 
By Lemma \ref{lem:1195} there is a canonical nondegenerate forward homomorphism
$\opn{I}_{g / A} : \opn{I}_{B / A} \to \opn{I}_{C / A}$; 
and by Proposition \ref{prop:1141} there is a canonical nondegenerate forward 
homomorphism 
$\Om^1_{g / A} : \Om^1_{B / A} \to \Om^1_{C / A}$.
Therefore we get canonical nondegenerate forward homomorphisms 
$\De^{\mrm{rs}}_{g / g^{\mrm{en}}} :
\De^{\mrm{rs}}_{B / B^{\mrm{en}}} \to \De^{\mrm{rs}}_{C / C^{\mrm{en}}}$
and 
$\Om^n_{g / A} : \Om^n_{B / A} \to \Om^n_{C / A}$. 
For them the diagram 
\begin{equation} \label{eqn:1208}
\begin{tikzcd} [column sep = 7ex, row sep = 5ex]
\De^{\mrm{rs}}_{B / B^{\mrm{en}}} \ot_B \Om^n_{B / A}
\ar[r, "{\simeq}"]
\ar[d, "{\De^{\mrm{rs}}_{g / g^{\mrm{en}}} \, \ot_{g} \, \Om^n_{g / A}}"']
&
B
\ar[d, "{g}"]
\\
\De^{\mrm{rs}}_{C / C^{\mrm{en}}} \ot_C \Om^n_{C / A} 
\ar[r, "{\simeq}"]
&
C
\end{tikzcd}
\end{equation}
in which the horizontal isomorphisms are (\ref{eqn:1215}), is 
commutative. If $C\simeq B[S^{-1}]$ for some multiplicative closed set $S\subset B$.
So $C$ is also flat over $A$.

Theorem \ref{thm:1080}, applied to the regular surjection of rings 
$\opn{mult}_{B / A}:B^{\mrm{en}}\to B$ and to the $B^{\mrm{en}}$-module 
$M := \Om^n_{B / A} \ot_A \Om^n_{B / A}$, produces a $B$-module isomorphism 
\begin{equation} \label{eqn:1210}
\opn{fund}_{B / A, M} :
\De^{\mrm{rs}}_{B / B^{\mrm{en}}} \ot_{B^{\mrm{en}}} 
\bigl( \Om^n_{B / A} \ot_A \Om^n_{B / A} \bigr) \iso
\opn{Ext}^n_{B^{\mrm{en}}} 
\bigl( B, \Om^n_{B / A} \ot_A \Om^n_{B / A} \bigr)  . 
\end{equation}

\begin{dfn}[Residue Isomorphism] \label{dfn:1215}
Under Setup \ref{set:1165}, the {\em residue isomorphism} is the 
$B$-module isomorphism
\[ \opn{res}_{B / A} : 
\Om^n_{B / A} \iso  
\opn{Ext}^n_{B^{\mrm{en}}} \bigl( B, \Om^n_{B / A} \ot_A \Om^n_{B / A} \bigr) \]
obtained by composing the isomorphisms 
\begin{equation} \label{eqn:1271}
\begin{aligned}
& 
\Om^n_{B / A} = B \ot_B \Om^n_{B / A}
\iso^{\mrm{(1)}}
\bigl( \De^{\mrm{rs}}_{B / B^{\mrm{en}}} \ot_B \Om^n_{B / A} 
\bigr) \ot_B \Om^n_{B / A}
\\
& \quad \iso^{\mrm{(2)}}
\De^{\mrm{rs}}_{B / B^{\mrm{en}}} \ot_{B^{\mrm{en}}} 
\bigl( \Om^n_{B / A} \ot_A \Om^n_{B / A} \bigr) 
\end{aligned} 
\end{equation}
with the isomorphism (\ref{eqn:1210}).
Here $\iso^{\mrm{(1)}}$ comes from 
applying $(-) \ot_B \Om^n_{B / A}$ to the inverse of the isomorphism 
(\ref{eqn:1215}), and $\iso^{\mrm{(2)}}$ is an instance of the canonical 
isomorphism (\ref{eqn:1486}). 
\end{dfn}

The residue isomorphism will be made explicit in the theorem below. 

Recall that for an element $b \in B$ we write 
$\til{\d}_{B / A}(b) = b \ot 1 - 1 \ot b \in \opn{I}_{B / A}$. For a 
sequence $\bb = (b_1, \ldots, b_n)$ in $B$ we write 
$\til{\d}_{B / A}(\bb) = \bigl( \til{\d}_{B / A}(b_1), \ldots \bigr)$
and 
$\d_{B / A}(\bb) = \bigl( \d_{B / A}(b_1), \ldots \bigr)$,
which are sequences in $\opn{I}_{B / A}$ and $\Om^1_{B / A}$ respectively. 
The element $\opn{wdg}(\d_{B / A}(\bb)) \in \Om^n_{B / A}$ 
was defined is in Definition \ref{dfn:1920}.

A localization ring homomorphism $g : B \to C$ induces nondegenerate forward 
homomorphisms 
$\Om^n_{g / A} : \Om^n_{B / A} \to \Om^n_{C / A}$
and 
\begin{equation} \label{eqn:1937}
\opn{Ext}^n_{B^{\mrm{en}}} 
\bigl( B, \Om^n_{B / A} \ot_A \Om^n_{B / A} \bigr) \to 
\opn{Ext}^n_{C^{\mrm{en}}} 
\bigl( C, \Om^n_{C / A} \ot_A \Om^n_{C / A} \bigr) . 
\end{equation}

\begin{thm} \label{thm:1935}
Under Setup \ref{set:1165}, the residue isomorphism $\opn{res}_{B / A}$ has the 
following two properties, and it is determined by them.
\begin{enumerate}
\item If $g : B \to C$ is a localization ring homomorphism, then the diagram  
\[ \begin{tikzcd} [column sep = 8ex, row sep = 5ex]
\Om^n_{B / A}
\ar[d, "{\Om^n_{g / A}}"']
\ar[r, "{\opn{res}_{B / A}}", "{\simeq}"']
&
\opn{Ext}^n_{B^{\mrm{en}}} 
\bigl( B, \Om^n_{B / A} \ot_A \Om^n_{B / A} \bigr)
\ar[d]
&
\\
\Om^n_{C / A}
\ar[r, "{\opn{res}_{C / A}}", "{\simeq}"']
&
\opn{Ext}^n_{C^{\mrm{en}}} 
\bigl( C, \Om^n_{C / A} \ot_A \Om^n_{C / A} \bigr)
\end{tikzcd} \]
in which the unlabeled arrow is the nondegenerate forward homomorphism 
(\ref{eqn:1937}), is commutative.

\item Suppose $B$ admits an \'etale coordinate system $\bb$ relative to $A$,
so that $\be := \opn{wdg}(\d_{B / A}(\bb))$ is a basis of the free $B$-module
$\Om^n_{B / A}$. Then 
\[ \opn{res}_{B / A}(\be) =  
\opn{kprox}^n_{\bb, M} 
\Bigl( \gfrac{\be \ot \be}{\til{\d}_{B / A}(\bb)} \Bigr) , \]
where $M := \Om^n_{B / A} \ot_A \Om^n_{B / A}$, 
$\smgfrac{\be \ot \be}{\til{\d}_{B / A}(\bb)}$
is the generalized fraction from Definition \ref{dfn:1922}, 
and $\opn{kprox}^n_{\bb, M}$ is 
the Koszul approximation from Corollary \ref{cor:1931}. 
\end{enumerate}
\end{thm}

\begin{proof}
(1) The ring homomorphism $g^{\mrm{en}} : B^{\mrm{en}} \to C^{\mrm{en}}$ is also a 
localization. The assertion is clear from Proposition \ref{prop:1141},
Proposition \ref{prop:1260} and formula (\ref{eqn:1208}).

\medskip \noindent
(2) This is just an unraveling of the Fundamental Local Isomorphism
(Theorem \ref{thm:1080}), specifically the formula in property
($\dag$) of the theorem, that refers to the formula in property ($*$) in Lemma 
\ref{lem:1100}; combined with formula (\ref{eqn:1271}) in the definition of 
$\opn{res}_{B / A}$.
\end{proof}

Recall the squaring operation from Theorem \ref{thm:631}. Because $A \to B$ is 
a flat ring homomorphism, for every $M \in \cat{D}(B)$ we have the isomorphism 
\begin{equation} \label{eqn:1216}
\opn{sq}^{B / A}_{B / A, M} : 
\opn{Sq}_{B / A}(M) \iso \opn{Sq}_{B / A}^{B / A}(M) = 
\opn{RHom}_{B^{\mrm{en}}}(B, M \ot^{\mrm{L}}_{A} M) 
\end{equation}
in $\cat{D}(B)$. See formula (\ref{eqn:625}) with 
$\til{B} / \til{A} = B / A$.

The vanishing of all cohomologies of the complex
$\opn{RHom}_{B^{\mrm{en}}} \bigl( B, \Om^n_{B / A} \ot_A \Om^n_{B / A} \bigr)$, 
except for $\opn{H}^n$, which is due to Theorem \ref{thm:1080}, implies that 
there is a unique isomorphism 
\begin{equation} \label{eqn:1225}
\opn{RHom}_{B^{\mrm{en}}} \bigl( B, \Om^n_{B / A} \ot_A \Om^n_{B / A} \bigr)
\cong 
\opn{Ext}^n_{B^{\mrm{en}}} \bigl( B, \Om^n_{B / A} \ot_A \Om^n_{B / A} 
\bigr)[-n]
\end{equation}
in $\cat{D}(B)$ that gives the identity after applying $\opn{H}^n$ to it. 
By translating some of the complexes in (\ref{eqn:1225}) we get an isomorphism 
\begin{equation} \label{eqn:1226}
\begin{aligned}
&
\opn{Ext}^n_{B^{\mrm{en}}} \bigl( B, \Om^n_{B / A} \ot_A \Om^n_{B / A} \bigr)[n]
\\
& \quad 
\iso  
\opn{RHom}_{B^{\mrm{en}}} \bigl( B, \Om^n_{B / A}[n] \ot_A \Om^n_{B / A}[n] 
\bigr) 
\iso   
\opn{Sq}^{B / A}_{B / A} \bigl( \Om^n_{B / A}[n] \bigr) 
\end{aligned}
\end{equation}
in $\cat{D}(B)$.

\begin{dfn} \label{dfn:1225}
Under Setup \ref{set:1165}, let 
\[ \rho_{B / A}^{\mrm{esm}} \, : \,
\Om^n_{B / A}[n] \, \iso \, \opn{Sq}_{B / A} \bigl( \Om^n_{B / A}[n] \bigr) \]
be the unique isomorphism in $\cat{D}(B)$ such that the diagram 
\[ \begin{tikzcd} [column sep = 8ex, row sep = 5ex]
\Om^n_{B / A}[n]
\ar[r, dashed, "{\rho_{B / A}^{\mrm{esm}}}", "{\simeq}"']
\ar[d, "{\opn{res}_{B / A}[n]}", "{\simeq}"'] 
&
\opn{Sq}_{B / A} \bigl( \Om^n_{B / A}[n] \bigr) 
\ar[d, "{\opn{sq}^{B / A}_{B / A, \Om^n_{B / A}[n]} }", , "{\simeq}"']
\\
\opn{Ext}^n_{B^{\mrm{en}}} 
\bigl( B, \Om^n_{B / A} \ot_A \Om^n_{B / A} \bigr)[n]
\ar[r, "{\simeq}"]
&
\opn{Sq}^{B / A}_{B / A} \bigl( \Om^n_{B / A}[n] \bigr) 
\end{tikzcd} \]
in which the bottom horizontal arrow is (\ref{eqn:1226}), 
the left vertical arrow is the residue isomorphism from Definition 
\ref{dfn:1215}, and the right vertical arrow is the isomorphism 
(\ref{eqn:1216}), is commutative.

The rigid complex 
\[ \bigl( \Om^n_{B / A}[n], \, \rho_{B / A}^{\mrm{esm}} \bigr)
\in \cat{D}(B)_{\mrm{rig} / A} \]
is called the {\em standard rigid complex of the essentially smooth ring 
homomorphism $A \to  B$}.
\end{dfn}

From here on in this section we abandon Setup \ref{set:1165}.

\begin{prop} \label{prop:1270}
Let $A$ be a noetherian ring, and let $B \to C$ be an essentially smooth 
homomorphism between EFT $A$-rings, of constant differential relative 
dimension $n$. Let $M \in \cat{D}(B)$ be a complex that has finite flat 
dimension over $A$. 
\begin{enumerate}
\item The complex $M \ot_B \Om^n_{C / B}[n]$ belongs to 
$\cat{D}^{\mrm{b}}_{\mrm{f}}(C)$, and it has 
finite flat dimension over $A$.

\item The cup product morphism 
\[ \begin{aligned}
&
\opn{cup}_{C / B / A, M, \Om^n_{C / B}[n]} \, : \, 
\opn{Sq}_{B / A}(M) \ot^{\mrm{L}}_{B}
\opn{Sq}_{C / B} \bigl( \Om^n_{C / B}[n] \bigr) 
\\
& \quad \qquad 
\to \,
\opn{Sq}_{C / A} \bigl( M \ot^{\mrm{L}}_{B} \Om^n_{C / B}[n] \bigr) 
\end{aligned} \]
from Theorem \ref{thm:780} is an isomorphism.
\end{enumerate}
\end{prop}

\begin{proof} 
(1) By Definition \ref{dfn:675}(a) we know that 
$M \in \cat{D}^{\mrm{b}}_{\mrm{f}}(B)$ and it has finite flat dimension over 
$A$. Since $C$ is flat over $B$ and $\Om^n_{C / B}$ is a finitely generated 
flat $C$-module, for every $i$ we have an isomorphism of $C$-modules 
\[ \opn{H}^i \big( M \ot_B \Om^n_{C / B}[n] \bigl) \cong 
\opn{H}^{i + n}(M) \ot_B \Om^n_{C / B} , \]
and this is a finitely generated $C$-module, vanishing for $\abs{i} \gg 0$. 
This means that \lb 
$M \ot_B \Om^n_{C / B}[n] \in \cat{D}^{\mrm{b}}_{\mrm{f}}(C)$.

Regarding the finite flat dimension over $A$: we must prove that there is a 
natural number $d$, such that for every $A$-module $L$ the cohomology of the 
complex 
$L \ot^{\mrm{L}}_{A} M \ot^{\mrm{L}}_{C} \Om^n_{C / B}$
is concentrated in the integer interval $[-d, d]$. 
Consider the complex $L \ot^{\mrm{L}}_{A} M \in \cat{D}(B)$. 
To calculate it we choose a K-flat resolution $P \to L$ over $A$; and then 
$L \ot^{\mrm{L}}_{A} M \cong P \ot_A M$ in $\cat{D}(B)$.
But if we just want to know the concentration of the cohomology of the complex 
$L \ot^{\mrm{L}}_{A} M$, we can view it as a complex in 
$\cat{D}(\Z)$ via the restriction functor. 
The complex $M$ has finite flat dimension over $A$, and thus there is a 
quasi-isomorphism $Q \to M$ in $\cat{C}_{\mrm{str}}(A)$ from a 
bounded complex of flat $A$-modules $Q$. Let $d$ be such that the complex $Q$ 
is concentrated in the integer interval $[-d, d]$. Since 
$L \ot^{\mrm{L}}_{A} M \cong L \ot_A Q$ in $\cat{D}(\Z)$, we see that the 
concentration of the cohomology of the complex 
$L \ot^{\mrm{L}}_{A} M$ is inside $[-d, d]$. Now for every $i$ we have 
\[ \opn{H}^i \big( L \ot^{\mrm{L}}_{A} M \ot_B \Om^n_{C / B} \bigl) \cong 
\opn{H}^{i}(L \ot^{\mrm{L}}_{A} M) \ot_B \Om^n_{C / B} , \]
and this vanishes for $i \notin [-d, d]$. 

\medskip \noindent
(2) Conditions (i), (ii), (iii) and (iv.a) of Theorem \ref{thm:810},
with $N := \Om^n_{C / B}[n]$, are satisfied, the last one by Corollary 
\ref{cor:1060}.
\end{proof}

\begin{dfn}[Twisted Induced Rigid Complex] \label{dfn:1235}
Let $A$ be a noetherian ring, and let $v : B \to C$ be an essentially smooth 
homomorphism between essentially finite type $A$-rings, of constant 
differential relative dimension $n$. Let $(M, \rho)$ be a rigid complex over 
$B$ relative to $A$. Define the complex 
\[ \opn{TwInd}^{\mrm{esm}}_{C / B}(M) = \opn{TwInd}^{\mrm{esm}}_{v}(M) := 
M \ot_B \Om^n_{C / B}[n] \in \cat{D}(C) \]
and the rigidifying isomorphism 
\[ \begin{aligned}
& 
\opn{TwInd}^{\mrm{rig, esm}}_{C / B / A}(\rho) = 
\opn{TwInd}^{\mrm{rig, esm}}_{v / A}(\rho) \, := \,
\rho \cupprod \rho_{C / B}^{\mrm{esm}} \, : 
\\
& \qquad \qquad \quad 
M \ot_B \Om^n_{C / B}[n] \, \iso \, 
\opn{Sq}_{C / A} \bigl( M \ot_B \Om^n_{C / B}[n] \bigr) ,
\end{aligned} \]
where 
$\bigl( \Om^n_{C / B}[n], \, \rho_{C / B}^{\mrm{esm}} \bigr)$
is the rigid complex from Definition \ref{dfn:1225}, and 
$\rho \cupprod \rho_{C / B}^{\mrm{esm}}$
is the rigidifying isomorphism from from Definition \ref{dfn:1260}(1). 

The rigid complex 
\[ \opn{TwInd}^{\mrm{rig, esm}}_{C / B / A}(M, \rho) = 
\opn{TwInd}^{\mrm{rig}}_{v / A}(M, \rho) :=
\bigl( M \ot_B \Om^n_{C / B}[n], \, \rho \cupprod \rho_{B / A}^{\mrm{esm}} 
\bigr)
\in \cat{D}(C)_{\mrm{rig} / A} \]
is called the {\em twisted induced rigid complex}.
\end{dfn}

Observe that Definitions \ref{dfn:1260} and \ref{dfn:675} apply, because of 
Proposition \ref{prop:1270}.

\begin{dfn} \label{dfn:1236}
Let $A$ be a noetherian ring, and let $v : B \to C$ be an essentially smooth 
homomorphism between essentially finite type $A$-rings. Let 
$C = \prod_{i = 1}^r C_i$ be the connected component decomposition of $C$;
so for each $i$ the homomorphism $B \to C_i$ has constant differential relative 
dimension, and Definition \ref{dfn:1235} applies to it.
Given a rigid complex 
$(M, \rho) \in \cat{D}(B)_{\mrm{rig} / A}$, 
define 
\[ \opn{TwInd}^{\mrm{rig, esm}}_{C / B / A}(M, \rho) =
\opn{TwInd}^{\mrm{rig, esm}}_{v / A}(M, \rho) :=
\bigoplus\nolimits_{i = 1}^r \, \opn{TwInd}^{\mrm{rig, esm}}_{C_i / B / A}
(M, \rho)\in \cat{D}(C)_{\mrm{rig} / A} , \]
using Proposition \ref{prop:1490}. 
\end{dfn}

\section{Essentially \'Etale Ring Homomorphisms}
\label{sec:etale}

In Section \ref{sec:smooth} we studied the algebra and geometry of essentially 
smooth ring homomorphisms. Here we specialize to essentially \'etale ring 
homomorphisms (see Definition \ref{dfn:1015}), and make a much deeper 
investigation of the structure, in particular of the diagonal embedding.
These results could be of independent interest, beyond the context of our paper 
e.g. the pro-\'etale site \cite{BS}.
 
According to Conventions \ref{conv:615} and \ref{conv:1070}, both assumed in 
this section, all rings are commutative and noetherian, and all ring 
homomorphisms are EFT. 

Essentially \'etale ring homomorphisms are ubiquitous. We already saw that 
localizations are essentially \'etale (Example \ref{exa:1050}). Here is another 
example. 

\begin{exa} \label{exa:1240}
Let $f : A \to B$ be a homomorphism between noetherian rings. 
If $\opn{Spec}(f) : \opn{Spec}(B) \to \opn{Spec}(A)$ is an open embedding, then 
$f$ is \'etale. Indeed, by \cite[Proposition 6.2.5]{EGA-I}, $B$ is a finite 
type $A$-ring. For every prime $\q \sub B$, with preimage 
$\p := f^{-1}(\q) \sub A$, the ring homomorphism $A_{\p} \to B_{\q}$ is 
bijective, so a fortiori it is formally \'etale.  By 
\cite[Theorem 17.6.1]{EGA-IV} the homomorphism $A \to B$ is \'etale. 
\end{exa}

Besides localizations and open embeddings, one also has the important class of
finite \'etale homomorphisms $A \to B$, that are high dimensional 
generalizations of finite separable field extensions. 
Such ring homomorphisms are studied in Section \ref{sec:fin-etale}.

From here until Corollary \ref{cor:1246} (inclusive) we assume the
following setup:

\begin{setup} \label{set:1280}
We are given an essentially \'etale homomorphism $A \to B$ between 
noetherian rings, and we write $X := \opn{Spec}(A)$ and $Y := \opn{Spec}(B)$.
\end{setup}

Recall (from Section \ref{sec:smooth}) the multiplication ring homomorphism 
$\opn{mult}_{B / A} : B \ot_A B \to  B$
and the diagonal embedding 
$\opn{diag}_{Y / X} : Y \to  Y \times_X Y$.
They are related by this formula: 
$\opn{diag}_{Y / X} = \opn{Spec}(\opn{mult}_{B / A})$.
Of course $\opn{diag}_{Y / X}$ is a closed embedding of affine schemes. 

\begin{thm} \label{thm:1240}
Under Setup \ref{set:1280}, the diagonal embedding 
$\opn{diag}_{Y / X} : Y \to  Y \times_X Y$ is closed and open. 
\end{thm}

\begin{proof}
Let $C := B \ot_A B$ and $Z := \opn{Spec}(C) = Y\times_X Y$. 
We need to prove that the subset
$\opn{diag}_{Y / X}(Y)$ is open in $Z$.

By Theorem \ref{thm:1015} there is an affine open covering 
$Y = \bigcup_i Y_i$, $Y_i = \opn{Spec}(B_i)$, such that for every $i$
the homomorphism $A \to B_i$ factors through an \'etale homomorphism 
$A \to \til{B}_i$ and a localization homomorphism $\til{B}_i \to B_i$.
Let's define the rings $C_i := B_i \ot_A B_i$ and the affine schemes
$Y_i := \opn{Spec}(B_i)$ and 
$Z_i := \opn{Spec}(C_i)$.
Then the schemes $Z_i = Y_i \times_X Y_i$ are open in $Z$, and 
$\opn{diag}(Y_i) \sub Z_i$. Thus it suffices to prove that for every $i$ the 
subset $\opn{diag}(Y_i)$ is open in $Z_i$.

Define $\til{C}_i := \til{B}_i \ot_A \til{B}_i$,
$\til{Y}_i := \opn{Spec}(\til{B}_i)$ and 
$\til{Z}_i := \opn{Spec}(\til{C}_i)$. 
We get a cartesian diagram of schemes
\begin{equation} \label{eqn:1525}
\begin{tikzcd} [column sep = 6ex, row sep = 4ex] 
Y_i
\arrow[r]
\ar[d, "{\mrm{diag}}"']
&
\til{Y}_i
\ar[d, "{\opn{diag}}"]
\\
Z_i
\ar[r]
&
\til{Z}_i
\end{tikzcd} 
\end{equation}
in which all the arrows are topological embeddings, namely the topological 
spaces $Y_i, Z_i, \til{Y}_i$ are homeomorphic to their images inside the 
topological space $\til{Z}_i$. Thus 
$Y_i = Z_i \cap \til{Y}_i \sub \til{Z}_i$,
and it suffices to prove that $\til{Y}_i$ is open in $\til{Z}_i$.
Warning: in general $Y_i$ is not a subscheme of $\til{Y}_i$; see Remark 
\ref{rem:1560}.  

Now the scheme map $\til{Y}_i \to X$ is \'etale, and 
$\til{Z}_i = \til{Y}_i \times_X \til{Y}_i$.
According to \cite[Corollary 17.4.2]{EGA-IV}, the diagonal embedding 
$\opn{diag}_{\til{Y}_i / X} : \til{Y}_i \to \til{Z}_i$
is an open embedding.
\end{proof}

\begin{figure}[h]
\centering
\includegraphics[scale=0.18]{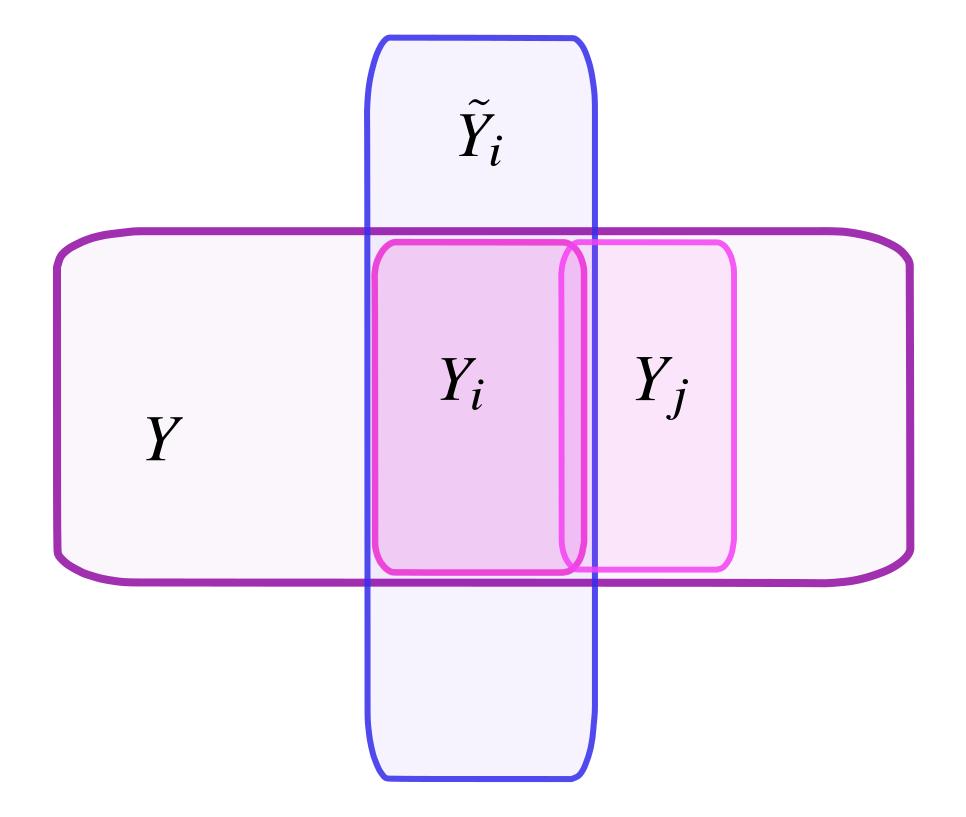}
\caption{An illustration for the proof of Theorem \ref{thm:1240}.}
\label{fig:10}
\end{figure}

\begin{figure}[h]
\centering
\includegraphics[scale=0.20]{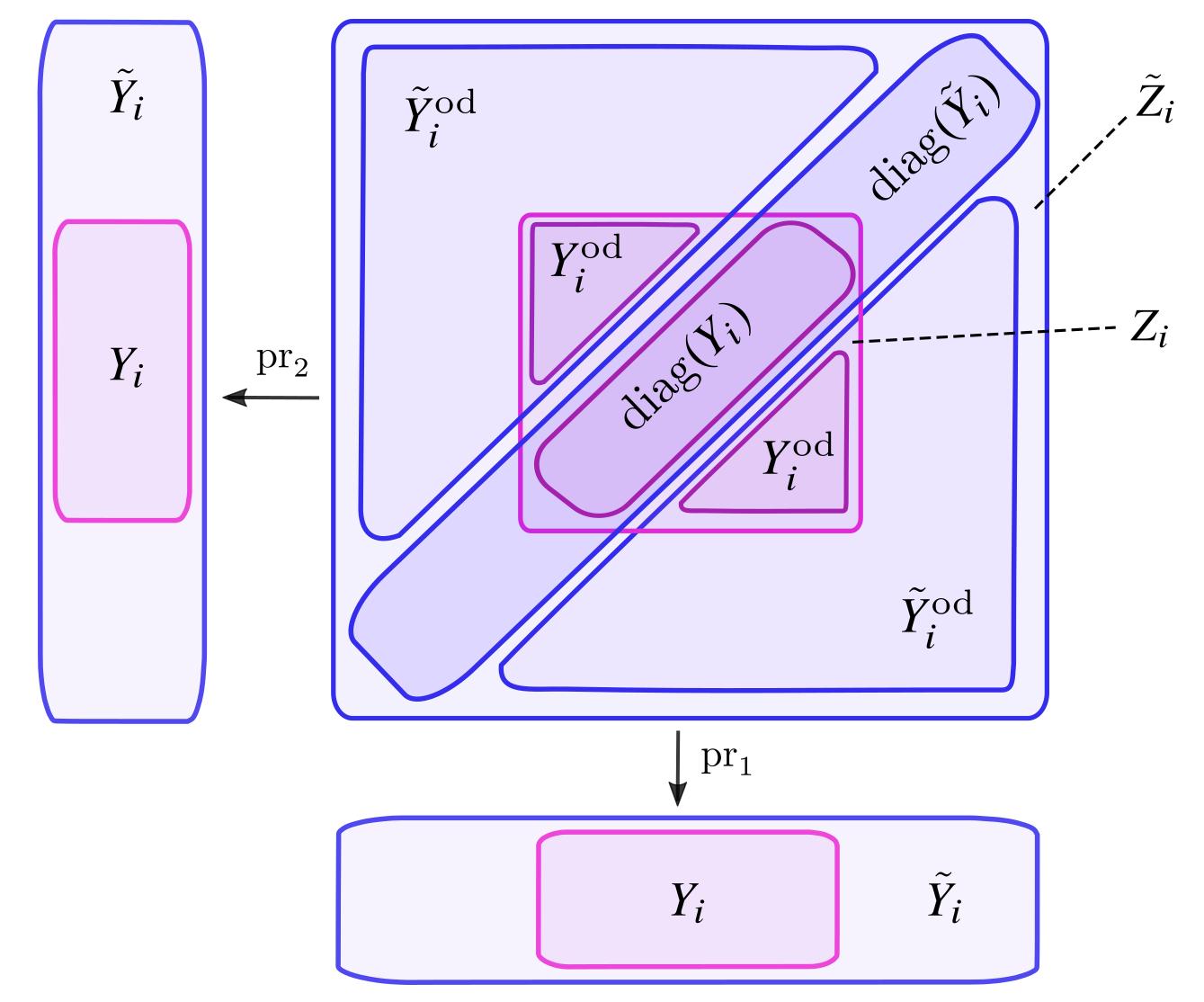}
\caption{An illustration for the proof of Theorem \ref{thm:1240}}
\label{fig:11}
\end{figure}

\newpage

\begin{cor} \label{cor:1280}
Under Setup \ref{set:1280}, there is a unique ring isomorphism 
\[ \tag{\dag} B \ot_A B \iso B \times B^{\mrm{od}} \]
such that the diagram 
\[ \tag{\ddag}
\begin{tikzcd} [column sep = 6ex, row sep = 4ex] 
B \ot_A B 
\arrow[r, "{\simeq}"]
\ar[dr, "{\mrm{mult}_{B / A}}"']
&
B \times B^{\mrm{od}} 
\ar[d, "{\mrm{pr}_1}"]
\\
&
B
\end{tikzcd} \]
of rings, where $\mrm{pr}_1$ is the projection to the first factor, is 
commutative.
\end{cor}

\begin{proof} 
Let's start with the existence of the ring isomorphism (\dag). By the theorem, 
and using the notation of its proof, we know that 
$Z = Y \times_X Y$ has a decomposition
$Z = \lb \opn{diag}_{Y / X}(Y) \, \sqcup \, Y^{\mrm{od}}$ 
into a disjoint union of two open and closed affine subschemes. Passing to 
rings this becomes the decomposition (\dag) with 
$B^{\mrm{od}} := \Ga(Y^{\mrm{od}}, \OO_Z)$. 

Now to the uniqueness of (\dag). Recall that 
$\opn{I}_{B / A} = \opn{Ker} \bigl (\opn{mult}_{B / A} : C \to B \bigr)$
is an ideal of the ring $C = B \ot_A B$. The commutativity of diagram (\ddag) 
implies that $B^{\mrm{od}} = \opn{I}_{B / A}$
as submodules (i.e.\ ideals) of $C$. 
Therefore the direct sum decomposition (\dag) of $C$ into $C$-submodules is 
$C = \opn{Ann}_{C}(\opn{I}_{B / A}) \oplus \opn{I}_{B / A}$.
This proves that the decomposition (\dag) is unique as $C$-modules, and 
therefore also as rings. 
\end{proof}

\begin{cor} \label{cor:1245}
Under Setup \ref{set:1280} the following hold:
\begin{enumerate}
\item The $(B \ot_A B)$-module $B$ is projective.

\item The ring homomorphism $\opn{mult}_{B / A} : B \ot_A B \to B$
is a localization. 
\end{enumerate}
\end{cor}

\begin{proof} 
(1) By Corollary \ref{cor:1280} we have a decomposition 
\begin{equation} \label{eqn:1255}
B \ot_A B \cong B \oplus B^{\mrm{od}}
\end{equation}
of $(B \ot_A B)$-modules.

\medskip \noindent
(2) The decomposition (\ref{eqn:1255}) exhibits $B$ and $B^{\mrm{od}}$ as 
ideals of the ring $B \ot_A B$; and by the ring isomorphism 
(\dag) in Corollary \ref{cor:1280} these ideals are generated
by idempotent elements $e, e^{\mrm{od}} \in B \ot_A B$ respectively,
which satisfy $e \cd e^{\mrm{od}} = 0$ and $e + e^{\mrm{od}} = 1$. 
Therefore 
$B \cong (B \ot_A B)_{e} = (B \ot_A B)[e^{-1}]$ as $(B \ot_A B)$-rings. 
\end{proof}

There is a canonical $(B \ot_A B)$-module isomorphism 
\begin{equation} \label{eqn:1245}
\opn{Hom}_{B \ot_A B}(B, B \ot_A B) \iso 
\opn{Ann}_{B \ot_A B}(\opn{I}_{B / A}) \sub B \ot_A B , \quad \phi \mapsto 
\phi(1) .
\end{equation}
This allows us to consider $\opn{Hom}_{B \ot_A B}(B, B \ot_A B)$ as a 
$(B \ot_A B)$-submodule of $B \ot_A B$.

\begin{cor} \label{cor:1246}
Under Setup \ref{set:1280}, the following hold:
\begin{enumerate}
\item There is a unique $(B \ot_A B)$-module homomorphism 
\[ \opn{sec}_{B / A} : B \to B \ot_A B  \]
such that 
$\opn{mult}_{B / A} \circ \opn{sec}_{B / A} = \opn{id}_B$.

\item The homomorphism $\opn{sec}_{B / A}$ induces an isomorphism of 
$(B \ot_A B)$-modules 
\[ \opn{sec}_{B / A} : B \iso \opn{Hom}_{B \ot_A B}(B, B \ot_A B) . \]
\end{enumerate}
\end{cor}

\begin{proof} 
(1) Existence: We have the $(B \ot_A B)$-module decomposition (\ref{eqn:1255}),
and we define $\opn{sec}_{B / A}$ to be the inclusion of the direct summand 
$B \sub B \ot_A B$ under this decomposition. Thus 
\begin{equation} \label{eqn:1256}
\opn{sec}_{B / A} : B \to \opn{Ann}_{B \ot_A B}(\opn{I}_{B / A}) = 
\opn{Hom}_{B \ot_A B}(B, B \ot_A B)
\end{equation}
is bijective. 

Uniqueness: Suppose $\si : B \to B \ot_A B$ is a $(B \ot_A B)$-module 
homomorphism s.t.\ 
$\opn{mult}_{B / A} \circ \, \si = \opn{id}_B$.
Because the $(B \ot_A B)$-module $B$ is annihilated by the ideal 
$\opn{I}_{B / A}$, we see that the image of $\si$ must lie inside 
$\opn{Ann}_{B \ot_A B}(\opn{I}_{B / A})$.
Since 
\[ \opn{mult}_{B / A}|_{\, \opn{Ann}_{B \ot_A B}(\opn{I}_{B / A})} : 
\opn{Ann}_{B \ot_A B}(\opn{I}_{B / A}) \to B \]
is bijective, it follows that $\si = \opn{sec}_{B / A}(B)$. 

\medskip \noindent
(2) This is the isomorphism (\ref{eqn:1256}).
\end{proof}

\begin{dfn} \label{dfn:}
The $(B \ot_A B)$-module homomorphism $\opn{sec}_{B / A} : B \to B \ot_A B$ 
is called the essentially \'etale section homomorphism of $B$ relative to $A$. 
\end{dfn}

\section{Squaring of Forward Morphisms}
\label{sec:forward}

The goal of this section is this: given an essentially \'etale homomorphism 
$v : B \to C$ in $\cat{Rng} \eftover A$, where $A$ is some noetherian base 
ring, and a forward morphism $\la : M \to N$ in $\cat{D}(B)$ over $v$,
we are going to construct a forward morphism 
\[ \opn{Sq}_{v / A}(\la) = \opn{Sq}_{C / B / A}(\la) : 
\opn{Sq}_{B / A}(M) \to \opn{Sq}_{C / A}(N)  \]
in $\cat{D}(B)$ over $v$. 
See Definition \ref{dfn:1290}. The important properties of this operation -- 
Theorems \ref{thm:1292}, \ref{thm:895} and \ref{thm:1300} --  
will be established in this section and in the subsequent one.
The main complication in the construction of the forward morphism 
$\opn{Sq}_{v / A}(\la)$ is this: unlike a localization homomorphism $v : B \to 
C$, where $C \ot_B C = C$, in the essentially \'etale case one has a ring 
decomposition $C \ot_B C = C \times C^{\mrm{od}}$, 
where $C^{\mrm{od}}$ is the ``off-diagonal'' ring; see Theorem \ref{thm:1240}. 
The challenge, as we shall see, is to carefully remove the off-diagonal 
component of the complex $\opn{Sq}_{B / A}(N)$. 

According to Conventions \ref{conv:615} and  \ref{conv:1070}, both of which are 
assumed here, all rings are commutative noetherian, and all ring homomorphisms 
are EFT. 

Throughout this section, until the end of the proof of Theorem \ref{thm:1860}, 
the next setup will be used.

\begin{setup} \label{set:1281}
We are given be a noetherian ring $A$, and an essentially \'etale 
homomorphism $v : B \to C$ between essentially finite type $A$-rings.
\end{setup}

Let's denote the ring homomorphism $A \to B$ by $u$. Then 
$(A \xar{u} B \xar{v} B) = C / B / A$
is a {\em triple of rings}; this is a slight modification of Definition 
\ref{dfn:630}(1). Moreover, since there is a fully faithful embedding 
$\cat{Rng} \sub \cat{DGRng}$, we can also consider $C / B  / A$ as a triple of 
DG rings.  

Recall the notion of a K-flat resolution 
\begin{equation} \label{eqn:1280}
t / s / r : \til{C} / \til{B} / \til{A} \to C / B / A 
\end{equation}
of the triple of DG rings $C / B / A$ from Definition \ref{dfn:630}. It says 
that there is a commutative diagram 
in $\cat{DGRng}$
\[ \begin{tikzcd} [column sep = 6ex, row sep = 4ex] 
\til{A}
\arrow[r, "{\til{u}}"]
\ar[d, "{r}"]
&
\til{B}
\ar[d, "{s}"]
\arrow[r, "{\til{v}}"]
&
\til{C}
\ar[d, "{t}"]
\\
A
\arrow[r, "{u}"]
&
B
\arrow[r, "{v}"]
&
C
\end{tikzcd} \]
such that $r, s, t$ are surjective quasi-isomorphisms, and $\til{u}$, $\til{v}$ 
are K-flat (i.e.\ $\til{B}$ and $\til{C}$ are K-flat as DG modules over 
$\til{A}$ and $\til{B}$, respectively).
Given a resolution (\ref{eqn:1280}), we define the auxiliary DG ring 
\begin{equation} \label{eqn:1320}
\til{D} := (\til{C}_1 \ot_{\til{A}} \til{C}_2)
\ot_{\til{B}_1 \ot_{\til{A}} \til{B}_2} \til{B}_0.
\end{equation}
Observe the similarity to the DG ring $\til{D}$ from Setup \ref{set:1395}.

\begin{lem} \label{lem:1315}
Given a resolution (\ref{eqn:1280}), the DG ring homomorphism 
\[ \begin{aligned}
& 
g : \til{D} = (\til{C}_1 \ot_{\til{A}} \til{C}_2)
\ot_{\til{B}_1 \ot_{\til{A}} \til{B}_2} {B}_0 \to C_1 \ot_{{B}_0} C_2 , 
\\
& 
g \bigl( (\til{c}_1 \ot \til{c}_2) \ot b_0 \bigr) :=  
(b_0 \cd t(\til{c}_1)) \ot t(\til{c}_2) = 
t(\til{c}_1) \ot (b_0 \cd t(\til{c_2}))
\end{aligned} \]
is a quasi-isomorphism.
\end{lem}

\begin{proof}
The homomorphism $g$ sits as the right vertical arrow inside this commutative 
diagram of DG rings: 
\[ \begin{tikzcd} [column sep = 8ex, row sep = 4ex] 
(\til{C}_1 \ot_{\til{A}} \til{C}_2) \ot_{\til{B}_1 \ot_{\til{A}} \til{B}_2} 
\til{B}_0 
\ar[d]
\ar[r, "{\opn{id} \ot \, s_0}"]
&
(\til{C}_1 \ot_{\til{A}} \til{C}_2) \ot_{\til{B}_1 \ot_{\til{A}} \til{B}_2} B_0
\ar[d, "{g}"]
\\
\til{C}_1 \ot_{\til{B}_0} \til{C}_2
\ar[r, "{t_1 \ot_{s_0} t_2}"]
&
C_1 \ot_{B_0} C_2
\end{tikzcd} \]
The top horizontal arrow is a quasi-isomorphism, because 
$\til{B}_1 \ot_{\til{A}} \til{B}_2 \to \til{C}_1 \ot_{\til{A}} \til{C}_2$
is K-flat and $\til{B} \to B$ is a quasi-isomorphism. 
The left vertical arrow is an isomorphism (it is the isomorphism $g$ in 
formula (\ref{eqn:1510})). 
The bottom horizontal arrow is a quasi-isomorphism, because 
$\til{B} \to \til{C}$ and $B \to C$ are both K-flat, and 
$\til{B} \to B$ and $\til{C} \to C$ are both quasi-isomorphisms; cf.\ 
\cite[Proposition 2.6$\tup{(1)}$]{Ye4}. Therefore the right vertical arrow
is a quasi-isomorphism.
\end{proof}

\begin{dfn} \label{dfn:1282}
Under Setup \ref{set:1281}, let 
$N \in \cat{D}(C)$, and let $\til{C} / \til{B} / \til{A}$ be a K-flat 
resolution of the triple of rings $C / B / A$. 
Define the morphism
\[ \opn{red}^{\til{C} / \til{B} / \til{A}}_{C/ B / A, N} : 
\opn{Sq}_{B / A}^{\til{B} / \til{A}}(N) \to 
\opn{Sq}_{C / A}^{\til{C} / \til{A}}(N) \]
in $\cat{D}(B)$, called the {\em resolved reduction morphism}, 
 to be the composition of the morphisms below:
\begin{equation} \label{eqn:1290}
\begin{aligned}
& 
\opn{Sq}_{B / A}^{\til{B} / \til{A}}(N)
= \opn{RHom}_{\til{B}_1 \ot_{\til{A}} \til{B}_2}
\bigl( B_0, N_1 \ot_{\til{A}}^{\mrm{L}} N_2 \bigr)
\\
& \quad 
\iso^{\mrm{(1)}} 
\opn{RHom}_{\til{C}_1 \ot_{\til{A}} \til{C}_2}
\bigl( (\til{C}_1 \ot_{\til{A}} \til{C}_2) 
\ot_{\til{B}_1 \ot_{\til{A}} \til{B}_2} B_0, 
N_1 \ot_{\til{A}}^{\mrm{L}} N_2 \bigr)
\\
& \quad 
\iso^{\mrm{(2)}} 
\opn{RHom}_{\til{C}_1 \ot_{\til{A}} \til{C}_2}
\bigl( C_1 \ot_{B_0} C_2, N_1 \ot_{\til{A}}^{\mrm{L}} N_2 \bigr)
\\
& \quad 
\to^{\mrm{(3)}} 
\opn{RHom}_{\til{C}_1 \ot_{\til{A}} \til{C}_2}
\bigl( C_0, N_1 \ot_{\til{A}}^{\mrm{L}} N_2 \bigr)
= \opn{Sq}_{C / A}^{\til{C} / \til{A}}(N) . 
\end{aligned} 
\end{equation}
Here we use Convention \ref{conv:1295} to distinguish, by number subscripts,
between the various positions that DG rings and modules appear in tensor 
products. The DG ring $B$ acts on all objects in (\ref{eqn:1290}) through 
$B_0$, 
with the exception of the objects in the last line, where it acts through $C_0$.
The isomorphism $\iso^{\mrm{(1)}}$ is Hom-tensor adjunction for the DG 
ring homomorphism 
$\til{B}_1 \ot_{\til{A}} \til{B}_2 \to \til{C}_1 \ot_{\til{A}} \til{C}_2$.
The isomorphism $\iso^{\mrm{(2)}}$ comes from Lemma \ref{lem:1315}.
Finally, the morphism $\to^{\mrm{(3)}}$ is 
$\opn{RHom}(\opn{sec}_{C / B}, \opn{id})$, 
see Corollary \ref{cor:1246}(1). 
\end{dfn}

\begin{thm} \label{thm:1285}
Under Setup \ref{set:1281}, given a complex $N \in \cat{D}(C)$, there is a 
unique morphism 
\[ \opn{red}_{C/ B / A, N} : 
\opn{Sq}_{B / A}(N) \to \opn{Sq}_{C / A}(N) \]
in $\cat{D}(B)$, called the {\em reduction to the diagonal}, with this 
property: 
\begin{itemize}
\item[($*$)] For every K-flat resolution 
$\til{C} / \til{B} / \til{A}$ of the triple of DG rings $C / B / A$,
the diagram 
\[ \begin{tikzcd} [column sep = 12ex, row sep = 5ex] 
\opn{Sq}_{B / A}(N)
\ar[r, "{\opn{red}_{C/ B / A, N}}"]
\ar[d, "{\opn{sq}^{\til{B} / \til{A}}_{B / A, N}}"', "{\simeq}"]
&
\opn{Sq}_{C / A}(N)
\ar[d, "{\opn{sq}^{\til{C} / \til{A}}_{C / A, N}}", "{\simeq}"']
\\
\opn{Sq}_{B / A}^{\til{B} / \til{A}}(N)
\ar[r, "{\opn{red}^{\til{C} / \til{B} / \til{A}}_{C/ B / A, N}}"]
&
\opn{Sq}_{C / A}^{\til{C} / \til{A}}(N)
\end{tikzcd} \]
in $\cat{D}(B)$ is commutative.  
\end{itemize}
\end{thm}

\begin{proof}
Choose a universal K-flat DG ring resolution
\[ t^{\mrm{u}} / s^{\mrm{u}} / r^{\mrm{u}} : 
\til{C}^{\mrm{u}} / \til{B}^{\mrm{u}} / \til{A}^{\mrm{u}} \to C / B / A  \]
of the triple of rings $C / B / A$, 
as in the proof of Theorem \ref{thm:780}, and define 
$\opn{red}_{C/ B / A, N}$
to be the unique morphism that makes the diagram in property ($*$) 
commutative for this resolution.

Given some other K-flat resolution $\til{C} / \til{B} / \til{A}$ of the triple 
of DG rings $C / B / A$, choose a lifting 
$\til{t} / \til{s} / \til{r}$ like in diagram (\ref{eqn:840}).
Going over the morphisms in (\ref{eqn:1290}) we see that the diagram 
\[ \begin{tikzcd} [column sep = 12ex, row sep = 5ex] 
\opn{Sq}_{B / A}^{\til{B} / \til{A}}(N)
\ar[r, "{\opn{red}^{\til{C} / \til{B} / \til{A}}_{C/ B / A, N}}"]
\ar[d, "{\opn{Sq}^{\til{s} / \til{r}}_{B / A}(\opn{id}_N)}"', "{\simeq}"]
&
\opn{Sq}_{C / A}^{\til{C} / \til{A}}(N)
\ar[d, "{\opn{Sq}^{\til{t} / \til{r}}_{C / A}(\opn{id}_N)}", "{\simeq}"']
\\
\opn{Sq}_{B / A}^{\til{B}^{\mrm{u}} / \til{A}^{\mrm{u}}}(N)
\ar[r, "{\opn{red}^{\til{C}^{\mrm{u}} / \til{B}^{\mrm{u}} / 
\til{A}^{\mrm{u}}}_{C/ B / A, N}}"]
&
\opn{Sq}_{C / A}^{\til{C}^{\mrm{u}} / \til{A}^{\mrm{u}}}(N)
\end{tikzcd} \]
in which the vertical morphisms are from formula (\ref{eqn:1490}),
in $\cat{D}(B)$ is commutative. This implies that the diagram in property ($*$) 
is commutative for the resolution $\til{C} / \til{B} / \til{A}$.
\end{proof}

\begin{dfn}[Squaring of a Forward Morphism] \label{dfn:1290} 
Under Setup \ref{set:1281}, let $M \in \cat{D}(B)$ and $N \in \cat{D}(C)$, and 
let $\la : M \to N$ be a forward morphism in $\cat{D}(B)$ over $v$.
Define the forward morphism 
\[ \opn{Sq}_{v / A}(\la) = \opn{Sq}_{C / B / A}(\la) : 
\opn{Sq}_{B / A}(M) \to \opn{Sq}_{C / A}(N) \]
in $\cat{D}(B)$ over $v$ to be the composition 
\[ \opn{Sq}_{v / A}(\la) :=  \opn{red}_{C/ B / A, N} \circ 
\opn{Sq}_{B / A}(\la) , \] 
where $\opn{Sq}_{B / A}(\la)$ 
is the morphism from Definition \ref{dfn:880}. 

In other words, the diagram 
\[ \begin{tikzcd} [column sep = 12ex, row sep = 6ex] 
\opn{Sq}_{B / A}(M)
\ar[rr, bend left = 15, start anchor = north, end anchor = north,
, "{\opn{Sq}_{v / A}(\la)}"]
\ar[r, "{\opn{Sq}_{B / A}(\la)}"]
&
\opn{Sq}_{B / A}(N)
\ar[r, "{\opn{red}_{C/ B / A, N}}"]
&
\opn{Sq}_{C / A}(N)
\end{tikzcd} \]
in $\cat{D}(B)$ is commutative.
\end{dfn}

In case $B = C$ and $v = \opn{id}_B$, there is potential for a conflict between 
Definitions \ref{dfn:1290} and \ref{dfn:880}. This is settled in the next 
proposition.  

\begin{prop} \label{prop:1671}
Assume $B = C$ and $v = \opn{id}_B$. 
Given a morphism $\la : M \to N$ in $\cat{D}(B)$, there is equality 
$\opn{Sq}_{v / A}(\la) = \opn{Sq}_{B / A}(\la)$
of morphisms $\opn{Sq}_{B / A}(M) \to \opn{Sq}_{B / A}(N)$
in $\cat{D}(B)$.
\end{prop}

\begin{proof}
This is because $\opn{red}_{B / B / A, N} = \opn{id}_{\opn{Sq}_{B / A}(N)}$.
\end{proof}

\begin{thm}[Squaring of Composable Forward Morphisms] \label{thm:1292}
Under Setup \ref{set:1281}, let $w : C \to D$ be another essentially \'etale 
ring homomorphism, let $\la : L \to M$ be a forward morphism in $\cat{D}(B)$ 
over $v$, and let $\mu : M \to N$ be a forward morphism in $\cat{D}(C)$ 
over $w$. Then 
\[ \opn{Sq}_{w / A}(\mu) \circ \opn{Sq}_{v / A}(\la) =
\opn{Sq}_{w \circ v / A}(\mu \circ \la) , \]
as forward morphisms
$\opn{Sq}_{B / A}(L) \to \opn{Sq}_{D / A}(N)$ in $\cat{D}(B)$ over $w \circ v$.
\end{thm}

A lemma is required before the proof of the theorem. 

\begin{lem} \label{lem:1395}
Let 
\[ \opn{id} \ot \opn{sec}_{C / B} : D \ot_C D \to D \ot_B D \]
be the $(D \ot_B D)$-module homomorphism gotten by applying 
$(D \ot_B D) \ot_{C \ot_B C} (-)$
to the $(C \ot_B C)$-module homomorphism 
$\opn{sec}_{C / B} : C \to C \ot_B C$. Then 
\[ \tag{$\heartsuit$}
(\opn{id} \ot \opn{sec}_{C / B}) \circ \opn{sec}_{D / C} = 
\opn{sec}_{D / B} , \]
as $(D \ot_B D)$-module homomorphisms $D \to D \ot_B D$. 
\end{lem}

\begin{proof}
We start with this commutative diagram: 
\begin{equation} \label{eqn:1395}
\begin{tikzcd} [column sep = 10ex, row sep = 6ex] 
C
\ar[rr, bend left = 18, start anchor = north east, end anchor = north west,
"{\opn{id}_{}}"]
\ar[r, "{\opn{sec}_{C / B} }"]
&
C \ot_B C
\ar[r, "{\opn{mult}_{C / B} }"]
&
C
\end{tikzcd} 
\end{equation}
in $\cat{M}(C \ot_B C)$. See Corollary \ref{cor:1246}(1). After applying 
$(D \ot_B D) \ot_{C \ot_B C} (-)$
to it, and inserting $D$ in three positions, we get this diagram 
\begin{equation} \label{eqn:1420}
\begin{tikzcd} [column sep = 10ex, row sep = 5ex] 
D
\ar[r, "{\opn{sec}_{D / C}}"]
\ar[dr, bend left = -15, start anchor = south, end anchor = west, "{\opn{id}}"']
&
D \ot_C D
\ar[d, "{\opn{mult}_{D/C}}"]
\ar[rr, bend left = 15, start anchor = north, end anchor = north,
"{\opn{id}_{}}"]
\ar[r, "{\opn{id} \ot \opn{sec}_{C / B}}"]
&
D \ot_B D
\ar[r, "{\opn{id} \ot \opn{mult}_{C / B}}"]
\ar[d, "{\opn{mult}_{D / B}}"]
&
D \ot_C D
\ar[d, "{\opn{mult}_{D / C}}"]
\\
&
D
\ar[r, "{\opn{id}}"]
&
D
\ar[r, "{\opn{id}}"]
&
D
\end{tikzcd}
\end{equation}
in $\cat{M}(D \ot_B D)$. The top curved area is commutative because
(\ref{eqn:1395}) is commutative. An easy calculation with elements shows that 
the rectangular region at the bottom right is commutative.
By Corollary \ref{cor:1246}(1) the homomorphism 
$\opn{mult}_{D / C} \circ \opn{sec}_{D / C}$
equals $\opn{id}_D$. This implies that the whole diagram is commutative.
The commutativity of the curved region at the bottom left, together with 
the uniqueness in Corollary \ref{cor:1246}(1), tell us that 
the equality ($\heartsuit$) holds. 
\end{proof}

\begin{proof}[Proof of Theorem \tup{\ref{thm:1292}}]
We need to prove that the diagram 
\begin{equation}\label{eqn:14001}
\begin{tikzcd} [column sep = 10ex, row sep = 5ex] 
\opn{Sq}_{B / A}(L)
\ar[r, "{\opn{Sq}_{v / A}(\la)}"]
\ar[dr, "{\opn{Sq}_{w \circ v / A}(\mu \circ \la)}"']
&
\opn{Sq}_{C / A}(M)
\ar[d, "{\opn{Sq}_{w / A}(\mu)}"]
\\
&
\opn{Sq}_{D / A}(N)
\end{tikzcd} 
\end{equation}
in $\cat{D}(B)$ is commutative. Let's expand this diagram using Definition 
\ref{dfn:1290}:
\begin{equation} \label{eqn:1400}
\begin{tikzcd} [column sep = 10ex, row sep = 5ex] 
\opn{Sq}_{B / A}(L)
\ar[r, "{\opn{Sq}_{B / A}(\la)}"]
\ar[dr, "{\opn{Sq}_{B / A}(\mu \circ \la)}"']
&
\opn{Sq}_{B / A}(M)
\ar[r, "{\opn{red}_{C / B / A, M}}"]
\ar[d, "{\opn{Sq}_{B / A}(\mu)}"]
&
\opn{Sq}_{C / A}(M)
\ar[d, "{\opn{Sq}_{C / A}(\mu)}"]
\\
&
\opn{Sq}_{B / A}(N)
\ar[r, "{\opn{red}_{C / B / A, N}}"]
\ar[dr, "{\opn{red}_{D / B / A, N}}"']
&
\opn{Sq}_{C / A}(N)
\ar[d, "{\opn{red}_{D / C / A, N}}"]
\\
&
&
\opn{Sq}_{D / A}(N)
\end{tikzcd}
\end{equation}
Thus the outer paths in (\ref{eqn:1400}) are the (\ref{eqn:14001}).
The top left triangle is commutative by Proposition \ref{prop:890}. To prove 
that the rectangle and the bottom right triangle are commutative, we will 
choose a K-flat resolution  
$\til{D} / \til{C} / \til{B} / \til{A}$ 
of the quadruple of DG rings $D / C / B  /A$ 
(this is the obvious modification of Definition \ref{dfn:630}(3)). 
Using this resolution we can replace 
$\opn{Sq}_{B / A}(L)$ with 
$\opn{Sq}_{B / A}^{\til{B} / \til{A}}(L) =
\opn{RHom}_{\til{B} \ot_{\til{A}} \til{B}}(B, L \ot^{\mrm{L}}_{\til{A}} L)$, 
etc. Consider the following diagram 
\[ \begin{tikzcd} [column sep = 13ex, row sep = 5ex] 
\opn{RHom}_{\til{B} \ot_{\til{A}} \til{B}}(B, M \ot^{\mrm{L}}_{\til{A}} M)
\ar[r, "{\opn{RHom}(\opn{id}, \mu \ot^{\mrm{L}} \mu)}"]
\ar[d, "{\opn{adj}}"', "{\simeq}"]
&
\opn{RHom}_{\til{B} \ot_{\til{A}} \til{B}}(B, N \ot^{\mrm{L}}_{\til{A}} N)
\ar[d, "{\opn{adj}}"', "{\simeq}"]
\\
\opn{RHom}_{\til{C} \ot_{\til{A}} \til{C}}
(C \ot_B C, M \ot^{\mrm{L}}_{\til{A}} M)
\ar[r, "{\opn{RHom}(\opn{id}, \mu \ot^{\mrm{L}} \mu)}"]
\ar[d, "{\opn{RHom}(\opn{sec}_{C / B}, \opn{id})}"]
&
\opn{RHom}_{\til{C} \ot_{\til{A}} \til{C}}
(C \ot_B C, N \ot^{\mrm{L}}_{\til{A}} N)
\ar[d, "{\opn{RHom}(\opn{sec}_{C / B}, \opn{id})}"']
\\
\opn{RHom}_{\til{C} \ot_{\til{A}} \til{C}}(C, M \ot^{\mrm{L}}_{\til{A}} M)
\ar[r, "{\opn{RHom}(\opn{id}, \mu \ot^{\mrm{L}} \mu)}"]
&
\opn{RHom}_{\til{C} \ot_{\til{A}} \til{C}}(C, N \ot^{\mrm{L}}_{\til{A}} N)
\end{tikzcd} \]
in $\cat{D}(B)$. The vertical arrows labeled ``adj'' are adjunction for the DG 
ring homomorphism 
$\til{B} \ot_{\til{A}} \til{B} \to \til{C} \ot_{\til{A}} \til{C}$.
 It is easy to see that this diagram is commutative. A comparison with 
formula (\ref{eqn:1290}) shows that the composed vertical arrows in it are 
$\opn{red}^{\til{C} / \til{B} / \til{A}}_{C/ B / A, M}$
and 
$\opn{red}^{\til{C} / \til{B} / \til{A}}_{C/ B / A, N}$, 
respectively. Also, the top and bottom horizontal arrows are 
$\opn{Sq}_{B / A}^{\til{B} / \til{A}}(\mu)$
and 
$\opn{Sq}_{C / A}^{\til{C} / \til{A}}(\mu)$, respectively.  
The conclusion is that the rectangular subdiagram in (\ref{eqn:1400}) is 
commutative. 

Finally we are going to handle the bottom right triangle in diagram 
(\ref{eqn:1400}). After passing to the parameterized version (the one with 
tildes), and applying the adjunction isomorphisms, we end up with this 
diagram: 
\[ \begin{tikzcd} [column sep = 17ex, row sep = 5ex] 
\opn{RHom}_{\til{D} \ot_{\til{A}} \til{D}}
(D \ot_B D, N \ot^{\mrm{L}}_{\til{A}} N)
\ar[r, "{\opn{RHom}(\opn{id} \ot \opn{sec}_{C / B}, \opn{id})}"]
\ar[dr, "{\opn{RHom}(\opn{sec}_{D / B}, \opn{id})}"']
&
\opn{RHom}_{\til{D} \ot_{\til{A}} \til{D}}
(D \ot_C D, N \ot^{\mrm{L}}_{\til{A}} N)
\ar[d, "{\opn{RHom}(\opn{sec}_{D / C}, \opn{id})}"]
\\
&
\opn{RHom}_{\til{D} \ot_{\til{A}} \til{D}}
(D, N \ot^{\mrm{L}}_{\til{A}} N)
\end{tikzcd} \]
This diagram is commutative by Lemma \ref{lem:1395}.
\end{proof}

Given a complex $M \in \cat{D}(B)$, the morphism 
$\eta^{\mrm{L}}_{M, C} : M \ot^{\mrm{L}}_{B} C \iso M \ot_{B} C$
from (\ref{eqn:621}) is an isomorphism; this is because $B \to C$ is flat. 
For convenience let us temporarily consider this isomorphism as an equality. 
Then there is the standard nondegenerate forward morphism  
\begin{equation} \label{eqn:1293}
\opn{q}^{\mrm{L}}_{C / B, M} = \opn{q}_{C / B, M} = \opn{q}_{v, M}
: M \to M \ot_B C , \quad \opn{q}_{v, M}(m) = m \ot 1 ,  
\end{equation}
see formula (\ref{eqn:656}). Recall that 
a forward morphism $\la : M \to N$ in $\cat{D}(B)$ over $v$ is called 
nondegenerate if the morphism
$\opn{fadj}^{\mrm{L}}_{v, M, N}(\la) : M \ot^{\mrm{L}}_B C \to N$
in $\cat{D}(C)$, which corresponds to $\la$ by forward adjunction, is an 
isomorphism. Here is the commutative diagram in $\cat{D}(B)$ depicting this 
situation:
\begin{equation} \label{eqn:1321}
\begin{tikzcd} [column sep = 10ex, row sep = 5ex] 
M
\ar[r, "{\opn{q}^{\mrm{L}}_{v, M}}"]
\ar[dr, "{\la}"']
&
M \ot^{\mrm{L}}_B C
\ar[d, "{\opn{fadj}^{\mrm{L}}_{v, M, N}(\la)}"]
\\
&
N
\end{tikzcd}
\end{equation}

The next theorem is a variant of Theorem \ref{thm:672}. 

\begin{thm} \label{thm:1860}
Under Setup \ref{set:1281}, let $M \in \cat{D}(B)$,  
$N \in \cat{D}(C)$, $\la : M \to  N$ a forward morphism in $\cat{D}(B)$ over 
$v$, and $c \in C$. Then 
\[ \opn{Sq}_{v / A}(c \cd \la) = c^2 \cd \opn{Sq}_{v / A}(\la) \]
as forward morphisms $\opn{Sq}_{B / A}(M) \to \opn{Sq}_{C / A}(N)$
in $\cat{D}(B)$ over $v$. 
\end{thm}

\begin{proof}
The ring homomorphism $v$ factors into $v = \opn{id}_C \circ \, v$, and the 
forward morphism $\la$ factors into $\la = \opn{id}_N \circ \, \la$.
See commutative diagrams below. 
\[ \begin{tikzcd} [column sep = 6ex, row sep = 4ex] 
B
\ar[rr, bend left = 30, 
start anchor = north east, end anchor = north west, "{v}"]
\ar[r, "{v}"]
&
C
\ar[r, "{\opn{id}_C}"]
&
C
\end{tikzcd} 
\qquad 
 \begin{tikzcd} [column sep = 6ex, row sep = 4ex] 
M
\ar[rr, bend left = 30, 
start anchor = north east, end anchor = north west, "{\la}"]
\ar[r, "{\la}"]
&
N
\ar[r, "{\opn{id}_N}"]
&
N
\end{tikzcd} \]
A morphism $\phi : N \to N$ in $\cat{D}(C)$ 
can be viewed either as a forward morphism 
or as a backward morphism in $\cat{D}(C)$ over $\opn{id}_C$. But according to 
Proposition \ref{prop:1671} its square 
$\opn{Sq}_{\opn{id}_C / A}(\phi) = \opn{Sq}_{C / A}(\phi)$
is the same in both interpretations, that of Definition \ref{dfn:1290} 
and that of Theorem \ref{thm:632} and Definition \ref{dfn:880}. We are going to 
use this for $\phi = c \cd \opn{id}_N$.

By Theorem \ref{thm:1292} we have equality
\begin{equation} \label{eqn:1861}
\opn{Sq}_{v / A}(c \cdot \la) = 
\opn{Sq}_{(\opn{id}_C \circ \, v) / A}
\bigl( (c \cd \opn{id}_N) \circ \la \bigr) =
\opn{Sq}_{\opn{id}_C / A}(c \cd \opn{id}_N) \circ \opn{Sq}_{v / A}(\la) 
\end{equation}
of forward morphisms $\opn{Sq}_{B / A}(M) \to \opn{Sq}_{C / A}(N)$
in $\cat{D}(B)$ over $v$. Next, by Theorem \ref{thm:672} we know that 
\[ \opn{Sq}_{\opn{id}_C / A}(c \cd \opn{id}_N) = 
c^2 \cd \opn{Sq}_{\opn{id}_C / A}(\opn{id}_N) = 
c^2 \cd \opn{id}_{\opn{Sq}_{C / A}(N)} . \]
Plugging this into formula (\ref{eqn:1861}) we obtain 
$\opn{Sq}_{v / A}(c \cdot \la) = c^2 \cd \opn{Sq}_{v / A}(\la)$.
\end{proof}

From here to the end of the section we abandon Setup 
\ref{set:1281}.

\begin{thm}[Forward and Backward Compatibility] \label{thm:895}
Let $A$ be a noetherian ring, and let $u : B \to C$ and 
$v : B \to  B'$ be homomorphisms in $\cat{Rng} \eftover A$.
Define the ring $C' := B' \ot_{B} C$, and let 
$u' : B' \to C'$ and $w : C \to C'$ be the induced ring homomorphisms.
Assume that $v : B \to  B'$ is an essentially \'etale homomorphism. 

Let $M \in \cat{D}(B)$ and $N \in \cat{D}(C)$ be complexes, and let 
$\th : N \to M$ be a backward morphism in $\cat{D}(B)$ over $u$. 
Define the complexes 
$M' := B' \ot_B M \in \cat{D}(B')$ and 
$N' := C' \ot_C N \in \cat{D}(C')$, and let 
$\th' : N' \to M'$ be the backward morphism in $\cat{D}(B')$ over $u'$ that's
induced from $\th$, so there is equality 
$\opn{q}^{\mrm{L}}_{v, M} \circ \, \th = \th' \circ \opn{q}^{\mrm{L}}_{w, N}$
of morphisms $N \to M'$ in $\cat{D}(B)$.  

Then there is equality 
\[ \opn{Sq}_{v / A}(\opn{q}^{\mrm{L}}_{v, M}) \circ \opn{Sq}_{u / A}(\th) = 
\opn{Sq}_{u' / A}(\th') \circ \opn{Sq}_{w / A}(\opn{q}^{\mrm{L}}_{w, N}) \]
of morphisms in $\cat{D}(B)$.
\end{thm}

Note that $w : C \to C'$ is an essentially \'etale ring homomorphism, so the 
forward morphism $\opn{Sq}_{w / A}(\opn{q}^{\mrm{L}}_{w, N})$ exists. 
Here are the commutative diagrams depicting the homomorphisms in 
$\cat{Rng} \eftover A$ and the morphisms in $\cat{D}(B)$ mentioned in the 
theorem.
\begin{equation} \label{eqn:1565}
\begin{tikzcd} [column sep = 6ex, row sep = 5ex] 
B
\ar[r, "{u}"]
\ar[d, "{v}"]
&
C
\ar[d, "{w}"]
\\
B'
\ar[r, "{u'}"]
&
C' 
\end{tikzcd} 
\quad 
\begin{tikzcd} [column sep = 6ex, row sep = 5ex] 
M
\ar[d, "{\opn{q}^{\mrm{L}}_{v, M}}"]
&
N
\ar[l, "{\th}"']
\ar[d, "{\opn{q}^{\mrm{L}}_{w, N}}"]
\\
M'
&
N' 
\ar[l, "{\th'}"']
\end{tikzcd}
\quad 
\begin{tikzcd} [column sep = 12ex, row sep = 5ex] 
\opn{Sq}_{B / A}(M)
\ar[d, "{\opn{Sq}_{v / A}(\opn{q}^{\mrm{L}}_{v, M})}"]
&
\opn{Sq}_{C / A}(N)
\ar[l, "{\opn{Sq}_{u / A}(\th)}"']
\ar[d, "{\opn{Sq}_{w / A}(\opn{q}^{\mrm{L}}_{w, N})}"]
\\
\opn{Sq}_{B' / A}(M')
&
\opn{Sq}_{C' / A}(N') 
\ar[l, "{\opn{Sq}_{u' / A}(\th')}"]
\end{tikzcd} 
\end{equation}

We need this lemma before proving the theorem. 

\begin{lem} \label{lem:1455}
In the situation of the theorem, the diagram 
\[ \tag{$\dag$}
\begin{tikzcd} [column sep = 8ex, row sep = 5ex] 
B'
\ar[d, "{u'}"]
\ar[r, "{\opn{sec}_{B' / B}}"]
&
B' \ot_B B' 
\ar[d, "{u' \ot_u u'}"]
\\
C'
\ar[r, "{\opn{sec}_{C' / C}}"]
&
C' \ot_C C' 
\end{tikzcd} \]

in $\cat{M}(B' \ot_B B')$ is commutative. 
\end{lem}

\begin{proof}
Consider the following commutative diagram in $\cat{M}(B' \ot_B B')$, coming from Corollary \ref{cor:1246}(1):
\begin{equation} \label{eqn:1455}
\begin{tikzcd} [column sep = 10ex, row sep = 6ex] 
B'
\ar[rr, bend left = 20, start anchor = north east, end anchor = north west,
"{\opn{id}_{B'}}"]
\ar[r, "{\opn{sec}_{B' / B} }"]
&
B' \ot_B B'
\ar[r, "{\opn{mult}_{B' / B} }"]
&
B'
\end{tikzcd} 
\end{equation}
After applying the functor $(-) \ot_B C$ to diagram (\ref{eqn:1455}), and using 
the ring isomorphism 
\[ (B' \ot_B B') \ot_B C \cong B' \ot_B C' , \
(b'_1 \ot b'_2) \ot c \mapsto b'_1 \ot (u'(b'_2) \cd w(c)) , \]
we obtain this commutative diagram 
\begin{equation} \label{eqn:1456}
\begin{tikzcd} [column sep = 12ex, row sep = 5.5ex] 
B'
\ar[rr, bend left = 15, start anchor = north east, end anchor = north west,
"{\opn{id}_{B'}}"]
\ar[r, "{\opn{sec}_{B' / B} }"]
\ar[d, "{u'}"]
&
B' \ot_B B'
\ar[r, "{\opn{mult}_{B' / B} }"]
\ar[d, "{\opn{id}_{B'} \ot_{\opn{id}_{B}} \, u'}"]
&
B'
\ar[d, "{u'}"]
\\
C'
\ar[rr, bend right = 15, start anchor = south east, end anchor = south west,
"{\opn{id}_{C'}}"']
\ar[r, "{\opn{sec}_{B' / B} \ot \opn{id}_C}"]
&
B' \ot_B C'
\ar[r, "{\opn{mult}_{B' / B}  \ot \opn{id}_C}"]
&
C'
\end{tikzcd} 
\end{equation}
in $\cat{M}(B' \ot_B B')$. Moreover, the bottom curved region of 
(\ref{eqn:1456}) is a commutative diagram in $\cat{M}(B' \ot_B C')$.

Next we patch the bottom curved region of diagram (\ref{eqn:1456}), with the  
``$C$ version'' of diagram (\ref{eqn:1455}), both flipped vertically.
This gives us the following solid commutative diagram 
\begin{equation} \label{eqn:1457}
\begin{tikzcd} [column sep = 14ex, row sep = 5.5ex] 
C'
\ar[rr, bend left = 18, start anchor = north east, end anchor = north west,
"{\opn{id}_{C'}}"]
\ar[r, "{\opn{sec}_{B' / B} \ot \opn{id}_C}"]
\ar[d, "{\opn{id}_{C'}}"]
\ar[dr, dashed, "{\tau}"]
&
B' \ot_B C'
\ar[r, "{\opn{mult}_{B' / B}  \ot \opn{id}_C}"]
\ar[d, dashed, "{u' \ot_u \opn{id}_{C'}}"]
&
C'
\ar[d, "{\opn{id}_{C'}}"]
\\
C'
\ar[rr, bend right = 15, start anchor = south east, end anchor = south west,
"{\opn{id}_{C'}}"']
\ar[r, "{\opn{sec}_{C' / C}}"]
&
C' \ot_C C'
\ar[r, "{\opn{mult}_{C' / C}}"]
&
C'
\end{tikzcd} 
\end{equation}
in $\cat{M}(B' \ot_B C')$. 
We insert two dashed arrows into diagram (\ref{eqn:1457}): the 
$(B' \ot_B C')$-module homomorphisms 
$\tau := (u' \ot_u \opn{id}_{C'}) \circ (\opn{sec}_{B' / B} \ot \opn{id}_C)$
and $u' \ot_u \opn{id}_{C'}$. 
By definition, the middle triangular region in (\ref{eqn:1457}) is 
commutative. A calculation with elements shows that the right rectangular 
region 
in (\ref{eqn:1457}) is also commutative. 

We claim that the homomorphism $\tau$ is $(C' \ot_C C')$-linear. This is shown 
by a calculation. Take some elements $c' \in C'$ and 
$c'_1 \ot c'_2 \in C' \ot_C C'$. We need to prove that 
\begin{equation} \label{eqn:1460}
(c'_1 \ot c'_2) \cd \tau(c') = \tau(c'_1 \cd c'_2 \cd c') 
\in C' \ot_C C' .
\end{equation}
We can assume that $c'_i \in C' = C \ot_B B'$ are $c'_i = c_i \cd b'_i$, for suitable 
$c_i \in C$ and $b'_i \in B'$. (The ring homomorphisms $w$ and $u'$ are kept 
implicit here, and we can replace tensors with pure tensors.) Then
\begin{equation} \label{eqn:1461}
(c'_1 \ot c'_2) \cd \tau(c') = 
(b'_1 \ot b'_2) \cd (c_1 \ot c_2) \cd \tau(c') .
\end{equation}
But $\tau(c') \in C' \ot_C C'$, so we have 
\begin{equation} \label{eqn:1462}
(c_1 \ot c_2) \cd \tau(c') = (1 \ot (c_1 \cd c_2)) \cd \tau(c')\in C'\ot_C C' .
\end{equation}
The homomorphism $\tau$ is $(B' \ot_B C')$-linear; therefore, using formulas 
(\ref{eqn:1461}) and (\ref{eqn:1462}), we get 
\[ (c'_1 \ot c'_2) \cd \tau(c') = 
(b'_1 \ot (b'_2 \cd c_1 \cd c_2)) \cd \tau(c') = 
\tau(b'_1 \cd b'_2 \cd c_1 \cd c_2 \cd c') = 
\tau((c'_1 \ot c'_2) \cd c')) . \]
This is the equality (\ref{eqn:1460}). 

After all these calculations we now know that $\tau$ is $(C' \ot_C C')$-linear.
The commutativity of the middle triangle and the rectangle in diagram 
(\ref{eqn:1457}) tells us that 
$\opn{mult}_{C' / C} \circ \, \tau = \opn{id}_{C'}$.
The uniqueness in Corollary \ref{cor:1246}(1) implies that 
$\tau = \opn{sec}_{C' / C}$. Thus the whole diagram (\ref{eqn:1457}), dashed 
arrows included, is commutative. Combining that with the commutativity of 
diagram (\ref{eqn:1456}), we conclude that diagram ($\dag$) is commutative.
\end{proof}

\begin{proof}[Proof of Theorem \tup{\ref{thm:895}}]
Let us expand the third diagram in (\ref{eqn:1565}) using Definition 
\ref{dfn:1290}: 
\begin{equation} \label{eqn:1463}
\begin{tikzcd} [column sep = 10ex, row sep = 5ex] 
\opn{Sq}_{B / A}(M)
\ar[d, "{\opn{Sq}_{B / A}(\opn{q}^{\mrm{L}}_{v, M})}"']
&
\opn{Sq}_{C / A}(N)
\ar[l, "{\opn{Sq}_{u / A}(\th)}"']
\ar[d, "{\opn{Sq}_{C / A}(\opn{q}^{\mrm{L}}_{w, N})}"]
\\
\opn{Sq}_{B / A}(M')
\ar[d, "{\opn{red}_{B' / B / A, M}}"']
&
\opn{Sq}_{C / A}(N') 
\ar[l, "{\opn{Sq}_{u / A}(\th')}"']
\ar[d, "{\opn{red}_{C' / C / A, N}}"]
\\
\opn{Sq}_{B' / A}(M')
&
\opn{Sq}_{C' / A}(N') 
\ar[l, "{\opn{Sq}_{u' / A}(\th')}"']
\end{tikzcd} 
\end{equation}
Consider the first diagram below, which is a commutative diagram in 
$\cat{Rng} \eftover A$~:
\begin{equation} \label{eqn:1464}
\begin{tikzcd} [column sep = 8ex, row sep = 5ex] 
B
\ar[r, "{u}"]
\ar[d, "{\opn{id}_B}"']
&
C
\ar[d, "{\opn{id}_C}"']
\\
B
\ar[r, "{u}"]
&
C
\end{tikzcd} 
\qquad
\begin{tikzcd} [column sep = 8ex, row sep = 5ex] 
M
\ar[d, "{\opn{q}^{\mrm{L}}_{v, M}}"']
&
N
\ar[l, "{\th}"']
\ar[d, "{\opn{q}^{\mrm{L}}_{w, N}}"]
\ar[dl, dashed]
\\
M'
&
N' 
\ar[l, "{\th'}"']
\end{tikzcd} 
\end{equation}
The solid part of the second diagram in (\ref{eqn:1464}) is just the second 
diagram in (\ref{eqn:1565});
it is a commutative diagram of backward morphisms in 
$\cat{D}(B)$, with respect to the diagram of rings in the first diagram. 
We add to it the dashed arrow that retains the commutativity.
Applying Proposition \ref{prop:890} twice to the second diagram in 
(\ref{eqn:1464}), we see that the top 
rectangular region in diagram (\ref{eqn:1463}) is commutative. 

To treat the bottom rectangular region in diagram (\ref{eqn:1463}), we will 
choose K-flat resolutions 
$\til{B} / \til{A} \to B / A$,
$\til{C} / \til{B} \to C / B$ and
$\til{B}' / \til{B} \to B' / B$
of of these pairs of DG rings. Then we let 
$\til{C}' := \til{C} \ot_{\til{B}} \til{B}'$. So there is a commutative diagram 
\begin{equation} \label{eqn:1465}
\begin{tikzcd} [column sep = 8.5ex, row sep = 4ex] 
\til{B}
\ar[r, "{\til{u}}"]
\ar[d, "{\til{v}}"']
&
\til{C}
\ar[d, "{\til{w}}"]
\\
\til{B'}
\ar[r, "{\til{u}'}"]
&
\til{C'} 
\end{tikzcd} 
\end{equation}
in $\cat{DGRng} \over \til{A}$, in which all homomorphisms (including the 
structural homomorphisms from $\til{A}$) are K-flat. 
The bottom rectangular region in diagram (\ref{eqn:1463}) becomes, with these 
resolutions, isomorphic to the outer boundary of the big diagram 
\begin{equation} \label{eqn:1466}
\begin{tikzcd} [column sep = 8ex, row sep = 4ex] 
\opn{RHom}_{\til{B} \ot_{\til{A}} \til{B}}(B, M' \ot^{\mrm{L}}_{\til{A}} M')
\ar[d, "{\opn{adj}}", "{\simeq}"']
&
\opn{RHom}_{\til{C} \ot_{\til{A}} \til{C}}(C, N' \ot^{\mrm{L}}_{\til{A}} N')
\ar[l, "{\opn{RHom}_{}(u, \th' \ot^{\mrm{L}} \th')}"']
\ar[d, "{\opn{adj}}"', "{\simeq}"]
\\
\opn{RHom}_{\til{B}' \ot_{\til{A}} \til{B}'}(B' \ot_B B', 
M' \ot^{\mrm{L}}_{\til{A}} M')
\ar[d, "{\opn{RHom}(\opn{sec}_{B' / B} , \opn{id})}"]
&
\opn{RHom}_{\til{C}' \ot_{\til{A}} \til{C}'}(C' \ot_C C', 
N' \ot^{\mrm{L}}_{\til{A}} N')
\ar[l, "{\al}"']
\ar[d, "{\opn{RHom}(\opn{sec}_{C' / C} , \opn{id})}"']
\\
\opn{RHom}_{\til{B}' \ot_{\til{A}} \til{B}'}(B', 
M' \ot^{\mrm{L}}_{\til{A}} M')
&
\opn{RHom}_{\til{C}' \ot_{\til{A}} \til{C}'}(C', 
N' \ot^{\mrm{L}}_{\til{A}} N')
\ar[l, "{\opn{RHom}_{}(u', \th' \ot^{\mrm{L}} \th')}"']
\end{tikzcd}
\end{equation}
where
$\al := \opn{RHom}_{}(u' \ot_u u', \th' \ot^{\mrm{L}} \th')$,
and the isomorphisms marked ``$\opn{adj}$'' come from DG ring adjunctions and 
quasi-isomorphisms, like the isomorphism marked
$\iso^{\mrm{(2)}}$ in formula (\ref{eqn:1290}) in Definition \ref{dfn:1282}, 
but with the required changes. 
The top rectangular region in diagram (\ref{eqn:1466}) is commutative, since 
the adjunctions and quasi-isomorphisms involved commute. The bottom rectangular 
region in diagram (\ref{eqn:1466}) is commutative due to Lemma \ref{lem:1455}.
We conclude that diagram (\ref{eqn:1466}) is commutative; and therefore so is 
the bottom rectangular region in diagram (\ref{eqn:1463}).

We have thus proved that diagram (\ref{eqn:1463}) is commutative. Therefore the 
third diagram in (\ref{eqn:1565}) is commutative. 
\end{proof}

\section{Induced Rigidity and Rigid Forward Morphisms}
\label{sec:induced-rigidity}

Suppose we are given a noetherian ring $A$ and an essentially \'etale 
homomorphism $v : B \to C$ in $\cat{Rng} \eftover A$; 
this is Setup \ref{set:1281}. Let 
$(M, \rho) \in \cat{D}(B)_{\mrm{rig} / A}$.
The purpose of this section is to endow the complex 
$N := C \ot_B M \in \cat{D}(C)$
with the {\em induced rigidifying isomorphism} 
$\si := \opn{Ind}^{\mrm{rig}}_{v / A}(\rho)$,
and to prove that the standard nondegenerate forward morphism 
$\opn{q}^{\mrm{L}}_{v, M} : (M, \rho) \to (N, \si)$
is a {\em rigid forward morphism} over $v / A$. 
See Definitions \ref{dfn:1300} and \ref{dfn:1293} and Proposition 
\ref{prop:1700}. 

Throughout this section we adhere to Conventions \ref{conv:615} and 
\ref{conv:1070}, so by default all rings are commutative noetherian, and all 
ring homomorphisms are EFT. Moreover, we shall mostly work under the 
assumptions of Setup \ref{set:1281}.

The next theorem is the key to induced rigidity. 

\begin{thm}[Nondegeneracy of forward squaring] \label{thm:1300}
Under Setup \ref{set:1281}, let $M \in \cat{D}(B)$, let $N \in \cat{D}(C)$, 
and let $\la : M \to N$ be a nondegenerate forward morphism in $\cat{D}(B)$ 
over $v$. Assume that $M$ has finite flat dimension over $A$. Then 
\[  \opn{Sq}_{v / A}(\la) : \opn{Sq}_{B / A}(M) \to \opn{Sq}_{C / A}(N) \]
is a nondegenerate forward morphism in $\cat{D}(B)$ over $v$. 
\end{thm}

We need a few lemmas before giving the proof of the theorem. 
Fix some K-flat DG ring resolution (\ref{eqn:1280}) of the triple 
of rings $C / B / A$. This choice will be used in the next lemmas and in the 
proof of the theorem.

According to  Corollary \ref{cor:1245}(2), the ring homomorphism 
$\opn{mult}_{C / B} : C \ot_B C \to C$ is a localization, and in particular it 
is flat. Given a complex $K \in \cat{D}(C \ot_{B} C)$ we have the standard 
nondegenerate forward morphism 
\[ \opn{q}^{\mrm{L}}_{\opn{mult}_{C / B}, K} = 
\opn{q}^{\mrm{L}}_{C / (C \ot_B C), K}  
: K \to K \ot_{C \ot_{B} C} C \]
in $\cat{D}(C \ot_{B} C)$ over $\opn{mult}_{C / B}$;
cf.\ (\ref{eqn:1293}). 

In what follows we are going to use Convention \ref{conv:1295} regarding the 
numbering of copies of DG rings and modules appearing in tensor products. 

\begin{lem} \label{lem:1310}
Given a DG module $L \in \cat{D}(\til{C}_1 \ot_{\til{A}} \til{C}_2)$, there is 
a unique isomorphism 
\[ \phi_L : 
\opn{RHom}_{\til{C}_1 \ot_{\til{A}} \til{C}_2}(C_0, L) \iso 
\opn{RHom}_{\til{C}_1 \ot_{\til{A}} \til{C}_2}(C_1 \ot_{B_0} C_2, L) 
\ot_{C_1 \ot_{B_0} C_2} C_0  \]
in $\cat{D}(C_1 \ot_{B_0} C_2)$, which is functorial in $L$, and makes the 
diagram 
\[ \begin{tikzcd} [column sep = 8ex, row sep = 6ex] 
\opn{RHom}_{\til{C}_1 \ot_{\til{A}} \til{C}_2}(C_1 \ot_{B_0} C_2, L) 
\ar[d, "{\opn{RHom}(\opn{sec}_{C / B}, \opn{id})}"']
\ar[dr, "{\opn{q}^{\mrm{L}}_{\opn{mult}_{C / B}, K}}"]
\\
\opn{RHom}_{\til{C}_1 \ot_{\til{A}} \til{C}_2}(C_0, L)
\ar[r, "{\phi_L}", "{\simeq}"']
&
\opn{RHom}_{\til{C}_1 \ot_{\til{A}} \til{C}_2}(C_1 \ot_{B_0} C_2, L) 
\ot_{C_1 \ot_{B_0} C_2} C_0
\end{tikzcd} \]
commutative. Here  
\[ \tag{$\star$} 
K := \opn{RHom}_{\til{C}_1 \ot_{\til{A}} \til{C}_2}(C_1 \ot_{B_0} C_2, L)
\in \cat{D}(C_1 \ot_{B_0} C_2) . \]
\end{lem}

\begin{proof}
By Corollary \ref{cor:1280} there is a decomposition
$C_1 \ot_B C_2 = C^{\mrm{diag}} \oplus C^{\mrm{od}}$ of the ring 
$C_1 \ot_B C_2 = C_1 \ot_{B_0} C_2$ into a direct sum of 
ideals, and these ideals are idempotent and perpendicular to each other. In 
fact the ideals $C^{\mrm{diag}}, C^{\mrm{od}} \sub C \ot_B C$ are generated idempotent 
elements $e, e^{\mrm{od}}$ respectively, which satisfy
$e + e^{\mrm{od}} = 1$ and $e \cd e^{\mrm{od}} = 0$.
This implies that when we view $C^{\mrm{diag}}$ and $C^{\mrm{od}}$ as projective 
$(C \ot_B C)$-modules, for every complex 
$K' \in \cat{D}(C \ot_{B} C)$
there is a decomposition
\begin{equation} \label{eqn:1421}
K' \cong (K' \ot_{C \ot_{B} C} C) \oplus 
(K' \ot_{C \ot_{B} C} C^{\mrm{od}}) .
\end{equation}
Furthermore $C \ot_{C \ot_{B} C} C = C$ 
and $C^{\mrm{od}} \ot_{C \ot_{B} C} C = 0$. 

Let's apply the decomposition (\ref{eqn:1421}) to the complex $K' := K$ from 
formula ($\star$), remembering that $C_1 \ot_{B_0} C_2$ acts on $K$ through 
itself (sitting in the first argument of $\opn{RHom}$). We get 
\[ \begin{aligned}
&
\opn{RHom}_{\til{C}_1 \ot_{\til{A}} \til{C}_2}(C, L) = 
\opn{RHom}_{\til{C}_1 \ot_{\til{A}} \til{C}_2} (C^{\mrm{diag}} \oplus C^{\mrm{od}}, L) 
\\
& \quad 
\cong \opn{RHom}_{\til{C}_1 \ot_{\til{A}} \til{C}_2}(C^{\mrm{diag}}, L) \oplus
\opn{RHom}_{\til{C}_1 \ot_{\til{A}} \til{C}_2}(C^{\mrm{od}}, L) . 
\end{aligned} \]
The operation 
$(-) \ot_{C_1 \ot_{B_0} C_2} C$ kills the second direct summand above, and 
retains the first direct summand unchanged. But this is precisely what the 
operation 
$\opn{RHom}(\opn{sec}_{C / B}, \opn{id}_L)$ does.

The uniqueness of $\phi_L$ is because 
$\opn{q}^{\mrm{L}}_{\opn{mult}_{C / B}, K}$
is a nondegenerate forward morphism. 
\end{proof}

\begin{lem} \label{lem:1340}
There is an isomorphism 
\[ \tag{\dag} 
\begin{aligned}
&
\opn{RHom}_{\til{B}_1 \ot_{\til{A}} \til{B}_2} 
(B_0, M_1 \ot_{\til{A}}^{\mrm{L}} M_2) \ot^{\mrm{L}}_{B_0} C_0
\\
& \quad \iso  
\bigl( \opn{RHom}_{\til{B}_1 \ot_{\til{A}} \til{B}_2} 
(B_0, M_1 \ot_{\til{A}}^{\mrm{L}} M_2)
\ot^{\mrm{L}}_{\til{B}_1 \ot_{\til{A}} \til{B}_2} 
(\til{C}_1 \ot_{\til{A}} \til{C}_2) \bigr) 
\ot^{\mrm{L}}_{\til{D}} C_0 \end{aligned} \]
in $\cat{D}(\til{D})$.
Here $\til{D}$ acts on the first object in ($\dag$) via the obvious DG ring 
homomorphism 
$\til{D} \to C_0 \ot_{\til{B}_1 \ot_{\til{A}} \til{B}_2} B_0$
and the actions of $B_0$ and $C_0$ on themselves; and $\til{D}$ acts on the 
second object  in ($\dag$) via $B_0$ and $\til{C}_1 \ot_{\til{A}} \til{C}_2$.
Moreover, after passing to $\cat{D}(B)$ by the restriction 
functor $\opn{Rest}_{\til{D} / B}$, 
the isomorphism \tup{($\dag$)} commutes with the morphisms from
$\opn{Sq}_{B / A}^{\til{B} / \til{A}}(M)$. 
\end{lem}

\begin{proof}
There is a commutative diagram of DG rings 
\begin{equation} \label{eqn:1450}
\begin{tikzcd} [column sep = 8ex, row sep = 5ex] 
B_0
\ar[rrr, bend left = 18, start anchor = north east, end anchor = north west,
"{v}"]
\ar[r]
&
\til{D}
\ar[r, "{g}"]
&
C_1 \ot_{B_0} C_2
\ar[r, "{\opn{mult}_{C / B} \, }"]
&
C_0
\end{tikzcd} 
\end{equation}
This tells us that the homomorphism $v : B \to C$ factors through $\til{D}$. 

Consider the sequence of isomorphisms 
\begin{equation} \label{eqn:1426}
\begin{aligned}
&
\opn{RHom}_{\til{B}_1 \ot_{\til{A}} \til{B}_2} 
(B_0, M_1 \ot_{\til{A}}^{\mrm{L}} M_2) \ot^{\mrm{L}}_{B_0} C_0
\\ 
& \quad \iso^{\mrm{(1)}}
\opn{RHom}_{\til{B}_1 \ot_{\til{A}} \til{B}_2} 
(B_0, M_1 \ot_{\til{A}}^{\mrm{L}} M_2) \ot^{}_{B_0} 
\til{D}_0 \ot^{\mrm{L}}_{\til{D}_0} C_0 
\\ 
& \quad \iso^{\mrm{(2)}}
\bigl( \opn{RHom}_{\til{B}_1 \ot_{\til{A}} \til{B}_2} 
(B_0, M_1 \ot_{\til{A}}^{\mrm{L}} M_2) \ot^{}_{B_0} 
\bigl( B_0 \ot^{}_{\til{B}_1 \ot_{\til{A}} \til{B}_2} 
(\til{C}_1 \ot_{\til{A}} \til{C}_2) \bigr) \bigr)
\ot^{\mrm{L}}_{\til{D}_0} C_0
\\ 
& \quad \iso^{\mrm{(3)}}
\bigl( \opn{RHom}_{\til{B}_1 \ot_{\til{A}} \til{B}_2} 
(B_0, M_1 \ot_{\til{A}}^{\mrm{L}} M_2) 
\ot^{}_{\til{B}_1 \ot_{\til{A}} \til{B}_2} 
(\til{C}_1 \ot_{\til{A}} \til{C}_2) \bigr)
\ot^{\mrm{L}}_{\til{D}_0} C_0 
\\ 
& \quad \iso^{\mrm{(4)}}
\bigl( \opn{RHom}_{\til{B}_1 \ot_{\til{A}} \til{B}_2} 
(B_0, M_1 \ot_{\til{A}}^{\mrm{L}} M_2) 
\ot^{\mrm{L}}_{\til{B}_1 \ot_{\til{A}} \til{B}_2} 
(\til{C}_1 \ot_{\til{A}} \til{C}_2) \bigr)
\ot^{\mrm{L}}_{\til{D}_0} C_0 
\end{aligned}
\end{equation}
in $\cat{D}(\til{D})$. 
The isomorphism $\iso^{\mrm{(1)}}$ is because the homomorphism $v : B \to C$ 
factors through $\til{D}$. The operation 
$\ot^{\mrm{L}}_{B_0} \til{D}$ can be replaced by $\ot^{}_{B_0} \til{D}$, 
because $\til{D}_0$ is K-flat over $B_0$. 
The isomorphism $\iso^{\mrm{(2)}}$ consists of expanding the DG ring 
$\til{D}$ according to its definition (see formula (\ref{eqn:1320})), with a 
permutation of the tensor factors. 
In the isomorphism $\iso^{\mrm{(3)}}$ we cancel the term 
$\ot^{}_{B_0} B_0$. Finally, in isomorphism $\iso^{\mrm{(4)}}$ 
we replace the operation 
$\ot^{}_{\til{B}_1 \ot_{\til{A}} \til{B}_2} 
(\til{C}_1 \ot_{\til{A}} \til{C}_2)$
with 
$\ot^{\mrm{L}}_{\til{B}_1 \ot_{\til{A}} \til{B}_2} 
(\til{C}_1 \ot_{\til{A}} \til{C}_2)$,
and this is permitted because $\til{C}_1 \ot_{\til{A}} \til{C}_2$ is K-flat 
over $\til{B}_1 \ot_{\til{A}} \til{B}_2$. 
The isomorphism ($\dag$) is defined to be the composition of the isomorphisms 
in (\ref{eqn:1426}). 

It is plain to see that the isomorphisms in (\ref{eqn:1426}) respect the 
morphisms from 
\[ \opn{Sq}_{B / A}^{\til{B} / \til{A}}(M) = 
\opn{RHom}_{\til{B} \ot_{\til{A}} \til{B}}
(B_0, M_1 \ot^{\mrm{L}}_{\til{A}} M_2) . \qedhere \]
\end{proof}

\begin{lem} \label{lem:1341}
There is an isomorphism 
\[ \tag{\ddag} 
\begin{aligned}
& 
\opn{RHom}_{\til{B}_1 \ot_{\til{A}} \til{B}_2} 
(B_0, M_1 \ot_{\til{A}}^{\mrm{L}} M_2)
\ot^{\mrm{L}}_{\til{B}_1 \ot_{\til{A}} \til{B}_2} (\til{C}_1 \ot_{\til{A}} 
\til{C}_2)
\\
& \quad 
\iso \opn{RHom}_{\til{B}_1 \ot_{\til{A}} \til{B}_2} 
(B_0, N'_1 \ot_{\til{A}}^{\mrm{L}} N'_2) 
\end{aligned} \]
in $\cat{D}(\til{D})$.
Here $N' := C \ot^{\mrm{L}}_{B} M \in \cat{D}(C)$. The DG ring 
$\til{D}$ acts on the first object via $B_0$ and 
$\til{C}_1 \ot_{\til{A}} \til{C}_2$, and on the second object via $B_0$ and 
$N'_1 \ot_{\til{A}}^{\mrm{L}} N'_2$.
Moreover, after passing to $\cat{D}(B)$ by the restriction functor 
$\opn{Rest}_{\til{D} / B_0}$, 
the isomorphism \tup{($\ddag$)} commutes with the morphisms from
$\opn{Sq}_{B / A}^{\til{B} / \til{A}}(M)$. 
\end{lem}

\begin{proof}
The morphism  (\ddag) is the composition of the derived tensor-evaluation 
morphism 
\[ \begin{aligned}
& 
\opn{ev}^{\mrm{R, L}} : 
\opn{RHom}_{\til{B}_1 \ot_{\til{A}} \til{B}_2} 
(B_0, M_1 \ot_{\til{A}}^{\mrm{L}} M_2)
\ot^{\mrm{L}}_{\til{B}_1 \ot_{\til{A}} \til{B}_2} (\til{C}_1 \ot_{\til{A}} 
\til{C}_2)
\\
& \quad \to
\opn{RHom}_{\til{B}_1 \ot_{\til{A}} \til{B}_2} 
\bigl( B_0, (M_1 \ot_{\til{A}}^{\mrm{L}} M_2)
\ot^{\mrm{L}}_{\til{B}_1 
\ot_{\til{A}} \til{B}_2} (\til{C}_1 \ot_{\til{A}} \til{C}_2) \bigr)
\end{aligned} \]
from \cite[Theorem 12.10.14]{Ye5}
with the application of 
$\opn{RHom}_{\til{B}_1 \ot_{\til{A}} \til{B}_2}(B_0, -)$
to the obvious isomorphisms 
\[ \begin{aligned}
&
(M_1 \ot_{\til{A}}^{\mrm{L}} M_2)
\ot^{\mrm{L}}_{\til{B}_1 
\ot_{\til{A}} \til{B}_2} (\til{C}_1 \ot_{\til{A}} \til{C}_2)
\\ & \quad 
\iso (M_1 \ot_{\til{B}_1}^{\mrm{L}} \til{C}_1) \ot^{\mrm{L}}_{\til{A}} 
(M_2 \ot_{\til{B}_2}^{\mrm{L}} \til{C}_2)
\iso N'_1 \ot_{\til{A}}^{\mrm{L}} N'_2
\end{aligned} \]
in $\cat{D}(\til{B}_1 \ot_{\til{A}} \til{B}_2)$. 

It is clear that the restriction of the morphism ($\ddag$) to $\cat{D}(B)$ 
commutes with the morphisms from $\opn{Sq}_{B / A}^{\til{B} / \til{A}}(M)$. 

To prove that ($\ddag$) is an isomorphism, we can pass to the category 
$\cat{D}(\Z)$ using the restriction functor. 
There is an obvious isomorphism 
\begin{equation} \label{eqn:1451}
\begin{aligned}
&
\opn{RHom}_{\til{B}_1 \ot_{\til{A}} \til{B}_2} 
(B_0, M_1 \ot_{\til{A}}^{\mrm{L}} M_2)
\ot^{\mrm{L}}_{\til{B}_1 \ot_{\til{A}} \til{B}_2} (\til{C}_1 \ot_{\til{A}} 
\til{C}_2)
\\ & \quad 
\iso \bigl( \opn{RHom}_{\til{B}_1 \ot_{\til{A}} \til{B}_2} 
(B_0, M_1 \ot_{\til{A}}^{\mrm{L}} M_2)
\ot^{\mrm{L}}_{\til{B}_1} C_1 \bigr) \ot_{\til{B}_2} C_2
\end{aligned}
\end{equation}
in $\cat{D}(\Z)$; see formula (\ref{eqn:1486}). 
The arguments used in the proof of Lemma \ref{lem:1436} 
can be used here as well, 
to show that the two derived tensor-evaluation morphisms 
\[ \begin{aligned}
&
\bigl( \opn{RHom}_{\til{B}_1 \ot_{\til{A}} \til{B}_2} 
(B_0, M_1 \ot_{\til{A}}^{\mrm{L}} M_2)
\ot^{\mrm{L}}_{\til{B}_1} C_1 \bigr) \ot_{\til{B}_2} C_2
\\ & \quad 
\to 
\opn{RHom}_{\til{B}_1 \ot_{\til{A}} \til{B}_2} 
\bigl( B_0, (M_1 \ot^{\mrm{L}}_{\til{B}_1} C_1) \ot_{\til{A}}^{\mrm{L}} M_2 
\bigr) \ot_{\til{B}_2} C_2
\\ & \quad 
\to 
\opn{RHom}_{\til{B}_1 \ot_{\til{A}} \til{B}_2} 
\bigl( B_0, (M_1 \ot_{\til{B}_1} C_1)
\ot_{\til{A}}^{\mrm{L}} (M_2 \ot^{\mrm{L}}_{\til{B}_2} C_2) \bigr) 
\end{aligned} \]
are isomorphisms in $\cat{D}(\Z)$.
\end{proof}

\begin{proof}[Proof of Theorem \tup{\ref{thm:1300}}]
Define the morphism
\[ \opn{Sq}_{C / B / A}^{\til{C} / \til{B} / \til{A}}(\la) := 
\opn{red}_{C / B / A, N}^{\til{C} / \til{B} / \til{A}} \circ 
\opn{Sq}_{B / A}^{\til{B} / \til{A}}(\la) : 
\opn{Sq}_{B / A}^{\til{B} / \til{A}}(M) \to 
\opn{Sq}_{C / A}^{\til{C} / \til{A}}(N) \]
in $\cat{D}(B)$. It suffices to prove that 
$\opn{Sq}_{C / B / A}^{\til{C} / \til{B} / \til{A}}(\la)$
is a nondegenerate forward morphism over $v = C / B$. 
Because $\la$ is a nondegenerate forward morphism, we may assume that 
$N = M \ot_B C$ and $\la = \opn{q}^{\mrm{L}}_{C / B, M}$;
cf.\ diagram (\ref{eqn:1321}). 
Thus what we need to prove is that the forward morphism
\begin{equation} \label{eqn:1350}
\opn{Sq}_{C / B / A}^{\til{C} / \til{B} / \til{A}} 
(\opn{q}^{\mrm{L}}_{C / B, M}) : 
\opn{Sq}_{B / A}^{\til{B} / \til{A}}(M) \to 
\opn{Sq}_{C / A}^{\til{C} / \til{A}}(M \ot^{\mrm{L}}_B C)
\end{equation}
in $\cat{D}(C)$ over $C / B$ is nondegenerate. 

Consider the following sequence of isomorphisms in $\cat{D}(C)$, in which the 
DG ring $\til{D}$ from (\ref{eqn:1320}) is used. 
\begin{equation} \label{eqn:1315}
\begin{aligned}
& 
\opn{Sq}_{B / A}^{\til{B} / \til{A}}(M) \ot^{\mrm{L}}_{B_0} C_0
= \opn{RHom}_{\til{B}_1 \ot_{\til{A}} \til{B}_2}
\bigl( B_0, M_1 \ot_{\til{A}}^{\mrm{L}} M_2 \bigr) \ot^{\mrm{L}}_{B_0} C_0
\\
& \quad 
\iso^{\mrm{(1)}} 
\opn{RHom}_{\til{B}_1 \ot_{\til{A}} \til{B}_2} 
\bigl( B_0, M_1 \ot_{\til{A}}^{\mrm{L}} M_2 \bigr) 
\ot^{\mrm{L}}_{\til{B}_1 \ot_{\til{A}} \til{B}_2} (\til{C}_1 \ot_{\til{A}} 
\til{C}_2)
\ot^{\mrm{L}}_{\til{D}_0} C_0
\\
& \quad 
\iso^{\mrm{(2)}} 
\opn{RHom}_{\til{B}_1 \ot_{\til{A}} \til{B}_2} 
(B_0, N_1 \ot_{\til{A}}^{\mrm{L}} N_2) \ot_{\til{D}_0}^{\mrm{L}} C_0
\\
& \quad 
\iso^{\mrm{(3)}} 
\opn{RHom}_{\til{C}_1 \ot_{\til{A}} \til{C}_2} 
\bigl( (\til{C}_1 \ot_{\til{A}} \til{C}_2) \ot_{\til{B}_1 \ot_{\til{A}} 
\til{B}_2} B_0,
N_1 \ot_{\til{A}}^{\mrm{L}} N_2 \bigr) \ot_{\til{D}_0}^{\mrm{L}} C_0
\\
& \quad 
\iso^{\mrm{(4)}} 
\opn{RHom}_{\til{C}_1 \ot_{\til{A}} \til{C}_2} 
\bigl( C_1 \ot_{B_0} C_2, N_1 \ot_{\til{A}}^{\mrm{L}} N_2 \bigr) 
\ot_{\til{D}_0}^{\mrm{L}} C_0
\\
& \quad 
\iso^{\mrm{(5)}} 
\opn{RHom}_{\til{C}_1 \ot_{\til{A}} \til{C}_2} 
\bigl( C_1 \ot_{B_0} C_2, N_1 \ot_{\til{A}}^{\mrm{L}} N_2 \bigr) 
\ot^{\mrm{L}}_{C_1 \ot_{B_0} C_2} C_0
\\
& \quad 
\iso^{\mrm{(6)}} 
\opn{RHom}_{\til{C}_1 \ot_{\til{A}} \til{C}_2} 
(C_0, N_1 \ot_{\til{A}}^{\mrm{L}} N_2) 
= \opn{Sq}_{C / A}^{\til{C} / \til{A}}(N) . 
\end{aligned} 
\end{equation}
The action of the ring $C$ on these objects is through $C_0$. 
Here are the explanations of the various isomorphisms.
The isomorphism $\iso^{\mrm{(1)}}$ is due to Lemma \ref{lem:1340}.
The isomorphism $\iso^{\mrm{(2)}}$ is due to Lemma \ref{lem:1341}.
$\iso^{\mrm{(3)}}$ comes from Hom-tensor
adjunction for the DG ring homomorphism 
$\til{B}_1 \ot_{\til{A}} \til{B}_2 \to \til{C}_1 \ot_{\til{A}} \til{C}_2$. 
For $\iso^{\mrm{(4)}}$ we use the DG ring quasi-isomorphism from 
Lemma \ref{lem:1315}.
The isomorphism $\iso^{\mrm{(5)}}$ uses the DG ring quasi-isomorphism from 
Lemma \ref{lem:1315}, and formula (\ref{eqn:719}). 
Lastly, $\iso^{\mrm{(6)}}$ is the inverse of the isomorphism 
$\phi_L$ in Lemma \ref{lem:1310}, for 
$L := N_1 \ot_{\til{A}}^{\mrm{L}} N_2$.

The object $\opn{Sq}_{B / A}^{\til{B} / \til{A}}(M) \in \cat{D}(B)$ has an 
obvious morphism in the category $\cat{D}(B)$ to each of the objects 
appearing in (\ref{eqn:1315}); and the isomorphisms in (\ref{eqn:1315}) commute 
with the morphisms from 
$\opn{Sq}_{B / A}^{\til{B} / \til{A}}(M)$. 
We obtain a commutative diagram 
\[ \begin{tikzcd} [column sep = 12ex, row sep = 5ex] 
\opn{Sq}_{B / A}^{\til{B} / \til{A}}(M)
\ar[r, "{\opn{q}^{\mrm{L}}_{C / B, \opn{Sq}_{B / A}^{\til{B} / \til{A}}(M)}}"]
\ar[dr, "{\opn{Sq}_{C / B / A}^{\til{C} / \til{B} / \til{A}} 
(\opn{q}^{\mrm{L}}_{C / B, M})}"']
&
\opn{Sq}_{B / A}^{\til{B} / \til{A}}(M) \ot^{\mrm{L}}_B C
\ar[d, "{\tup{(\ref{eqn:1315})}}", "{\simeq}"']
\\
&
\opn{Sq}_{C / A}^{\til{C} / \til{A}}(N) 
\end{tikzcd} \]
in $\cat{D}(B)$. Since the forward morphism
$\opn{q}^{\mrm{L}}_{C / B, \opn{Sq}_{B / A}^{\til{B} / \til{A}}(M)}$
is nondegenerate, the same holds for  
$\opn{Sq}_{C / B / A}^{\til{C} / \til{B} / \til{A}} 
(\opn{q}^{\mrm{L}}_{C / B, M})$. 
\end{proof}

Recall that the ring homomorphism $B \to C$ is flat, so 
$M \ot^{}_B C = M \ot^{L}_B C$. 
Because the standard forward morphism 
$\opn{q}^{}_{v, M} : M \to M \ot^{}_B C$ 
in $\cat{D}(B)$ over $v$ is nondegenerate, Theorem \ref{thm:1300} says that the 
forward morphism 
\[ \opn{Sq}_{v / A}(\opn{q}_{v, M}) : 
\opn{Sq}_{B / A}(M)  \to \opn{Sq}_{C / A}(M \ot^{}_B C) \]
is nondegenerate too. (It is just the morphism (\ref{eqn:1350}), written 
without 
resolutions.) 
This means that the morphism 
\begin{equation} \label{eqn:1352}
\opn{fadj}^{\mrm{L}}_{v} \bigl( \opn{Sq}_{v / A}
(\opn{q}^{}_{v, M}) \bigr) : 
\opn{Sq}_{B / A}(M) \ot^{}_B C \to 
\opn{Sq}_{C / A}(M \ot^{}_B C) 
\end{equation}
in $\cat{D}(C)$, which corresponds to
$\opn{Sq}_{v / A}(\opn{q}^{}_{v, M})$
by adjunction, is an isomorphism. 

\begin{dfn}[Induced Rigidity] \label{dfn:1300} 
Under Setup \ref{set:1281}, let $(M, \rho) \in \cat{D}(B)_{\mrm{rig} / A}$. 
Define the {\em induced rigidifying isomorphism} 
\[ \opn{Ind}^{\mrm{rig }}_{v / A}(\rho) = 
\opn{Ind}^{\mrm{rig }}_{C / B / A}(\rho) :
M \ot^{}_B C \iso \opn{Sq}_{C / A}(M \ot^{}_B C) \]
in $\cat{D}(C)$ to be
\[ \opn{Ind}^{\mrm{rig}}_{v / A}(\rho) := 
\opn{fadj}^{\mrm{L}}_{v} \bigl( \opn{Sq}_{v / A}
(\opn{q}^{}_{v, M}) \bigr) \circ (\rho \ot^{}_{B} \opn{id}_C) . \]
The rigid complex 
\[ \opn{Ind}^{\mrm{rig}}_{v / A}(M, \rho) := 
\bigl( M \ot^{}_B C, \, \opn{Ind}^{\mrm{rig}}_{v / A}(\rho) \bigr) 
\in \cat{D}(C)_{\mrm{rig} / A} \]
is called the {\em induced rigid complex}. 
\end{dfn}

Observe that the complex $M \ot^{}_B C$ belongs to 
$\cat{D}^{\mrm{b}}_{\mrm{f}}(C)$, and it has 
finite flat dimension over $A$; see Proposition \ref{prop:1270}(1).
Therefore $\opn{Ind}^{\mrm{rig}}_{v / A}(M, \rho)$ is really an object of 
$\cat{D}(C)_{\mrm{rig} / A}$. 
Here is Definition \ref{dfn:1300}, shown as a commutative diagram of 
isomorphisms in $\cat{D}(C)$~:
\begin{equation} \label{eqn:1700}
\begin{tikzcd} [column sep = 10ex, row sep = 5ex]
M \ot^{}_B C
\ar[r, "{\rho \, \ot^{}_{B} \, \opn{id}_C}", "{\simeq}"']
\ar[dr, "{\opn{Ind}^{\mrm{rig}}_{v / A}(\rho)}"', "{\simeq}"] 
&
\opn{Sq}_{B / A}(M) \ot^{}_B C  
\ar[d, "{\opn{fadj}^{\mrm{L}}_{v}(\opn{Sq}_{v / A}
(\opn{q}^{\mrm{L}}_{C / B, M}))}", , "{\simeq}"']
\\
&
\opn{Sq}_{C / A}(M \ot^{}_B C)
\end{tikzcd}
\end{equation}
See Question \ref{que:1660} below regarding how Definition \ref{dfn:1300}
relates to Definition \ref{dfn:1235}. 

\begin{prop} \label{prop:1300}
Under Setup \ref{set:1281},  
the formula 
$\opn{Ind}^{\mrm{rig}}_{C / B / A}(M, \rho)$ 
from Definition \ref{dfn:1300}
for an object $(M, \rho)$ in $\cat{D}(B)_{\mrm{rig} / A}$, and 
$\opn{Ind}^{\mrm{rig}}_{C / B / A}(\phi) := \opn{Ind}_v(\phi) 
= \phi \ot_B \opn{id}_C$
for a morphism 
$\phi : (M, \rho) \to (N, \si)$
in $\cat{D}(B)_{\mrm{rig} / A}$, 
is a functor
\[ \opn{Ind}^{\mrm{rig}}_{C / B / A} : 
\cat{D}(B)_{\mrm{rig} / A} \to \cat{D}(C)_{\mrm{rig} / A} . \]
\end{prop}

\begin{proof}
The proof relies on Theorem \ref{thm:895}, and therefore it is going 
to be more convenient to change notation to the setup of that theorem, but with 
$C = B$ and $u = \opn{id}_B$. Thus we are going to consider an essentially 
\'etale homomorphism of $A$-rings $v : B \to B'$, and a morphism 
$\th : (N, \si) \to (M, \rho)$
in $\cat{D}(B)_{\mrm{rig} / A}$.
We define the rigid complexes 
$(N', \si') := \opn{Ind}^{\mrm{rig}}_{v / A}(N, \si)$
and 
$(M', \rho')  := \opn{Ind}^{\mrm{rig}}_{v / A}(M, \rho)$
in $\cat{D}(B)_{\mrm{rig} / A}$, and the morphism 
$\th' := \opn{Ind}_v(\th) = \th \ot_B \opn{id}_C : N' \to  M'$.
See the first commutative diagram in (\ref{eqn:1875}) below. 
Our first goal is to prove that 
$\th' : (N', \si') \to (M', \rho')$
is a morphism in $\cat{D}(B')_{\mrm{rig} / A}$,
i.e.\ that it is a rigid morphism over $B' / A$.

Theorem \ref{thm:895} tells us that the second diagram in (\ref{eqn:1875})
is commutative. Indeed, the two diagrams below are precisely the second and 
third commutative diagrams in (\ref{eqn:1565}), but with $C = B$,  
$u = \opn{id}_B$ and $w = v$.  
\begin{equation} \label{eqn:1875}
\begin{tikzcd} [column sep = 6ex, row sep = 5ex] 
M
\ar[d, "{\opn{q}^{\mrm{L}}_{v, M}}"']
&
N
\ar[l, "{\th}"']
\ar[d, "{\opn{q}^{\mrm{L}}_{v, N}}"]
\\
M'
&
N' 
\ar[l, "{\th'}"']
\end{tikzcd}
\quad 
\begin{tikzcd} [column sep = 10ex, row sep = 5ex] 
\opn{Sq}_{B / A}(M)
\ar[d, "{\opn{Sq}_{v / A}(\opn{q}^{\mrm{L}}_{v, M})}"']
&
\opn{Sq}_{B / A}(N)
\ar[l, "{\opn{Sq}_{B / A}(\th)}"']
\ar[d, "{\opn{Sq}_{v / A}(\opn{q}^{\mrm{L}}_{v, N})}"]
\\
\opn{Sq}_{B' / A}(M')
&
\opn{Sq}_{B' / A}(N') 
\ar[l, "{\opn{Sq}_{B' / A}(\th')}"']
\end{tikzcd} 
\end{equation}

Let us examine the following diagram in $\cat{D}(B')$. 
The top square is commutative because $\th$ is a rigid morphism over $B / A$. 
The bottom square is obtained from the second diagram in (\ref{eqn:1875}) by 
forward adjunction, so it is commutative too. 
\begin{equation} \label{eqn:1876}
\begin{tikzcd}[column sep = 15ex, row sep = 5ex] 
M' = M \ot_B B'
\ar[d, "{\rho \lsp \ot_B \lsp \opn{id}_{B'}}"', , "{\simeq}"]
&
N' = N \ot_B B'
\ar[d, "{\si \lsp \ot_B \lsp \opn{id}_{B'}}", "{\simeq}"']
\ar[l, "{\th' \lsp = \lsp \th \lsp \ot_B \lsp \opn{id}_{B'}}"']
\\
\opn{Sq}_{B / A}(M)\ot_B B'
\ar[d, "{\opn{fadj}^{\mrm{L}}_{v}(\opn{Sq}_{v / A}
(\opn{q}^{\mrm{L}}_{v, M}))}"', "{\simeq}"]
&
\opn{Sq}_{B / A}(N)\ot_B B'
\ar[l, "{\opn{Sq}_{B / A}(\th) \lsp \ot_B \lsp \opn{id}_{B'}}"']
\ar[d, "{\opn{fadj}^{\mrm{L}}_{v}(\opn{Sq}_{v / A}
(\opn{q}^{\mrm{L}}_{v, N}))}", "{\simeq}"']
\\
\opn{Sq}_{B' / A}(M')
&
\opn{Sq}_{B' / A}(N')
\ar[l, "{\opn{Sq}_{B' / A}(\th')}"']
\end{tikzcd}
\end{equation}
We conclude that the outer square is  commutative. But by Definition 
\ref{dfn:1300} the composed vertical arrows in (\ref{eqn:1876}) are 
$\opn{Ind}^{\mrm{rig}}_{B' / B / A}(\rho)$ and 
$\opn{Ind}^{\mrm{rig}}_{B' / B / A}(\si)$, respectively.  
Thus $\th'$ is rigid, as claimed.

It remains to show that $\opn{Ind}^{\mrm{rig}}_{v / A}$ respects 
compositions and identity automorphism; but this is clear, since on morphisms 
we have by definition 
$\opn{Ind}^{\mrm{rig}}_{v / A}(\phi) = \opn{LInd}_v(\phi)$.
\end{proof}

\begin{dfn}[Rigid Forward Morphism] \label{dfn:1293} 
Under Setup \ref{set:1281}, let $(M, \rho) \in \cat{D}(B)_{\mrm{rig} / A}$ and 
$(N, \si) \in \cat{D}(C)_{\mrm{rig} / A}$. A {\em rigid forward morphism}
\[ \la : (M, \rho) \to (N, \si) \]
{\em over $v / A$}, or {\em over over $C / B / A$}, is a forward morphism 
$\la : M \to N$ in $\cat{D}(B)$ over $v$, such that the diagram 
\[ \tag{$*$}
\begin{tikzcd} [column sep = 8ex, row sep = 5ex] 
M
\ar[r, "{\rho}", "{\simeq}"']
\ar[d, "{\la}"']
&
\opn{Sq}_{B / A}(M)
\ar[d, "{\opn{Sq}_{v / A}(\la)}"]
\\
N
\ar[r, "{\si}", "{\simeq}"']
&
\opn{Sq}_{C / A}(N)
\end{tikzcd} \]
in $\cat{D}(B)$ is commutative.
\end{dfn}

\begin{prop} \label{prop:1700}
Under Setup \ref{set:1281}, let $(M, \rho) \in \cat{D}(B)_{\mrm{rig} / A}$, and 
define 
$(N, \si) := \lb \opn{Ind}^{\mrm{rig}}_{v / A}(M, \rho) \in 
\cat{D}(C)_{\mrm{rig} / A}$,
as in Definition \ref{dfn:1300}. Then the standard nondegenerate forward 
morphism
$\opn{q}^{\mrm{L}}_{v, M} : (M, \rho) \to (N, \si)$
is a rigid forward morphism over $v / A$. 
\end{prop}

\begin{proof}
This is immediate from Definitions \ref{dfn:1300}  and \ref{dfn:1293}. Cf.\ the 
commutative diagram (\ref{eqn:1700}). 
\end{proof}

\begin{prop} \label{prop:1675}
Under Setup \ref{set:1281}, let $w : C \to D$ be another essentially \'etale 
ring homomorphism. Suppose 
$(L, \rho) \in  \cat{D}(B)_{\mrm{rig} / A}$, 
$(M, \si) \in  \cat{D}(D)_{\mrm{rig} / A}$ and 
$(N, \tau) \in  \cat{D}(D)_{\mrm{rig} / A}$ are rigid complexes,  
$\la : (L, \rho) \to (M, \si)$
is a rigid forward morphism over $v / A$, and $\mu : (M, \si) \to (N, \tau)$
is a rigid forward morphism over $w / A$. Then 
$\mu \circ \la : (L, \rho) \to (N, \tau)$
is a rigid forward morphism over $(w \circ v) / A$.
\end{prop}

\begin{proof}
Clear from Theorem \ref{thm:1292}. 
\end{proof} 

Here is an analogue of Theorem \ref{thm:2030}; 
it is another generalization of Theorem \ref{thm:675}. 

\begin{thm} \label{thm:1863} 
Under Setup \ref{set:1281}, let $(M, \rho) \in \cat{D}(B)_{\mrm{rig} / A}$
and $(N, \si) \in \cat{D}(C)_{\mrm{rig} / A}$. Assume that $N$ has the derived 
Morita property over $C$. Then there is at most one nondegenerate rigid 
forward morphism $\la : (M, \rho) \to  (N, \si)$
over $v$ relative to $A$.
\end{thm}

\begin{proof}
Let $\la, \la' : (M, \rho) \to (N, \si)$ be two nondegenerate rigid forward 
morphisms over $v$. By Proposition \ref{prop:1871}(2) there 
is a unique element $c \in C^{\times}$  such that 
$\la' = c \cd \la$ in $\opn{Hom}_{\cat{D}(B)}(M, N)$.
Proceeding like in the proof of Theorem \ref{thm:675}, but using Theorem 
\ref{thm:1860} instead of Theorem \ref{thm:672}, 
we show that $c^2 = c$, and hence $c = 1$ and $\la' = \la$. 
\end{proof}

\begin{prop} \label{prop:1867}
Under Setup \ref{set:1281}, let 
$(M, \rho) \in  \cat{D}(B)_{\mrm{rig} / A}$.
Define $N := C \ot_B M \in \cat{D}(C)$, and assume that the complex $N$
has the derived Morita property over $C$. Then 
$\si := \opn{Ind}^{\mrm{rig}}_{v / A}(\rho)$ 
is the only rigidifying isomorphism 
$N \iso \opn{Sq}_{C / A}(N)$ in $\cat{D}(C)$ 
for which the standard nondegenerate forward morphism 
$\opn{q}^{\mrm{L}}_{v, M} : (M, \rho) \to (N, \si)$
is a rigid forward morphism over $v / A$. 
\end{prop}

\begin{proof}
According to Proposition \ref{prop:1700}, 
$\la := \opn{q}^{}_{v, M} : (M, \rho) \to (N, \si)$
is a nondegenerate rigid forward morphism over $v / A$. 
Suppose $\la' :  (M, \rho) \to (N, \si)$
is some nondegenerate rigid forward morphism over $v / A$; we must prove 
that $\la' = \la$. 

By Proposition \ref{prop:1871}(2) the $C$-module 
$\opn{Hom}_{\cat{D}(B)}(M, N)$ is free of rank $1$, with bases $\la$ and 
$\la'$. Therefore  there is a unique element $c \in C^{\times}$ such that 
$\la' = c \cd \la$. 
Because $\la$ is rigid we have 
$\opn{Sq}_{v / A}(\la) \circ \rho = \si \circ \la$. 
Likewise, because $\la'$ is rigid we have 
$\opn{Sq}_{v / A}(\la') \circ \rho = \si \circ \la'$. 
By  Theorem \ref{thm:1860} we know that 
$\opn{Sq}_{v / A}(\la') = \opn{Sq}_{v / A}(c \cd \la) 
= c^2 \cd \opn{Sq}_{v / A}(\la)$.
Combining these equalities we obtain 
\[ c \cd (\si \circ \la) = \si \circ \la' = 
\opn{Sq}_{v / A}(\la') \circ \rho = 
c^2 \cd (\opn{Sq}_{v / A}(\la) \circ \rho) = 
c^2 \cd (\si \circ \la) . \]
Since $\si : N \to \opn{Sq}_{C / A}(N)$ is an isomorphism, it follows that 
the forward morphism $\si \circ \la$ is nondegenerate. Using Proposition 
\ref{prop:1700} once more we see that 
$\opn{Hom}_{\cat{D}(B)}(M, \opn{Sq}_{C / A}(N))$ is a free $C$-module of 
rank $1$, with basis $\si \circ \la$. We conclude that $c = c^2$, and hence 
$c = 1$ and $\la' = \la$. 
\end{proof}

\begin{que} \label{que:1660}
In the situation of Setup \ref{set:1281}, where $v : B \to C$ is an essentially 
\'etale ring homomorphism, let 
$(M, \rho) \in \cat{D}(B)_{\mrm{rig} / A}$. 
In Definition \ref{dfn:1300} we have the induced rigid complex 
\[ \opn{Ind}^{\mrm{rig}}_{v / A}(M, \rho) := 
\bigl( M \ot^{}_B C, \, \opn{Ind}^{\mrm{rig}}_{v / A}(\rho) \bigr) 
\in \cat{D}(C)_{\mrm{rig} / A} . \]
But $v$ is essentially smooth of differential relative dimension $0$, 
so according to Definition \ref{dfn:1235} we also have the twisted induced 
rigid 
complex
\[ \opn{TwInd}^{\mrm{rig}}_{C / B / A}(M, \rho) :=
\bigl( M \ot_B C, \, \rho \cupprod \rho_{B / A}^{\mrm{esm}} 
\bigr) \in \cat{D}(C)_{\mrm{rig} / A} . \]
The obvious question is whether the rigidifying isomorphisms 
$\opn{Ind}^{\mrm{rig}}_{v / A}(\rho)$ and 
$\rho \cupprod \rho_{B / A}^{\mrm{esm}}$ on $M \ot_B C$ are equal. 

We believe the answer is positive, yet we were only able to prove it in
the special case when $A \to B$ is flat; see Appendix \ref{sec:flat-case}. 
\end{que}

\begin{thm}[Cup-Induced Rigid Forward Morphism] \label{thm:1905} 
Let $A$ be a noetherian ring, let $B \to C \xar{w} C'$ be homomorphisms
in $\cat{Rng} \eftover A$, and assume that $w$ is essentially \'etale. 
Let $(M, \rho) \in \cat{D}(B)_{\mrm{rig} / A}$,
$(N, \si) \in \cat{D}(C)_{\mrm{rig} / B}$ and 
$(N', \si') \in \cat{D}(C')_{\mrm{rig} / B}$
be rigid complexes, and let 
$\la : (N, \si) \to (N', \si')$
be a rigid forward morphism in $\cat{D}(C)$ over $w / B$.
Assume that the cup product morphisms 
$\opn{cup}_{C / B / A, M, N}$ and $\opn{cup}_{C' / B / A, M, N'}$ are 
isomorphisms. Then 
\[ \opn{id}_M \ot^{\mrm{L}}_{B} \, \la :
\bigl( M \ot^{\mrm{L}}_{B} N, \rho \cupprod \si \bigr) \to
\bigl( M \ot^{\mrm{L}}_{B} N', \rho \cupprod \si' \bigr)  \]
is a rigid forward morphism in $\cat{D}(C)$ over $w / A$. 
\end{thm}

\begin{proof}
According to Definition \ref{dfn:1293} we want to prove that the outer 
boundary in diagram below is commutative
\begin{equation} \label{eqn:newwa}
\begin{tikzcd} [column sep = 12ex, row sep = 5.5ex] 
M\ot^{\mrm{L}}_{B}N
\ar[rr, bend left = 15, start anchor = north east, end anchor = north west,
"{\rho \cupprod \sigma}", "{\simeq}"']
\ar[r, "{\rho\ot^{\mrm{L}}_{B}\sigma}", "{\simeq}"']
\ar[d, "{\opn{id}\ot^{\mrm{L}}_{B} \lambda}"]
&
\opn{Sq}_{B/A} (M)\ot^{\mrm{L}}_{B} \opn{Sq}{C/A} (N)
\ar[r, "{\opn{cup}}", "{\simeq}"']
\ar[d, "{\opn{id}\ot^{\mrm{L}}_{B}\opn{Sq}_{w/B}( \lambda)}"]
&
\opn{Sq}_{C/A} (M\ot^{\mrm{L}}_{B} N)
\ar[d, "{\opn{Sq}_{w/B}(\opn{id}\ot^{\mrm{L}}_{B} \lambda)}"]
\\
M\ot^{\mrm{L}}_{B}N'
\ar[rr, bend right = 15, start anchor = south east, end anchor = south west,
"{\rho \cupprod \sigma'}", "{\simeq}"']
\ar[r, "{\rho\ot^{\mrm{L}}_{B}\sigma'}", "{\simeq}"']
&
\opn{Sq}_{B/A} (M)\ot^{\mrm{L}}_{B} \opn{Sq}_{C'/A} (N')
\ar[r, "{\opn{cup}'}", "{\simeq}"']
&
\opn{Sq}_{C'/A} (M\ot^{\mrm{L}}_{B} N')
\end{tikzcd} 
\end{equation}
Here $\opn{cup}:=\opn{cup}_{C/B/A,M,N}$ and $\opn{cup}':=\opn{cup}_{C'/B/A,M,N'}$.
The two curved regions are commutative by Definition \ref{dfn:1260}.
The left rectangular region is commutative because $\lambda$ 
is rigid forward morphism.
It remains to prove that the right rectangular region in commutative. 
By Definition \ref{dfn:1290} the right rectangular diagram is given by: 
\begin{equation} \label{eqn:newwb}
\begin{tikzcd}[column sep = 15ex, row sep = 5ex] 
\opn{Sq}_{B/A} (M)\ot^{\mrm{L}}_{B} \opn{Sq}_{C/A} (N)
\ar[r, "{\opn{cup}}", "{\simeq}"']
\ar[d, "{\opn{id}\ot^{\mrm{L}}_{B} \opn{Sq}_{C/B} (\lambda)}"]
&
\opn{Sq}_{C/A} (M\ot^{\mrm{L}}_{B} N)
\ar[d, "{\opn{Sq}_{C/B}(\opn{id}\ot^{\mrm{L}}_{B}\lambda)}"]
\\
\opn{Sq}_{B / A}(M)\ot^{\mrm{L}}_{B} \opn{Sq}_{C / A}(N')
\ar[r, "{\opn{cup}'}", "{\simeq}"']
\ar[d, "{\opn{id}\ot^{\mrm{L}}_{B} \opn{red}_{C'/C/B,N'}}"]
&
\opn{Sq}_{C / A}(M\ot^{\mrm{L}}_{B} N')
\ar[d, "{\opn{red}_{C'/C/A,M\ot^{\mrm{L}}_{B} N'}}"]
\\
\opn{Sq}_{B / A}(M)\ot^{\mrm{L}}_{B} \opn{Sq}_{C / A}(N')
\ar[r, "{\opn{cup}''}", "{\simeq}"']
&
\opn{Sq}_{C' / A}(M\ot^{\mrm{L}}_{B} N')
\end{tikzcd}
\end{equation}
Here $\opn{cup}'':=\opn{cup}_{C/B/A,M,N'}$. The upper diagram is commutative 
according to Theorem \ref{thm:845} considering $\lambda$ as a rigid backward morphism.
The commutativity of the lower diagram is given by the definition of 
Cup Product Morphism (cf. Theorem \ref{thm:780}) and by the following commutative diagram
\begin{equation} \label{eqn:newwc}
\begin{tikzcd}[column sep = 7ex, row sep = 5ex] 
\opn{Sq}_{B / A}(M)\ot^{\mrm{L}}_{B} \opn{RHom}_{\til{G}}(C,N'')
\ar[r, "{\opn{ev}'}", "{\simeq}"']
\ar[d, "{\opn{id}\ot^{\mrm{L}}_{B}\opn{adj}}"]
&
\opn{RHom}_{\til{G}}(C,N''\ot^{\mrm{L}}_{B} \opn{Sq}_{B / A}(M))
\ar[d, "{\opn{adj}}"]
\\
\opn{Sq}_{B / A}(M)\ot^{\mrm{L}}_{B} \opn{RHom}_{\til{G}}(C'\ot_C C',N'')
\ar[r, "{\opn{ev}''}", "{\simeq}"']
&
\opn{RHom}_{\til{G}}(C'\ot_C C',N''\ot^{\mrm{L}}_{B} \opn{Sq}_{B / A}(M))
\end{tikzcd}
\end{equation}
Here $N'':= N'\ot^{\mrm{L}}_{\til{B}} N'$, $G':= \til{C}'_1\ot^{\mrm{L}}_{\til{B}} \til{C}'_2$, 
$\opn{ev}':=\opn{ev}^{R,L}_{\tiny\opn{Sq}_{B / A}(M),C,N''}$ and $\opn{ev}''$ 
denotes $\opn{ev}^{R,L}_{\tiny\opn{Sq}_{B / A}(M),C'\ot_C C',N''}$.
Therefore the whole diagram (\ref{eqn:newwa}) is commutative.
\end{proof}

\section{Rigid Dualizing Complexes} 
\label{sec:dualizing}

In this section we apply the material from the previous sections to prove 
existence and uniqueness of rigid dualizing complexes; see Theorem 
\ref{thm:1550}. Then we discuss the functoriality of rigid dualizing complexes 
with respect to finite ring homomorphisms (Theorem \ref{thm:1551}) and 
essentially \'etale ring homomorphisms (Theorem \ref{thm:1552}), 
and we prove these functorialities commute (Theorem \ref{thm:1553}). 
Throughout this section the next convention will be in force.

\begin{conv} \label{conv:1550}
There is a base ring $\K$, which is a nonzero finite-dimensional regular 
noetherian ring. All rings and ring homomorphisms are in the category 
$\cat{Rng} \eftover \K$ of essentially finite type $\K$-rings (see Definition \ref{dfn:1065} 
and \ref{dfn:1066}).
\end{conv}

Observe that this implies that all rings are noetherian and all 
homomorphisms between them are EFT. 
Because the base ring $\K$ has finite global cohomological dimension,
every complex $M \in \cat{D}^{\mrm{b}}(\K)$ has finite flat dimension over 
$\K$. 

The notion of {\em dualizing complex} goes back to Grothendieck \cite{RD}. 
Given a ring $A$, a dualizing complex over $A$ is a complex 
$R \in \cat{D}^{\mrm{b}}_{\mrm{f}}(A)$ that has finite injective dimension and 
the derived Morita property (Definition \ref{dfn:676}). A rather detailed 
discussion of dualizing complexes can be found in \cite[Chapter 13]{Ye5}.
 
\begin{dfn}\label{defn:rigidducomp}
Let $A$ be an EFT $\K$-ring. A {\em rigid dualizing complex} over $A / \K$
is a rigid complex 
$(R_A, \rho_{A}) \in \cat{D}(A)_{\mrm{rig} / \K}$,
in the sense of Definition \ref{dfn:675}, such that $R$ itself is a dualizing 
complex over $A$. 
\end{dfn}

The next result is stronger than \cite[Theorems 13.5.4 and and 13.5.7]{Ye5}, 
in which it was necessary to assume that the ring $A$ is {\em flat} over $\K$. 

\begin{thm} \label{thm:1550}
Let $A$ be an EFT $\K$-ring. Then $A$ has a rigid dualizing complex
$(R_A, \rho_{A})$ relative to $\K$, and it is unique, up to a unique 
isomorphism in $\cat{D}(A)_{\mrm{rig} / \K}$. 
\end{thm}

\begin{proof}
The structure homomorphism $u : \K \to A$ can be factored 
into $u = u_3 \circ u_2 \circ u_1$, where
$u_1 : \K \to A_1 = \K[t_1, \ldots, t_n]$ is a homomorphism to a polynomial 
ring, $u_2 : A_1 \to A_2$ is surjective, and 
$u_3 : A_2 \to A_3 = A$ is a localization. 
We start with the tautological rigid dualizing complex 
$(\K, \rho^{\mrm{tau}}_{\K / \K}) \in \cat{D}(\K)_{\mrm{rig} / \K}$.
By twisted induction (Definition \ref{dfn:1225}, a special case of Definition 
\ref{dfn:1235}) there is a rigidifying isomorphism 
$\rho_{A_1}$ on the complex $R_{A_1} := \Om^n_{A_1 / \K}[n]$, yielding 
$(R_{A_1}, \rho_{A_1}) \in \cat{D}(A_1)_{\mrm{rig} / \K}$. 
The complex $R_1$ is dualizing because the ring $A_1$ is regular and 
$\Om^n_{A_1 / \K}$ is a free $A_1$-module of rank $1$. 
Consider the complex 
$R_{A_2} := \opn{RHom}_{A_1}(A_2, R_{A_1}) \in \cat{D}(A_2)$. 
It is a dualizing complex over $A_2$ (see \cite[Proposition 13.1.28]{Ye5}), and 
hence it has bounded cohomology, and so finite flat dimension over $\K$. 
By derived coinduction (Definition \ref{dfn:760}) for the finite ring 
homomorphism $A_1 \to A_2$,  there is a rigidifying 
isomorphism $\rho_{A_2}$ on the complex $R_{A_2}$, yielding a rigid dualizing 
complex $(R_{A_2}, \rho_{A_2}) \in \cat{D}(A_2)_{\mrm{rig} / \K}$. 
Finally, by induction (Definition \ref{dfn:1300}) for the essentially \'etale 
ring homomorphism $A_2 \to A_3$, there is a rigidifying isomorphism 
$\rho_{A_3}$ on the complex 
$R_{A_3} := A_3 \ot_{A_2} R_{A_2}$, yielding a rigid dualizing complex 
$(R_{A_3}, \rho_{A_3}) \in \cat{D}(A_3)_{\mrm{rig} / \K}$. 

The uniqueness of the rigid dualizing complex $(R_A, \rho_{A})$
 is proved in \cite[Theorem 13.5.4]{Ye5}; one just has to replace 
the ring $A^{\mrm{en}}$ used there with the DG ring $\til{A} \ot_{\K} \til{A}$, 
where $\til{A} \to A$ is a K-flat DG ring resolution relative to $\K$. 
\end{proof}

Because of the uniqueness in the theorem above, we can talk about {\em the} 
rigid dualizing complex $(R_A, \rho_A)$ over $A / \K$. 

\begin{thm} \label{thm:1551}
Let $u : A \to B$ be a finite homomorphism in $\cat{Rng} \eftover \K$,
and let $(R_A, \rho_{A})$ and $(R_B, \rho_{B})$ be the corresponding rigid 
dualizing complexes relative to $\K$. 
Then there is a unique nondegenerate rigid backward morphism 
\[ \opn{tr}^{\mrm{rig}}_{B / A}  = \opn{tr}^{\mrm{rig}}_{u} :
(R_B, \rho_{B}) \to (R_A, \rho_{A}) \]
in $\cat{D}(A)$ over $u$ relative to $\K$. It is called the {\em rigid trace 
morphism}. 
\end{thm}

\begin{proof}
Define the complex 
$N := \opn{RCInd}_u(R_A) = \opn{RHom}_A(B, R_A) \in \cat{D}(B)$. 
We know, by \cite[Proposition 13.1.28]{Ye5}, that $N$ is a dualizing complex 
over $B$. Thus both $R_A$ and $N$ have finite flat dimensions over $\K$. 
According to Theorem \ref{thm:680}, the complex $N$ has a unique rigidifying 
isomorphism $\si$, such that the standard trace morphism 
$\opn{tr}^{\mrm{R}}_{u, M} : N \to M$ becomes a nondegenerate rigid backward 
morphism $\opn{tr}^{\mrm{R}}_{u, M} : (N, \si) \to (R_A, \rho_A)$
over $u$ relative to $\K$. Now $(N, \si)$ is a rigid dualizing complex over 
$A / \K$, and therefore, by Theorem \ref{thm:1550} there is a unique rigid 
isomorphism $\phi : (R_B, \rho_{B}) \iso (N, \si)$. Then 
\[ \opn{tr}^{\mrm{rig}}_{B / A} := \opn{tr}^{\mrm{R}}_{u, M} \circ \, \phi :
(R_B, \rho_{B}) \to (R_A, \rho_A) \]
is a nondegenerate rigid backward morphism. 

The uniqueness of $\opn{tr}^{\mrm{rig}}_{B / A}$ is by 
Theorem \ref{thm:2030}.
\end{proof}

\begin{cor} \label{cor:1551}
Suppose $A \xar{u} B \xar{v} C$ are finite homomorphisms in 
$\cat{Rng} \eftover \K$, 
and let $(R_A, \rho_{A})$, $(R_B, \rho_{B})$ and $(R_C, \rho_{C})$ be the 
corresponding rigid dualizing complexes relative to $\K$. Then 
\[ \opn{tr}^{\mrm{rig}}_{B / A} \circ \opn{tr}^{\mrm{rig}}_{C / B} =
\opn{tr}^{\mrm{rig}}_{C / A} \]
as backward morphisms $R_C \to R_B$ 
in $\cat{D}(A)$ over $v \circ u$.                 
\end{cor}

\begin{proof}
This is due to the uniqueness of the rigid trace. 
\end{proof}

\begin{thm} \label{thm:1552}
Let $w : A \to A'$ be an essentially \'etale homomorphism in 
$\cat{Rng} \eftover \K$,
and let $(R_A, \rho_{A})$ and $(R_{A'}, \rho_{A'})$ be the corresponding rigid 
dualizing complexes relative to $\K$. 
Then there is a unique nondegenerate rigid forward morphism 
\[ \opn{q}^{\mrm{rig}}_{A' / A}  = \opn{q}^{\mrm{rig}}_{w} :
(R_A, \rho_{A}) \to (R_{A'}, \rho_{A'}) \]
in $\cat{D}(A)$ over $w$ relative to $\K$. It is called the {\em rigid 
\'etale-localization morphism}. 
\end{thm}

\begin{proof}
$M' := \opn{Ind}_w(R_A) = A' \ot_A R_A \in \cat{D}(A')$. 
The complex $M'$ has an induced rigidifying isomorphism 
$\rho' := \opn{Ind}^{\mrm{rig}}_{w / \K}(\rho_A)$; see Definition 
\ref{dfn:1300}. By Proposition \ref{prop:1700}, the standard nondegenerate
forward morphism 
$\opn{q}^{}_{w, R_A} : (R_A, \rho_A) \to (M', \rho')$
is a  rigid forward morphism over $w / A$. 
Now $(M', \rho')$ is a rigid dualizing complex over 
$A' / \K$, and therefore, by Theorem \ref{thm:1550}, there is a unique 
rigid isomorphism $\phi : (M', \rho') \iso (R_{A'}, \rho_{A'})$. 
Then 
\[ \opn{q}^{\mrm{rig}}_{A' / A} := \phi \circ \opn{q}^{}_{w, R_A} :
(R_A, \rho_{A}) \to (R_{A'}, \rho_{A'}) \]
is a nondegenerate rigid forward morphism. 

The uniqueness of $\opn{q}^{\mrm{rig}}_{A' / A}$ is by Theorem   
\ref{thm:1863}.
\end{proof}

\begin{cor} \label{cor:1552}
Suppose $A \xar{w} A' \xar{w'} A''$ are essentially \'etale homomorphisms in 
$\cat{Rng} \eftover \K$, 
and let $(R_A, \rho_{A})$, $(R_{A'}, \rho_{A'})$ and $(R_{A''}, \rho_{A''})$ be 
the corresponding rigid dualizing complexes relative to $\K$. Then 
\[ \opn{q}^{\mrm{rig}}_{A'' / A'} \circ \opn{q}^{\mrm{rig}}_{A' / A} =
\opn{q}^{\mrm{rig}}_{A'' / A} \]
as forward morphisms $R_{A} \to R_{A''}$ 
in $\cat{D}(A)$ over $w' \circ w$.                 
\end{cor}

\begin{proof}
This is due to the uniqueness of the rigid \'etale-localization morphism.
\end{proof}

\begin{thm} \label{thm:1553}
Let $u : A \to B$ be a finite homomorphism,
and let $v : A \to A'$ be an essentially \'etale homomorphism, both in 
$\cat{Rng} \eftover \K$. Define 
$B' := A' \ot_{A} B$, and let 
$u' : A' \to B'$ and $w : B \to B'$ be the induced ring homomorphisms. 
The rigid dualizing complexes of these rings, relative to $\K$, are 
$(R_A, \rho_{A})$, $(R_{B}, \rho_{B})$, $(R_{A'}, \rho_{A'})$ and
$(R_{B'}, \rho_{B'})$ respectively. Then 
\[ \opn{q}^{\mrm{rig}}_{v} \circ \opn{tr}^{\mrm{rig}}_{u} = 
\opn{tr}^{\mrm{rig}}_{u'} \circ \opn{q}^{\mrm{rig}}_{w} \]
as morphisms 
$R_B \to R_{A'}$ in $\cat{D}(A)$.
\end{thm}

The theorem asserts that given the first commutative diagram below in 
$\cat{Rng} \eftover \K$, the second diagram in $\cat{D}(A)$ is also 
commutative. 
\begin{equation} \label{eqn:1780}
\begin{tikzcd} [column sep = 6ex, row sep = 4ex] 
A
\ar[r, "{u}"]
\ar[d, "{v}"']
&
B
\ar[d, "{w}"]
\\
A'
\ar[r, "{u'}"]
&
B' 
\end{tikzcd} 
\qquad \qquad 
\begin{tikzcd} [column sep = 6ex, row sep = 4ex] 
R_A
\ar[d, "{\opn{q}^{\mrm{rig}}_{v}}"']
&
R_B
\ar[l, "{\opn{tr}^{\mrm{rig}}_{u}}"']
\ar[d, "{\opn{q}^{\mrm{rig}}_{w}}"]
\\
R_{A'}
&
R_{B'}
\ar[l, "{\opn{tr}^{\mrm{rig}}_{u'}}"']
\end{tikzcd} 
\end{equation}

We need several lemmas before we can prove the theorem. In them we assume the 
situation of the theorem.

\begin{lem} \label{lem:1590}
Let $\th : N \to M$ be a backward morphism in $\cat{D}(A)$ over $u$, 
let $\mu : M \to M'$ be a nondegenerate forward morphism in $\cat{D}(A)$ over 
$v$, and let $\nu : N \to N'$ be a nondegenerate forward morphism in 
$\cat{D}(B)$ over $w$.
\begin{enumerate}
\item There is a unique backward morphism $\th' : N' \to M'$ in $\cat{D}(A')$ 
over $u'$, such that $ \th' \circ \nu = \mu \circ \th$. 

\item If $\th$ is a nondegenerate backward morphism, then $\th'$ is also a 
nondegenerate backward morphism.
\end{enumerate}
\end{lem}

Item (1) of the lemma says that there is a unique morphism $\th'$, the 
dashed arrow, making this diagram in $\cat{D}(A)$ is commutative:
\begin{equation} \label{eqn:1800}
\begin{tikzcd} [column sep = 8ex, row sep = 4ex] 
M
\ar[d, "{\mu}"']
&
N
\ar[l, "{\th}"']
\ar[d, "{\nu}"]
\\
M'
&
N'
\ar[l, dashed, "{\th'}"']
\end{tikzcd} 
\end{equation}

\begin{proof} \mbox{}

\smallskip \noindent 
(1) We can assume that $M' = A' \ot_A M$, 
$N' = B' \ot_B N \cong A' \ot_A N$, and $\mu = \opn{q}_{v, M}$, 
$\nu = \opn{q}_{w, N}$ are the standard nondegenerate backward morphisms. Then 
the only option for $\th'$ is $\opn{id}_{A'} \ot_{A} \, \th$. 

\medskip \noindent 
(2) Here we can also assume that $N = \opn{RHom}_A(B, M)$, and 
$\th = \opn{tr}^{\mrm{R}}_{v, M}$ is the standard nondegenerate backward 
morphism. Applying $A' \ot_A (-)$ to the isomorphism 
$\opn{badj}^{\mrm{R}}_{u, M, N}(\th) : N \iso \opn{RHom}_A(B, M)$
in $\cat{D}(B)$ we obtain the isomorphism
\begin{equation} \label{eqn:1590}
N' = A' \ot_A N \iso A' \ot_A \opn{RHom}_A(B, M) \cong 
\opn{RHom}_{A'}(B', M') . 
\end{equation}
Here we use the fact that $A$ is noetherian and $B$ is a finite $A$-module. 
It is not hard to see (using resolutions) that the composed isomorphism 
(\ref{eqn:1590}) is $\opn{badj}^{\mrm{R}}_{u', M', N'}(\th')$. Thus $\th'$ is 
nondegenerate.
\end{proof}

\begin{lem} \label{lem:1591} 
Let $\th : (N, \si) \to (M, \rho)$ be a rigid backward morphism in $\cat{D}(A)$ 
over $u / \K$, let $\mu : (M, \rho) \to (M', \rho')$ be a nondegenerate rigid 
forward morphism in $\cat{D}(A)$ over $v / \K$, and let 
$\nu : (N, \si) \to (N', \si')$ be a nondegenerate rigid forward morphism in 
$\cat{D}(B)$ over $w / \K$. Then the unique backward morphism 
$\th' : N' \to M'$ from Lemma \ref{lem:1590}(1) is a rigid backward 
morphism $\th' : (N', \si') \to (M', \rho')$  in $\cat{D}(A')$ 
over $u' / \K$.
\end{lem}

The next diagram illustrates the lemma. It is given that that solid arrows are 
rigid, and the assertion is that the dashed arrow is also rigid.
\[ \begin{tikzcd} [column sep = 8ex, row sep = 4ex] 
(M, \rho)
\ar[d, "{\mu}"']
&
(N, \si)
\ar[l, "{\th}"']
\ar[d, "{\nu}"]
\\
(M', \rho')
&
(N', \si')
\ar[l, dashed, "{\th'}"']
\end{tikzcd} \]

\begin{proof}
We need to prove that the bottom face of this cubical diagram is commutative:
\[ \begin{tikzcd}[row sep = 4ex, column sep = 10ex]
& 
M
\arrow[dl, "{\rho}"'] 
\arrow[dd, "{\mu}"' near start] 
& 
& 
N
\ar[ll, "{\th}"' near start]
\arrow[dl, "{\si}"'] 
\arrow[dd, "{\nu}"' near start] 
\\
\opn{Sq}_{A / \K}(M)
\arrow[dd, "{\opn{Sq}_{v / \K}(\mu)}"' near start] 
& 
& 
\opn{Sq}_{B / \K}(N) 
\arrow[ll, crossing over, "{\opn{Sq}_{u / \K}(\th)}"' near start] 
\\
& 
M' 
\arrow[dl, "{\rho'}"'] 
& 
& 
N' 
\arrow[dl, "{\si'}"'] 
\arrow[ll, "{\th'}"' near start] 
\\
\opn{Sq}_{A' / \K}(M')
& 
& 
\opn{Sq}_{B' / \K}(N')
\arrow[ll, "{\opn{Sq}_{u' / \K}(\th')}"' near start] 
\arrow[from=uu, crossing over, "{\opn{Sq}_{w / \K}(\nu)}"' near start] 
\end{tikzcd} \]
The top face is commutative because $\th$ is a rigid backward morphism; 
the rear face is commutative by the definition of $\th'$
(it is diagram (\ref{eqn:1800})); the right face is 
commutative because $\nu$ is a rigid forward morphism; and the left face is 
commutative because $\mu$ is a rigid forward morphism.
As for the front face: like in the proof of Lemma \ref{lem:1590}, we can assume 
that $M' = A' \ot_A M$, $N' = B' \ot_B N$, 
$\mu = \opn{q}^{}_{v, M}$ and $\nu = \opn{q}^{}_{w, N}$. Then, according  to   
Theorem \ref{thm:895}, the front face is commutative. 
Define $\al_0 := \rho \circ \th$ and 
$\al_1 := \opn{Sq}_{u / \K}(\th) \circ \si$; these morphisms are on the  
diagonal of the top face, and $\al_0 = \al_1$ because this face is commutative. 
Define $\al'_0 := \rho' \circ \th'$ and 
$\al'_1 := \opn{Sq}_{u' / \K}(\th') \circ \si'$; 
these morphisms are on the diagonal of the bottom face, and we need to 
prove they are equal. 
The commutativity of the rear and left faces implies that 
$\opn{Sq}_{v / \K}(\mu) \circ \al_0 = \al'_0 \circ \nu$.  
The commutativity of the right and front faces implies that 
$\opn{Sq}_{v / \K}(\mu) \circ \al_1 = \al'_1 \circ \nu$. We know that $\nu$ and 
$\opn{Sq}_{v / \K}(\mu)$ are nondegenerate forward morphisms. By the uniqueness 
in Lemma \ref{lem:1590}(1) it follows that $\al'_0 = \al'_1$. 
\end{proof}

\begin{proof}[Proof of Theorem \tup{\ref{thm:1553}}]
We shall use Lemmas \ref{lem:1590} and \ref{lem:1591}.
Define the rigid complexes $(M, \rho) := (R_A, \rho_A)$, 
$(M', \rho') := (R_{A'}, \rho_{A'})$,
$(N, \si) := (R_{B}, \rho_{B})$ and
$(N', \si') := (R_{B'}, \rho_{B'})$.
There is a nondegenerate rigid backward morphism 
$\th := \opn{tr}^{\mrm{rig}}_{u}$,
and nondegenerate rigid forward morphisms 
$\mu := \opn{q}^{\mrm{rig}}_{v}$ and $\nu := \opn{q}^{\mrm{rig}}_{w}$. 
Let $\th' : N' \to M'$ be the unique nondegenerate backward morphism from Lemma 
\ref{lem:1590}; so $\mu \circ \th = \th' \circ \nu$. 
According to Lemma \ref{lem:1591} the nondegenerate backward morphism 
$\th' : (R_{B'}, \rho_{B'}) \to (R_{A'}, \rho_{A'})$
is rigid. But then, by Theorem \ref{thm:2030}, there is equality 
$\th' = \opn{tr}^{\mrm{rig}}_{u'}$. 
\end{proof}

\begin{lem} \label{lem:1720}
Let $A$ be a local ring, with maximal ideal $\m$ and residue field 
$\bk(\m)$. Suppose $M \in \cat{D}^{+}_{\mrm{f}}(A)$ and 
$q_0 \leq q_1$ are integers such that 
$\opn{inf}(\opn{H}(M)) \geq q_0$ and 
\[ \opn{sup} \bigl( \opn{H} \bigl (\opn{RHom}_A(\bk(\m), M) \bigr) \bigr) \leq 
q_1 . \]
Then the injective concentration of $M$ is inside the integer interval 
$[q_0, q_1]$, and $M$ has injective dimension $\leq q_1 - q_0$. 
\end{lem}

The notions we use in this lemma (the numbers $\opn{inf}(N)$ and $\opn{sup}(N)$
for a graded module $N$, such as $N := \opn{H}(M)$ for a complex 
$M \in \cat{D}(A)$, the injective concentration of $M$, and the injective 
dimension of $M$) are taken from \cite[Sections 12.1 and 12.4]{Ye5}. 
The assertion about the injective concentration of $M$ means that for every 
$L \in \cat{M}(A)$ the cohomology of the complex 
$\opn{RHom}_A(L, M)$ is concentrated inside the integer interval $[q_0, q_1]$. 

\begin{proof}
The proof is in a few steps. It is based on the proof of 
\cite[Proposition V.3.4]{RD}, but we felt more details are necessary, 
especially regarding the inductive procedure.  

\medskip \noindent 
Step 1. Here we establish the inductive procedure. The induction is on the 
set $\N^2$ with the lexicographic order, which makes it a well-ordered set. 

We know that every nonzero finitely generated $A$-module $L$ admits a 
filtration 
\begin{equation} \label{eqn:1720}
0 = L_0 \subsetneqq L_1 \subsetneqq \cdots \subsetneqq L_n = L 
\end{equation}
where $L_i / L_{i - i} \cong A / \p_i$ for some prime ideal $\p_i$. See 
\cite[Theorem 6.4]{Ma}. Let's call the smallest such number $n$ the {\em prime 
filtration length} of $L$, and denote it by $\opn{len}(L)$. To the module $L$ 
we attach the pair of natural numbers 
$\opn{dl}(L) := \bigl( \opn{dim}(L), \opn{len}(L) \bigr) \in \N^2$,
where $\opn{dim}(L)$ is the Krull dimension of the support of $L$. 
If $\opn{len}(L) = 1$ then $L$ is called a {\em critical} module. 

We are going to prove that  
\begin{equation} \label{eqn:1721}
\opn{sup} \bigl( \opn{H} \bigl (\opn{RHom}_A(L, M) \bigr) \bigr) \leq q_1 
\end{equation}
for every $L \in \cat{M}_{\mrm{f}}(A)$, by induction on 
$\opn{dl}(L) \in \N^2$. 
This means that given $(d, l) \in \N^2$, and assuming that 
(\ref{eqn:1721}) holds for every $L \in \cat{M}_{\mrm{f}}(A)$ 
with $\opn{dl}(L) < (d, l)$, we need to prove that (\ref{eqn:1721}) holds 
for every $L \in \cat{M}_{\mrm{f}}(A)$ with $\opn{dl}(L) = (d, l)$.

\medskip \noindent 
Step 2. Here we prove the inductive assertion. We break the proof of this 
assertion into two cases: when $L$ is not critical, and when $L$ is critical. 

When $L$ is not critical, then from a filtration (\ref{eqn:1720}) of minimal 
length we obtain a short exact sequence 
$0 \to L' \to L \to L'' \to 0$ 
with $\opn{len}(L'), \opn{len}(L'') < l$.
Because $\opn{dim}(L'), \opn{dim}(L'') \leq d$, 
it follows that $\opn{dl}(L'), \opn{dl}(L'') < (d, l)$.
In the exact sequence 
\[ \cdots \to
\opn{Ext}^q_A(L'', M) \to \opn{Ext}^q_A(L, M) \to 
\opn{Ext}^q_A(L', M) \to \cdots \]
for $q \geq q_1 + 1$ the extreme terms vanish due to the induction hypothesis.
Therefore so does the middle term. We see that (\ref{eqn:1721}) holds for
$L$. 

When $L$ is critical and $d = 0$, then $L \cong \bk(\m)$, 
and (\ref{eqn:1721}) holds for $L$ by assumption. When $L$ is critical and 
$d > 1$, then $L \cong A / \p$ for a prime $\p \subsetneqq \m$. 
Take some element $a \in \m - \p$. The element $a$ is regular on $L$, so there 
is a short exact sequence 
$0 \to L \xar{a \cd (-)} L \to L'' \to 0$ 
in $\cat{M}_{\mrm{f}}(A)$. Now $\opn{dim}(L'') < d$, and hence   
$\opn{dl}(L'') < (d, l)$. In the exact sequence 
\[ \cdots \to
\opn{Ext}^q_A(L, M) \xar{a \cd (-)}
\opn{Ext}^q_A(L, M) \to \opn{Ext}^{q + 1}_A(L'', M)  \to \cdots  \]
the third term vanishes for $q \geq q_1$, by the induction hypothesis. 
Since $a \in \m$ and $\opn{Ext}^q_A(L, M)$ is a finitely generated $A$-module, 
the Nakayama Lemma implies that \lb $\opn{Ext}^q_A(L, M) = 0$. So in this case 
too we see that (\ref{eqn:1721}) holds for $L$. 

\medskip \noindent 
Step 3. According to \cite[Corollary 11.5.27]{Ye5}
there is a quasi-isomorphism 
$M \to J$, where $J$ is a complex of injective $A$-modules, and 
$\opn{inf}(J) \geq q_0$. By formula (\ref{eqn:1721}) with 
$L := A$ we see that $\opn{sup}(\opn{H}(M)) \leq q_1$. Let $I$ the smart 
truncation $I := \opn{smt}^{\leq q_1}(J)$, so $I \to J$ is a quasi-isomorphism.
For $q < q_1$ we have $I^q = J^q$. We claim that 
$I^{q_1} = \opn{Z}^{q_1}(J)$ is also an injective 
$A$-module. Arguing like in the proof of \cite[Proposition 12.4.13]{Ye5}, but 
taking the test modules $L$ only from $\cat{M}_{\mrm{f}}(A)$, 
we deduce that $\opn{Ext}^q_A(L, I^{q_1}) = 0$ for all 
$L \in \cat{M}_{\mrm{f}}(A)$ and $q > 0$. This implies that $I^{q_1}$ is 
injective. 

We have shown that $I$ is a complex of injective $A$-modules, concentrated in 
the degree interval $[q_0, q_1]$. The quasi-isomorphisms 
$M \to J$ and $I \to J$ yield a quasi-isomorphism $M \to I$. Therefore 
$M$ has injective concentration inside $[q_0, q_1]$.
\end{proof}

\begin{lem} \label{lem:1625}
Let $u : A \to B$ be an essentially smooth ring homomorphism, and let 
$M \in \cat{D}^{\mrm{b}}_{\mrm{f}}(A)$ be a complex of finite injective 
dimension over $A$. Then 
$B \ot_A M \in \cat{D}^{\mrm{b}}_{\mrm{f}}(B)$
has finite injective dimension over $B$. 
\end{lem}

The result and its proof were communicated privately to us by L. Shaul. 

\begin{proof}
Let's write $N := B \ot_A M$.
Say the injective concentration of the $A$-module $M$ is inside the 
integer interval $[d_0, d_1]$, and the relative 
differential dimension of $u$ is $\leq e$. We will show that the injective 
concentration of the $B$-module $N$ is inside the integer interval 
$[d_0, d_1 + e]$,

For a prime ideal $\q \sub B$ and a module $L \in \cat{M}_{\mrm{f}}(B)$
we have   
\begin{equation} \label{eqn:1715}
\opn{RHom}_{B}(L, N) \ot_{B} B_{\q} \cong 
\opn{RHom}_{B_{\q}}(L_{\q}, N_{\q}) .
\end{equation}
So it suffices to show that for every $\q$ the 
$B_{\q}$-module $N_{\q}$ has injective concentration inside $[d_0, d_1 + e]$.

Given $\q \in \opn{Spec}(B)$, let $\p := u^{-1}(\q) \sub A$, so  
$N_{\q} \cong B_{\q} \ot_{A_{\p}} M_{\p}$. 
Because every finitely generated 
$A_{\p}$-module is induced from a  finitely generated $A$-module, an 
isomorphism like (\ref{eqn:1715}) shows that the injective 
concentration of the $A_{\p}$-module $M_{\p}$ is inside $[d_0, d_1]$.
Note also that the ring homomorphism $A_{\p} \to B_{\q}$ is essentially smooth
of relative differential dimension $\leq e$. 

As the two paragraphs above show, we can assume that $A$ is local 
with maximal ideal $\p$, $B$ is local with maximal ideal $\q$, $u : A \to B$ 
is an essentially smooth local ring homomorphism of relative differential 
dimension $\leq e$, and $M \in \cat{D}^{\mrm{b}}_{\mrm{f}}(A)$ has injective 
concentration inside $[d_0, d_1]$.
We need to prove that $N := B \ot_A M$ has injective 
concentration inside $[d_0, d_1 + e]$.

According to Lemma \ref{lem:1720} it suffices to prove that 
\[ \opn{con} \bigl( \opn{H} \bigl( \opn{RHom}_B(\bk(\q), N) \bigr) \bigr) 
\sub [d_0, d_1 + e] , \]
where $\bk(\q)$ is the residue field of $B$, and $\opn{con}(-)$ is the 
concentration of a graded module (see \cite[Definition 12.1.4]{Ye5}). 

Define the ring $C := \bk(\p) \ot_A B$. Here is the main calculation:
\begin{equation} \label{eqn:1625}
\begin{aligned}
&
\opn{RHom}_B(\bk(\q), N) \cong^{(1)}
\opn{RHom}_{C} \bigl( \bk(\q), \opn{RHom}_{B}(C, N) \bigr)  
\\ & \quad 
\cong^{(2)} \opn{RHom}_{C} \bigl( \bk(\q), \opn{RHom}_{A}( \bk(\p), N) \bigr) 
\\ & \quad 
\cong^{(3)} \opn{RHom}_{C} \bigl( \bk(\q), 
B \ot_{A} \opn{RHom}_{A}( \bk(\p), M) \bigr)
\\ & \quad 
\cong^{(4)} \opn{RHom}_{C} \bigl( \bk(\q), 
C \ot_{\bk(\p)} \opn{RHom}_{A}( \bk(\p), M) \bigr)
\\ & \quad 
\cong^{(5)} \opn{RHom}_{C}(\bk(\q), C) 
\ot_{\bk(\p)} \opn{RHom}_{A}( \bk(\p), M)  . 
\end{aligned} 
\end{equation}
Here are the explanations of the various isomorphisms above: 
the isomorphism $\cong^{(1)}$ is by adjunction for the ring homomorphism 
$B \to C$, noting that $B \to \bk(\q)$ factors through $C$. 
The isomorphism $\cong^{(2)}$ is by adjunction for the ring homomorphism 
$A \to B$. The isomorphism $\cong^{(3)}$ relies on derived tensor-evaluation, 
noting that $N = B \ot_A M$, $B$ is flat over $A$, and $A$ is noetherian; see 
\cite[Theorem 12.10.14]{Ye5}. The isomorphism $\cong^{(4)}$ is because 
for the complex of $\bk(p)$-modules $K := \opn{RHom}_{A}( \bk(\p), M)$
we have $B \ot_A K \cong C \ot_{\bk(\p)} K$. 
Since $M$ has injective concentration inside the integer interval $[d_0, d_1]$, 
it follows that the cohomology of the complex $K$ is inside $[d_0, d_1]$.
The isomorphism $\cong^{(5)}$ is another instance of derived 
tensor-evaluation.

The ring $C$ is essentially smooth over $\bk(\p)$ of relative differential 
dimension $\leq e$. Theorem \ref{thm:1015} implies that $C$ is a 
regular local ring of dimension $\leq e$. Therefore the cohomology of the 
complex $\opn{RHom}_{C}(\bk(\q), C)$ is concentrated inside the integer 
interval $[0, e]$. We conclude that the cohomology of the complex 
(\ref{eqn:1625}) is concentrated in the interval 
$[d_0, d_1] + [0, e] = [d_0, d_1 + e]$.
\end{proof}

\begin{thm} \label{thm:1625}
Let $u : A \to B$ be be an essentially smooth ring homomorphism.
\begin{enumerate}
\item Let $R \in \cat{D}^{\mrm{b}}_{\mrm{f}}(A)$ be a dualizing complex over 
$A$. Then $B \ot_A R \in \cat{D}^{\mrm{b}}_{\mrm{f}}(B)$ is a dualizing complex 
over $B$.

\item Let $(R_A, \rho_A)$ be the rigid dualizing complex of $A$ relative to 
$\K$, and assume that $B$ has constant differential relative dimension $n$ over 
$A$. Then the rigid complex 
\[ \opn{TwInd}^{\mrm{rig}}_{u / \K}(R_A, \rho_A)
= \bigl( R_A \ot_A \Om^n_{B / A}[n], \rho_A \cupprod \rho_{u}^{\mrm{esm}} \bigr)
\]
from Definition \ref{dfn:1235} is a rigid dualizing complex over $B$ relative 
to $\K$. 
\end{enumerate}
\end{thm}

\begin{proof} \mbox{}
\smallskip \noindent
(1) This is \cite[Theorem V.8.3]{RD}, but the proof given there is very 
complicated. We give an easier proof, communicated to us by L. Shaul. 
Lemma \ref{lem:1625} says that $B \ot_A R$ has finite injective dimension over 
$B$. It remains to verify that $B \ot_A R$ has the derived Morita property over 
$B$, and here is the calculation:
\begin{equation} \label{eqn:1626}
\begin{aligned}
&
\opn{RHom}_B(B \ot_A R, B \ot_A R) \cong^{\mrm{(i)}} 
\opn{RHom}_A(R, B \ot_A R) 
\\ & \quad 
\cong^{\mrm{(ii)}}  B \ot_A \opn{RHom}_A(R, R) \cong^{\mrm{(iii)}}  B .
\end{aligned} 
\end{equation}
The isomorphism $\cong^{\mrm{(i)}}$ is by adjunction for the ring homomorphism 
$u$; the isomorphism $\cong^{\mrm{(ii)}}$ is by derived tensor-evaluation; 
and the isomorphism $\cong^{\mrm{(iii)}}$ is due to the derived Morita 
property of $R$. The isomorphisms in (\ref{eqn:1626}) respect the derived 
homothety morphisms from $B$. 

\medskip \noindent
(2) We need to show that the complex 
$R_B := R_A \ot_A \Om^n_{B / A}[n] \in \cat{D}^{\mrm{b}}_{\mrm{f}}(B)$ 
is dualizing. We know by item (1) that $R_A \ot_A B$ is a dualizing complex 
over $B$. Since $\Om^n_{B / A}$ is an invertible $B$-module, the same is true 
for $R_B$. 
\end{proof}

\begin{cor} \label{cor:1705}
Let $u : A \to B$ be be an essentially smooth ring homomorphism of constant 
relative differential dimension $n$. There is a unique isomorphism 
\[ (R_B, \rho_B) \cong (R_A, \rho_A) 
\ot^{\mrm{L, rig}}_{B / A/ \K} 
\bigl( \Om^n_{B / A}[n], \rho_{B / A}^{\mrm{esm}} \bigr) \]
in $\cat{D}(B)_{\mrm{rig} / \K}$. Here 
$(- \ot^{\mrm{L, rig}}_{B / A/ \K} -)$ is the derived tensor product of rigid 
complexes from Definition \ref{dfn:1260}.

In particular we get a canonical isomorphism 
$R_B \cong R_A \ot_{B} \Om^n_{B / A}[n]$
in $\cat{D}(B)$. 
\end{cor}

\begin{proof}
This is a consequence of Definition \ref{dfn:1260}, Theorem \ref{thm:1550} and Theorem \ref{thm:1625}(2).
\end{proof}

\section{The Twisted Induction Pseudofunctor} 
\label{sec:twisted}

Here we use the rigid dualizing complexes from the previous section to 
construct the {\em twisted induction pseudofunctor}. Its geometric counterpart, 
the {\em twisted inverse image pseudofunctor}, is discussed in Remark 
\ref{rem:1555}. We prove that the twisted induction pseudofunctor has specific 
behaviors with respect to finite and to essentially smooth ring homomorphisms. 

Throughout this section Convention \ref{conv:1550} will be in force, namely all 
rings are EFT over the regular base ring $\K$.  

\begin{dfn}[Rigid Autoduality Functor] \label{dfn:1556}
Let $A$ be an EFT $\K$-ring, with rigid dualizing complex $(R_A, \rho_A)$ 
relative to $\K$. The {\em rigid autoduality functor} of $A$ relative to $\K$ 
is the functor 
\[ \opn{D}_{A}^{\mrm{rig}} : \cat{D}_{\mrm{f}}(A)^{\mrm{op}} \to 
\cat{D}_{\mrm{f}}(A) , \quad
\opn{D}_{A}^{\mrm{rig}} := \opn{RHom}_{A}(-, R_A) . \]
\end{dfn}

\begin{prop} \label{prop:1556}
Let $A$ be an EFT $\K$-ring. 
\begin{enumerate}
\item The derived Hom-evaluation morphism 
$\opn{ev} : \opn{Id} \to \opn{D}_{A}^{\mrm{rig}} \circ \opn{D}_{A}^{\mrm{rig}}$ 
is an isomorphism of triangulated functors from $\cat{D}_{\mrm{f}}(A)$ to 
itself.

\item For every boundedness condition $\star$, with negated 
boundedness condition $-\star$, the functor 
$\opn{D}_{A}^{\mrm{rig}} : \cat{D}^{\star}_{\mrm{f}}(A)^{\mrm{op}} \to 
\cat{D}_{\mrm{f}}^{-\star}(A)$ 
is an equivalence of triangulated categories. 
\end{enumerate}
\end{prop}

\begin{proof}
These assertions are true for any dualizing complex over $A$. See \cite{RD}, or 
\cite[Theorem 13.1.18 and Corollary 13.1.19]{Ye5}.
\end{proof}

Recall that to a homomorphism $u : A \to B$ in $\cat{Rng} \eftover \K$
we associate the left derived induction functor
\begin{equation} \label{eqn:1610}
\opn{LInd}_u : \cat{D}_{A} \to \cat{D}_{B}, \quad 
\opn{LInd}_u = B \ot^{\mrm{L}}_{A} (-) .
\end{equation}

\begin{prop} \label{prop:1557}
Let $u : A \to B$ be a homomorphism in $\cat{Rng} \eftover \K$. The 
left derived induction functor restricts to a triangulated functor 
\[  \opn{LInd}_u : \cat{D}_{\mrm{f}}^{-}(A) \to 
\cat{D}_{\mrm{f}}^{-}(B) . \]
\end{prop}

\begin{proof}
Every complex $M \in \cat{D}_{\mrm{f}}^{-}(A)$ admits a 
resolution $P \to M$, where $P$ is a bounded above complex of finitely 
generated free $A$-modules. Cf.\ \cite[Section 11.4]{Ye5}, where such a complex 
$P$ is called pseudo-finite semi-free. Then 
$\opn{LInd}_u(M) \cong B \ot_A P \in\cat{D}_{\mrm{f}}^{-}(B)$. 
\end{proof}

In the next few paragraphs we shall talk about $2$-categories. 
It will be convenient to denote $2$-morphisms in a $2$-category
$\cat{D}$ by $\twoto$,
and vertical composition of $2$-morphisms by $\ast$, as was done in 
\cite[Section 1]{Ye1}.

\begin{dfn} \label{dfn:1555}
The set of $\K$-linear triangulated categories is denoted by 
$\cat{TrCat} \over \K$. It is the set of objects of a strict $2$-category, 
in which the $1$-morphisms are the $\K$-linear triangulated functors 
$F : \cat{C} \to \cat{D}$, and the $2$-morphisms 
$\ze : F \twoto G$ are the morphisms of triangulated functors. See 
\cite[Sections 5.1-5.2]{Ye5} for details. 
\end{dfn}

Let's recall the notion of {\em normalized pseudofunctor} 
$F : \cat{C} \to \cat{D}$ from a category $\cat{C}$ to a (strict) $2$-category 
$\cat{D}$, following \cite[Section 1]{Ye1}. 
The pseudofunctor $F$ has three 
components. To every object $C \in \cat{C}$, $F$ assigns an object $F(C) \in 
\cat{D}$. To every morphism 
$f : C_0 \to C_1$ in $\cat{C}$ it assigns a $1$-morphism 
$F(f) : F(C_0) \to F(C_1)$ in $\cat{D}$. And to every composable pair of 
morphisms $C_0 \xar{f_1} C_1 \xar{f_2} C_2$ in $\cat{C}$ it assigns a 
$2$-isomorphism 
$F(f_1, f_2) : F(f_2) \circ F(f_1) \twoiso F(f_2 \circ f_1)$ in $\cat{D}$.  
The normalization condition says that for every $C$ there is an equality of 
$1$-morphisms $F(\opn{id}_C) = \opn{id}_{F(C)}$, and for every 
$C_0 \xar{f} C_1$ there are equalities of $2$-morphisms 
$F(\opn{id}_{C_0}, f) = F(f, \opn{id}_{C_1}) = \opn{id}_{F(f)}$.  
Finally there is the coherence condition: for every 
$C_0 \xar{f_1} C_1 \xar{f_2} C_2 \xar{f_3} C_3$
there is equality 
\begin{equation} \label{eqn:1607}
F(f_2 \circ f_1, f_3) \ast F(f_1, f_2) = 
F(f_1, f_3 \circ f_2) \ast F(f_2, f_3) 
\end{equation}
of $2$-isomorphisms 
\[ F(f_3) \circ F(f_2) \circ F(f_1) \twoiso F(f_3 \circ f_2 \circ f_1) . \]

Note that given functors $F_1, G_1 : \cat{C}_0 \to \cat{C}_1$
and $F_2, G_2 : \cat{C}_1 \to \cat{C}_2$,
and morphisms of functors 
$\al_i : F_i \twoto G_i$, there is equality 
\begin{equation} \label{eqn:1613}
\al_2 \circ \al_1 = 
(\opn{id}_{G_2} \circ \, \al_1) \ast (\al_2 \circ \opn{id}_{F_1}) 
= (\al_2 \circ \opn{id}_{G_1}) \ast (\opn{id}_{F_2} \circ \, \al_1) 
\end{equation}
of morphisms
$F_2 \circ F_1 \twoto G_2 \circ G_1$
of functors $\cat{C}_0 \to \cat{C}_2$. This is the {\em exchange condition}.

An unnormalized pseudofunctor $F$ has a $2$-isomorphism 
$\al_C : F(\opn{id}_C) \twoiso \opn{id}_{F(C)}$
instead of equality, and these isomorphisms $\al_C$ need to satisfy some 
conditions. See \cite[Definition tag = \texttt{003N}]{SP} for details.
It is always possible to normalize a pseudofunctor $F$, by declaring 
$F(\opn{id}_C) := \opn{id}_{F(C)}$, and changing the $2$-isomorphisms 
$F(\opn{id}_{C_0}, f)$ and $F(f, \opn{id}_{C_1})$ using the given 
$2$-isomorphisms $\al_{C_0}$ and $\al_{C_1}$. 
We now stop using the special $2$-categorical symbols. 

Given ring homomorphisms $A \to B \to C$, there is a unique isomorphism 
\begin{equation} \label{eqn:1605}
\ga : C \ot^{\mrm{L}}_{B} (B \ot^{\mrm{L}}_{A} (-)) \iso C \ot^{\mrm{L}}_{A} (-)
\end{equation}
of triangulated functors 
$\cat{D}(A) \to \cat{D}(C)$, such that for every K-flat complex of 
$A$-modules $P$, the isomorphism 
$\ga : C \ot_{B} (B \ot_{A} P) \iso C \ot_{A} P$
is $\ga(c \ot b \ot p) = (c \cd b) \ot p$. 

\begin{prop} \label{prop:1605}
There is a unique normalized pseudofunctor 
\[ \opn{LInd} : \cat{Rng} \eftover \K \to \cat{TrCat} \over \K , \]
called {\em left derived induction}, with these properties:
\begin{enumerate}
\item To an object $A \in \cat{Rng} \eftover \K$ it assigns the category
$\opn{LInd}(A) := \cat{D}^{-}_{\mrm{f}}(A) \in \cat{TrCat} \over \K$.

\item To a morphism $u : A \to B$ in $\cat{Rng} \eftover \K$ it assigns the 
triangulated functor  
$\opn{LInd}(u) := \opn{LInd}_{u}$
from formula (\ref{eqn:1610}). 
Except if $B = A$ and $u = \opn{id}_A$, where we take 
$\opn{LInd}(\opn{id}_A) := \opn{Id}_{\opn{LInd}(A)}$. 

\item To composable homomorphisms $u : A \to B$ and $v : B \to C$ in 
$\cat{Rng} \eftover \K$, the isomorphism of triangulated functors 
\[ \opn{LInd}(u, v) : \opn{LInd}(v) \circ \opn{LInd}(u) 
\iso \opn{LInd}(v \circ u) \]
is from formula (\ref{eqn:1605}). Except if $u = \opn{id}_A$ or 
$v = \opn{id}_B$, where we take 
$\opn{LInd}(u, v) := \opn{id}_{\opn{LInd}(v \circ u)}$. 
\end{enumerate}
\end{prop}

\begin{proof}
The only nontrivial verification is that the coherence condition 
(\ref{eqn:1607}) holds. Let 
$A \to B \to C \to D$ be ring homomorphisms. For a K-flat complex of 
$A$-modules $P$, both composed morphisms send 
$d \ot c \ot b \ot p \mapsto (d \cd c \cd b) \ot p$, so they are equal.
\end{proof}

\begin{thm}[Twisted Induction] \label{thm:1555}
There is a unique normalized pseudofunctor 
\[ \opn{TwInd} : \cat{Rng} \eftover \K \to \cat{TrCat} \over \K , \]
called {\em twisted induction}, with these properties:
\begin{enumerate}
\item To an object $A \in \cat{Rng} \eftover \K$ it assigns the category
\[ \opn{TwInd}(A) := \cat{D}^{+}_{\mrm{f}}(A) \in \cat{TrCat} \over \K . \]

\item To a morphism $u : A \to B$ in $\cat{Rng} \eftover \K$ it assigns the 
triangulated functor 
\[ \opn{TwInd}(u) = \opn{TwInd}_{u} := 
\opn{D}_{B}^{\mrm{rig}} \circ \opn{LInd}_u \circ \opn{D}_{A}^{\mrm{rig}} . \]
Except if $B = A$ and $u = \opn{id}_A$, where we take 
$\opn{TwInd}(\opn{id}_A) := \opn{Id}_{\opn{TwInd}(A)}$. 

\item To composable homomorphisms $u : A \to B$ and $v : B \to C$ in 
$\cat{Rng} \eftover \K$, the isomorphism of triangulated functors 
\[ \opn{TwInd}(u, v) = \opn{TwInd}_{u, v} : 
\opn{TwInd}_{v} \circ \opn{TwInd}_{u} \iso 
\opn{TwInd}_{v \circ u} \]
is gotten by first applying the inverse of the evaluation isomorphism 
$\opn{ev}_B : \opn{Id} \to \opn{D}_{B}^{\mrm{rig}} \circ 
\opn{D}_{B}^{\mrm{rig}}$, 
and then applying the isomorphism $\opn{LInd}(u, v)$ from 
Proposition \lb \ref{prop:1605}(3). 
Except if $u = \opn{id}_A$ or $v = \opn{id}_B$, where we take 
$\opn{TwInd}(u, v) := \opn{id}_{\opn{TwInd}(v \circ u)}$. 
\end{enumerate}
\end{thm}

\begin{proof}
Again, the only nontrivial verification is that the coherence condition 
(\ref{eqn:1607}) holds. Let 
$A_0 \xar{u_1} A_1 \xar{u_2} A_2 \xar{u_3} A_3$ be ring homomorphisms. 
Consider the diagram of functors and isomorphisms between them.  
\begin{equation} \label{eqn:1611}
\begin{tikzcd} [column sep = 0ex, row sep = 5ex]
&
\begin{minipage}[]{12em}
$\opn{D}_{A_3} \circ \opn{LI}_{u_3} \circ \opn{D}_{A_2} 
\circ \opn{D}_{A_2} \circ \opn{LI}_{u_2} \\
\circ \opn{D}_{A_1} \circ \opn{D}_{A_1} 
\circ \opn{LI}_{u_1} \circ 
\opn{D}_{A_0}$ 
\end{minipage}
\ar[dl, "{\opn{ev}_{A_2}^{-1}}"']
\ar[dr, "{\opn{ev}_{A_1}^{-1}}"]
\\
\begin{minipage}[]{10em}
$\opn{D}_{A_3} \circ \opn{LI}_{u_3} \circ \opn{LI}_{u_2} \\
\circ \opn{D}_{A_1} \circ \opn{D}_{A_1} \circ  \opn{LI}_{u_1} \circ 
\opn{D}_{A_0}$ 
\end{minipage}
\ar[d, "{\opn{LI}_{u_2, u_3}}"']
\ar[dr, "{\opn{ev}_{A_1}^{-1}}"]
&
&
\begin{minipage}[]{10em}
$\opn{D}_{A_3} \circ \opn{LI}_{u_3} \circ \opn{D}_{A_2} \\ 
\circ \opn{D}_{A_2} \circ \opn{LI}_{u_2}  \circ  \opn{LI}_{u_1} \circ 
\opn{D}_{A_0}$ 
\end{minipage}
\ar[d, "{\opn{LI}_{u_1, u_2}}"]
\ar[dl, "{\opn{ev}_{A_2}^{-1}}"']
\\
\begin{minipage}[]{8em}
$\opn{D}_{A_3} \circ \opn{LI}_{u_3 \circ u_2} \circ \opn{D}_{A_1} \\
\circ \opn{D}_{A_1} \circ  \opn{LI}_{u_1} \circ \opn{D}_{A_0}$ 
\end{minipage}
\ar[d, "{\opn{ev}_{A_1}^{-1}}"]
&
\begin{minipage}[]{8em}
$\opn{D}_{A_3} \circ \opn{LI}_{u_3} \circ \opn{LI}_{u_2} \\
\circ  \opn{LI}_{u_1} \circ \opn{D}_{A_0}$ 
\end{minipage}
\ar[dl, "{\opn{LI}_{u_2, u_3}}"]
\ar[dr, "{\opn{LI}_{u_1, u_2}}"']
&
\begin{minipage}[]{9em}
$\opn{D}_{A_3} \circ \opn{LI}_{u_3} \circ \opn{D}_{A_2} \\ 
\circ \opn{D}_{A_2} \circ \opn{LI}_{u_2 \circ u_1}  \circ \opn{D}_{A_0}$ 
\end{minipage}
\ar[d, "{\opn{ev}_{A_2}^{-1}}"']
\\
\begin{minipage}[]{9em}
$\opn{D}_{A_3} \circ \opn{LI}_{u_3 \circ u_2} \\
\circ  \opn{LI}_{u_1} \circ \opn{D}_{A_0}$ 
\end{minipage}
\ar[dr, "{\opn{LI}_{u_1, u_3 \circ u_2}}"]
&
&
\begin{minipage}[]{9em}
$\opn{D}_{A_3} \circ \opn{LI}_{u_3} \\ 
\circ \opn{LI}_{u_2 \circ u_1}  \circ \opn{D}_{A_0}$ 
\end{minipage}
\ar[dl, "{\opn{LI}_{u_2 \circ u_1, u_3}}"']
\\
&
\begin{minipage}[]{9em}
$\opn{D}_{A_3} \circ \opn{LI}_{u_3 \circ u_2 \circ u_1} \circ \opn{D}_{A_0}$ 
\end{minipage}
\end{tikzcd} 
\end{equation}
We are using the abbreviations 
$\opn{D}_{A_i} := \opn{D}^{\mrm{rig}}_{A_i}$,
$\opn{LI}_{u_i} := \opn{LInd}_{u_i}$ and 
$\opn{LI}_{u_i, u_j} := \opn{LInd}(u_i, u_j)$. 
The path along the left edge of this diagram is the isomorphism 
$\opn{TwInd}(u_1, u_3 \circ u_2) \circ \opn{TwInd}(u_2, u_3)$, 
and the path along the right edge is the isomorphism 
$\opn{TwInd}(u_2 \circ u_1, u_3) \circ \opn{TwInd}(u_1, u_2)$; 
we need to prove they are equal. By the exchange property the top diamond and 
middle two triangles in diagram (\ref{eqn:1611}) are commutative. By 
Proposition \ref{prop:1605}, i.e.\ by the coherence property of the 
pseudofunctor $\opn{LInd}$,  the bottom diamond is commutative. Therefore the 
whole diagram is commutative. 
\end{proof}

\begin{lem} \label{lem:1620}
Suppose $R \in \cat{D}^{\mrm{b}}_{\mrm{f}}(A)$ is some dualizing complex, and 
define the functor $\opn{D} := \opn{RHom}_A(-, R)$. 
Given $M, N \in \cat{D}^{}_{}(A)$, there is a morphism 
\[ \tau_{M, N} : \opn{RHom}_{A}(M, N) \iso 
\opn{RHom}_{A} \bigl( \opn{D}(N), \opn{D}(M) \bigr) \]
in $\cat{D}(A)$, which is functorial is $M$ and $N$, and such that 
the diagram 
\[ \begin{tikzcd} [column sep = 10ex, row sep = 5ex] 
\opn{H}^0 \bigl( \opn{RHom}_{A}(M, N) \bigr)
\ar[r, "{\opn{H}^0(\tau_{M, N})}"]
\ar[d, "{}"', "{\simeq}"]
&
\opn{H}^0 \bigl( \opn{RHom}_{A} \bigl( \opn{D}(N), \opn{D}(M) \bigr) \bigr)
\ar[d, "{}", "{\simeq}"']
\\
\opn{Hom}_{\cat{D}(A)}(M, N)
\ar[r, "{\opn{D}}"]
&
\opn{Hom}_{\cat{D}(A)} \bigl( \opn{D}(N), \opn{D}(M) \bigr)
\end{tikzcd} \]
is commutative. If $M, N \in \cat{D}^{}_{\mrm{f}}(A)$,
then $\tau_{M, N}$ is an isomorphism. 
\end{lem}

\begin{proof}
Fix a K-injective resolution $R \to K$, so 
$\opn{D} \cong \opn{Hom}_A(-, K)$. Given $M$ and $N$, let 
$N \to J$ be a K-injective resolution, and let $P \to M$ be a K-projective 
resolution. The complex of $A$-modules 
$\opn{Hom}_A(P, K)$ is also K-injective. The morphism $\tau_{M, N}$ is defined 
to be the composition of these morphisms: 
\[ \begin{aligned}
& \opn{RHom}_{A}(M, N) \iso \opn{Hom}_{A}(P, J) 
\xar{\til{\tau}} 
\opn{Hom}_{A} \bigl( \opn{Hom}_{A}(J, K), \opn{Hom}_{A}(P, K) \bigr)
\\ & \qquad 
\iso \opn{RHom}_{A} \bigl( \opn{D}(N), \opn{D}(M) \bigr) , 
\end{aligned} \]
where $\til{\tau} := \opn{Hom}_A(-, K)$. For a morphism 
$\phi : M \to  N$ in $\cat{D}(A)$, 
represented by a homomorphism $\til{\phi} : P \to J$
in $\cat{C}_{\mrm{str}}(A)$, 
the morphism 
$\opn{D}(\phi) : \opn{D}(N) \to \opn{D}(M)$ is represented by 
$\til{\tau}(\til{\phi})$. Therefore the diagram is commutative. 
Functoriality of $\tau_{M, N}$ is clear. 

If $M, N \in \cat{D}^{}_{\mrm{f}}(A)$, then 
\[ \opn{D} : \opn{Hom}_{\cat{D}(A)}(M, N) \to 
\opn{Hom}_{\cat{D}(A)} \bigl( \opn{D}(N), \opn{D}(M) \bigr) \]
is bijective; see Proposition \ref{prop:1556} and its proof. Hence 
$\opn{H}^0(\tau_{M, N})$ is an isomorphism. 
By replacing $N$ with $N[q]$, we see that $\opn{H}^q(\tau_{M, N})$ is an 
isomorphism for every $q$. Therefore $\tau_{M, N}$ is an isomorphism. 
\end{proof}

\begin{prop} \label{prop:1610}
Let $u : A \to B$ be a finite ring homomorphism. Given a complex 
$M \in \cat{D}^{+}_{\mrm{f}}(A)$, there is an 
isomorphism 
\[ \opn{conc}^{\mrm{fin}}_{u, M} : \opn{TwInd}_u(M) \iso \opn{RHom}_A(B, M) \]
in $\cat{D}(B)$. The isomorphism $\opn{conc}^{\mrm{fin}}_{u, M}$ is functorial 
in $M$. 
\end{prop}

The notation $\opn{conc}^{\mrm{fin}}_{u, M}$ stands for ``concretization'' -- 
it is the concrete incarnation of the twisted induction functor in the finite 
case. Similar notation will be used for other kinds of homomorphisms $u$. 

\begin{proof}
The isomorphism $\opn{conc}^{\mrm{fin}}_{u, M}$
is the composition of the following isomorphisms: 
\[ \begin{aligned}
&
\opn{TwInd}_u(M) = 
\opn{RHom}_{B} \bigl( B \ot^{\mrm{L}}_{A} \opn{RHom}_{A}(M, R_A), R_B \bigr)
\\ & \quad 
\iso^{\mrm{(1)}} 
\opn{RHom}_{A} \bigl( \opn{RHom}_{A}(M, R_A), R_B \bigr)
\\ & \quad 
\iso^{\mrm{(2)}} 
\opn{RHom}_{A} \bigl( \opn{RHom}_{A}(M, R_A), 
\opn{RHom}_{A}(B, R_A) \bigr)
\iso^{\mrm{(3)}} 
\opn{RHom}_{A}(B, M) 
\end{aligned} \]
Here $\iso^{\mrm{(1)}}$ is adjunction for the ring homomorphism 
$A \to B$; the isomorphism $\iso^{\mrm{(2)}}$ is the one coming from the 
isomorphism 
\[ \opn{badj}^{\mrm{R}}_{u, M, N}(\opn{tr}^{\mrm{rig}}_{u}) : 
R_B \iso \opn{RHom}_{A}(B, R_A) , \]
where $\opn{tr}^{\mrm{rig}}_{u} : R_B \to R_A$ is the nondegenerate rigid 
backward morphism from Theorem \ref{thm:1551}; and 
the isomorphism $\iso^{\mrm{(3)}}$ is from Lemma \ref{lem:1620}.
\end{proof}

\begin{dfn} \label{dfn:1610}
Let $u : A \to B$ be a finite ring homomorphism.
For  $M \in \cat{D}^{+}_{\mrm{f}}(A)$
define the morphism 
\[ \opn{tr}^{\mrm{twind}}_{u, M} : \opn{TwInd}_u(M) \to  M \]
in $\cat{D}(A)$ to be the unique morphism making the diagram 
\[ \begin{tikzcd} [column sep = 10ex, row sep = 5ex] 
\opn{TwInd}_u(M)
\ar[r, "{\opn{conc}^{\mrm{fin}}_{u, M}}", "{\simeq}"']
\ar[d, "{\opn{tr}^{\mrm{twind}}_{u, M}}"', dashed]
&
\opn{RHom}_A(B, M) 
\ar[d, "{\opn{tr}^{\mrm{R}}_{u, M}}"]
\\
M
\ar[r, "{\opn{id}_M}", "{\simeq}"']
&
M
\end{tikzcd} \]
commutative. Here $\opn{tr}^{\mrm{R}}_{u, M}$ is the standard nondegenerate 
backward morphism, see formula (\ref{eqn:655}). 
\end{dfn}

Note that 
$\opn{tr}^{\mrm{twind}}_{u} : \opn{Rest}_u \circ \opn{TwInd}_u \to \opn{Id}$
is a morphism of functors from $\cat{D}^{+}_{\mrm{f}}(A)$ to itself.  

\begin{prop} \label{prop:1620}
Let $u : A \to B$ be a finite ring homomorphism, and let 
$M \in \cat{D}^{+}_{\mrm{f}}(A)$. Then 
$\opn{tr}^{\mrm{twind}}_{u, M} : \opn{TwInd}_u(M) \to M$
is a nondegenerate backward morphism in $\cat{D}(A)$ over $u$. 
\end{prop}

\begin{proof}
This is because $\opn{tr}^{\mrm{R}}_{u, M}$ is a nondegenerate backward 
morphism.
\end{proof}

The next proposition gives a concrete formulation of twisted induction in the 
essentially smooth case. 

\begin{prop} \label{prop:1611} 
Let $u : A \to B$ be an essentially smooth ring 
homomorphism,  of constant differential relative dimension $n$.
Given a complex 
$M \in \cat{D}^{+}_{\mrm{f}}(A)$, there is an isomorphism
\[ \opn{conc}^{\mrm{esm}}_{u, M} : 
\opn{TwInd}_u(M) \iso M \ot_A \Om^n_{B / A}[n] \]
in $\cat{D}^{+}_{\mrm{f}}(B)$. The 
isomorphism $\opn{conc}^{\mrm{esm}}_{u, M}$ is functorial in $M$.

In particular, if $u$ is essentially \'etale, this becomes an isomorphism 
\[ \opn{conc}^{\mrm{eet}}_{u, M} : \opn{TwInd}_u(M) \iso M \ot_A B \]
in $\cat{D}^{+}_{\mrm{f}}(B)$.
\end{prop}

\begin{proof}
The isomorphism $\opn{conc}^{\mrm{esm}}_{u, M}$ is the composition of the 
following isomorphisms: 
\[ \begin{aligned}
& \opn{TwInd}_u(M) = 
(\opn{D}_{B}^{\mrm{rig}} \circ \opn{LInd}_u \circ \opn{D}_{A}^{\mrm{rig}})(M) 
= \opn{RHom}_{B} \bigl(B \ot^{\mrm{L}}_{A} \opn{RHom}_{A}(M, R_A), R_B \bigr)
\\
& \quad \iso^{\mrm{(1)}}
\opn{RHom}_{A} \bigl(\opn{RHom}_{A}(M, R_A), R_B \bigr)
\\
& \quad \iso^{\mrm{(2)}}
\opn{RHom}_{A} \bigl(\opn{RHom}_{A}(M, R_A), R_A \ot^{\mrm{L}}_{B} 
\Om^n_{B / A}[n] \bigr)
\\
& \quad \iso^{\mrm{(3)}}
\opn{RHom}_{A} \bigl(\opn{RHom}_{A}(M, R_A), R_A \bigr)
\ot^{\mrm{L}}_{B} \Om^n_{B / A}[n] 
\\
& \quad = (\opn{D}_{A}^{\mrm{rig}} \circ \opn{D}_{A}^{\mrm{rig}})(M)
\ot^{\mrm{L}}_{B} \Om^n_{B / A}[n] 
\iso^{\mrm{(4)}} M \ot^{\mrm{L}}_{B} \Om^n_{B / A}[n] .
\end{aligned} \]
The isomorphism $\iso^{\mrm{(1)}}$ is Hom-tensor adjunction for the ring 
homomorphism $u$; the isomorphism $\iso^{\mrm{(2)}}$ comes from Corollary 
\ref{cor:1705}; the isomorphism $\iso^{\mrm{(3)}}$ is derived 
tensor-evaluation, see \cite[Theorem 12.10.14]{Ye5}; and 
the isomorphism $\iso^{\mrm{(4)}}$ is due to the 
Hom-evaluation isomorphism 
$\opn{ev} : \opn{Id} \to \opn{D}_{A}^{\mrm{rig}} \circ \opn{D}_{A}^{\mrm{rig}}$ 
from Proposition \ref{prop:1556}.

According to Proposition \ref{prop:1055}, $u$ is essentially \'etale iff it is 
essentially smooth of relative dimension $0$.
\end{proof}

\begin{dfn} \label{dfn:1650}
Let $u : A \to B$ be an essentially \'etale ring homomorphism, and let 
$M \in \cat{D}^{+}_{\mrm{f}}(A)$. Define 
\[ \opn{q}^{\mrm{twind}}_{u, M} : M \to \opn{TwInd}_{u}(M) \] 
to be the unique forward morphism in $\cat{D}(A)$ over $u$ such that the 
diagram 
\[ \begin{tikzcd} [column sep = 10ex, row sep = 6ex] 
M
\ar[r, "{\opn{id}_M}", "{\simeq}"']
\ar[d, "{\opn{q}^{\mrm{twind}}_{u, M}}"', dashed]
&
M
\ar[d, "{\opn{q}^{}_{u, M}}"]
\\
\opn{TwInd}_u(M)
\ar[r, "{\opn{conc}^{\mrm{eet}}_{u, M}}", "{\simeq}"']
&
M \ot_A B 
\end{tikzcd} \]
commutative. Here $\opn{q}^{}_{u, M}$ is the standard nondegenerate 
forward morphism from equation (\ref{eqn:1476}). 
\end{dfn}

\begin{prop} \label{prop:1612} 
If $u : A \to B$ is an essentially \'etale ring homomorphism, then  
for every $M \in \cat{D}^{+}_{\mrm{f}}(A)$ the morphism
\[ \opn{q}^{\mrm{twind}}_{u, M} : M \to \opn{TwInd}_{u}(M) \] 
is a nondegenerate forward morphism in $\cat{D}(A)$ over $u$. 
\end{prop}

\begin{proof}
This is because $\opn{q}^{}_{u, M}$ is a nondegenerate forward 
morphism.
\end{proof}

\begin{prop} \label{prop:1614} 
Assume that $u : A \to B$ is a finite \'etale ring homomorphism, such that 
$B$ has constant rank $r$ as an $A$-module. Then for every 
$M \in \cat{D}^{+}_{\mrm{f}}(A)$ the composed morphism
\[ M \xar{\opn{q}^{\mrm{twind}}_{u, M}} \opn{TwInd}_u(M) 
\xar{\opn{tr}^{\mrm{twind}}_{u, M}} M \]
equals multiplication by $r$.  
\end{prop}

\begin{proof}
Since $u$ is a finite homomorphism, by using Proposition \ref{prop:1610}, 
we know that $\opn{TwInd}_{u}(M)\simeq\opn{RHom}_{A}(B, M)$.
Moreover $u$ is \'etale so $B$ is a projective $A$-modules hence 
$\opn{RHom}_{A}(B, M)\simeq \opn{Hom}_{A}(B, M)$.
For every prime $\mathfrak{p} \in A$ we get a finite and \'etale homomorphism 
of $A$-modules $u_{\mathfrak{p}}:A_{\mathfrak{p}} \to B_{\mathfrak{p}}$ 
where $B_{\mathfrak{p}}:=B\ot_{A}A_{\mathfrak{p}}$.
By Corollary \ref{cor:1245} we have that $B_{\mathfrak{p}}$ 
is a finite projective $A_{\mathfrak{p}}$.\
Since $A_{\mathfrak{p}}$ is local by Nakayama Lemma we have that 
$B_{\mathfrak{p}}$ is a free $A_{\mathfrak{p}}$-module of rank $r$, 
in formula $B_{\mathfrak{p}}\simeq\oplus_rA_{\mathfrak{p}}$. 
So $\opn{TwInd}_{u_\mathfrak{p}}(A_\mathfrak{p})=B_\mathfrak{p}\simeq \oplus_{r} A_\mathfrak{p}$, 
and we have:
\[ \begin{aligned}
A_{\mathfrak{p}} \xar{\opn{q}^{\mrm{twind}}_{u_{A_{\mathfrak{p}}}, A_{\mathfrak{p}}}}& \opn{TwInd}_{u_{\mathfrak{p}}}(A_{\mathfrak{p}}) 
\xar{\opn{tr}^{\mrm{twind}}_{u_{A_{\mathfrak{p}}}, A_{\mathfrak{p}}}} A_{\mathfrak{p}}, \\
a\mapsto &\oplus_r a \mapsto r\cdot a.
\end{aligned} \]
In the same vein we have $M_{\mathfrak{p}}:=M\otimes A_{\mathfrak{p}}$
and $\opn{TwInd}_{u_{\mathfrak{p}}}(M_{\mathfrak{p}})=M\ot_{A_{\mathfrak{p}}}B_{\mathfrak{p}}$.
On the other side the corresponding homomorphisms are given by 
$\opn{tr}^{\mrm{twind}}_{u_{A_{\mathfrak{p}}}}=\opn{tr}_{B_{\mathfrak{p}}}\ot_{A_{\mathfrak{p}}}\opn{id}_{A_{\mathfrak{p}}}$,
and $\opn{q}^{\mrm{twind}}_{u_{A_{\mathfrak{p}}}}(\overline{m})=\overline{m}\ot 1_{B_{\mathfrak{p}}}$. 
We get
\[ \begin{aligned}
M_{\mathfrak{p}} \xar{\opn{q}^{\mrm{twind}}_{u_{A_{\mathfrak{p}}},M_{\mathfrak{p}}}} &\opn{TwInd}_{u_{\mathfrak{p}}}(M_{\mathfrak{p}}) 
\xar{\opn{tr}^{\mrm{twind}}_{u_{A_{\mathfrak{p}}}, A_{\mathfrak{p}}}} M_{\mathfrak{p}},\\
m\mapsto &\oplus_r m\mapsto r\cdot m.
\end{aligned}
\]
To conclude the proof we consider the following commutative diagram:
\[ \begin{tikzcd} [column sep = 10ex, row sep = 6ex] 
M
\ar[d]
\ar[r, "{\opn{q}^{\mrm{twind}}_{u, M}}"]
&
\opn{TwInd}_{u}(M) 
\ar[d]
\ar[r, "{\opn{tr}^{\mrm{twind}}_{u,M}}"]
&
M
\ar[d]
\\
M_{\mathfrak{p}}
\ar[r, "{\opn{q}^{\mrm{twind}}_{u_{A_{\mathfrak{p}}}, M_{\mathfrak{p}}}}"]
&
\opn{TwInd}_{u_{\mathfrak{p}}}(M_{\mathfrak{p}}) 
\ar[r, "{\opn{tr}^{\mrm{twind}}_{u_{A_{\mathfrak{p}}}, M_{\mathfrak{p}}}}"]
&
M_{\mathfrak{p}}
\end{tikzcd} \]
Here the vertical arrows are the localizations at $\mathfrak{p}$.
The left square is commutative since $q$ commutes with the localizations and 
the right square is commutative by definition of $\opn{tr}^{\mrm{twind}}_{u_{A_{\mathfrak{p}}}, M_{\mathfrak{p}}}$ and we are done.
\end{proof}

\begin{rem} \label{rem:1785}
Let $u : A \to B$ be a finite ring homomorphism, let 
$v : A \to A'$ be an essentially \'etale ring homomorphism, and let
$B' := A' \ot_A B$, with induced finite ring homomorphism
$u' : A' \to B'$ and  induced essentially \'etale ring homomorphism
$w : B \to B'$. Theorem \ref{thm:1553} tells us that the rigid traces 
$\opn{tr}^{\mrm{rig}}_{(-)}$ and the rigid \'etale-localization morphisms 
$\opn{q}^{\mrm{rig}}_{(-)}$ commute. 

One would like to make a similar statement for the morphisms 
$\opn{tr}^{\mrm{twind}}_{(-, -)}$ and 
$\opn{q}^{\mrm{twind}}_{(-, -)}$
from Definitions \ref{dfn:1610} and \ref{dfn:1650}.
However, since $\opn{TwInd}$ is a pseudofunctor, such a formula becomes 
rather messy. Therefore we have decided not to include this result in the 
paper.
\end{rem}

\begin{rem} \label{rem:1786}
In Propositions \ref{prop:1610} and \ref{prop:1611}, the concretization 
isomorphisms $\opn{conc}^{\mrm{fin}}_{u, M}$ and $\opn{conc}^{\mrm{esm}}_{u, M}$
are pseudofunctorial in the ring homomorphism $u$. But stating and then proving 
these facts requires unpleasant formulas and long diagram chases; so we opted 
not to include these facts in the paper. 
\end{rem}

\begin{rem} \label{rem:1555}
Here is the geometric interpretation of Theorem \ref{thm:1555}. 
Consider a ring homomorphism $u : A \to B$. Let $X := \opn{Spec}(A)$,
$Y := \opn{Spec}(B)$ be the corresponding affine schemes, and let 
$f := \opn{Spec}(u) : Y \to X$ be the corresponding map of affine 
schemes. We know that the functor 
$\mrm{R} \Ga(X, -) : \cat{D}^{+}_{\mrm{c}}(X) \to \cat{D}^{+}_{\mrm{f}}(A)$
is an equivalence, and likewise for $\mrm{R} \Ga(Y, -)$. 
The {\em twisted inverse image functor} 
$f^! : \cat{D}^{+}_{\mrm{c}}(X) \to \cat{D}^{+}_{\mrm{c}}(Y)$
is defined by this commutative diagram: 
\[ \begin{tikzcd} [column sep = 8ex, row sep = 5ex] 
\cat{D}^{+}_{\mrm{c}}(X)
\ar[r, "{f^!}"]
\ar[d, "{\mrm{R} \Ga(X, -)}"', "{\approx}"]
&
\cat{D}^{+}_{\mrm{c}}(Y)
\ar[d, "{\mrm{R} \Ga(Y, -)}", "{\approx}"']
\\
\cat{D}^{+}_{\mrm{f}}(A)
\ar[r, "{\opn{TwInd}_u}"]
&
\cat{D}^{+}_{\mrm{f}}(B)
\end{tikzcd} \]
Theorem \ref{thm:1555} says that $f \mapsto f^!$ is a contravariant 
pseudofunctor from the category of affine strictly EFT $\K$-schemes to 
$\cat{TrCat} \over \K$. (An affine $\K$-scheme $X$ is called strictly EFT if 
$\Ga(X, \OO_X)$ is an EFT $\K$-ring.)

In the original book \cite{RD} the twisted inverse image pseudofunctor
$f \mapsto f^!$ from $\cat{Sch}$ to $\cat{TrCat}$
is constructed in a totally different way, geometrically. For a finite map of 
schemes $f : Y \to X$ there is the functor
\[ f^{\flat} : \cat{D}^{+}_{\mrm{c}}(X) \to \cat{D}^{+}_{\mrm{c}}(Y) \, , \quad 
f^{\flat}(\MM) := f^{-1} \bigl( \mrm{R} \Hom_{\OO_X}(f_*(\OO_Y), \MM) \bigr) . 
\]
For a smooth map $f : Y \to X$ of constant relative dimension $n$ there is the 
functor 
\[ f^{\sharp} : \cat{D}^{+}_{\mrm{c}}(X) \to \cat{D}^{+}_{\mrm{c}}(Y) \, , 
\quad f^{\sharp}(\MM) := f^{*}(\MM) \ot_{\OO_Y} \Om^n_{Y / X}[n] . \]
The immense difficulty in \cite{RD} is to prove that there is a pseudofunctor
$f^!$ such that $f^! \cong f^{\flat}$ when $f$ is finite, and 
$f^! \cong f^{\sharp}$ when $f$ is smooth. This is proved using global and 
local duality, and an enormous web of compatibilities. 

Our approach to constructing the twisted inverse image pseudofunctor
$f \mapsto f^!$ from $\cat{Sch} \eftover \K$ to $\cat{TrCat} / \K$
in the subsequent papers \cite{Ye6} and \cite{Ye7} is much easier, given the 
results of the present paper.
We first pass from rigid dualizing complexes $(R_A, \rho_A)$ to {\em rigid 
residue complexes} $(\KK_A, \rho_A)$, in the paper \cite{OY}. The assignment 
$A \mapsto \KK_A$ is a complex of quasi-coherent sheaves  on 
the big \'etale site of $\cat{Rng} \eftover \K$; this is a direct consequence of
Theorem \ref{thm:1552} and Corollary \ref{cor:1552}. Therefore the 
rigid residue complexes can be glued on schemes, and even on DM stacks.
Once we have a rigid residue complex $\KK_X$ on every scheme $X$ (or DM 
stack $\XX$), we can define the rigid auto-duality functor 
$\opn{D}_{X}^{\mrm{rig}} := \Hom_{\OO_X}(-, \KK_X)$. For a map of schemes (or 
DM stacks) $f : Y \to X$ we define 
\[ f^! := 
\opn{D}_{Y}^{\mrm{rig}} \circ \, \mrm{L} f^* \circ \opn{D}_{X}^{\mrm{rig}} :
\cat{D}^{+}_{\mrm{c}}(X) \to \cat{D}^{+}_{\mrm{c}}(Y) . \]
This is a pseudofunctor (the proof of Theorem \ref{thm:1555} works here too).
The isomorphisms $f^! \cong f^{\flat}$ and $f^! \cong f^{\sharp}$ come  
from the concretizations (Propositions \ref{prop:1610} and \ref{prop:1611} 
respectively). There is also the {\em ind-rigid trace homomorphism} 
$\opn{tr}^{\mrm{irig}}_{f} : f_*(\KK_Y) \to \KK_X$,
coming from the rigid trace 
$\opn{tr}^{\mrm{rig}}_{u}$ of Theorem \ref{thm:1551} and a limiting process, 
which eventually accounts for global duality when $f$ is proper. 
\end{rem}

\section{Rigid Relative Dualizing Complexes} 
\label{sec:relative}

In this final section we discuss a relative variant of rigid dualizing 
complexes. Given a ring homomorphism $u : A \to B$ of finite flat dimension, 
the complex $R_{B / A} := \opn{TwInd}_{u}(A) \in  \cat{D}(B)$
admits a rigidifying isomorphic $\rho_{B / A}$ relative to $A$, and the 
pair $(R_{B / A}, \rho_{B / A})$ is called the {\em relative rigid dualizing 
complex} of $B / A$. The rigid complex 
$(R_{B / A}, \rho_{B / A})$ has forward and backward functorialities, just like 
the absolute variant (i.e.\ when $A = \K$). 
We work out several examples, and indicate further applications of this 
relative construction.

Convention \ref{conv:1550} is assumed throughout this section; so there is a
regular noetherian base ring $\K$, and all rings and ring homomorphisms 
are in the category $\cat{Rng} \eftover \K$. 

The next definition will simplify many of our formulas. 

\begin{dfn} \label{dfn:1705}
Given a ring homomorphism $u : A \to B$, we write 
\[ u^! := \opn{TwInd}_u : \cat{D}^{+}_{\mrm{f}}(A) \to 
\cat{D}^{+}_{\mrm{f}}(B) , \]
where $\opn{TwInd}_u$ is the functor from Theorem \ref{thm:1555}(2). 
\end{dfn}

The notation $u^!$ is borrowed from the geometric setting, see 
Remark \ref{rem:1555}.

We are going to use various concepts of cohomological dimensions of functors 
and objects, following \cite[Sections 11.1 and 12.4]{Ye5}. 
Here is a brief recollection of them. By an {\em integer interval} we mean a 
subset 
$S = [i_0, i_1] := \{ i \in \Z \mid i_0 \leq i \leq i_1 \}$ 
for some  $-\infty \leq i_0 \leq i_1 \leq \infty$, or $S = \varnothing$. The 
{\em length} of $S = [i_0, i_1]$ is $i_1 - i_0 \in \N \cup \{ \infty \}$, and 
the length of $S = \varnothing$ is $-\infty$. 

For a graded abelian group $M = \boplus_{i \in \Z} M^i$, its {\em 
concentration} $\opn{con}(M)$ is the smallest integer interval $S$ such that 
$\{ i \mid M^i \neq 0 \} \sub S$, and its {\em amplitude} $\opn{amp}(M)$ is the 
length of $\opn{con}(M)$.
For a complex $M \in \cat{D}(\Z)$, its {\em cohomological concentration} is 
$\opn{con}(\opn{H}(M))$, the concentration of the graded abelian group 
$\opn{H}(M)$, and its {\em cohomological amplitude} is $\opn{amp}(\opn{H}(M))$.

Suppose now that $A$ and $B$ are DG rings. 
Let $\cat{E} \sub \cat{D}(A)$ be a full subcategory (not necessarily 
triangulated) and let 
$F : \cat{E} \to \cat{D}(B)$ be an additive functor. 
The {\em cohomological displacement} of $F$ relative to $\cat{E}$ is the 
smallest integer interval $S$ such that 
$\opn{con}(\opn{H}(F(M))) \sub \opn{con}(\opn{H}(M)) + S$
for every $M \in \cat{E}$. 
If $F : \cat{E}^{\mrm{op}} \to \cat{D}(B)$ is an additive functor, then the 
cohomological displacement $S$ has to satisfy the modified equation
$\opn{con}(\opn{H}(F(M))) \sub -\opn{con}(\opn{H}(M)) + S$.
The {\em cohomological dimension} of $F$ relative to $\cat{E}$ is the length of 
its cohomological displacement. If the subcategory $\cat{E}$ is not mentioned, 
then it is assumed to be $\cat{D}(A)$. 

Let $A$ be a DG ring. 
For a DG module $M \in \cat{D}(A)$, its {\em projective}, {\em injective} and 
{\em flat dimensions} are the cohomological dimensions of the functors 
$\opn{RHom}_A(M, -)$, $\opn{RHom}_A(-, M)$ and 
$M \ot^{\mrm{L}}_{A} (-)$, respectively. 

In case $A$ is a ring, then for an $A$-module $M$ the dimensions in the 
previous paragraph coincide with the classical definitions. 
A DG $A$-module $M$, i.e.\ a complex of $A$-modules, has finite projective 
dimension iff it admits a quasi-isomorphism $P \to M$ from a bounded complex of 
projective modules $P$; likewise for injective and flat. See \cite[Section 
12.4]{Ye5}. 
A ring homomorphism $A \to B$ has finite flat (resp.\ projective) 
dimension if $B$ has finite flat (resp.\ projective) dimension as a complex of 
$A$-modules. 

\begin{lem} \label{lem:1685}
Let $A \to B$ be a ring homomorphism of finite flat dimension. 
Then the rigid dualizing complex $R_B$ has finite injective dimension over $A$. 
To be precise, if the flat dimension of $B$ over $A$ is $\leq d_1$, and the 
injective dimension of $R_B$ over $B$ is $\leq d_2$, then the injective 
dimension of $R_B$ over $A$ is $\leq d_1 + d_2$. 
\end{lem}

\begin{proof}
If suffices to compute the cohomological amplitude of 
$\opn{RHom}_A(M, R_B)$ for $M \in \cat{M}(A)$, see 
\cite[Proposition 12.4.13]{Ye5}. 
But using derived Hom-tensor 
adjunction for the ring homomorphism $A \to B$, we have 
\[ \opn{RHom}_A(M, R_B) \cong \opn{RHom}_B(B \ot^{\mrm{L}}_{A} M, R_B) \]
in $\cat{D}(B)$. The cohomological amplitude of 
$B \ot^{\mrm{L}}_{A} M$ is $\leq d_1$, and hence the 
cohomological amplitude of 
$\opn{RHom}_B(B \ot^{\mrm{L}}_{A} M, R_B)$
is $\leq d_1 + d_2$. 
\end{proof}

\begin{lem} \label{lem:1686}
Let $u : A \to B$ be a ring homomorphism of finite flat dimension. 
Then the functor 
$u^! = \opn{TwInd}_u : \cat{D}^{+}_{\mrm{f}}(A) \to \cat{D}^{+}_{\mrm{f}}(B)$
has finite cohomological dimension. 
To be precise, if the flat dimension of $B$ over $A$ is $\leq d_1$, the 
injective dimension of $R_B$ over $B$ is $\leq d_2$, and the 
injective dimension of $R_A$ over $A$ is $\leq d_3$, 
then the cohomological dimension of $u^!$ is $\leq d_1 + d_2 + d_3$. 
\end{lem}

\begin{proof}
If the cohomology of $M \in \cat{D}^{+}_{\mrm{f}}(A)$ is not bounded above, 
then we do not need to worry about the cohomological amplitude of the complex 
$u^!(M)$. 
For a complex $M \in \cat{D}^{\mrm{b}}_{\mrm{f}}(A)$, we can assume 
(by smart truncation) that $M$ has the same amplitude as its cohomology, say 
$d_0 \in \N \cup \{ -\infty \}$.  
By definition of the functor $u^!$, and using Hom-tensor adjunction, we have
\begin{equation} \label{eqn:1840}
u^!(M) = \opn{RHom}_B \bigl( B \ot^{\mrm{L}}_{A} 
\opn{RHom}_A(M, R_A), R_B \bigr) 
\cong \opn{RHom}_A \bigl( \opn{RHom}_A(M, R_A), R_B \bigr)
\end{equation}
in $\cat{D}(B)$. Choose injective resolutions 
$R_A \to I$ and $R_B \to J$ over $A$, with amplitudes 
$\opn{amp}(I) \leq d_3$ and $\opn{amp}(J) \leq d_1 + d_2$, see Lemma 
\ref{lem:1685}. We get an isomorphism
$u^!(M) \cong \opn{Hom}_A \bigl( \opn{Hom}_A(M, I), J \bigr)$ 
in $\cat{D}(A)$. 
The last complex has amplitude $\leq d_0 + d_1 + d_2 + d_3$. 
\end{proof}

The next lemma is a variation of the theorem on derived tensor evaluation, see 
\cite[Theorem 12.9.10]{Ye5}. 

\begin{lem} \label{lem:1835}
Let $A \to B$ be a ring homomorphism. Given complexes 
$L, M \in \cat{D}(A)$ and $N \in \cat{D}(B)$, there is a morphism 
\[ \tag{$*$} \ze_{L, M, N} : \opn{RHom}_A(L, N) \ot^{\mrm{L}}_{A} M \to
\opn{RHom}_A \bigl( \opn{RHom}_A(M, L),N \bigr) \]
in $\cat{D}(B)$ with these properties: 
\begin{itemize}
\item[$\triangleright$] The morphism $\ze_{L, M, N}$ is functorial in the 
complexes $L, M, N$. 

\item[$\triangleright$] If $L \in \cat{D}^{\mrm{b}}(A)$,
$M \in \cat{D}^{-}_{\mrm{f}}(A)$ and $N$ has finite 
injective dimension over $A$, then $\ze_{L, M, N}$ is an isomorphism.
\end{itemize}
\end{lem}

\begin{proof} \mbox{}

\smallskip \noindent Step 1.
We start by explaining how the two objects in formula ($*$) are obtained. 
The complex $X := \opn{RHom}_A(L, N) \in \cat{D}(B)$ is calculated using a 
K-projective resolution of $L$ over $A$, and then the complex  
$X \ot^{\mrm{L}}_{A} M \in \cat{D}(B)$,
the source of $\ze_{L, M, N}$, is calculated using a K-projective 
resolution of $M$ over $A$. 
The object $Y := \opn{RHom}_A(M, L) \in \cat{D}(A)$ is 
calculated using a K-projective resolution of $M$ over $A$, and then the 
complex 
$\opn{RHom}_A(Y, N) \in \cat{D}(B)$, the target of $\ze_{L, M, N}$, 
is calculated using a K-projective resolution of $Y$ over $A$.

\medskip \noindent Step 2.
Now we construct the morphism $\ze_{L, M, N}$.
It turns out that the resolutions in step 1 are not 
sufficient to do that (see Remark \ref{rem:1835}). 
 
Let us choose a semi-free DG ring resolution $\til{B} \to B$ over $A$. 
This allows us to consider $N$ as an object of $\cat{C}_{\mrm{str}}(\til{B})$.
Choose a K-injective resolution $N \to \til{J}$ over $\til{B}$, and a 
K-projective resolution $P \to M$ over $A$. Because the DG ring $\til{B}$ is 
K-flat over $A$, it follows that the DG module $\til{J}$ is K-injective over 
$A$.

The restriction functor 
$\opn{Rest}_{B / \til{B}} : \cat{D}(B) \to \cat{D}(\til{B})$
is an equivalence of triangulated categories, and it respects 
$\opn{RHom}$ and $\ot^{\mrm{L}}$; see \cite[Theorem 12.7.2]{Ye5}.
Therefore there are \lb isomorphisms 
$\opn{Hom}_A(L, \til{J}) \cong \opn{RHom}_A(L, N) = X$
and 
$\opn{Hom}_A(L, \til{J}) \ot^{}_{A} P \cong X \ot^{\mrm{L}}_{A} M $
in $\cat{D}(\til{B})$. 
Similarly there is an isomorphism 
$\opn{Hom}_A(P, L) \cong \opn{RHom}_A(M, L) = Y$
in $\cat{D}(A)$ and an isomorphism 
$\opn{Hom}_A \bigl( \opn{Hom}_A(P, L), \til{J} \bigr) \cong 
\opn{RHom}_A(Y, N)$ in $\cat{D}(\til{B})$. 
All these isomorphisms are canonical.

Consider the canonical homomorphism 
\begin{equation} \label{eqn:1841}
\til{\ze}_{L, P, \til{J}}  : \opn{Hom}_A(L, \til{J}) \ot^{}_{A} P \to
\opn{Hom}_A \bigl( \opn{Hom}_A(P, L), \til{J} \bigr) 
\end{equation}
in $\cat{C}_{\mrm{str}}(\til{B})$; the formula is 
\[ \til{\ze}_{L, P, \til{J}}(\chi \ot p)(\xi) = 
(-1)^{c \cd e} \cd \chi(\xi(p)) \in \til{J}^{c + d + e} \]
for $p \in P^c$, $\chi \in \opn{Hom}_A(L, \til{J})^d$ and 
$\xi \in \opn{Hom}_A(P, L)^e$. 
As explained in the paragraph above, this gives us a morphism
\[ \opn{Q}(\til{\ze}_{L, P, \til{J}}) : 
\opn{RHom}_A(L, N) \ot^{\mrm{L}}_{A} M \to
\opn{RHom}_A \bigl( \opn{RHom}_A(M, L),N \bigr) \]
in $\cat{D}(\til{B})$. 

To finish this step we take $\ze_{L, M, N}$ to be the unique morphism in 
$\cat{D}(B)$ such that 
$\opn{Rest}_{B / \til{B}}(\ze_{L, M, N}) = \opn{Q}(\til{\ze}_{L, P, \til{J}})$
in $\cat{D}(\til{B})$. It is functorial in $L, M, N$, because all the DG module
resolutions we have chosen are unique up to homotopy. 
(See Remark \ref{rem:1835} regarding the possible dependence on the DG ring 
resolution $\til{B} \to B$.)

\medskip \noindent Step 3.
In this step we assume that the complexes $L, M, N$ satisfy the extra 
conditions, and we prove that $\ze_{L, M, N}$ is an isomorphism. Since 
$L \in \cat{D}(A)$ 
has bounded cohomology, we can assume (by replacing it with a suitable smart 
truncation) that $L$ is itself bounded. Since 
$M \in \cat{D}^{-}_{\mrm{f}}(A)$, we can choose its K-projective resolution 
$P$, in step 2, to be a bounded above complex of finite rank free $A$-modules. 
Let $N \to \til{J}$ be a K-injective resolution over $\til{B}$, as in step 2. 
Because $N$ has finite injective dimension over $A$, there is a resolution 
$N \to J$ over $A$, where $J$ is a bounded complex of injective $A$-modules. 
Let $\phi : \til{J} \to J$ be a homomorphism in 
$\cat{C}_{\mrm{str}}(A)$ that commutes up to homotopy with the 
homomorphisms from $N$. Because both $J$ and $\til{J}$ are K-injective over 
$A$, the morphism $\phi$ is a homotopy equivalence. We obtain this commutative 
diagram 
\begin{equation} \label{eqn:1839}
\begin{tikzcd} [column sep = 8ex, row sep = 6ex] 
\opn{Hom}_A(L, \til{J}) \ot^{}_{A} P
\ar[r, "{\til{\ze}_{L, P, \til{J}}}"]
\ar[d, "{\opn{Hom}(\opn{id}, \phi) \ot \opn{id}}"']
&
\opn{Hom}_A \bigl( \opn{Hom}_A(P, L), \til{J} \bigr)
\ar[d, "{\opn{Hom}(\opn{Hom}(\opn{id}, \opn{id}), \phi)}"]
\\
\opn{Hom}_A(L, J) \ot^{}_{A} P
\ar[r, "{\til{\ze}_{L, P, J}}"]
&
\opn{Hom}_A \bigl( \opn{Hom}_A(P, L), J \bigr)
\end{tikzcd}
\end{equation}
in $\cat{C}_{\mrm{str}}(A)$. 
The morphism $\ze_{L, M, N}$ is an isomorphism in $\cat{D}(B)$ iff the 
homomorphism $\til{\ze}_{L, P, \til{J}}$ is a quasi-isomorphism; and it does 
not matter if we consider $\til{\ze}_{L, P, \til{J}}$ as a homomorphism in
$\cat{C}_{\mrm{str}}(\til{B})$ or in $\cat{C}_{\mrm{str}}(A)$. 
Since $\phi$ is a homotopy equivalence, the vertical arrows in diagram 
(\ref{eqn:1837}) are quasi-isomorphisms. It remains to prove that 
$\til{\ze}_{L, P, J}$ is a quasi-isomorphism. 

But in fact $\til{\ze}_{L, P, J}$ is an isomorphism in 
$\cat{C}_{\mrm{str}}(A)$.
To see this, let's note that in each degree $c$, the $A$-modules  
$\bigl( \opn{Hom}_A (L, J)\ot_A P \bigr)^c$ and 
$\opn{Hom}_A \bigl( \opn{Hom}_A(P, L), J \bigr)^c$
involve only finitely many of the $A$-modules $P^d$; this is due to the fact 
that $L$ and $J$ are bounded. Hence it is enough to verify that for every $d$ 
the homomorphism 
\[ \opn{Hom}_A (L, J) \ot_A P^d[-d] \to
\opn{Hom}_A \bigl( \opn{Hom}_A(P^d[-d], L), J \bigr) \]
is bijective. This is true because $P^d$ is a finite rank free $A$-module. 
\end{proof}

\begin{prop} \label{prop:1690}
Let $u : A \to B$ be a ring homomorphism of finite flat dimension. 
Then for every $M \in \cat{D}^{\mrm{b}}_{\mrm{f}}(A)$ there is an isomorphism
$u^!(M) \cong u^!(A) \ot^{\mrm{L}}_{A} M$
in $\cat{D}(B)$. This isomorphism is functorial in $M$. 
\end{prop}

\begin{proof}
From formula (\ref{eqn:1840}) we have an isomorphism 
\begin{equation} \label{eqn:1837}
u^!(M) 
\cong \opn{RHom}_{A} \bigl( \opn{RHom}_{A}(M, R_A), R_B \bigr) 
\end{equation}
in $\cat{D}(B)$. Likewise there are isomorphisms 
$u^!(A) \cong \opn{RHom}_{A}(R_A, R_B)$
and 
\begin{equation} \label{eqn:1838}
u^!(A) \ot^{\mrm{L}}_{A} M \cong \opn{RHom}_{A}(R_A, R_B) 
\ot^{\mrm{L}}_{A} M
\end{equation}
in $\cat{D}(B)$. We know that $M \in \cat{D}^{-}_{\mrm{f}}(A)$
and $R_A \in \cat{D}^{\mrm{b}}(A)$.
According to Lemma \ref{lem:1685}, the complex $R_B$ has finite injective 
dimension over $A$. Thus Lemma \ref{lem:1835} provides an isomorphism in 
$\cat{D}(B)$ from the object in (\ref{eqn:1837}) to the object in 
(\ref{eqn:1838}).
\end{proof}

\begin{rem} \label{rem:1835}
The reason we take $M \in \cat{D}^{\mrm{b}}_{\mrm{f}}(A)$ in Proposition 
\ref{prop:1690} is that the functor $u^!$ is defined only on 
$\cat{D}^{+}_{\mrm{f}}(A)$, whereas in Lemma \ref{lem:1835} we get an 
isomorphism only for $M \in \cat{D}^{-}_{\mrm{f}}(A)$.

The reader might wonder why we need the DG ring resolution 
$\til{B} \to B$ is step 2 of the proof of Lemma \ref{lem:1835}. 
Here is the explanation. The complexes 
$\opn{Hom}_A(L, J) \ot^{}_{A} P$ and 
$\opn{Hom}_A \bigl( \opn{Hom}_A(P, L), J \bigr)$
appearing in step 3 of the proof of the lemma do not have $B$-module 
structures, and this is why diagram (\ref{eqn:1839}) is in the category 
$\cat{C}_{\mrm{str}}(A)$. 
On the other hand, the presentations of the complexes
$\opn{RHom}_A(L, N) \ot^{\mrm{L}}_{A} M$ and 
$\opn{RHom}_A \bigl( \opn{RHom}_A(M, L),N \bigr)$
as objects of $\cat{D}(B)$, which are used in step 1 of that proof, do not 
permit a DG presentation of the morphism 
$\ze_{L, M, N}$. The only way we can do that is by the homomorphism 
$\til{\ze}_{L, P, \til{J}}$ in formula (\ref{eqn:1841}), which exists in 
the category $\cat{C}_{\mrm{str}}(\til{B})$. 

We do not stipulate that the isomorphism in Proposition \ref{prop:1690}  is 
canonical; this is not needed for Corollaries \ref{cor:1695} 
and \ref{cor:1690}.
However, it is possible to prove that the isomorphism in the proposition is 
indeed canonical. This is done by showing that the morphism 
$\ze_{L, M, N}$ in Lemma \ref{lem:1835} does not depend on the choice of DG 
ring resolution $\til{B} \to B$ in step 2 of the proof of that lemma. 
The idea is this: given any other semi-free DG ring resolution 
$\til{B}' \to B$, it is possible to lift it to a DG ring homomorphism 
$\til{B}' \to \til{B}$; see \cite[Theorem 3.22]{Ye4}. We obtain restriction 
functors 
$\cat{D}(B) \to \cat{D}(\til{B}) \to \cat{D}(\til{B}')$, 
and a calculation shows that the morphism 
$\ze_{L, M, N}$ constructed in $\cat{D}(\til{B})$ is sent to the morphism 
$\ze'_{L, M, N}$ constructed in $\cat{D}(\til{B}')$. Such calculations were 
made in \cite[Section 5]{Ye4}.
\end{rem}

\begin{cor} \label{cor:1695}
If $u : A \to  B$ has finite flat dimension, then the complex $u^!(A)$ has 
finite flat dimension over $A$. 
\end{cor}

\begin{proof}
It is enough to bound the flat dimension of $u^!(A)$ relative to 
$\cat{M}_{\mrm{f}}(A)$, which is, by definition, the cohomological dimension of 
the functor 
$u^!(A)  \ot^{\mrm{L}}_{A} (-) : \cat{M}_{\mrm{f}}(A) \to \cat{D}(B)$; 
cf.\ \cite[Proposition 12.4.19]{Ye5}. 
Now the flat dimension of $u^!(A)$ relative to $\cat{M}_{\mrm{f}}(A)$
is at most the flat dimension of $u^!(A)$ relative to the bigger category
$\cat{D}^{\mrm{b}}_{\mrm{f}}(A)$.
By Proposition \ref{prop:1690}, the flat dimension of $u^!(A)$ relative to 
$\cat{D}^{\mrm{b}}_{\mrm{f}}(A)$ equals the cohomological dimension of the 
functor $u^! : \cat{D}^{\mrm{b}}_{\mrm{f}}(A) \to \cat{D}(B)$, and this is 
finite by Lemma \ref{lem:1686}. 
\end{proof}

\begin{cor} \label{cor:1690}
If $u : A \to  B$ has finite flat dimension, then there is an isomorphism 
$R_B \cong u^!(A) \ot^{\mrm{L}}_{A} R_A$
in $\cat{D}(B)$.
\end{cor}

\begin{proof}
Almost by definition we have $R_B \cong u^!(R_A)$. Now 
take $M := R_A$ in Proposition \ref{prop:1690}. 
\end{proof}

\begin{lem} \label{lem:1820}
Let $u : A \to B$ and $v : B \to C$ be ring homomorphisms, both of finite 
flat dimension. 
\begin{enumerate}
\item The homomorphisms $v \circ u$ has finite flat dimension. 

\item Let $A[\bt] = A[t_1, \ldots, t_n]$ be the polynomial ring in $n$ 
variables, and let 
$w : A[\bt] \to B$ be some $A$-ring 
homomorphism. Then $w$ has finite flat dimension. 

\item Suppose $S \sub B$ is a multiplicatively closed set such that 
$v(S) \sub C^{\times}$. Let $B_{S}$ be the localized ring, and let
$v_S : B_S \to C$ the corresponding $B$-ring homomorphism. Then 
$v_S$ has finite flat dimension.
\end{enumerate}  
\end{lem}

\begin{proof} \mbox{} 

\smallskip \noindent
(1) Say $u$ has flat dimension $\leq d$ and $v$ has flat dimension $\leq e$. 
Take some $M \in \cat{M}(A)$. We have an isomorphism 
$M \ot^{\mrm{L}}_{A} C \cong M \ot^{\mrm{L}}_{A} B \ot^{\mrm{L}}_{B} C$
in $\cat{D}(A)$. The cohomological amplitude of $M \ot^{\mrm{L}}_{A} B$ is at 
most $d$, so the cohomological amplitude of $M \ot^{\mrm{L}}_{A} C$ is at most 
$d + e$.

\medskip \noindent
(2) The proof of this item was communicated to us by Asaf Yekutieli.
In view of item (1) we can assume that $n = 1$; and let's write $t := t_1$.
Consider the $A$-ring homomorphism $\til{u} : A[t] \to B[t]$ such that 
$\til{u}(t) = t$, 
and the $B$-ring homomorphism $\til{w} : B[t] \to B$ such that 
$\til{w}(t) = w(t) \in B$. Then $w = \til{w} \circ \til{u}$ 
as homomorphisms 
$A[t] \to B$. An easy calculation shows that for every $A[t]$-module $M$
the obvious homomorphism 
$M \ot_A B \to M \ot_{A[t]} B[t]$ is bijective. Since
$A \to A[t]$ is flat, it follows that the flat dimension of $\til{u}$ is at 
most the flat dimension of $u$, which is finite by assumption, say $d < \infty$. 
The flat dimension of $\til{w}$ is at most $1$.  Again using item (1) we see 
that the flat dimension of $w$ is $\leq d + 1$. 

\medskip \noindent
(3) Suppose $M$ is some $B_S$-module. An easy calculation shows that 
$M \ot_{B} C \cong M \ot_{B_S} C$. 
Since $B \to B_S$ is flat, this implies that 
$M \ot^{\mrm{L}}_{B} C \cong M \ot^{\mrm{L}}_{B_S} C$
in $\cat{D}(B)$. We see that the flat dimension of 
$v_S$ is at most that of $v$. 
\end{proof}

\begin{lem} \label{lem:1821}
Let $u : A \to C$ be an EFT ring homomorphism. 
\begin{enumerate}
\item The homomorphism $u$ can be factored as $u = w \circ v$, where
$B := A[t_1, \ldots, t_n]$, the polynomial ring in $n$ variables for some $n$;
$S \sub B$ is a  multiplicatively closed set; 
$B_S$ is the localization; $v : A \to B_S$ is the canonical 
homomorphism; and $w : B_S \to C$ is a surjection. 

\item If $u$ is of finite flat dimension, then $C$ is 
perfect as a complex of $B_S$-modules.
\end{enumerate}
\end{lem}

\begin{proof} \mbox{} 

\smallskip \noindent
(1) By definition there exists a factorization 
$u = h \circ g \circ f$,
where $f : A \to B$ is the embedding of $A$ in a polynomial ring
$B = A[t_1, \ldots, t_n]$;
$g : B \to \bar{B}$ is a surjection; and 
$h : \bar{B} \to C$ is the localization of $\bar{B}$ at some multiplicatively 
closed set $\bar{S} \sub \bar{B}$.
Let $S := g^{-1}(\bar{S}) \sub B$, and let $B_S$ be the localized ring. The 
corresponding homomorphism $w : B_S \to C$ is surjective, and we obtain a new 
factorization $u = w \circ v$ as required.

\medskip \noindent
(2) By Lemma \ref{lem:1820}(2), the ring $C$ has finite flat dimension over 
$B$, and by Lemma \ref{lem:1820}(3) we know that $C$ has finite flat dimension 
over $B_S$. The ring $B_S$ is noetherian, and $C$ is a finite module over 
it, so the projective dimension of $C$ as a $B_S$-module equals its flat 
dimension. Finally, the perfect complexes over $B_S$ are the 
complexes in $\cat{D}^{\mrm{b}}_{\mrm{f}}(B_S)$ of finite projective 
dimension (see \cite[Theorem 14.1.33]{Ye5}).
\end{proof}

\begin{prop} \label{prop:1815}
Let $u : A \to B$ be a ring homomorphism of finite flat dimension.
Then the complex $u^!(A)$ has the derived Morita property over $B$. 
\end{prop}

\begin{proof}
By Lemma \ref{lem:1821}(1) we can factor the ring homomorphism $u$ into
$u = w \circ v$,
where $v : A \to \til{B}$ is essentially smooth of differential relative 
dimension $n$, and $w : \til{B} \to B$ is surjective. 
Define $L := \Om^n_{\til{B} / A}[n] \in \cat{D}(\til{B})$.
According to Proposition \ref{prop:1611} there is an isomorphism
$v^!(A) \cong L$ in $\cat{D}(\til{B})$, and according to 
Theorem \ref{thm:1555} and Proposition \ref{prop:1610} there are isomorphisms
\begin{equation} \label{eqn:1695}
u^!(A) \cong w^!(v^!(A)) \cong \opn{RHom}_{\til{B}}(B, L) 
\end{equation}
in $\cat{D}(B)$. 

Let $\cat{E} \sub \cat{D}^{\mrm{b}}_{\mrm{f}}(B)$ be the full subcategory on 
the complexes that are perfect over $\til{B}$. Since 
$L \cong  \til{B}[n]$, the 
functor $G : \cat{E}^{\mrm{op}} \to \cat{E}$,
$G := \opn{RHom}_{\til{B}}(-, L)$, is an equivalence, and 
$G \circ G \cong \opn{Id}_{\cat{E}}$;
see \cite[Lemma 5.4]{PSY} and its proof. 
By formula (\ref{eqn:1695}) we know that 
$u^!(A) \cong G(B) \in \cat{D}(B)$, and by Lemma \ref{lem:1821}
we have $B \in \cat{E}$; hence $u^!(A) \in \cat{E}$.

To verify that $u^!(A)$ has the derived Morita property over $B$ is suffices 
to prove that the homothety ring homomorphism 
\begin{equation} \label{eqn:1696}
B \to \opn{End}_{\cat{D}(B)} \bigl( u^!(A) \bigr)
\end{equation}
is bijective, and that 
\begin{equation} \label{eqn:1825}
\opn{Hom}_{\cat{D}(B)} \bigl( u^!(A), u^!(A)[i] \bigr) = 0
\end{equation}
for all $i \neq 0$. In these formulas we can replace $\cat{D}(B)$ with its full 
subcategory $\cat{E}$. 
Using the duality $G$ of the category $\cat{E}$ we get NC ring isomorphisms 
\[ \opn{Hom}_{\cat{E}} \bigl( u^!(A), u^!(A) \bigr) \cong 
\opn{Hom}_{\cat{E}} \bigl( G(B), G(B) \bigr) \cong 
\opn{Hom}_{\cat{E}}(B, B)^{\mrm{op}} \cong B , \]
and these respect the homothety ring homomorphisms from $B$. Thus we have an 
isomorphism in formula (\ref{eqn:1696}). 
Formula (\ref{eqn:1825}) for $i \neq 0$ 
is proved
similarly.
\end{proof}

\begin{dfn} \label{dfn:1697}
Let $u : A \to B$ be a ring homomorphism of finite flat dimension.
A {\em rigid relative dualizing complex over $B / A$}, or {\em over $u$}, is a 
rigid complex 
$(R_{B / A}, \rho_{B / A}) \in \cat{D}(B)_{\mrm{rig} / A}$,
such that $R_{B / A} \cong u^!(A)$ in $\cat{D}(B)$. 
\end{dfn}

\begin{thm} \label{thm:1695}
Let $A \to B$ be a ring homomorphism of finite flat dimension.
There exists a rigid relative dualizing complex 
$(R_{B / A}, \rho_{B / A})$ over $B / A$, and it is unique, up to a unique 
isomorphism in $\cat{D}(B)_{\mrm{rig} / A}$. 
\end{thm}

\begin{proof}
Denoting by $u : A \to B$ the given ring homomorphism, define the complex 
$R_{B / A} := u^!(A) \in \cat{D}^{+}_{\mrm{f}}(B)$. 
According to Corollary \ref{cor:1695} the complex $R_{B / A}$ has finite flat 
dimension over $A$. 
We need to produce a rigidifying isomorphism 
$\rho_{B / A} : R_{B / A} \iso \opn{Sq}_{B / A}(R_{B / A})$
in $\cat{D}(B)$, and then to prove uniqueness. 

Choose a factorization $u = w \circ v$, in which $v : A \to \til{B}$ is 
essentially smooth of differential relative dimension $n$ for some $n$, and 
$w : \til{B} \to B$ is surjective; this is possible according to 
Lemma \ref{lem:1821}.
From Definition \ref{dfn:1225} we have the rigid complex 
\[ \bigl( \Om^n_{\til{B} / A}[n], \, \rho_{\til{B} / A}^{\mrm{esm}} 
\bigr) \in \cat{D}(\til{B})_{\mrm{rig} / A} . \]
According to Propositions \ref{prop:1611} and \ref{prop:1610} and Theorem 
\ref{thm:1555} we know that 
\[ R_{B / A} = u^!(A) \cong w^!(v^!(A)) \cong 
\opn{RHom}_{\til{B}} \bigl( B, \Om^n_{\til{B} / A}[n] \bigr) \]
in $\cat{D}(B)$. Both complexes 
$\Om^n_{\til{B} / A}[n]$ and $R_{B / A}$ have finite flat dimensions over 
$A$ (for $R_{B / A}$ we use Corollary \ref{cor:1695}). Theorem \ref{thm:680} 
provides a rigidifying isomorphism 
\[ \rho_{B / A} :=
\opn{RCInd}^{\mrm{rig}}_{w / A}(\rho_{\til{B} / A}^{\mrm{esm}})
: R_{B / A} \iso \opn{Sq}_{B / A}(R_{B / A}) \]
in $\cat{D}(B)$. 

Regarding uniqueness: suppose $(R, \rho)$ is some other 
rigid relative dualizing complex over $B / A$. The existence of an isomorphism 
$R \cong u^!(A)$ gives an isomorphism $\phi : R_{B / A} \iso R$ 
in $\cat{D}(B)$. By Proposition \ref{prop:1815} the complex $R_{B / A}$ has the 
derived Morita property. Hence there is a unique element 
$b \in B^{\times}$ such that 
$b \cd \phi : (R_{B / A}, \rho_{B / A}) \iso (R, \rho)$ 
is an isomorphism in $\cat{D}(B)_{\mrm{rig} / A}$; cf.\ proof of Theorem 
\ref{thm:675}.
\end{proof}

\begin{rem} \label{rem:1695}
In general the rigid relative dualizing complex  complex $R_{B / A}$ is {\em 
not a dualizing complex over $B$}. 
It does belong to $\cat{D}^{\mrm{b}}_{\mrm{f}}(B)$, and it has the derived 
Morita property over $B$; what is missing is having 
finite injective dimension over $B$. It is not hard to prove that when $A$ is a 
Gorenstein ring (i.e.\ it has finite injective dimension over itself), then 
$R_{B / A}$ is a dualizing complex over $B$. 

What is always true (but is beyond the scope of this paper) is that 
$R_{B / A}$ induces rigid dualizing complexes on the derived fibers of $u$. 
By this we mean that for every $\p \in \opn{Spec}(A)$ the DG module 
$\bk(\p) \ot^{\mrm{L}}_{A} R_{B / A}$
is a rigid dualizing DG module over the DG ring 
$\bk(\p) \ot^{\mrm{L}}_{A} B$ relative to $A$, in the sense of 
\cite[Section 7]{Ye3}.
\end{rem}

\begin{thm} \label{thm:1845}
Let $B$ and $C$ be $A$-rings of finite flat dimension, and let 
$v : B \to C$ be a finite $A$-ring homomorphism. Then there exists a unique 
nondegenerate rigid backward morphism 
\[ \opn{tr}^{\mrm{rig}}_{v / A} = \opn{tr}^{\mrm{rig}}_{C / B / A} : 
(R_{C / A}, \rho_{C / A}) \to (R_{B / A}, \rho_{B / A}) \]
over $v$ relative to $A$. It is called the {\em relative rigid trace morphism}. 
\end{thm}

\begin{proof}
Let us denote the structural homomorphisms by 
$u_B : A \to B$ and $u_C : A \to C$. 
By definition there are isomorphisms 
$R_{B / A} \cong u_B^!(A)$ and $R_{C / A} \cong u_C^!(A)$
in $\cat{D}(B)$ and $\cat{D}(C)$, respectively. 
Because $u_C = v \circ u_B$, pseudofunctoriality says that 
$u_C^!(A) \cong v^!(u_B^!(A))$, and thus 
$R_{C / A} \cong v^!(R_{B / A})$ in $\cat{D}(C)$. 

According to Definition \ref{dfn:1610} there is a nondegenerate backward 
morphism 
$\th := \opn{tr}^{\mrm{twind}}_{v, R_{B / A}} : R_{C / A} \to R_{B / A}$
in $\cat{D}(B)$ over $v$. By Proposition \ref{prop:1815} the complex 
$R_{C / A}$ has the derived Morita property over $C$. Hence, by Proposition 
\ref{prop:1871}, $\opn{Hom}_{\cat{D}(B)}(R_{C / A}, R_{B / A})$ 
is a free $C$-module of 
rank $1$, with basis $\th$. As in the proof of Theorem \ref{thm:675}, there 
is a unique element $c \in C^{\times}$ such that 
$\opn{tr}^{\mrm{rig}}_{v / A} := c \cd \th$ is a rigid backward morphism
$(R_{C / A}, \rho_{C / A}) \to (R_{B / A}, \rho_{B / A})$
relative to $A$.
\end{proof}

\begin{cor} \label{cor:1845}
Let $B, C$ and $D$ be $A$-rings of finite flat dimension, and let 
$v : B \to C$ and $w : C \to D$ be finite $A$-ring homomorphisms. Then there 
is equality 
\[ \opn{tr}^{\mrm{rig}}_{v / A} \circ \opn{tr}^{\mrm{rig}}_{w / A} =
\opn{tr}^{\mrm{rig}}_{w \circ v / A} \]
of backward morphisms 
$R_{D / A} \to R_{B / A}$ over $w \circ v$. 
\end{cor}

\begin{proof}
This is due to the uniqueness in Theorem \ref{thm:1845}.
\end{proof}

Next we have a relative version of Theorem \ref{thm:1552}.

\begin{thm} \label{thm:1850}
Let $B$ and $B'$ be $A$-rings of finite flat dimension, and let 
$v : B \to B'$ be an essentially \'etale $A$-ring homomorphism. Then 
there exists a unique nondegenerate rigid forward morphism 
\[ \opn{q}^{\mrm{rig}}_{v / A} = \opn{q}^{\mrm{rig}}_{B' / B / A} : 
(R_{B / A}, \rho_{B / A}) \to (R_{B' / A}, \rho_{B' / A}) \]
over $v$ relative to $A$. It is called the {\em relative rigid 
\'etale-localization morphism}. 
\end{thm}

\begin{proof}
Let us denote the structural homomorphisms by 
$u_B : A \to B$ and $u_{B'} : A \to B'$. 
By definition there are isomorphisms 
$R_{B / A} \cong u_B^!(A)$ and $R_{B' / A} \cong u_{B'}^!(A)$
in $\cat{D}(B)$ and $\cat{D}(B')$, respectively. 
Because $u_{B'} = v \circ u_B$, pseudofunctoriality says that 
$u_C^!(A) \cong v^!(u_B^!(A))$, and thus 
$R_{C / A} \cong v^!(R_{B / A})$ in $\cat{D}(B')$. 

According to Proposition \ref{prop:1612} there is a nondegenerate forward 
morphism 
$\la := \opn{q}^{\mrm{twind}}_{v, R_{B / A}} : R_{B / A} \to R_{B' / A}$
in $\cat{D}(B)$ over $v$. By Proposition \ref{prop:1815} the complex 
$R_{B' / A}$ has the derived Morita property over $B'$. Hence, by Proposition 
\ref{prop:1871}(2), $\opn{Hom}_{\cat{D}(B)}(R_{B / A}, R_{B' / A})$ is a 
free $B'$-module of rank $1$, with basis $\la$. As in the proof of Theorem 
\ref{thm:675}, there is a unique element 
$b' \in (B')^{\times}$ such that 
$\opn{q}^{\mrm{rig}}_{v / A} := b' \cd \la$ is a nondegenerate rigid forward 
morphism
$(R_{B / A}, \rho_{B / A}) \to (R_{B' / A}, \rho_{B' / A})$
relative to $A$.
\end{proof}

\begin{cor} \label{cor:1850}
Let $B, B', B''$ be $A$-rings of finite flat dimension, and let 
$v : B \to B'$ and $v' : B' \to B''$ be essentially \'etale $A$-ring 
homomorphisms. Then there is equality 
\[ \opn{q}^{\mrm{rig}}_{v' / A} \circ \opn{q}^{\mrm{rig}}_{v / A} =
\opn{q}^{\mrm{rig}}_{v' \circ v / A} \]
of forward morphisms 
$R_{B / A} \to R_{B'' / A}$ over $v' \circ v$. 
\end{cor}

\begin{proof}
This is due to the uniqueness in Theorem \ref{thm:1850}.
\end{proof}

Here is a relative variant of Theorem \ref{thm:1553}.

\begin{thm} \label{thm:1851}
Let $B$, $C$ and $B'$ be $A$-rings of finite flat dimension,
Let $u : B \to C$ be a finite homomorphism,
and let $v : B \to B'$ be an essentially \'etale homomorphism.
Define $C' := B' \ot_{B} C$, and let 
$u' : B' \to C'$ and $w : C \to C'$ be the induced ring homomorphisms. 
The relative rigid dualizing complexes of these rings are 
$(R_{B / A}, \rho_{B / A})$,
$(R_{C / A}, \rho_{C / A})$,
$(R_{B' / A}, \rho_{B' / A})$ and
$(R_{C' / A}, \rho_{C' / A})$, respectively. Then 
\[ \opn{q}^{\mrm{rig}}_{v / A} \circ \opn{tr}^{\mrm{rig}}_{u / A} = 
\opn{tr}^{\mrm{rig}}_{u' / A} \circ \opn{q}^{\mrm{rig}}_{w / A} \]
as morphisms 
$R_{C / A} \to R_{B' / A}$ in $\cat{D}(B)$.
\end{thm}

The theorem asserts that given the first commutative diagram below in 
$\cat{Rng} \eftover A$, the second diagram in $\cat{D}(B)$ is also 
commutative. 
\begin{equation} \label{eqn:1850}
\begin{tikzcd} [column sep = 6ex, row sep = 4ex] 
B
\ar[r, "{u}"]
\ar[d, "{v}"']
&
C
\ar[d, "{w}"]
\\
B'
\ar[r, "{u'}"]
&
C' 
\end{tikzcd} 
\qquad \qquad 
\begin{tikzcd} [column sep = 6ex, row sep = 4ex] 
R_{B / A}
\ar[d, "{\opn{q}^{\mrm{rig}}_{v / A}}"']
&
R_{C / A}
\ar[l, "{\opn{tr}^{\mrm{rig}}_{u / A}}"']
\ar[d, "{\opn{q}^{\mrm{rig}}_{w / A}}"]
\\
R_{B' / A}
&
R_{C' / A}
\ar[l, "{\opn{tr}^{\mrm{rig}}_{u' / A}}"']
\end{tikzcd} 
\end{equation}

\begin{proof}
The proof of Theorem \ref{thm:1553}, including the proofs of the 
auxiliary lemmas \ref{lem:1590} and \ref{lem:1591}, hold in the relative setting 
as well. 
\end{proof}

\section{An Example: Finite \'Etale Ring Homomorphisms} \label{sec:fin-etale}

The purpose of this section is to prove that when $A \to B$ is a finite 
\'etale homomorphism of noetherian rings, the operator trace is a nondegenerate 
rigid backward morphism.  

Recall our Conventions \ref{conv:615} and \ref{conv:1070}, both assumed in this 
section; they say that all rings are noetherian commutative by default, and all 
ring homomorphisms are EFT.  
When talking about free $A$-modules, matrices, etc.\ the default assumption 
is that the ring $A$ is nonzero (thus avoiding silly trivialities, such as the 
rank of a free module not being defined). 

Let $A \to B$ be a finite flat ring homomorphism.  The {\em operator trace 
homomorphism} 
$\opn{tr}^{\mrm{oper}}_{B / A} : B \to A$ 
is the $A$-linear homomorphism defined as follows. 
First assume that $A \neq 0$ and $B$ is a free $A$-module of rank $n \geq 1$. 
By choosing a basis for $B$ we get an isomorphism of $A$-modules $B \cong A^n$, 
and thus an $A$-ring isomorphism $g : \opn{End}_A(B) \iso \opn{Mat}_n(A)$. An 
element $b \in B$ acts on $B$ by multiplication; this is an $A$-linear 
operator, so there is a matrix $g(b) \in \opn{Mat}_n(A)$, and we let 
$\opn{tr}^{\mrm{oper}}_{B / A}(b) \in A$ be the usual trace of the matrix 
$g(b)$.
Since the trace of a matrix is invariant under conjugation, the element
$\opn{tr}^{\mrm{oper}}_{B / A}(b) \in A$ does not depend on the basis of $B$ 
that was chosen. In the general situation, with $A$ nonzero, we can find a 
covering sequence $(s_1, \ldots, s_m)$ of $A$ s.t.\ $A_{s_i} \neq 0$, and 
$B_{s_i} := A_{s_i} \ot_A B$ is a free 
module over $A_{s_i}$ of rank $n_i \geq 0$. 
If $n_i \geq 1$ then the $A_{s_i}$-linear trace 
$\opn{tr}^{\mrm{oper}}_{B_{s_i} / A_{s_i}} : B_{s_i} \to A_{s_i}$
is as explained above. In case $n_i = 0$ then of course we take 
$\opn{tr}^{\mrm{oper}}_{B_{s_i} / A_{s_i}} := 0$. 
The homomorphisms $\opn{tr}^{\mrm{oper}}_{B_{s_i} / A_{s_i}}$
agree on the double intersections 
$\opn{Spec}(A_{s_i \cd s_j}) \sub \opn{Spec}(A)$, and therefore 
they can be glued uniquely to an $A$-linear homomorphism 
$\opn{tr}^{\mrm{oper}}_{B / A} : B \to A$.

We know that if a matrix $\ba \in \opn{Mat}_n(A)$ has a diagonal block 
decomposition
\[ \ba = 
\scalebox{0.8}{$\bmat{ \ba_1 & \dots & 0 
\\ \vdots & \ddots & \vdots
\\ 0 & \dots & \ba_m}$} , \]
where $\ba_i \in \opn{Mat}_{n_i}(A)$
and $\sum_i n_i = n$, then 
$\opn{tr}(\ba) = \sum_i \opn{tr}(\ba_i)$.
Therefore, if the ring $B$ decomposes into 
$B = \prod_i B_i$, then 
$\opn{tr}^{\mrm{oper}}_{B / A} = \sum_i \opn{tr}^{\mrm{oper}}_{B_i / A}$.

From the construction it is clear that given any ring homomorphism 
$A \to A'$, letting $B' := A' \ot_A B$, and letting $A' \to B'$ be the induced 
finite flat ring homomorphism, the diagram 
\begin{equation} \label{eqn:1370}
\begin{tikzcd} [column sep = 8ex, row sep = 4ex]
A
\ar[d]
&
B
\ar[d]
\ar[l, "{\opn{tr}^{\mrm{oper}}_{B / A}}"'] 
\\
A'
&
B'
\ar[l, "{\opn{tr}^{\mrm{oper}}_{B' / A'}}"'] 
\end{tikzcd} 
\end{equation}
in $\cat{M}(A)$ is commutative. 

\begin{dfn} \label{dfn:1379}
Let us say that a ring homomorphism $A \to B$ is {\em faithfully \'etale} if 
it is \'etale and faithfully flat. Geometrically this means that the map of 
schemes $\opn{Spec}(B) \to \opn{Spec}(A)$ is \'etale and surjective. 
\end{dfn}

\begin{thm} \label{thm:1325}
Assume $A \to B$ is finite \'etale. Then the operator trace 
$\opn{tr}^{\mrm{oper}}_{B / A} : B \to A$
is a nondegenerate backward homomorphism in $\cat{M}(A)$ over $B / A$. 
\end{thm}

We need these lemmas first. 

\begin{lem} \label{lem:1330}
Let $B$ be a finite flat $A$-ring. Then $B$ is faithfully flat iff for 
each $\p \in \opn{Spec}(A)$ the rank of 
the free $A_{\p}$-module $B_{\p} := A_{\p} \ot_A B$ is at least $1$. 
\end{lem}

\begin{proof}
The scheme map $\opn{Spec}(B) \to \opn{Spec}(A)$ is surjective iff all the 
fibers $\opn{Spec}(\bk(\p) \ot_A B)$ are nonempty. 
But by Nakayama's Lemma, $\bk(\p) \ot_A B$ is nonzero iff 
$\opn{rank}_{A_{\p}}(B_{\p}) \geq 1$. 
\end{proof}

\begin{lem}[Splitting] \label{lem:1325}
Let $A \to B$ be finite \'etale, and assume that as a projective $A$-module, 
$B$ has constant rank $n \geq 1$. Then there is a finite faithfully \'etale 
$A$-ring $A'$, s.t.\ the $A'$-ring $B' := A' \ot_A B$ decomposes into 
$B' \cong \prod\nolimits_{i = 1, \ldots, n} A'$.
\end{lem}

\begin{proof}
The proof is by induction on $n$. If $n = 1$ then $B = A$, so there is nothing 
to prove. Now assume $n \geq 2$ and the lemma holds for lower ranks. 
Define $A^{\diamondsuit} := B$, which, by Lemma \ref{lem:1330}, is a finite 
faithfully \'etale $A$-ring. 
Then define 
$B^{\diamondsuit} := A^{\diamondsuit} \ot_A B$,
which is a finite \'etale $A^{\diamondsuit}$-ring of constant rank $n$. 
By Corollary \ref{cor:1280} there is a ring isomorphism 
$B^{\diamondsuit} = A^{\diamondsuit} \ot_A B = B \ot_A B \cong 
B \times B^{\mrm{od}}$.
Writing $B^{\diamondsuit}_1 := B = A^{\diamondsuit}$ and 
$B^{\diamondsuit}_2 := B^{\mrm{od}}$,
we obtain the ring decomposition 
$B^{\diamondsuit} \cong B_1^{\diamondsuit} \times B_2^{\diamondsuit}$. 
The rank of $B_1^{\diamondsuit}$ as an $A^{\diamondsuit}$-module is $1$, and 
therefore the $A^{\diamondsuit}$-module $B_2^{\diamondsuit}$
has constant rank $n - 1$. 

By induction there is a finite faithfully \'etale
ring homomorphism $A^{\diamondsuit} \to A'$ s.t.\ 
the $A'$-ring $B'_2 := A' \ot_{A^{\diamondsuit}} B_2^{\diamondsuit}$
has a decomposition 
$B'_2 \cong \prod\nolimits_{i = 1, \ldots, n - 1} A'$.
Then
\[ \begin{aligned}
&
A' \ot_A B \cong A' \ot_{A^{\diamondsuit}} (A^{\diamondsuit} \ot_A B) \cong
A' \ot_{A^{\diamondsuit}} B^{\diamondsuit} 
\\
& \quad 
\cong  A' \ot_{A^{\diamondsuit}} (A^{\diamondsuit} \times B_2^{\diamondsuit})
\cong A' \times (A' \ot_{A^{\diamondsuit}} B_2^{\diamondsuit})
= A' \times B'_2 \cong \prod\nolimits_{i = 1, \ldots, n} A' . 
\end{aligned} \]

Lastly, we note that $A \to A'$ is finite faithfully \'etale. 
\end{proof}

\begin{proof}[Proof of Theorem \tup{\ref{thm:1325}}]
The condition that $\opn{tr}^{\mrm{oper}}_{B / A}$ is a nondegenerate backward 
homomorphism, namely that the $B$-module homomorphism
$\opn{badj}_{B / A}(\opn{tr}^{\mrm{oper}}_{B / A}) : B \to \opn{Hom}_A(B, A)$
is an isomorphism, is local on $\opn{Spec}(A)$. Hence we can assume that 
$\opn{Spec}(A)$ is connected. This implies that the rank of $B$ as an 
$A$-module is constant. 

Because of the commutativity of diagram (\ref{eqn:1370}), and by 
the standard properties of faithfully flat 
homomorphisms, it suffices to prove that for some  
faithfully flat ring homomorphism $A \to A'$, with 
$B' := A' \ot_A B$, the trace
$\opn{tr}^{\mrm{oper}}_{B' / A'} : B' \to A'$ 
is a nondegenerate backward homomorphism.

Let $A'$ be a finite faithfully \'etale $A$-ring that splits $B$, as in Lemma 
\ref{lem:1325}. Thus $B' := \lb A' \ot_A B$ decomposes into 
$B' = \prod\nolimits_{i = 1, \ldots, n} B_i'$,
where $B'_i \cong A'$. We know that 
$\opn{tr}^{\mrm{oper}}_{B' / A'} = \sum_i \opn{tr}^{\mrm{oper}}_{B'_i / A'}$;
and also that 
$\opn{Hom}_{A'}(B', A') \cong \bigoplus\nolimits_i \opn{Hom}_{A'}(B'_i, A')$.
Therefore it suffices to prove that 
$\opn{tr}^{\mrm{oper}}_{B'_i / A'}$ is nondegenerate for every $i$.
But $A' \to B'_i$ is an isomorphism of rings, so 
$\opn{tr}^{\mrm{oper}}_{B'_i / A'}$ is an $A$-module isomorphism.
\end{proof}

\begin{rem} \label{rem:1325}
Suppose $A \to B$ is a {\em Galois} ring homomorphism; namely it is finite and 
faithfully \'etale, with automorphism group $G := \opn{Aut}_A(B)$, such that 
the ring homomorphism 
$B \ot_A B \to \prod_{g \in G} B$, 
$b_1 \ot b_2 \mapsto \{ b_1 \cd g(b_2) \}_{g \in G}$, is bijective. 
Then the trace satisfies 
$\opn{tr}^{\mrm{oper}}_{B / A}(b) = \sum\nolimits_{g \in G} g(b)$,
just as in Galois theory of fields. 

A proof of the nondegeneracy of the trace using the Galois approach can be 
found in \cite[Lemma V.5.12]{Mi}. 
\end{rem}

Again let $A \to B$ be a finite \'etale ring homomorphism. 
There are two rigid complexes associated to this situation:
$\bigl( B, \rho^{\mrm{eet}}_{B / A} \bigr)$ and 
$\bigl( \De^{\mrm{fifl}}_{B / A}, \rho^{\mrm{fifl}}_{B / A} \bigr)$
in $\cat{D}(B)_{\mrm{rig} / A}$.
The rigidifying isomorphism 
$\rho_{B / A}^{\mrm{eet}} : B \iso \opn{Sq}_{B / A}(B)$
in $\cat{D}(B)$ is from Definition \ref{dfn:1225} (since it is also essentially smooth).
The dual module of $B$ relative to $A$ is 
$\De^{\mrm{fifl}}_{B / A} = \opn{Hom}_A(B, A)$,
and its  rigidifying isomorphism 
$\rho^{\mrm{fifl}}_{B / A} : \De^{\mrm{fifl}}_{B / A} \iso 
\opn{Sq}_{B / A}(\De^{\mrm{fifl}}_{B / A})$
is from Definition \ref{dfn:1505}. 
We want to compare these two rigid complexes. This will be done in Theorem 
\ref{thm:1336} below.

It turns out that in the present situation the squares 
$\opn{Sq}_{B / A}(B)$ and $\opn{Sq}_{B / A}(\De^{\mrm{fifl}}_{B / A})$, as well 
as the rigidifying isomorphisms 
$\rho^{\mrm{eet}}_{B / A}$ and $\rho^{\mrm{fifl}}_{B / A}$,
can all be described within $\cat{M}(B)$, namely without recourse to derived 
categories. 

As K-flat DG ring resolution of $B / A$ we choose
$\til{B} / \til{A} := B / A$.
We know that $B$ is a projective $(B \ot_A B)$-module. Also both $B$ and 
$\De^{\mrm{fifl}}_{B / A}$ are flat $A$-modules. Therefore we have 
\[ \opn{Sq}_{B / A}^{B / A}(B) = \opn{Hom}_{B \ot_A B}(B, B \ot_A B) \]
and 
\[ \opn{Sq}_{B / A}^{B / A}(\De^{\mrm{fifl}}_{B / A}) = 
\opn{Hom}_{B \ot_A B}(B, \De^{\mrm{fifl}}_{B / A} \ot_A \De^{\mrm{fifl}}_{B / 
A}) \]
in $\cat{M}(B)$. We shall work with the rigidifying isomorphisms 
$\rho^{\mrm{eet}}_{B / A} : B \iso \opn{Sq}_{B / A}^{B / A}(B)$
and 
$\rho^{\mrm{fifl}}_{B / A} : \De^{\mrm{fifl}}_{B / A} \iso 
\opn{Sq}_{B / A}^{B / A}(\De^{\mrm{fifl}}_{B / A})$
in $\cat{M}(B)$. This involves a slight abuse of notation, because we are 
neglecting the isomorphisms 
$\opn{sq}_{B / A, B}^{B / A}$ and 
$\opn{sq}_{B / A, \De^{\mrm{fifl}}_{B / A}}^{B / A}$,
see Theorem \ref{thm:631}.

For the next three lemmas we consider a finite flat ring homomorphism 
$f : A \to B$, and an arbitrary EFT ring homomorphism $g : A \to A'$.
Let $B' := A' \ot_A B$. We get this cocartesian commutative diagram of rings 
\begin{equation} \label{eqn:1380}
\begin{tikzcd} [column sep = 8ex, row sep = 4ex]
A
\ar[d, "{g}"']
\ar[r, "{f}"] 
&
B
\ar[d, "{h}"]
\\
A'
\ar[r, "{f'}"] 
&
B'
\end{tikzcd} 
\end{equation}
and $f' : A' \to B'$ is also finite flat.

Let 
$\be : \De^{\mrm{fifl}}_{B / A} \to \De^{\mrm{fifl}}_{B' / A'}$ 
be the composed homomorphism 
\[ \De^{\mrm{fifl}}_{B / A} = \opn{Hom}_{A}(B, A) 
\xar{\opn{Hom}_{}(\opn{id}, g)} 
\opn{Hom}_{A}(B, A') \iso^{(\dag)} \opn{Hom}_{A'}(B', A') =  
\De^{\mrm{fifl}}_{B' / A'} ,  \]
where the isomorphism $\iso^{(\dag)}$ is adjunction for the ring homomorphism 
$g$. Let 
$\ga : B \to \De^{\mrm{fifl}}_{B / A}$ be the homomorphism 
$\ga := \opn{badj}_{B / A}(\opn{tr}^{\mrm{oper}}_{B / A})$,
and similarly we have the homomorphism
$\ga' : B' \to \De^{\mrm{fifl}}_{B' / A'}$. 

\begin{lem} \label{lem:1380}
In the situation described above, these diagrams in $\cat{M}(B)$ are 
commutative:
\[ \tag{1} 
\begin{tikzcd} [column sep = 8ex, row sep = 4ex]
B
\ar[d, "{h}"']
\ar[r, "{\ga}"] 
&
\De^{\mrm{fifl}}_{B / A}
\ar[d, "{\be}"]
\\
B'
\ar[r, "{\ga'}"] 
&
\De^{\mrm{fifl}}_{B' / A'}
\end{tikzcd} \]
\[ \tag{2} 
\begin{tikzcd} [column sep = 8ex, row sep = 5ex]
\De^{\mrm{fifl}}_{B / A}
\ar[d, "{\be}"]
\ar[r, "{\rho^{\mrm{fifl}}_{B / A}}", "{\simeq}"']
&
\opn{Hom}_{B \ot_A B}(B, 
\De^{\mrm{fifl}}_{B / A} \ot_A \De^{\mrm{fifl}}_{B / A})
\ar[d, "{\opn{adj} \circ \opn{Hom}(\opn{id}, \be \ot \be)}"]
\\
\De^{\mrm{fifl}}_{B' / A'}
\ar[r, "{\rho^{\mrm{fifl}}_{B' / A'}}", "{\simeq}"']
&
\opn{Hom}_{B' \ot_{A'} B'}(B', 
\De^{\mrm{fifl}}_{B' / A'} \ot_{A'} \De^{\mrm{fifl}}_{B' / A'})
\end{tikzcd} \]
\[ \tag{3} 
\begin{tikzcd} [column sep = 14ex, row sep = 5ex]
\opn{Hom}_{B \ot_A B}(B, B \ot_A B)
\ar[d, "{\opn{adj} \circ \opn{Hom}(\opn{id}, h \ot h)}"]
\ar[r, "{\opn{Hom}(\opn{id}, \ga \ot \ga)}"]
&
\opn{Hom}_{B \ot_A B}(B, 
\De^{\mrm{fifl}}_{B / A} \ot_A \De^{\mrm{fifl}}_{B / A})
\ar[d, "{\opn{adj} \circ \opn{Hom}(\opn{id}, \be \ot \be)}"']
\\
\opn{Hom}_{B' \ot_{A'} B'}(B', B' \ot_{A'} B')
\ar[r, "{\opn{Hom}(\opn{id}, \ga' \ot \ga')}"]
&
\opn{Hom}_{B' \ot_{A'} B'}(B', 
\De^{\mrm{fifl}}_{B' / A'} \ot_{A'} \De^{\mrm{fifl}}_{B' / A'})
\end{tikzcd} \]
Here the expression $\opn{adj}$ appearing is some vertical homomorphisms
refers to adjunction for the ring homomorphism $g : A \to A'$. 
\end{lem}

\begin{proof} 
The commutativity of these diagrams can be checked locally on $\opn{Spec}(A)$. 
Therefore we can assume that $B$ is a free $A$-module of constant rank. Then 
the verification is easy linear algebra. 
\end{proof}

\begin{lem} \label{lem:1381}
Assume the ring homomorphism $f : A \to B$ is finite \'etale. Then the 
diagram 
\[ \begin{tikzcd} [column sep = 8ex, row sep = 5ex]
B
\ar[d, "{h}"']
\ar[r, "{\rho^{\mrm{eet}}_{B / A}}", "{\simeq}"'] 
&
\opn{Hom}_{B \ot_A B}(B, B \ot_A B)
\ar[d, "{\opn{adj} \circ \opn{Hom}(\opn{id}, h \ot h)}"]
\\
B'
\ar[r, "{\rho^{\mrm{eet}}_{B' / A'}}", "{\simeq}"'] 
&
\opn{Hom}_{B' \ot_{A'} B'}(B', B' \ot_{A'} B')
\end{tikzcd} \]
is commutative. 
\end{lem}

\begin{proof}
Recall, from Definition \ref{dfn:1225} that the isomorphism 
$\rho^{\mrm{eet}}_{B / A}$ is the homomorphism 
$\opn{sec}_{B / A} : B \to B \ot_A B$; the only change is that the target is 
not $B \ot_A B$ but a submodule of it. 
Likewise for $\rho^{\mrm{eet}}_{B' / A'}$. 
Hence it is enough to prove that the solid subdiagram below
\begin{equation} \label{eqn:1387}
\begin{tikzcd} [column sep = 10ex, row sep = 5ex]
B
\ar[d, "{h}"']
\ar[r, "{\opn{sec}_{B / A}}"] 
\ar[rr, dashed, bend left = 20, start anchor = north east, end anchor = north 
west, "{\opn{id}_{B}}"]
&
B \ot_A B
\ar[d, "{h \ot h}"]
\ar[dashed, r, "{\opn{mult}_{B / A}}"] 
&
B
\ar[dashed, d, "{h}"']
\\
B'
\ar[r, "{\opn{sec}_{B' / A'}}"]
\ar[rr, dashed, bend right = 20, start anchor = south east, 
end anchor = south west, "{\opn{id}_{B'}}"']
&
B' \ot_{A'} B'
\ar[dashed, r, "{\opn{mult}_{B' / A'}}"]
&
B'
\end{tikzcd}
\end{equation}
is commutative. Now if we were to replace the homomorphism 
$\opn{sec}_{B' / A'}$
with the homomorphism $\opn{id}_{A'} \ot \opn{sec}_{B / A}$
that's induced from $\opn{sec}_{B / A}$, then the whole diagram would be 
commutative. Note that $\opn{id}_{A'} \ot \opn{sec}_{B / A}$ is a
$(B' \ot_{A'} B')$-linear homomorphism. According to the uniqueness clause in 
Corollary \ref{cor:1246}(1), the commutativity of the lower 
portion of diagram (\ref{eqn:1387}) implies that 
$\opn{id}_{A'} \ot \opn{sec}_{B / A} = \opn{sec}_{B' / A'}$.
\end{proof}

\begin{lem} \label{lem:1387}
Suppose $B = \prod_{i = 1, \ldots, n} B_i$.   
\begin{enumerate}
\item There are obvious decompositions
$\De^{\mrm{fifl}}_{B / A} \cong \bigoplus_i \De^{\mrm{fifl}}_{B_i / A}$,
\[ \opn{Sq}_{B / A}^{B / A}(\De^{\mrm{fifl}}_{B / A}) \cong 
\bigoplus\nolimits_i \opn{Sq}_{B_i / A}^{B_i / A}(\De^{\mrm{fifl}}_{B_i / A}) \]
and 
\[ \opn{Sq}_{B / A}^{B / A}(B) \cong 
\bigoplus_i \opn{Sq}_{B_i / A}^{B_i / A}(B_i) . \] 

\item With respect to the decompositions in item (1) we have 
$\rho^{\mrm{fifl}}_{B / A} = \sum_i \rho^{\mrm{fifl}}_{B_i / A}$
and
$\ga = \sum_i \ga_i$, where 
$\ga_i :=  \opn{badj}_{B_i / A}(\opn{tr}^{\mrm{oper}}_{B_i / A})$. 

\item Assume $A \to B$ is finite \'etale. Then, w.r.t.\ the decompositions in 
item (1), there is equality 
$\rho^{\mrm{eet}}_{B / A} = \sum_i \rho^{\mrm{eet}}_{B_i / A}$.
\end{enumerate}
\end{lem}

\begin{proof}
These are easy calculations. 
\end{proof}

\begin{thm} \label{thm:1336}
Let $A \to B$ be a finite \'etale ring homomorphism. Then the isomorphism 
\[ \opn{badj}_{B / A}(\opn{tr}^{\mrm{oper}}_{B / A}) : B \iso 
\De^{\mrm{fifl}}_{B / A} \]
in $\cat{M}(B)$, which we have due to Theorem \ref{thm:1325}, becomes an 
isomorphism of rigid complexes
\[ \opn{Q} \bigl( \opn{badj}_{B / A}(\opn{tr}^{\mrm{oper}}_{B / A}) \bigr) :
(B, \rho^{\mrm{eet}}_{B / A}) \iso 
(\De^{\mrm{fifl}}_{B / A}, \rho^{\mrm{fifl}}_{B / A}) \]
over $B$ relative to $A$. 
\end{thm}

\begin{proof}
Consider the diagram 
\begin{equation} \label{eqn:1373}
\begin{tikzcd} [column sep = 8ex, row sep = 5ex]
B
\ar[r, "{\rho^{\mrm{eet}}_{B / A}}", "{\simeq}"']
\ar[d, "{\ga}"', "{\simeq}"]
&
\opn{Hom}_{B \ot_A B}(B, B \ot_A B)
\ar[d, "{\opn{Hom}(\opn{id}_B, \ga \ot \ga)}", "{\simeq}"'] 
\\
\De^{\mrm{fifl}}_{B / A}
\ar[r, "{\rho^{\mrm{fifl}}_{B / A}}", "{\simeq}"']
\ar[r]
&
\opn{Hom}_{B \ot_A B}(B, 
\De^{\mrm{fifl}}_{B / A} \ot_A \De^{\mrm{fifl}}_{B / A})
\end{tikzcd}
\end{equation}
in $\cat{M}(B)$, were 
$\ga := \opn{badj}_{B / A}(\opn{tr}^{\mrm{oper}}_{B / A})$.
We need to prove it is commutative.

According to Lemmas \ref{lem:1380} and \ref{lem:1381}, diagram (\ref{eqn:1373}) 
is functorial w.r.t.\ any base change $g : A \to A'$. Thus it suffices to prove 
that for some faithfully flat ring homomorphism $g : A \to A'$ the induced 
diagram 
\begin{equation} \label{eqn:1365}
\begin{tikzcd} [column sep = 8ex, row sep = 5ex]
B'
\ar[r, "{\rho^{\mrm{eet}}_{B' / A'}}", "{\simeq}"']
\ar[d, "{\ga'}"', "{\simeq}"]
&
\opn{Hom}_{B' \ot_{A'} B'}(B', B' \ot_{A'} B')
\ar[d, "{\opn{Hom}(\opn{id}_{B'}, \ga' \ot \ga')}", "{\simeq}"'] 
\\
\De^{\mrm{fifl}}_{B' / A'}
\ar[r, "{\rho^{\mrm{fifl}}_{B' / A'}}", "{\simeq}"']
\ar[r]
&
\opn{Hom}_{B' \ot_{A'} B'}(B', 
\De^{\mrm{fifl}}_{B' / A'} \ot_{A'} \De^{\mrm{fifl}}_{B' / A'})
\end{tikzcd}
\end{equation}
of isomorphisms  $\cat{M}(B')$, where 
$B' := A' \ot_A B$, is commutative. 
Moreover, if $A' = \prod_{i = 1, \ldots, n} A'_i$, then it is enough to prove 
that diagram (\ref{eqn:1365}) is commutative for each $A'_i$ separately
(namely after replacing $A'$ with $A'_i$ in (\ref{eqn:1365})).

Consider the connected component decomposition 
$A = \prod_{i = 1, \ldots, n} A_i$.
As noted above, it is enough to verify the commutativity of 
diagram (\ref{eqn:1365}) after the base change $A \to A_i$
for each $i$. Thus we may assume that $\opn{Spec}(A)$ is connected. 
In this case $B$ has constant rank as a projective $A$-module.

Since $B$ has constant rank as a projective $A$-module, we can take 
$A'$ to be the finite faithfully \'etale $A$-ring from Lemma 
\ref{lem:1325}. For this choice we have an $A'$-ring isomorphism
$B' \cong \prod\nolimits_{i = 1, \ldots, n} A'$.
By Lemma \ref{lem:1387} we can replace $B'$ with $A'$. 

Finally, we have reduced matters to the case where $A \to B$ is bijective. 
So we can assume $A = B$ and $f = \opn{id}_A$. 
Now the four modules in diagram (\ref{eqn:1373}) are canonically isomorphic
to $A$; and the four homomorphisms in diagram become $\opn{id}_A$. Thus 
diagram (\ref{eqn:1373}) is commutative.
\end{proof}

\begin{cor} \label{cor:1336}
Let $A \to B$ be a finite \'etale ring homomorphism. Then the operator trace
\[ \opn{tr}^{\mrm{oper}}_{B / A} : (B, \rho^{\mrm{eet}}_{B / A}) \to 
(A, \rho^{\mrm{tau}}_{A / A}) \]
is a nondegenerate rigid backward morphism in $\cat{D}(A)$ over $B / A$. 
\end{cor}

\begin{proof}
Clear from Theorem \ref{thm:1336} and the fact that the standard nondegenerate 
backward morphism 
\[ \opn{tr}_{B / A, A} :  
\bigl( \De^{\mrm{fifl}}_{B / A}, \rho^{\mrm{fifl}}_{B / A} \bigr) \to 
(A, \rho^{\mrm{tau}}_{A / A}) \]
is a rigid backward morphism in $\cat{D}(A)$ over $B / A$
(see Theorem \ref{thm:680} and Definition \ref{dfn:1505}).
\end{proof}

\appendix 

\section{Induced Rigidity -- Solution of the Question in a Special Case}
\label{sec:flat-case}

In this appendix we give a positive answer to Question \ref{que:1660} under a 
flatness assumption. Namely

\begin{thm} [Easy Case] \label{thm:1570}
Let $A$ be a noetherian ring, let $u : A \to B$ and $v : B \to C$ be respectively
flat and essentially \'etale ring homomorphisms, and let 
$(M, \rho) \in \cat{D}(B)_{\mrm{rig} / A}$ with $M$ a flat essentially finite type $B$-module. 
Define the rigid complex 
\[ (N, \si) := \opn{TwInd}^{\mrm{rig}}_{v / A}(M, \rho) 
\in \cat{D}(C)_{\mrm{rig} / A} , \]
as in Definition \ref{dfn:1235}. 
Then the standard nondegenerate forward morphism 
\[ \opn{q}^{\mrm{L}}_{v, M} : (M, \rho) \to (N, \si) \]
is a rigid forward morphism over $v / A$. 
\end{thm}

We shall use the shortcut (a slight abuse of notation) 
$\opn{Sq}_{C / B}(C) := \opn{Sq}^{C / B}_{C / B}(C)$. 

Here are two lemmas needed for the proof. 
Recall the rigidifying isomorphism 
$\rho_{C / B}^{\mrm{esm}}$ from Definition \ref{dfn:1225}. It was then 
used in Definition \ref{dfn:1235} of twisted induced rigidity.  

\begin{lem} \label{lem:1596}
There is equality 
\[ \rho_{C / B}^{\mrm{esm}} = \opn{sec}_{C / B} \]
of $C$-module isomorphisms 
\[ C \iso \opn{Sq}_{C / B}(C) = \opn{Hom}_{C \ot_B C}(C, C \ot_B C) .  \]
\end{lem}

\begin{proof}
By Definition \ref{dfn:1225}, and since $B\to C$ is \'etale, $\rho_{C / B}^{\mrm{esm}}$ is the unique isomorphism in $\cat{M}(C)$ 
fitting in the following diagram:
\begin{equation}\label{eqn:MO5}
\begin{tikzcd} [column sep = 8ex, row sep = 5ex]
C
\ar[r, "{\rho_{C / B}^{\mrm{esm}}}", "{\simeq}"']
&
\opn{Sq}_{C / B} \bigl( C \bigr) 
\ar[d, "{\opn{sq}^{C / B}_{C / B, C} }", , "{\simeq}"']
\ar[r, "{}", ,"{=}"']
&
\opn{Hom}_{C^{\mrm{en}}} \bigl( C, C \ot_B C \bigr)
\ar[d, "{}", ,"{=}"']
\\
\opn{Ext}^0_{C^{\mrm{en}}} 
\bigl( C, C \ot_B C \bigr)
\ar[u, "{\opn{res}_{C / B}}"', "{\simeq}"] 
&
\opn{Sq}^{C / B}_{C / B} \bigl( C \bigr) 
\ar[l, "{\simeq}"]
&
\opn{Hom}_{C^{\mrm{en}}} \bigl( C, C \ot_B C \bigr)
\ar[l, "{}", ,"{=}"]
\end{tikzcd}
\end{equation}
Let us focus on the residue isomorphism $\opn{res}_{C / B}$, i.e. the left vertical arrow in the diagram above.\\
Since $\opn{Ext}^0_{C^{\mrm{en}}} \bigl( C, C \ot_B C \bigr) =\opn{Hom}_{C \ot_B C}(C, C \ot_B C)$, we have:
\[ \opn{res}_{C / B} : 
\opn{Hom}_{C^{\mrm{en}}} \bigl( C, C \ot_B C \bigr)
 \iso C. \]
By Definition \ref{dfn:1215} $\opn{res}_{C / B}$ is given by the homomorphism of $C$-modules
\begin{equation}\label{eqn:MO4}
\opn{res}_{C / B} : 
\opn{Hom}_{C \ot_B C}(C, C \ot_B C)
 \to C,  \quad
c_1\ot c_2\mapsto \opn{mult}(c_1\ot c_2).
\end{equation}
Now we recall that since the map $B\to C$ is \'etale, by Corollary \ref{cor:1246}(1) we have:
\[ \begin{tikzcd} [column sep = 8ex, row sep = 5ex]
C\ar[rr, "\opn{Id}_C"', bend right=25]
\ar[r, "{\opn{sec}_{C/B}}"]
&
C\ot_BC
\ar[r, "{\opn{mult}_{C/B}}"]
&
C
\end{tikzcd}
\]
in the category $\cat{M}(C\ot_BC)$.\ 
Applying the functor $\opn{Hom}_{C\ot_BC}(C,\opn{-})$ we get:
\[ \begin{tikzcd} [column sep = 18ex, row sep = 5ex]
C\ar[rr, "\opn{Id}_C"', bend right=25]
\ar[r, "{\opn{Hom}_{C \ot_B C}(C,\opn{sec}_{C/B})}"]
&
\opn{Sq}_{C/B}(C)
\ar[r, "{\opn{Hom}_{C \ot_B C}(C,\opn{mult}_{C/B})}"]
&
C
\end{tikzcd}
\]
in the category $\cat{M}(C)$.\ It means that $\opn{mult}_{C/B}$ and $\opn{sec}_{C/B}$ are mutually inverse in $\cat{M}(C)$.\\ Since $\opn{res}_{C / B}=\opn{Hom}_{C \ot_B C}(C,\opn{mult}_{C/B})$ by (\ref{eqn:MO5}) we have $\rho_{C / B}^{\mrm{esm}}=\opn{sec}_{C/B}$ as $C$-modules isomorphism.
\end{proof}

Recalling that $M$ is flat over $B$, and hence over $A$, we note that 
\[ \opn{Sq}_{B / A}(M) = \opn{Hom}_{B \ot_A B}(B, M \ot_A M) \sub M \ot_A M . \]
Because $B$ is a cyclic $(B \ot_A B)$-module with generator $1_B$, 
an element $\phi \in \opn{Sq}_{B / A}(M)$ is determined by its 
value on the generator $1_B$. Now $\phi(1_B) \in M \ot_A M$,
so there are finitely many elements 
$m'_i, m''_i \in M$ such that 
\begin{equation} \label{eqn:1595}
\phi(1_B) = \sum_i m'_i \ot m''_i \in M \ot_A M . 
\end{equation}
(Of course this expansion is not unique.)

\begin{lem} \label{lem:1595}
Suppose $(L, \tau) \in \cat{D}(C)_{\mrm{rig} / B}$ where $L$ is a flat 
$C$-module. Consider the rigidifying isomorphism 
\[ \rho \cupprod \tau : M \ot_B L \iso \opn{Sq}_{C / A}(M \ot_B L) . \]
Let $m \in M$ and $l \in L$. 
Suppose the elements $\rho(m) \in \opn{Sq}_{B / A}(M)$ and 
$\tau(l) \in \opn{Sq}_{C / B}(L)$ 
have these expansions, like in (\ref{eqn:1595}):
\[ \rho(m)(1_B) = \sum_i m'_i \ot m''_i \in M \ot_A M \]
and 
\[ \tau(l)(1_C) = \sum_j l'_j \ot l''_j \in L \ot_B L . \]
Then the element 
\[ (\rho \cupprod \tau) (m \ot l) \in \opn{Sq}_{C / A}(M 
\ot_B L) \]
has this expansion: 
\[ ((\rho \cupprod \tau)(m \ot l))(1_C) = \sum_{i, j} (m'_i \ot l'_j) \ot 
(m''_i \ot l''_j) \in 
(M \ot_B L) \ot_A (M \ot_B L) .  \]
\end{lem}

\begin{proof}
By definition $\rho\smile\tau$ is the unique isomorphism fitting the diagram
\begin{equation}  \label{eqn:MO2}
\begin{tikzcd} [column sep = 6ex, row sep = 5ex] 
M\ot_BL
\ar[rr, "{\rho\smile \tau}", "{\simeq}"']
\ar[d, "{\rho\ot\tau}"', "{\simeq}"]
&&
\opn{Sq}_{C / A}(M\ot_BL)
\\
\opn{Sq}_{B / A}(M)\ot \opn{Sq}_{C / B}(L) 
\ar[d, "\opn{ev}_1"']
&&
\\
\opn{Hom}_{C \ot_B C}(C, (L\ot_B L) \ot_B \opn{Sq}_{B / A}(M))
\ar[rr, "{\opn{Hom}(\opn{id},\varepsilon)}", "{\simeq}"']
&&
\opn{Hom}_{C \ot_B C}(C,\opn{Sq}_{B / A}(M\ot_BL))
\ar[uu, "{\opn{adj}_2\circ\opn{Hom}(\opn{id},\opn{adj}_1)}"', "{\simeq}"]
\end{tikzcd} 
\end{equation}
in $\cat{M}(C)$.\ Where $\opn{ev}_1$ is the tensor-evaluation, $\opn{adj}_1$ is adjunction 
for the ring homomorphism $B\ot_A B\to C\ot_A C$, $\opn{adj}_2$ is adjunction for the ring 
homomorphism $C\ot_A C\to C\ot_B C$ and $\varepsilon$ denotes the two times application 
of tensor-evaluation as in the following diagram in $\cat{M}((C\ot_A C)\ot_{B\ot_AB}B)$:
\begin{equation} \label{eqn:MO7}
\begin{tikzcd} [column sep = 6ex, row sep = 5ex] 
(L\ot_B L) \ot_B \opn{Sq}_{B / A}(M)
\ar[rr, "{\dagger}", "{\simeq}"']
\ar[d, "{\varepsilon}"', "{\simeq}"]
&&
L\ot_B (L \ot_B \opn{Sq}_{B / A}(M))
\ar[d, "{\opn{id}_L\ot\opn{ev}_3}"', "{\simeq}"]
\\
\opn{Sq}_{B / A}(M\ot_B L) 
&&
L\ot_B\opn{Hom}_{B \ot_A B}(B, M\ot_A (M\ot_B L)).
\ar[ll, "{\opn{ev}_2}"']
\end{tikzcd} 
\end{equation}
Here $\dagger$ is the isomorphism described in Lemma \ref{lem:1431}, $\opn{ev}_2$ and $\opn{ev}_3$ are the tensor-evaluation morphisms.\ 
We note that $\varepsilon$ is an homomorphism $C\ot_B C$-modules, the fact that (\ref{eqn:MO7}) is in $\cat{M}((C\ot_A C)\ot_{B\ot_AB}B)$ 
is just because the target of $\dagger$ does not belong to $\cat{M}(C\ot_B C)$.

Taking $m\in M$, $l\in L$ we want to calculate $\bigl((\rho\smile\tau)(m\ot l)\bigr)(1_C)$.\ 
We point out that the next calculations in this proof are in $\cat{M}(C)$.\\ 
First we fix $c\in C$, going down in the diagram \ref{eqn:MO2} we have:
$$\bigl(\opn{ev}_1\circ(\rho\ot_B \tau)\bigr)(m\ot l)(c)=\tau(l)(c)\ot_B\rho(m)\in\opn{Hom}_{C \ot_B C}(C, (L\ot_B L) \ot_B \opn{Sq}_{B / A}(M))$$
Then going right we have:
\begin{equation} \label{eqn:MO4}
\begin{aligned}
& 
\opn{Hom}_{}(\opn{id}, \varepsilon)
\bigl( \tau(l)(c)\ot_B\rho(m) \bigr)
\\ & \quad 
= \opn{Hom}_{}(\opn{id}, \varepsilon)
\bigl( (c\ot_B1)\cd\tau(l)(1)\ot_B\rho(m) \bigr)
\\ & \quad 
= (c\ot_B1)\cd\varepsilon
\bigl( \tau(l)(1)\ot_B\rho(m) \bigr)
\\ & \quad 
= (c\ot_B1)\cd\varepsilon
\bigl( \tau(l)(1)\ot_B (\opn{id}_B\ot_A1)\cd\rho(m)(1)) \bigr)
\\ & \quad 
=(c\ot_B1)\cd(\opn{id}_B\ot_A1)\cd \sum\nolimits_{i, j} (m'_i \ot l'_j) \ot (m''_i \ot l''_j) . 
\end{aligned}
\end{equation}
Now going up we have 
\begin{equation} \label{eqn:1603MO}
\begin{aligned}
& 
\bigl(\opn{adj}_2\circ\opn{Hom}_{}(\opn{id}, \opn{adj}_1)
\circ
\opn{Hom}_{}(\opn{id}, \varepsilon)
\bigr)
\bigl( \tau(l)(c)\ot_B\rho(m) \bigr)
\\ & \quad 
= 
\opn{adj}_2\circ\opn{Hom}_{}(\opn{id}, \opn{adj}_1)
\bigl(
(c\ot_B1)\cd(\opn{id}_B\ot_A1)\cd \sum\nolimits_{i, j} (m'_i \ot l'_j) \ot (m''_i \ot l''_j)
\bigr)
\\ & \quad 
= 
(c\ot_A1)\cd\sum\nolimits_{i, j} (m'_i \ot l'_j) \ot (m''_i \ot l''_j).
\end{aligned}
\end{equation}
It implies that
$(\rho\smile\tau)(m\ot l)(1)=\sum\nolimits_{i, j} (m'_i \ot l'_j) \ot (m''_i \ot l''_j)$
and we are done.
\end{proof}


\begin{proof}[Proof of Theorem \tup{\ref{thm:1570}}]

The rigidifying isomorphism of $C$ is 
\[ \rho_{C / B}^{\mrm{esm}} = \opn{sec}_{C / B} : C \iso 
\opn{Sq}_{C / B}(C) =  \opn{Hom}_{C \ot_B C}(C, C \ot_B C) . \]
See Lemma \ref{lem:1596}. Therefore the rigidifying isomorphism of 
$N = M \ot_B C$ is 
\[ \si = \rho \cupprod \opn{sec}_{C / B} : N \iso 
\opn{Sq}_{C / A}(N) =  \opn{Hom}_{C \ot_A C}(C, N \ot_A N) . \]
The standard nondegenerate forward homomorphism
$\opn{q}^{}_{v, M} : M \to N$ is 
$\opn{q}^{}_{v, M}(m) = m \ot 1_C$
for an element $m \in M$. 
Our goal is to prove that the diagram 
\begin{equation} \label{eqn:1600}
\begin{tikzcd} [column sep = 6ex, row sep = 5ex] 
M
\ar[r, "{\rho}"]
\ar[d, "{\opn{q}^{}_{v, M}}"']
&
\opn{Sq}_{B / A}(M)
\ar[d, "{\opn{Sq}_{v / A}(\opn{q}^{}_{v, M})}"]
\\
N
\ar[r, "{\si}"]
&
\opn{Sq}_{C / A}(N)
\end{tikzcd} 
\end{equation}
in $\cat{M}(B)$ is commutative. 

Let's embed diagram (\ref{eqn:1600}) as the left rectangular region of the 
next bigger diagram. 
\begin{equation} \label{eqn:1601}
\begin{tikzcd} [column sep = 10ex, row sep = 6ex] 
M
\ar[r, "{\rho}"]
\ar[d, "{\opn{q}^{}_{}}"']
&
\opn{Sq}_{B / A}(M)
\ar[d, "{\opn{Sq}_{v / A}(\opn{q}^{}_{})}"]
\ar[r, "{\opn{Sq}_{B / A}(\opn{q}^{}_{})}"]
&
\opn{Sq}_{B / A}(N)
\ar[d, "{\mrm{adj}}"', "{\simeq}"]
\\
N
\ar[r, "{\rho \cupprod \opn{sec}_{}}"]
&
\opn{Sq}_{C / A}(N)
&
\opn{Hom}_{C \ot_A C}(C \ot_B C, N \ot_A N)
\ar[l, "{\opn{Hom}_{}(\opn{sec}, \opn{id})}"']
\end{tikzcd} 
\end{equation}
We use the abbreviations $\opn{q} := \opn{q}^{}_{v, M}$ and 
$\opn{sec} := \opn{sec}_{C / B}$. The isomorphism $\opn{adj}$ is adjunction for 
the ring homomorphism $B \ot_A B \to C \ot_A C$. 
The right rectangular subdiagram is 
commutative by the definition of $\opn{Sq}_{v / A}(\opn{q})$; see Definitions \ref{dfn:1282}
and \ref{dfn:1290} and Theorem \ref{thm:1285}. Take an arbitrary element $m \in M$, and 
expand the element $\rho(m)$ to a finite sum as in Lemma \ref{lem:1595}:
\[ \rho(m)(1_B) = \sum_i m'_i \ot m''_i \in M \ot_A M . \] 
Similarly we expand the element $\opn{sec}(1_C)$ into a finite sum:
\[ \opn{sec}(1_C)(1_C) = \sum_j c'_j \ot c''_j \in C \ot_B C . \]
Let's calculate where $m$ goes by the down-then-right path.  
According to Lemma \ref{lem:1595} we have
\begin{equation} \label{eqn:1604}
\begin{aligned}
&
\bigl( ((\rho \cupprod \opn{sec}) \circ \opn{q})(m) \bigr) (1_C) = 
\bigl( (\rho \cupprod \opn{sec})(m \ot 1_C) \bigr) (1_C) 
\\ & \quad
= \sum_{i, j} (m'_i \ot c''_j) \ot (m''_i \ot c''_j) \in N \ot_A N .
\end{aligned} 
\end{equation}
Going all the way right, then down, i.e.\ applying 
$\opn{adj} \circ \opn{Sq}_{B / A}(\opn{q}) \circ \rho$, 
the image of $m$ expands to  
\begin{equation} \label{eqn:1602}
\begin{aligned}
& 
\bigl( (\opn{adj} \circ \opn{Sq}_{B / A}(\opn{q}) \circ \rho)(m) \bigr) 
(1_C \ot 1_C) 
\\ & \quad 
= \sum_i (m'_i \ot 1_C) \ot (m''_i \ot 1_C) \in N \ot_A N .  
\end{aligned}
\end{equation}
Therefore, moving to the left by $\opn{Hom}_{}(\opn{sec}, \opn{id})$, we have 
this element
\begin{equation} \label{eqn:1603}
\begin{aligned}
& 
\opn{Hom}_{}(\opn{sec}, \opn{id})
\bigl( (\opn{adj} \circ \opn{Sq}_{B / A}(\opn{q}) \circ \rho)(m) \bigr)(1_C) 
\\ & \quad 
= \bigl((\opn{adj} \circ \opn{Sq}_{B / A}(\opn{q}) \circ \rho)(m) \bigr) 
(\opn{sec}(1_C))
\\ & \quad 
= \bigl((\opn{adj} \circ \opn{Sq}_{B / A}(\opn{q}) \circ \rho)(m) \bigr)
\bigl( \sum\nolimits_j c'_j \ot c'_j \bigr)
\\ & \quad 
= \bigl( \sum\nolimits_j c'_j \ot c''_j \bigr) \cd 
\bigl((\opn{adj} \circ \opn{Sq}_{B / A}(\opn{q}) \circ \rho)(m) \bigr)
(1_C \ot 1_C)
\\ & \quad 
= \bigl( \sum\nolimits_j c'_j \ot c''_j \bigr) \cd 
\bigl( \sum\nolimits_i (m'_i \ot 1_C) \ot (m''_i \ot 1_C) \bigr) 
\\ & \quad 
= \sum\nolimits_{i, j} (m'_i \ot c'_j) \ot (m''_i \ot c''_j) \in N \ot_A N . 
\end{aligned}
\end{equation}
We see that (\ref{eqn:1603}) and (\ref{eqn:1604}) are equal. 
\end{proof}

\end{document}